\documentclass[12pt]{amsart}
\usepackage{geometry} 
\usepackage{amscd}
\usepackage{am
ssymb}
\usepackage{amsmath}
\usepackage{amsthm}

\newtheorem{thm}{Theorem}[section]
\newtheorem{lemma}[thm]{Lemma}
\newtheorem{corollary}[thm]{Corollary}
\newtheorem{prop}[thm]{Proposition}

\newtheorem{conjecture}[thm]{Conjecture}
\theoremstyle{definition}

\newtheorem{rem}[thm]{Remark}

\newtheorem{defn}[thm]{Definition}

\makeatletter

\newcommand{\isom}{\overset{\sim}{\rightarrow}}           
\def\coloneqq{\mathrel{\mathop:}=}%

\title{Zeta morphisms for rank two universal deformations}
\author{Kentaro Nakamura}
\date{} 

\begin{document}

\maketitle
\pagestyle{plain}
\footnote{2010 Mathematical Subject Classification 11F80, 11F85, 11S25.
Keywords: Kato's Euler system, $p$-adic Langlands correpondence, completed cohomology.}
\begin{abstract}
In this article, we construct zeta morphisms for the universal deformations of odd absolutely irreducible two dimensional mod $p$ 
Galois representations satisfying some mild assumptions, and prove that our zeta morphisms interpolate Kato's zeta morphisms for Galois representations 
associated to Hecke eigen cusp newforms. The existence of such morphisms was predicted by Kato's generalized Iwasawa main conjecture. 
Based on Kato's original construction, we construct our zeta morphisms using many deep results in 
the theory of $p$-adic (local and global) Langlands correspondence for $\mathrm{GL}_{2/\mathbb{Q}}$. As an application of our zeta morphisms and the recent article \cite{KLP19}, 
we prove a theorem which roughly states that, under some $\mu=0$ assumption, Iwasawa main conjecture without $p$-adic $L$-function for $f$ holds if this conjecture holds for one $g$ which is congruent to $f$.


\end{abstract}
\setcounter{tocdepth}{2}
\tableofcontents

\section{Introduction}

\subsection{Zeta morphisms for rank two universal deformations}
Let $p$ be a prime. We fix embeddings $\iota_{\infty} : \overline{\mathbb{Q}}\hookrightarrow \mathbb{C}$ and 
$\iota_p : \overline{\mathbb{Q}}\hookrightarrow \overline{\mathbb{Q}}_p$. 

In his celebrated article \cite{Ka04} on Iwasawa main conjecture for the motive associated to a normalized Hecke eigen cusp new form $f
=\sum_{n=1}^{\infty}a_n(f)q^n$ of level $N_f\geqq 1$, Kato defined a non trivial  Euler system relating with the $L$-function associated to $f$. 
Let $\rho_f : G_{\mathbb{Q}}\rightarrow \mathrm{GL}_2(\mathcal{O})$ be the $p$-adic representation of $G_{\mathbb{Q}}:=\mathrm{Gal}(\overline{\mathbb{Q}}/\mathbb{Q})$ associated to $f$, and $\rho_f^*(1)$ be its $\mathcal{O}$-linear Tate dual, where $\mathcal{O}$ is the integer ring of a finite extension $E\subset \overline{\mathbb{Q}}_p$ of $\mathbb{Q}_p$. Then, Kato defined a system of Galois cohomology classes $c_n\in H^1(\mathbb{Z}[1/N_fp, \zeta_n], \rho_f^*(1))$ for infinitely many 
integers $n\geqq 1$ satisfying the so-called Euler system relation. He proved the non triviality of his Euler system by relating it with the special values of 
the $L$-function associated to $f$.
By the general theory of Euler system developed by 
Kato, Perrin-Riou, and Rubin, this non triviality implies that the Iwasawa cohomology 
$$H^i_{\mathrm{Iw}}(\mathbb{Z}[1/N_fp], \rho_f^*(1)):=\varprojlim_{m\geqq 1}H^1(\mathbb{Z}[1/N_fp, \zeta_{p^m}], \rho_f^*(1))$$
is torsion (resp. generically free of rank one)  if $i=2$ (resp. $i=1$) as a module over the Iwasawa algebra $\Lambda_\mathcal{O}:=\mathcal{O}[[\Gamma]]$ of 
$\Gamma:=\mathrm{Gal}(\mathbb{Q}(\zeta_{p^{\infty}})/\mathbb{Q})$. By the Euler system relation, $(c_{p^m})_{m\geqq 1}$ is an element in $H^1_{\mathrm{Iw}}(\mathbb{Z}[1/N_fp], \rho_f^*(1))$, but we remark that it (and $\{c_n\}_n$) depends on many 
auxiliary choices appearing in Kato's construction (see Example 13.3 of \cite{Ka04}). Removing its auxiliary factors suitably, Kato defined a canonical $\mathcal{O}$-linear map 
which we call the zeta morphism for $\rho_f$
$$\bold{z}(f) : \rho_f^*\rightarrow H^1_{\mathrm{Iw}}(\mathbb{Z}[1/N_fp], \rho_f^*(1))[1/p]$$ 
in Theorem 12.5 \cite{Ka04} depending only on $f$, and showed that its image is contained in $H^1_{\mathrm{Iw}}(\mathbb{Z}[1/N_fp], \rho_f^*(1))$ 
when $\rho_f$ is residually absolutely irreducible. This map is the most important object for Iwasawa main conjecture for $f$ by the following two properties. 
First, this map is related with the collect special values of the $L$-function (precisely, the $L$-fuction removing its $p$-th Euler factor) associated to $f$ via Bloch-Kato's dual exponential map. When $f$ is of finite slope at $p$, this property 
enables us to construct a $p$-adic $L$-function $\mathcal{L}(f)$ associated to $f$ from the map $\bold{z}(f)$ via Perrin-Riou's general machine constructing 
$p$-adic $L$-functions. Second, Kato showed that if $p$ is odd and $\mathrm{Im}(\rho_f)$ contains a conjugate of $\mathrm{SL}_2(\mathbb{Z}_p)$, then this map satisfies the following inclusion 
\begin{equation}\label{conj}
\mathrm{Char}_{\Lambda_{\mathcal{O}}}\left(H^1_{\mathrm{Iw}}(\mathbb{Z}[1/N_fp], \rho_f^*(1))/\Lambda_{\mathcal{O}}\cdot \mathrm{Im}(\bold{z}(f))\right)\subset 
\mathrm{Char}_{\Lambda_{\mathcal{O}}}\left(H^2_{\mathrm{Iw}}(\mathbb{Z}[1/p], \rho_f^*(1))\right)
\end{equation}
of characteristic ideals of finite generated torsion $\Lambda_{\mathcal{O}}$-modules as a consequence of the general theory of Euler system, and conjectured that this inclusion should be an equality, where we set 
$$H^2_{\mathrm{Iw}}(\mathbb{Z}[1/p], \rho_f^*(1)):=\mathrm{Ker}(H^2_{\mathrm{Iw}}(\mathbb{Z}[1/N_fp], \rho_f^*(1))
\rightarrow \oplus_{l|N_f, l\not=p}H^2_{\mathrm{Iw}}(\mathbb{Q}_l, \rho_f^*(1))).$$ When $f$ is ordinary at $p$, Kato proved the inclusion 
\begin{equation}(\mathcal{L}(f))\subset \mathrm{Char}_{\Lambda_{\mathcal{O}}}(\text{Pontryagin dual of the cyclotomic Selmer group of } \rho_f)
\end{equation}
as a consequence of these two properties of the map $\bold{z}(f)$. In fact, it is known that the existence of the map $\bold{z}(f)$ with these two properties is essentially equivalent to the validity of Iwasawa main conjecture for $f$ when we can formulate it (e.g. when $f$ is ordinary at $p$, or more generally,  $f$ is of finite slope at $p$). We remark that it is not known (up to now) how to formulate Iwasawa main conjecture with $p$-adic $L$-function for general $f$ (e.g. we don't know the existence of ``good" $p$-adic $L$-functions and Selmer groups when $f$ is not of finite slope), but one has the map $\bold{z}(f)$ and 
one can formulate the conjecture on the equality of the inclusion (\ref{conj}) for arbitrary $f$.

In \cite{Ka93a}, Kato furthermore made a remarkable conjecture called the generalized Iwasawa main conjecture unifying both Iwasawa main conjecture and Bloch-Kato's Tamagawa number conjecture. 
This conjecture predicts that such zeta morphisms exist for arbitrary geometric $p$-adic 
representations of $G_{\mathbb{Q}}$, more importantly for our purpose, for arbitrary families of 
$p$-adic representations of $G_{\mathbb{Q}}$ which are unramified outside finite sets of primes. In particular, this conjecture predicts that 
there exists a zeta morphism for any family of rank two odd $p$-adic representations of $G_{\mathbb{Q}}$ which are unramified outside a finite set of primes, and it interpolates
Kato's zeta morphisms $\bold{z}(f)$ for arbitrary $\rho_f$ containing in the family. The main result of this article concerns this 
problem. 

Let $\Sigma$ be a finite set of primes containing $p$. 
Let $\overline{\rho} : G_{\mathbb{Q}} \rightarrow \mathrm{GL}_2(\mathbb{F})$ be an odd absolutely irreducible representation over a 
finite field $\mathbb{F}$ of characteristic $p$ which is unramified outside $\Sigma$. We consider deformations of $\overline{\rho}$ which are unramified outside $\Sigma$ with no condition at the primes $l\in\Sigma$. 
Since $\overline{\rho}$ is absolutely irreducible, there exists the universal deformation ring $R_{\overline{\rho}, \Sigma}$, and the universal 
deformation $\rho^u : G_{\mathbb{Q}}\rightarrow \mathrm{GL}_2(R_{\overline{\rho}, \Sigma})$ (defined up to isomorphism) for this deformation problem. 
Let $\varepsilon : G_{\mathbb{Q}}\rightarrow \mathbb{Z}_p^{\times}$ be the $p$-adic cyclotomic character. For each integer $n\geqq 1$, we write 
$\Sigma_n:=\Sigma\cup\mathrm{prime}(n)$ to denote the union of $\Sigma$ and the set of the prime divisors $\mathrm{prime}(n)$ of $n$. 
For each prime $l\not\in \Sigma$, we set $P_l(u):=\mathrm{det}(1-\mathrm{Frob}_l \cdot u\mid \rho^u)\in R_{\overline{\rho}, \Sigma}[u]$. For any $a\in \mathbb{Z}$ such that $(a, np)=1$, we write $\sigma_a\in \mathrm{Gal}(\mathbb{Q}(\zeta_{np^{\infty}})/\mathbb{Q})$ 
to denote the element corresponding to 
$a\in \varprojlim_{k\geqq 1}(\mathbb{Z}/np^k\mathbb{Z})^{\times}$ via the cyclotomic character. 
The main result of this article is the following (see Corollary \ref{3.11} and Theorem \ref{3.2} for the precise statement). 
\begin{thm}\label{0.1}Assume the following $:$ 
\begin{itemize}
\item[(1)]$p\geqq 5$.
\item[(2)]$\rho|_{G_{\mathbb{Q}(\zeta_p)}}$ is absolutely irreducible. 
\item[(3)]$\mathrm{End}_{\mathbb{F}[G_{\mathbb{Q}_p}]}(\overline{\rho}_p)=\mathbb{F}$ for $\overline{\rho}_p:=\overline{\rho}|_{G_{\mathbb{Q}_p}}$.
\item[(4)]$\overline{\rho}_p$ is not of the form $\begin{pmatrix}1& *\\ 0 & \overline{\varepsilon}^{\pm 1}\end{pmatrix}
\otimes \chi\,  $ for any character $\chi : G_{\mathbb{Q}_p}\rightarrow \mathbb{F}^{\times}$.
\item[(5)]$\overline{\rho}_p$ is not of the form $\begin{pmatrix}1& *\\ 0 & 1\end{pmatrix}
\otimes \chi\,  $  for any character $\chi : G_{\mathbb{Q}_p}\rightarrow \mathbb{F}^{\times}$.
\end{itemize}
Then, there exists a family of $R_{\overline{\rho}, \Sigma}$-linear maps
$$\bold{z}_{\Sigma, \,n}(\rho^u) : (\rho^u)^*\rightarrow H^1_{\mathrm{Iw}}(\mathbb{Z}[1/\Sigma_n, \zeta_n], 
(\rho^u)^*(1))$$
for each $n\geqq 1$ such that $(n, \Sigma)=1$ satisfying the following $:$ 
\begin{itemize}
 \item[(i)]One has 
 $$\bold{z}_{\Sigma, n}(\rho^u)(\tau(v))=\sigma_{-1}(\bold{z}_{\Sigma, n}(\rho^u)(v))$$
 for any $v\in  (\rho^u)^*$, where $\tau\in \mathrm{Gal}(\mathbb{C}/\mathbb{R})\subseteq G_{\mathbb{Q}}$ is 
 the complex conjugation. 
\item[(ii)]For each integer $n\geqq 1$ and prime $l$ such that $(nl, \Sigma)=1$, one has 
$$\mathrm{Cor}\circ \bold{z}_{\Sigma,\,nl}(\rho^u)=\begin{cases}\bold{z}_{\Sigma, \,n}(\rho^u) & \text{ if } l|n \\
P_l(\sigma_l^{-1})\cdot \bold{z}_{\Sigma, \,n}(\rho^u) & \text{ if } (l, n)=1
\end{cases}$$
for the corestriction map
$$\mathrm{Cor} : H^1_{\mathrm{Iw}}(\mathbb{Z}[1/\Sigma_{nl}, \zeta_{nl}], 
(\rho^u)^*(1))\rightarrow H^1_{\mathrm{Iw}}(\mathbb{Z}[1/\Sigma_{nl}, \zeta_n], 
(\rho^u)^*(1)).$$
\item[(iii)]For every point $x_f\in \mathrm{Spec}(R_{\overline{\rho}, \Sigma})(\overline{\mathbb{Q}}_p)$ corresponding to a normalized Hecke eigen cusp newform
$f$, the specialization $\bold{z}_{\Sigma, \,1}(f)$ of the map $\bold{z}_{\Sigma, \,1}(\rho^u)$ for $n=1$ at $x_f$ satisfies the equality
$$\bold{z}_{\Sigma, \,1}(f)=\prod_{l\in \Sigma\setminus\{p\}}P_{f, l}(\sigma_l^{-1})\cdot \bold{z}(f).$$
 Here, for every prime $l\in \Sigma\setminus \{p\}$, we set $I_l\subset G_{\mathbb{Q}_l}$ the inertia subgroup, and 
 $P_{f, l}(u):=\mathrm{det}(1-\mathrm{Frob}_l \cdot u\mid \rho_f^{I_l})$.

 \end{itemize}
\end{thm}
\begin{rem}
The condition (i) means that the map $\bold{z}_{\Sigma, n}(\rho^u)$ induces a $R_{\overline{\rho}, \Sigma}\widehat{\otimes}_{\mathbb{Z}_p}\Lambda_n$-linear map 
$$\bold{z}_{\Sigma, n}(\rho^u) :  (\rho^u)^*\widehat{\otimes}_{\mathbb{Z}_p[\{\pm 1\}]}\Lambda_n\rightarrow H^1_{\mathrm{Iw}}(\mathbb{Z}[1/\Sigma_n, \zeta_n], 
(\rho^u)^*(1)), $$ where $-1\in \{\pm 1\}$ acts on $(\rho^u)^*$ via $\tau$, on $\Lambda_n$ by $\sigma_{-1}$. Such a property is demanded for zeta morphisms in Kato's generalized Iwasawa main conjecture \cite{Ka93a}. 
\end{rem}

We give some remarks on the assumptions in the theorem. First, one can remove the assumptions (2) and (5)  if we replace the ring $R_{\overline{\rho}, \Sigma}$ with the universal promodular deformation ring $\mathbb{T}_{\overline{\rho}, \Sigma}$ which is defined as 
a limit of Hecke algebras acting on the \'etale cohomology of modular curves of some infinitely many different levels. We need other assumptions (1), (3), (4) to use some theorems in the theory of $p$-adic (local and global) Langlands correspondence for $\mathrm{GL}_{2/\mathbb{Q}}$, on which our construction heavily depend. For example, we need the assumption (4) for $\overline{\varepsilon}^{\pm 1}= \overline{\varepsilon}$ to apply Colmez' and Emerton's results in \cite{Co10} and \cite{Em}, 
and the assumptions (1), (3) and (4) for $\overline{\varepsilon}^{\pm 1}= \overline{\varepsilon}^{-1}$ to apply Pa\v{s}k\={u}nas' results in 
\cite{Pas13} and \cite{Pas15}. Therefore, it seems to be possible to remove these assumptions  if some modified versions of their theorems are obtained for general case. On the other hand, the assumption that $\overline{\rho}$ is absolutely irreducible seems to be crucial for our construction. For example, Emerton's result \cite{Em} which is fundamental for our method 
is only known for this case. Moreover, some results of Kato in \cite{Ka04} are valid only for the residually absolutely irreducible case. For example, 
the $\Lambda$-freeness of the Iwasawa cohomology $H^1_{\mathrm{Iw}}(\mathbb{Z}[1/N_fp], \rho_f^*(1))$ seems to be important for our method, 
 but such a freeness is not known for residually reducible case. 

We next give some remarks on the relation of our result with the previous known results on the extension of Kato's Euler systems or zeta morphisms to families of Galois representations. 
For Hida families, Ochiai defined such a zeta morphism which he calls the optimization of the two-variable Beilinson-Kato element in Theorem 6.11 of \cite{Oc06} by glueing Kato's zeta morphisms. Fukaya-Kato \cite{FK06} also defined it by a different method, which is important for our work. They first defined a zeta morphism for 
the \'etale cohomology $H^1_{\text{\'et}}(Y_1(N)_{\overline{\mathbb{Q}}}, \mathbb{Z}_p(1))$ of modular curves $Y_1(N)_{\overline{\mathbb{Q}}}$ over $\overline{\mathbb{Q}}$. Precisely, they defined a map
$$z^{\mathrm{Iw}}_{1, N, n} : H^1_{\text{\'et}}(Y_1(N)_{\overline{\mathbb{Q}}}, \mathbb{Z}_p)\rightarrow 
H^1_{\mathrm{Iw}}(\mathbb{Z}[1/nNp, \zeta_n], H^1_{\text{\'et}}(Y_1(N)_{\overline{\mathbb{Q}}}, \mathbb{Z}_p(1))\otimes_{\Lambda_n}\mathrm{Frac}(\Lambda_n)$$ for every $n\geqq 1$ such that $(n, Np)=1$ satisfying an Euler system relation. Here, we set $\Lambda_n:=\mathbb{Z}_p[[\mathrm{Gal}(\mathbb{Q}(\zeta_{np^{\infty}})/\mathbb{Q})]]$. They also showed that this map is equivariant for the Hecke action. 
Taking its ordinary part, and taking its limit with respect to the levels $Np^k$ ($k\geqq 1$), they defined their zeta morphisms for Hida families. Furthermore, there are some works on E, uler systems or zeta morphisms for the Galois representations over Coleman-Mazur eigencurves, e.g by Hansen \cite{Ha15}, Ochiai \cite{Oc18}, and Wang \cite{Wa13}. In particular, Hansen conjectured that there exist zeta morphisms for the Galois representations over Coleman-Mazur eigencurves in Conjecture 1.3.1 of \cite{Ha15}. As an application of our main theorem, it seems to be 
possible  to prove his conjecture simply as base changes of our zeta morphism from the universal deformation spaces to Coleman-Mazur eigencurves. 
Finally, there are some works on Euler systems or zeta morphisms 
for the universal deformations, 
e.g. by Fouquet \cite{Fo16}, and Wang. 
However,
the author can't find the precise relationship between their constructions with Kato's 
zeta morphisms $\bold{z}(f)$ containing in the family. Whereas our zeta morphism $\bold{z}_{\Sigma, \,n}(\rho^u)$ is directly related with each $\bold{z}(f)$ by its base change to the point 
corresponding to $f$, and 
its existence is compatible with Kato's generalized Iwasawa main conjecture for the universal deformations. 
We remark that they used Emerton's theory of completed cohomology of modular curves \cite{Em}, which is also a foundation for our work. 

There seem to be several applications of our main theorem to some problems related with the generalized Iwasawa main conjecture. 
In this article, we give an application to Iwasawa main conjecture, which can be regarded as a generalization of 
the results of e.g. \cite{GV00}, \cite{EPW06}, and \cite{KLP19}. We consider the following conjecture, which we call Kato's main conjecture 
(Conjecture 12.10 of \cite{Ka04}).

\begin{conjecture}\label{conj01}Assume $p$ is odd and $\overline{\rho}_f : G_{\mathbb{Q}}\rightarrow \mathrm{GL}_2(\mathbb{F})$ is absolutely irreducible. 
Then the equality 
$$\mathrm{Char}_{\Lambda_{\mathcal{O}}}\left(H^1_{\mathrm{Iw}}(\mathbb{Z}[1/N_fp], \rho_f^*(1))/\Lambda_{\mathcal{O}}\cdot \mathrm{Im}(\bold{z}(f))\right)=
\mathrm{Char}_{\Lambda_{\mathcal{O}}}\left(H^2_{\mathrm{Iw}}(\mathbb{Z}[1/p], \rho_f^*(1))\right)$$
holds.
\end{conjecture}
Our theorem roughly state that this conjecture is true for  $f$ if this conjecture is true for one $g$ which is congruent to $f$. 
In this introduction, we say that two normalized Hecke eigen cusp newforms $f=\sum^{\infty}_{n=1}a_nq^n$ and $g=\sum^{\infty}_{n=1}b_nq^n$ are congruent 
 if $a_l\equiv b_l\bmod \varpi$ for all but finitely many primes $l$ and a finite extension $E\subset \overline{\mathbb{Q}}_p$ with a uniformizer $\varpi$. 
 As a direct consequence of Theorem \ref{0.1}, we can immediately obtain congruences of zeta morphisms
 $\prod_{l\in \Sigma_0}P_{f,l}(\sigma_l^{-1})\cdot \bold{z}(f)$ between congruent Hecke eigen cusp newforms 
 (see Corollary \ref{cong} for the precise and more general statement).
 As a consequence of this congruence and the arguments of the recent article \cite{KLP19}, we obtain the following theorem 
 (see \S 5.3 for the precise and more general statement).
 
 The author would like to Chan-Ho Kim for sending him this article, and discussing applications of our main theorem to Iwasawa main conjecture.

 \begin{thm}\label{NNN}
 Assume all the assumptions $(1), (3), (4)$ in Theorem $\ref{0.1}$ are satisfied for $\overline{\rho}=\overline{\rho}_f$, 
 If there exist $g_1, g_2, g_3$ which are congruent to $f$ satisfying :
 \begin{itemize}
 \item[(1)]There exists an element $\tau\in \mathrm{Gal}(\overline{\mathbb{Q}}/\mathbb{Q}(\zeta_{p^{\infty}}))$ such that 
 $\rho_{g_1}/(\tau-1)\rho_{g_1}$ is a free $\mathcal{O}$-module of rank one.
 \item[(2)]The $\mu$-invariant of $H^1_{\mathrm{Iw}}(\mathbb{Z}[1/N_{g_2}p], \rho_{g_2}^*(1))/\Lambda_{\mathcal{O}}\cdot \mathrm{Im}(\bold{z}(g_2))$ is zero. 
 \item[(3)]Conjecture $\ref{conj01}$ holds for $g_3$. 
 \end{itemize}
 Then, the $\mu$-invariant of $H^1_{\mathrm{Iw}}(\mathbb{Z}[1/N_{f}p], \rho_{f}^*(1))/\Lambda_{\mathcal{O}}\cdot \mathrm{Im}(\bold{z}(f))$ is zero, and 
 Conjecture $\ref{conj01}$ holds for $f$. 
 
 \end{thm}

  In the pioneering works \cite{GV00} and \cite{EPW06}, they consider such problems for congruent modular forms which are ordinary at $p$. 
 They consider the usual Iwasawa main conjecture on the equality of the $p$-adic $L$-functions with Selmer groups. 
 In \cite{GV00}, they studied such problems for congruent elliptic curves with good ordinary reduction at $p$. In \cite{EPW06}, they studied such problems 
 for congruent modular forms in Hida families. In ordinary case, it is known that Iwasawa main conjecture 
 with $p$-adic $L$-function is equivalent to Kato's main conjecture, and the $\mu$-invariant of the $p$-adic $L$-function of $f$ is larger or equal to 
 that of the quotient $$H^1_{\mathrm{Iw}}(\mathbb{Z}[1/N_{f}p], \rho_{f}^*(1))/\Lambda_{\mathcal{O}}\cdot \mathrm{Im}(\bold{z}(f)).$$ Hence, our Theorem 
 \ref{NNN} can be regarded as a generalization of their results.  In the recent work \cite{KLP19}, they consider Kato's main conjecture, and studies congruences of 
 Kato's zeta morphisms $\bold{z}(f)$ for congruent modular forms with levels not dividing $p$, and a fixed weight $0\leqq k\leqq p-1$. 
 Our Theorem \ref{NNN} can treat all the congruent modular forms with arbitrary levels and weights. Hence, our Theorem 
 \ref{NNN} can be also regarded as a generalization of their results. However, the proof of Theorem \ref{NNN} is essentially same 
 as the proof of their main theorem in \cite{KLP19}. Our contribution is just generalizing the congruences of zeta morphisms $\bold{z}(f)$ for 
 arbitrary congruent modular forms.



In the remaining, we give comments on some other possible applications 
of our main theorem, which we hope to study in future works.  First, we remark in the previous paragraph 
that one can apply the theorem to Hansen's conjecture on the existence of the zeta morphisms for the Galois representations over Coleman-Mazur eigencurves. 
Then, it seems to be possible to obtain two variable $p$-adic $L$-functions over Coleman-Mazur eigencurves using some rigid analytic family version 
of Perrin-Riou theory constructing $p$-adic distributions on $\mathbb{Z}_p^{\times}$ from Iwawasa cohomology classes. Such 
two variable
$p$-adic $L$-functions  were previously defined by Bella\"iche \cite{Be12}, Emerton \cite{Em06a}, Panchishkin \cite{Pan03}, and Stevens \cite{St} via different methods. It seems to be an interesting problem to study the precise relationship between our $p$-adic $L$-fucntion and others. 
Next, Kato \cite{Ka93b} conjectured that zeta morphisms satisfy the functional equation of a similar form as that of the associated $L$-functions (Kato's  global $\varepsilon$-conjacture). 
In our previous works \cite{Na17b} (resp.  \cite{Na17a}), we defined the $p$-th local factors
called the local $\varepsilon$-isomorphisms appearing in the functional equation of zeta morphisms for rank two universal deformations (resp. 
the Galois representations over Coleman-Mazur eigencurves). Using these local factors, we hope to obtain the conjectural functional equation
of our zeta morphisms for rank two universal deformations and the Galois representations over Coleman-Mazur eigencurves. In \cite{Na17b}, we proved such functional equation for the zeta morphism $\bold{z}(f)$. In the proof, we finally reduced it to the functional equation of the $L$-function of $f$. 
 Using density of modular points, it seems to be possible to obtain 
the functional equation for our zeta morphisms $\bold{z}_{\Sigma, n}(\rho^{\mathfrak{m}})$. However, we want to search a more direct proof 
of such functional equation by some geometric arguments, which was the first motivation of this article.


\subsection{Outline of the proof of the main theorem} 
In \cite{Ka04}, Kato constructed his Euler systems for Hecke eigen new forms $f$ from a system of elements 
${}_{c,d}z_{N}\in K_2(Y(N))$ called Beilinson elements  in the $K$-groups of the modular curves $Y(N)$ for every integers $N\geqq 3$ satisfying an Euler system relation. We remark that these elements depend on a choice of integers $c, d\geqq 2$ such that $(N, 6cd)=1$. For every $N_0\geqq1$ which are prime to $p$, 
the Euler system relation gives us an element $${}_{c,d}z_{N_0p^{\infty}}:=({}_{c,d}z_{N_0p^k})_{k\geqq 1} \in \varprojlim_{k\geqq 1}K_2(Y(N_0p^k))$$ 
in the projective limit with respect to the norm maps $$K_2(Y(N_0p^{k+1}))\rightarrow K_2(Y(N_0p^k)).$$ 
As the image of the element
${}_{c,d}z_{N_0p^{\infty}}$ by the composite of the projective limit of Chern class maps and the edge maps in Hochschild-Serre spectral sequences, one obtains an
element ${}_{c,d}z^{(p)}_{N_0p^{\infty}}$ which we call Kato's element in the Galois cohomology group 
$$H^1(\mathbb{Z}[1/ N_0p], \widetilde{H}^{BM}_1(K^p(N_0))(1))$$
of the projective limit $$\widetilde{H}^{BM}_1(K^p(N_0))(1):=\varprojlim_{k\geqq 1}H^1_{\text{\'et}}(Y(N_0p^k)_{\overline{\mathbb{Q}}}, \mathbb{Z}_p(2))$$ with respect to the 
corestriction maps $$H^1_{\text{\'et}}(Y(N_0p^{k+1})_{\overline{\mathbb{Q}}}, \mathbb{Z}_p(2))\rightarrow H^1_{\text{\'et}}(Y(N_0p^k)_{\overline{\mathbb{Q}}}, \mathbb{Z}_p(2))$$ for all $k\geqq1$. 

Since one has a canonical isomorphism $$H^1_{\text{\'et}}(Y(N)_{\overline{\mathbb{Q}}}, \mathbb{Z}_p(1))\isom H_{1, \text{\'et}}(X(N), \{\mathrm{cusps}\}, \mathbb{Z}_p)$$
by Poincar\'e duality, such a projective limit is called the completed (relative, or Borel-Moore) homology, and it is (precisely, its dual, i.e. the completed cohomogoly is) studied by Emerton in his work \cite{Em} on the local-global compatibility of $p$-adic Langlands correspondence for $\mathrm{GL}_{2/\mathbb{Q}}$. 
By (the dual version of) his theorem on the refined local-global compatibility, the inductive limit 
$$\widetilde{H}^{BM}_{1, \Sigma}:=\varinjlim_{N_0}\widetilde{H}^{BM}_1(K^p(N_0))$$ 
with respect to the restriction maps $$\widetilde{H}^{BM}_1(K^p(N_0))\rightarrow \widetilde{H}^{BM}_1(K^p(N'_0))$$
for all $N_0|N'_0$ such that $(N'_0, p)=1$ and $\mathrm{prime}(N_0p)=\mathrm{prime}(N'_0p)=\Sigma$
as a Hecke module
(with respect to the Hecke actions at the primes $l\not\in \Sigma$) 
with actions of $G_{\mathbb{Q}}$ and $G_l:=\mathrm{GL}_2(\mathbb{Q}_l)$ for $l\in \Sigma$ can be described 
by using the universal (promodular) deformation $\rho^u$ and the representations $\pi_l^u$ of $G_l$ corresponding to $\rho^u|_{G_{\mathbb{Q}_l}}$ via 
the family version of the local Langlands correspondence by Colmez \cite{Co10} for $l=p$ and Emerton-Helm \cite{EH14} for $l\not=p$. 
Precisely, 
 there exists a 
$R_{\overline{\rho}, \Sigma}[G_{\mathbb{Q}}\times G_{\Sigma}]$-linear topological 
isomorphism $$\widetilde{H}^{BM}_{1, \overline{\rho}, \Sigma}\isom  (\rho^u)^*\widehat{\otimes}_{R_{\overline{\rho}, \Sigma}}
(\pi_p^u)^*\otimes_{R_{\overline{\rho}, \Sigma}} \widetilde{\pi}_{\Sigma_0}^u$$
 for $G_{\Sigma}:=\prod_{l\in \Sigma}G_l$ under some mild assumptions, where $\widetilde{H}^{BM}_{1, \overline{\rho}, \Sigma}$ is the localization of $\widetilde{H}^{BM}_{1, \Sigma}$ at the maximal ideal $\mathfrak{m}$ of the Hecke algebra corresponding to the fixed $\overline{\rho}$, $(\rho^u)^*$ and $(\pi_p^u)^*$ are the $R_{\overline{\rho}, \Sigma}$-linear duals of $\rho^u$ and $\pi_p^u$ respectively, and $\widetilde{\pi}_{\Sigma_0}^u$ is the $R_{\overline{\rho}, \Sigma}$-linear 
smooth contragradient of $\otimes_{l\in \Sigma_0}\pi^u_l$ for $\Sigma_0:=\Sigma\setminus\{p\}$


Therefore, it seems to be natural 
to use the completed homology or the completed cohomology to construct Euler systems or zeta morphisms for 
the universal deformation $\rho^u$. 
In fact, there are some preceding works e.g. by Fouquet \cite{Fo16} and Wang, in which they studied Euler systems or a kind of zeta morphisms for the universal deformation $\rho^u$ using the elements like ${}_{c,d}z^{(p)}_{p^{\infty}}$ and the completed cohomology. However, I could not find (at least in the literatures) a precise relation between their constructions with Kato's zeta morphisms $\bold{z}(f)$ for arbitrary $f$ containing in the family. 
In particular. I could not find how to obtain a system of conjectural morphisms
$$\bold{z}_{\Sigma, n}(\rho^u) :  (\rho^u)^*\rightarrow H^1_{\mathrm{Iw}}(\mathbb{Z}[1/ \Sigma_n, \zeta_n],  (\rho^u)^*(1))$$
for every $n\geqq1$ such that $(n,\Sigma)=1$ satisfying the precise interpolation property as in the condition (iii) of our main theorem. 

Our idea to overcome this problem 
is a modification of Fukaya-Kato's construction \cite{FK06} of zeta morphisms for Hida families using 
the theory of Pa\v{s}k\={u}nas \cite{Pas13}, \cite{Pas15} (resp. Emerton-Helm \cite{EH14} and Helm \cite{He16}) 
on the representation $\pi^u_l$ for $l=p$ (resp. $l\not=p$) appearing in the description of 
$\widetilde{H}^{BM}_{1, \overline{\rho}, \Sigma}$ by Emerton \cite{Em}. 

To be precise, we first generalize Fukaya-Kato's zeta morphisms
$$z^{\mathrm{Iw}}_{1, N, n} : H^1_{\text{\'et}}(Y_1(N)_{\overline{\mathbb{Q}}}, \mathbb{Z}_p(1))\rightarrow 
H^1_{\mathrm{Iw}}(\mathbb{Z}[1/nNp, \zeta_n], H^1_{\text{\'et}}(Y_1(N)_{\overline{\mathbb{Q}}}, \mathbb{Z}_p(2)))\otimes_{\Lambda_n}\mathrm{Frac}(\Lambda_n)$$ to the modular curve $Y(N)_{\overline{\mathbb{Q}}}$ with the full level structure. Precisely, we could not follow some arguments
in the proof of the construction of $z^{\mathrm{Iw}}_{1, N, n}$ in \cite{FK}, hence we instead define a Hecke equivariant map 
$$z^{\mathrm{Iw}}_{N, n} : H^1_{\text{\'et}}(X(N)_{\overline{\mathbb{Q}}}, \mathbb{Z}_p(1))\rightarrow 
H^1_{\mathrm{Iw}}(\mathbb{Z}[1/nNp, \zeta_n], H^1_{\text{\'et}}(X(N)_{\overline{\mathbb{Q}}}, \mathbb{Z}_p(2)))\otimes_{\Lambda_n}\mathrm{Frac}(\Lambda_n)$$
replacing $Y(N)$ to the compactified modular curve $X(N)$. 
To proceed our construction, we need to show the integrality of this map, i.e. we need to show that  the image of this map is contained in the 
$\Lambda_n$-lattice $$H^1_{\mathrm{Iw}}(\mathbb{Z}[1/nNp, \zeta_n], H^1_{\text{\'et}}(X(N)_{\overline{\mathbb{Q}}}, \mathbb{Z}_p(2))).$$ 
This problem is a little bit subtle problem, and complicates our construction. In fact, we don't prove this claim itself, and 
 we don't know whether the claim is true of not. Instead, we show the integrality after taking the 
localization $z^{\mathrm{Iw}}_{N, n, \overline{\rho}}$ of the map  $z^{\mathrm{Iw}}_{N, n}$ at the maximal ideal $\mathfrak{m}$ of Hecke algebra corresponding to $\overline{\rho}$, to obtain 
a $R_{\overline{\rho}, \Sigma}$-linear map 
$$z^{\mathrm{Iw}}_{N, n, \overline{\rho}} : H^1_{\text{\'et}}(Y(N)_{\overline{\mathbb{Q}}}, \mathbb{Z}_p(1))_{\overline{\rho}}\rightarrow 
H^1_{\mathrm{Iw}}(\mathbb{Z}[1/\Sigma_n, \zeta_n], H^1_{\text{\'et}}(Y(N)_{\overline{\mathbb{Q}}}, \mathbb{Z}_p(2))_{\overline{\rho}})$$
for every sufficiently large $N$ such that $\mathrm{prime}(N)=\Sigma$. Here, we remark that one has 
$H^1_{\text{\'et}}(X(N)_{\overline{\mathbb{Q}}}, \mathbb{Z}_p(1))_{\overline{\rho}}\isom H^1_{\text{\'et}}(Y(N)_{\overline{\mathbb{Q}}}, \mathbb{Z}_p(1))_{\overline{\rho}}$ since we assume that $\overline{\rho}$ is absolutely irreducible. 
In our proof of this integrality, this assumption is essential, and 
we don't know whether the same claim is true or not when $\overline{\rho}$ is absolutely reducible case. 
By this integrality, one obtains a $R_{\overline{\rho}, \Sigma}$-linear map 
$$z^{\mathrm{Iw}}_{N_0p^{\infty}, n, \overline{\rho}} : \widetilde{H}^{BM}_1(K^p(N_0))_{\overline{\rho}}\rightarrow 
H^1_{\mathrm{Iw}}(\mathbb{Z}[1/\Sigma_n, \zeta_n], \widetilde{H}^{BM}_1(K^p(N_0))_{\overline{\rho}}(1))$$
as the projective limit of the maps $z^{\mathrm{Iw}}_{N_0p^k, n, \overline{\rho}}$ with respect to all $k\geqq 1$ for every sufficiently large $N_0$ such that 
$(N_0, p)=1$ and $\mathrm{prime}(N_0p)=\Sigma$, and also obtains
a $R_{\overline{\rho}, \Sigma}$-linear map 
$$z^{\mathrm{Iw}}_{\Sigma, n, \overline{\rho}} : \widetilde{H}^{BM}_{1, \overline{\rho}, \Sigma}\rightarrow 
H^1_{\mathrm{Iw}}(\mathbb{Z}[1/\Sigma_n, \zeta_n], \widetilde{H}^{BM}_{1, \overline{\rho}, \Sigma}(1))$$
as the inductive limit of the maps $z^{\mathrm{Iw}}_{N_0p^{\infty}, n, \overline{\rho}}$ with respect to all such $N_0$.
Furthermore, one can show that this map $z^{\mathrm{Iw}}_{\Sigma, n, \overline{\rho}}$ 
is continuous and $G_{\Sigma}$-equivariant. By Emerton's isomorphism 
$$\widetilde{H}^{BM}_{1, \overline{\rho}, \Sigma}\isom  (\rho^u)^*\widehat{\otimes}_{R_{\overline{\rho}, \Sigma}}
(\pi_p^u)^*\otimes_{R_{\overline{\rho}, \Sigma}} \widetilde{\pi}_{\Sigma_0}^u,$$ one can regard the map $z^{\mathrm{Iw}}_{\Sigma, n, \overline{\rho}}$ 
as a continuous $R_{\overline{\rho}, \Sigma}[G_{\Sigma}]$-linear map 
$$z^{\mathrm{Iw}}_{\Sigma, n, \overline{\rho}} :  (\rho^u)^*\widehat{\otimes}_{R_{\overline{\rho}, \Sigma}}
(\pi_p^u)^*\otimes_{R_{\overline{\rho}, \Sigma}} \widetilde{\pi}_{\Sigma_0}^u
\rightarrow H^1_{\mathrm{Iw}}(\mathbb{Z}[1/\Sigma_n, \zeta_n],  (\rho^u)^*(1)\widehat{\otimes}_{R_{\overline{\rho}, \Sigma}}
(\pi_p^u)^*\otimes_{R_{\overline{\rho}, \Sigma}} \widetilde{\pi}_{\Sigma_0}^u),$$
which we call the $G_{\Sigma}$-equivariant zeta morphism.

We construct our zeta morphisms $\bold{z}_{\Sigma,n}(\rho^u)$ by factoring out the $(\rho^u)^*$-part from the maps 
$z^{\mathrm{Iw}}_{\Sigma, n, \overline{\rho}}$ for each $n$. It turns out that the main theorems concerning
the representation $\pi_p^u$ (resp. $\pi_l^u$ for $l\in \Sigma_0$) by Pa\v{s}k\={u}nas \cite{Pas13} (resp. Emerton-Helm \cite{EH14} and 
Helm \cite{He16}) are crucial for this purpose. 

In \cite{EH14} and \cite{He16}, they study a correspondence between the continuous representations of $G_{\mathbb{Q}_l}$ and 
the smooth admissible representations of $G_l$ with coefficients in noether local complete $\mathbb{Z}_p$-algebras for $l\not=p$, 
and give a characterization of this correspondence. In its characterization, a generalization of the classical notion of genericity is very important. 
In \cite{EH14}, they defined an exact functor $\Psi_l$ from the smooth $\mathbb{Z}_p[G_l]$-modules to 
$\mathbb{Z}_p$-modules such that, for $\pi$ defined over $\overline{\mathbb{Q}}_p$, $\Psi_l(\pi)$ is the maximal quotient of $\pi$ on which the subgroup 
$\begin{pmatrix}1& * \\ 0 &1\end{pmatrix}$ of $G_l$ acts by a fixed non trivial character $\psi_l : \mathbb{Q}_l\rightarrow \overline{\mathbb{Q}}^{\times}_p$. We recall that an irreducible smooth admissible representation $\pi$ of $G_l$ 
defined over $\overline{\mathbb{Q}}_p$ is called generic if $\Psi_l(\pi)$ is a one dimensional $\overline{\mathbb{Q}}_p$-vector space. 
Their correspondence demand such a property for representations with general coefficients. In particular, the representation 
$\pi_l^u$ of $G_l$ corresponding to $\rho^u|_{G_{\mathbb{Q}_l}}$ satisfies 
$\Psi_l(\widetilde{\pi}_l^u)\isom R_{\overline{\rho}, \Sigma}$, 
by which one obtains a $R_{\overline{\rho}, \Sigma}[G_p]$-linear morphism 
$$\Psi_{\Sigma_0}(z^{\mathrm{Iw}}_{\Sigma, n, \overline{\rho}}):  (\rho^u)^*\widehat{\otimes}_{R_{\overline{\rho}, \Sigma}}
(\pi_p^u)^*
\rightarrow H^1_{\mathrm{Iw}}(\mathbb{Z}[1/\Sigma_n, \zeta_n],  (\rho^u)^*(1)\widehat{\otimes}_{R_{\overline{\rho}, \Sigma}}
(\pi_p^u)^*)$$
applying the composite $\Psi_{\Sigma_0}:=\prod_{l\in \Sigma_0}\Psi_l$ of the functors $\Psi_l$ for every $l\in \Sigma_0$ to the $\prod_{l\in \Sigma_0}G_l$ equivariant map $z^{\mathrm{Iw}}_{\Sigma, n, \overline{\rho}}$ . Therefore, it finally remains to remove 
the $p$-th factor $(\pi_p^u)^*$. 

In \cite{Pas13}, he studies a deformation theory of the representation $\overline{\pi}_p$ of $G_p$ corresponding to 
$\overline{\rho}_p:=\overline{\rho}|_{G_{\mathbb{Q}_p}}$ via the mod $p$ local Langlands correspondence. The most important object in his study is 
the projective envelope $\widetilde{P}$ of the socle of the Pontryagin dual $\overline{\pi}_p^{\vee}$ of $\overline{\pi}_p$ in a suitable category $\mathfrak{C}(\mathcal{O})$ 
of representations of $G_p$ on compact $\mathcal{O}$-modules. He proves that, under some mild assumptions, there is a natural isomorphism 
$$R_p\isom \mathrm{End}_{\mathfrak{C}(\mathcal{O})}(\widetilde{P})$$ for the universal deformation ring $R_p$ of $\overline{\rho}_p$, and 
this isomorphism makes $\widetilde{P}$ to be the $R_p$-dual of the universal deformation of $\overline{\pi}_p$. In particular, 
one has a topological $R_{\overline{\rho}, \Sigma}[G_p]$-linear isomorphism 
$$\widetilde{P}\widehat{\otimes}_{R_p}R_{\overline{\rho}, \Sigma}\isom (\pi_p^u)^*.$$
One can show that this isomorphism induces natural isomorphisms
$$ (\rho^u)^*\isom \mathrm{Hom}_{\mathfrak{C}(\mathcal{O})}(\widetilde{P},  (\rho^u)^*\widehat{\otimes}_{R_{\overline{\rho}, \Sigma}}
(\pi_p^u)^*)$$
and also 
$$H^1_{\mathrm{Iw}}(\mathbb{Z}[1/\Sigma_n, \zeta_n],  (\rho^u)^*(1))\isom 
 \mathrm{Hom}_{\mathfrak{C}(\mathcal{O})}(\widetilde{P}, H^1_{\mathrm{Iw}}(\mathbb{Z}[1/\Sigma_n, \zeta_n],  (\rho^u)^*(1)\widehat{\otimes}_{R_{\overline{\rho}, \Sigma}}
(\pi_p^u)^*))$$
essentially by the projectiveness of $\widetilde{P}$. Then, we define our zeta morphism 
$$\bold{z}_{\Sigma,n}(\rho^u) : (\rho^u)^*\rightarrow H^1_{\mathrm{Iw}}(\mathbb{Z}[1/\Sigma_n, \zeta_n],  (\rho^u)^*(1))$$
to be the map
$\mathrm{Hom}_{\mathfrak{C}(\mathcal{O})}(\widetilde{P}, \Psi_{\Sigma_0}(z^{\mathrm{Iw}}_{\Sigma, n, \overline{\rho}}))$. 

Finally, we need to compare our zeta morphism $\bold{z}_{\Sigma, 1}(\rho^u)$ (for $n=1$) with Kato's zeta morphism 
$\bold{z}(f)$ for every $f$ containing in the family. For this purpose, we need to describe the map 
$x_f : R^u\rightarrow \mathcal{O}$ using the objects appearing in our 
construction, in particular, using the projective envelope $\widetilde{P}$. In \cite{Pas15}, Pa\v{s}k\={u}nas re-defined 
the quotients of $R_p$ parametrizing de Rham representations with  fixed Hodge-Tate and inertia types (i.e. Kisin's potentially semi-stable deformation rings) 
using $\widetilde{P}$. It turns out that his re-construction completely fits for our proof of the comparison of our zeta morphisms with Kato's ones. 

\subsection{Contents of the article}Now, we briefly describe the contents of different sections. 

In \S 2, we recall the definition and some basic properties of the completed cohomology and 
the completed homology of modular curves. 
In \S2.1, we recall the definition of the completed cohomology of modular curves following Emerton's 
article \cite{Em}. In \S 2.2. we recall his main theorem on a refined local-global compatibility on
the $p$-adic Langlands correspondence for $\mathrm{GL}_{2/\mathbb{Q}}$. 
In \S 2.3, we define a dual version of the completed cohomology, which we call the completed (Borel-Moore) homology. 
Using the duality theorem proved by \cite{CE12} 
between the completed cohomology and 
the completed Borel-Moore homology, we also prove a dual version of the refined local-global compatibility, 
which is our foundation for the construction of our zeta morphisms. 
In \S 2.4. we define the completed homology with coefficients and prove its basic properties, which we need 
to compare our zeta morphisms with Kato's zeta morphisms. 

\S 3 is the main technical body of this article, where we construct the $G_{\Sigma}$-equivariant zeta morphisms $z^{\mathrm{Iw}}_{\Sigma, n, \overline{\rho}}$. 
In \S 3.1. we recall the definition of Beilinson's elements and Kato's elements, and recall its zeta values formula proved in \cite{Ka04}, which is 
crucial for our construction of the $G_{\Sigma}$-equivariant zeta morphisms  $z^{\mathrm{Iw}}_{\Sigma, n, \overline{\rho}}$. In \S 3.2 and \S 3.3,  we 
define the zeta morphisms $z^{\mathrm{Iw}}_{N, n, \overline{\rho}}$ for the localization $H^1(Y(N), \mathbb{Z}_p(1))_{\overline{\rho}}$ of $H^1(Y(N), \mathbb{Z}_p(1))$ 
at the maximal ideal of the Hecke algebra corresponding to $\overline{\rho}$ as a generalization of Fukaya-Kato's zeta morphism 
for $H^1(Y_1(N), \mathbb{Z}_p(1))$. More generally, we define such morphisms for the localization
$H^1(Y(N), \mathcal{F}(1))_{\overline{\rho}}$ of the cohomology $H^1(Y(N), \mathcal{F}(1))$ with coefficients $\mathcal{F}$ corresponding to 
an integral lattice of algebraic representations of $\mathrm{GL}_{2/\mathbb{Q}}$. We need such a generalization to compare our zeta morphisms with Kato's ones. 
In \S 3.2, we first define zeta morphisms for the parabolic 
cohomology, i.e. for the
cohomology $H^1(X(N), j_*\mathcal{F}(1))$ of the compactified modular curve $X(N)$ with coefficient in 
the push forward $j_*(\mathcal{F})$ with respect to the inclusion $j : Y(N)\hookrightarrow X(N)$. In \S 3.3, 
we consider its localization at the maximal ideal corresponding to $\overline{\rho}$, to obtain 
the zeta morphisms for the localization $H^1(Y(N), \mathcal{F}(1))_{\overline{\rho}}$. Then, we furthermore show its integrality by reducing 
to the integrality of Kato's zeta morphisms $\bold{z}(f)$. Taking its projective limit, then its inductive limit, in \S 3.4, we finally define
$G_{\Sigma}$-equivariant zeta morphisms $z^{\mathrm{Iw}}_{\Sigma, n, \overline{\rho}}$ for the completed homology 
$\widetilde{H}^{BM}_{1, \overline{\rho}, \Sigma}$. 

In \S4, we define the zeta morphisms $\bold{z}_{\Sigma, n}(\rho^u)$, and prove the main theorem of the article. 
In \S 4.1, we define the zeta morphisms $\bold{z}_{\Sigma, n}(\rho^u)$  from the $G_{\Sigma}$-equivariant zeta morphisms $z^{\mathrm{Iw}}_{\Sigma, n, \overline{\rho}}$
following the method described in the previous sub section. In \S 4.2, we state the main theorem of this article, 
which is the precise version of theorem \ref{0.1} (iii). In \S 4.3, we recall the results of Pa\v{s}k\={u}nas \cite{Pas15} on 
the description of potentially semi-stable deformation rings using the projective envelope $\widetilde{P}$, which is crucial 
for the proof of our main theorem. In \S 4.4, we prove the main theorem. 

In \S5, we give an application of our main theorem to Iwasawa main conjecture. 
In \S5.1, we recall the statement of Kato's main conjecture for $f$. In \S 5.2, we prove congruences of 
Kato's zeta morphisms $\bold{z}(f)$. In \S5.3, we apply the congruences to Kato's main conjecture. 

Moreover, we have two appendixes of this article. In the appendix A, we recall and slightly generalized Kato's theorem on the existence of 
zeta morphisms associated to Hecke eigen new forms. In the appendix B, we collect basic notions appearing in family version of 
$p$-adic local Langlands correspondence for $G_l$ both for $l=p$ and $l\not=p$.

\subsection*{Acknowledgements}I would  like to thank Seidai Yasuda for discussing applications of $p$-adic Langalnds correspondence to a series of 
Kato's conjectures (the generalized Iwasawa main conjecture, the local and the global $\varepsilon$-conjecture) over the years. I would also like to thank 
Chan-Ho Kim for reading the first version of the article, and discussing applications of the main theorem to Iwasawa main conjecture, in particular, suggesting me the appication in \S 6. This application 
seems to make this article much more interesting than the first version.
I started to announce the results of this article about three years ago. 
I would also like to thank David Burns, Shinichi Kobayashi, Teruhisa Koshikawa, Robert Pollack, Ryotaro Sakamoto, and Shanwen Wang for their interests in my results and some helpful comments. 

\subsection*{Notation}For any field $F$, we write $G_F$ to denote its absolute Galois group.  
Throughout the article, we fix a prime $p$, and embeddings 
$\iota_{\infty}:\overline{\mathbb{Q}}\hookrightarrow \mathbb{C}$ and $\iota_p:\overline{\mathbb{Q}}\rightarrow \overline{\mathbb{Q}}_p$.
Fix $i:=\sqrt{-1}\in \mathbb{C}$, and we set $\zeta_n:=\mathrm{exp}\left(\frac{2\pi i}{n}\right)\in \mathbb{C}$ for each 
$n\in \mathbb{Z}_{\geqq 1}$, which we freely regards the elements in $\overline{\mathbb{Q}}$ or $\overline{\mathbb{Q}}_p$ by the embeddings 
$\iota_{\infty}$ and $\iota_{p}$. For each prime $l$, let $I_l\subseteq G_{\mathbb{Q}_l}$ be the inertia subgroup, 
$\mathrm{Frob}_l\in G_{\mathbb{F}_l}$ the geometric Frobenius, and $G_l:=\mathrm{GL}_2(\mathbb{Q}_l)$. 
We write $\psi_{l}:\mathbb{Q}_l\rightarrow \overline{\mathbb{Q}}^{\times}$ to denote the additive character defined by $\psi_l\left(\frac{1}{l^n}\right)=\zeta_{l^n}$ for all $n\geqq 1$. We use the notation $E\subset \overline{\mathbb{Q}}_p$ to denote a finite extension of $\mathbb{Q}_p$, $\mathcal{O}$ its integer ring, $\varpi\in \mathcal{O}$ a
uniformizer, $\mathbb{F}:=\mathcal{O}/(\varpi)$ its residue field.
We write $\mathrm{Comp}(\mathcal{O})$ to denote the category of commutative local completed Noetherian $\mathcal{O}$-algebras with  finite residue fields. 
We use the notation $\Sigma$ to denote a finite set of primes. We always assume that $\Sigma$ contains $p$, and 
write $\Sigma_0:=\Sigma\setminus\{p\}$. We set $G_{\Sigma}:=\prod_{l\in \Sigma}G_l$ and $G_{\Sigma_0}:=\prod_{l\in \Sigma_0}G_l$, 
hence $G_{\Sigma}=G_{\Sigma_0}\times G_p$. For any field $F$ and a topological $G_F$-module $T$, we write 
$C^{\bullet}(F, T)$ to denote the complex of continuous cochains of $G_F$ with values in $T$, and 
$H^i(F, T)$ to denote its $i$-th cohomology. For any finite extension $F\subset \overline{\mathbb{Q}}$ of $\mathbb{Q}$ and any finite set 
$\Sigma$ as above, we write $F_{\Sigma}\subseteq \overline{\mathbb{Q}}$ to denote the largest
extension of $F$ which is unramified outside the places of $F$ above $\Sigma\cup\{\infty\}$, $G_{F, \Sigma}:=\mathrm{Gal}(F_{\Sigma}/F)$. 
For any topological $G_{F, \Sigma}$-module $T$, we set $C^{\bullet}(\mathcal{O}_F[1/\Sigma], T)$ 
to denote the complex of continuous cochains of $G_{F, \Sigma}$ with coefficients in $T$, and 
$H^i(\mathcal{O}_F[1/\Sigma], T)$ to denote its $i$-th cohomology. We write 
$\varepsilon : G_{\mathbb{Q}}\rightarrow \mathbb{Z}_p^{\times}$ to denote the cyclotomic character 
defined by $g(\zeta_{p^n})=\zeta_{p^n}^{\varepsilon(g)}$ for all $n\geqq 1$. For each integer $n\geqq 1$, 
we identify $\mathrm{Gal}(\mathbb{Q}(\zeta_n)/\mathbb{Q})\isom (\mathbb{Z}/n\mathbb{Z})^{\times}$ by the 
cyclotomic character, and, for each integer $a$ which is prime to $n$, we write $\sigma_{n,a}\in \mathrm{Gal}(\mathbb{Q}(\zeta_n)/\mathbb{Q})$ 
to denote the elements corresponding to $a\in (\mathbb{Z}/n\mathbb{Z})^{\times}$. We simply write $\sigma_a:=\sigma_{n,a}$ if there is no risk of confusing 
about $n$. We set $\Sigma_n:=\Sigma\cup\mathrm{prime}(n)$ the union of $\Sigma$ and the set of prime divisors $\mathrm{prime}(n)$ of $n$. 
For each integer $n\geqq 1$ which is prime to $\Sigma$, we set $\Lambda_n:=\mathbb{Z}_p[[\mathrm{Gal}(\mathbb{Q}(\zeta_{np^{\infty}})/\mathbb{Q})]]$ 
the Iwasawa algebra of $\mathrm{Gal}(\mathbb{Q}(\zeta_{np^{\infty}})/\mathbb{Q})$. For any topological $G_{\mathbb{Q}, \Sigma}$-module $T$ on a compact $\mathbb{Z}_p$-module, we set $\bold{Dfm}_n(T):=T\widehat{\otimes}_{\mathbb{Z}_p}\Lambda_n$ the completed tensor product on 
which $G_{\mathbb{Q}, \Sigma_n}$ acts by $g(u\otimes v):=g\cdot u\otimes [\overline{g}]^{-1}v$, where $\overline{g}\in \mathrm{Gal}(\mathbb{Q}(\zeta_{np^{\infty}})/\mathbb{Q})$ is the 
image of $g$ by the natural map $G_{\mathbb{Q},\Sigma_n}\rightarrow \mathrm{Gal}(\mathbb{Q}(\zeta_{np^{\infty}})/\mathbb{Q})$. We 
set $C^{\bullet}_{\mathrm{Iw}}(\mathbb{Z}[1/\Sigma_n, \zeta_n], T):=C^{\bullet}(\mathbb{Z}[1/\Sigma_n], \bold{Dfm}_n(T))$ and write its cohomology by 
$H^i_{\mathrm{Iw}}(\mathbb{Z}[1/\Sigma_n, \zeta_n], T)$. For each prime $l$, we write $\mathrm{rec}_l : \mathbb{Q}_l^{\times}\rightarrow G_{\mathbb{Q}_l}^{\mathrm{ab}}$ to denote the reciprocity map normalized so that $\mathrm{rec}_l(l)$ is a lift of the geometric Frobenius, by which we identify the set of  continuous characters $\chi : \mathbb{Q}_l^{\times}\rightarrow A^{\times}$ for $A\in \mathrm{Comp}(\mathcal{O})$ with the set of continuous characters 
$\chi : G^{\mathrm{ab}}_{\mathbb{Q}_l}\rightarrow A^{\times}$. 


\section{Completed homology of modular curves}

\subsection{Completed cohomology of modular curves}
We start this section by recalling the definition of the completed cohomology of 
modular curves exactly following the original article \cite{Em}. 
\subsubsection{Definition of the completed cohomology}

Let $\mathbb{A}_f\coloneqq\widehat{\mathbb{Z}}\otimes\mathbb{Q}$ be the ring of finite adeles. Let $\mathcal{H}^{\pm}\coloneqq\mathbb{C}\setminus \mathbb{R}$
 be the union of the upper and lower half planes, on which $\mathrm{GL}_2(\mathbb{R})$ acts by 
$g\tau=\frac{a\tau+b}{c\tau+d}$ for $g=\begin{pmatrix}a& b\\ c & d\end{pmatrix}\in \mathrm{GL}_2(\mathbb{R})$ and $\tau\in \mathcal{H}^{\pm}$.
For any compact open subgroup $K_f$ of $\mathrm{GL}_2(\mathbb{A}_f)$, we set
$$Y(K_f)(\mathbb{C})\coloneqq\mathrm{GL}_2(\mathbb{Q})\backslash \mathcal{H}^{\pm} \times (\mathrm{GL}_2(\mathbb{A}_f)/K_f),$$
where $\mathrm{GL}_2(\mathbb{Q})$ acts on $\mathcal{H}^{\pm} \times (\mathrm{GL}_2(\mathbb{A}_f)/K_f)$ by 
$h(\tau, [g])\coloneqq(h\tau, [hg])$ for $h\in \mathrm{GL}_2(\mathbb{Q}), \tau\in \mathcal{H}^{\pm}$, and $[g]\in \mathrm{GL}_2(\mathbb{A}_f)/K_f$.
As is well known, for any sufficiently small $K_f$, $Y(K_f)(\mathbb{C})$ is a smooth manifold, and admits a canonical model $Y(K_f)_{\mathbb{Q}}$ as an algebraic curve over $\mathbb{Q}$. We denote its base change to $F$ by $Y(K_f)_{F}$ for any $\mathbb{Q}$-algebra $F$. For any commutative ring $A$, we denote its $i$-th singular cohomology with coefficient $A$ by 
$$H^i(K_f)_A:=H^i(Y(K_f)(\mathbb{C}), A).$$ Since we mainly consider the case where $A=\mathcal{O}$, we simply write 
$$H^i(K_f):=H^i(K_f)_{\mathcal{O}}$$ to simplify the notation if there is no risk of confusion. 

For each $g\in \mathrm{GL}_2(\mathbb{A}_f)$, the multiplication by $g$ on the right induces a homeomorphism 
$$Y(gK_fg^{-1})(\mathbb{C})\isom Y(K_f)(\mathbb{C})\colon (\tau,[h])\mapsto (\tau, [hg]),$$ and 
this induces an $A$-linear isomorphism 
$$g^*\colon H^1(K_f)_A\isom H^i(gK_fg^{-1})_A.$$

If $K^p$  is some fixed compact open subgroup of $\mathrm{GL}_2(\mathbb{A}_f^p)$, then we write 
$$H^i(K^p)_A\coloneqq\varinjlim_{K_p}H^i(K_pK^p)_A, $$
where the inductive limit is taken over all the compact open subgroups $K_p$ of $G_p=\mathrm{GL}_2(\mathbb{Q}_p)$. 
The isomorphisms 
$g^*\colon H^i(K_pK^p)_A\isom H^i((gK_pg^{-1})K^p)_A$ for all $g\in G_p$ and $K_p$ induce a smooth action of $G_p$ on $H^i(K^p)_A$.

For $A$, we mainly consider a finite extension $E\subset \overline{\mathbb{Q}}_p$ of $\mathbb{Q}_p$, and its integer ring $\mathcal{O}$ 
in this article, For such $A$, 
$H^i(K_f)_A$ is canonically isomorphic to the \'etale cohomology $H^i_{\text{\'et}}(Y(K_f)_{\overline{\mathbb{Q}}}, A)$ by 
the composite 
$$H^i_{\text{\'et}}(Y(K_f)_{\overline{\mathbb{Q}}}, A)\isom H^i_{\text{\'et}}(Y(K_f)_{\overline{\mathbb{C}}}, A)
\isom 
H^i(Y(K_f)(\mathbb{C}), A),$$
where the first isomorphism is induced by the fixed embedding $\iota_{\infty} : \overline{\mathbb{Q}}\hookrightarrow \mathbb{C}$, 
and the second one is the comparison isomorphism between Betti and \'etale cohomologies. We freely identify $H^i(K_f)_A$ with 
$H^i_{\text{\'et}}(Y(K_f)_{\overline{\mathbb{Q}}}, A)$ by this isomorphism. In particular, $H^i(K_f)_A$ is equipped with a continuous $A$-linear $G_{\mathbb{Q}}$-action, and the induced $G_{\mathbb{Q}}$-action on $H^i(K^p)_A$ commutes with the $G_p$-action for any $K^p$.



We write 
$$\widehat{H}^i(K^p)_{\mathcal{O}}\coloneqq\varprojlim_{s}H^i(K^p)_{\mathcal{O}}/\varpi^s H^i(K^p)_{\mathcal{O}}$$
to denote the $\varpi$-adic completion of $H^i(K^p)_{\mathcal{O}}$. We simply write 
$$\widehat{H}^i(K^p):=\widehat{H}^i(K^p)_{\mathcal{O}}$$ 
if there is no risk of confusing about $\mathcal{O}$. The $G_{\mathbb{Q}}\times G_p$-action on $H^i(K^p)$ extends to a 
$\varpi$-adically continuous $G_{\mathbb{Q}}\times G_p$-action on $\widehat{H}^i(K^p)$.

For any $g\in \mathrm{GL}_2(\mathbb{A}_f^p)$, the isomorphisms $g^*\colon H^i(K_pK^p)_A\isom 
H^i(K_p(gK^pg^{-1}))_A$ for all 
$K_p$ induce an isomorphism 
$$g^*\colon H^i(K^p)_A\isom H^i(gK^pg^{-1})_A$$ commuting with the actions of $G_p$ and $G_{\mathbb{Q}}$. 
The isomorphism 
$g^*\colon H^i(K^p)_{\mathcal{O}}\isom H^i(gK^pg^{-1})_{\mathcal{O}}$ extends to an $\mathcal{O}$-linear topological
 isomorphism 
$$g^*\colon \widehat{H}^i(K^p)\isom \widehat{H}^i(gK^pg^{-1}).$$ 


Let $\Sigma_0$ be a finite set of primes containing $p$. We set $\Sigma_0:=\Sigma_0\setminus\{p\}$, $G_{\Sigma_0}:=\prod_{l\in  \Sigma_0}G_l=\prod_{l\in  \Sigma_0}\mathrm{GL}_2(\mathbb{Q}_l)$, $G_{\Sigma}:=G_{\Sigma_0}\times G_p$ and $K_0^{\Sigma}:=\prod_{l\not \in  \Sigma}\mathrm{GL}_2(\mathbb{Z}_l)$, 
where the product is taken over all the primes $l$ not containing in $\Sigma$. For any open subgroup $K_{\Sigma_0}$ of $G_{\Sigma_0}$, we write 
$$H^i(K_{\Sigma_0})_A:=H^i(K_{\Sigma_0}K_0^{\Sigma})_A, \quad \widehat{H}^i(K_{\Sigma_0})
:=\widehat{H}^i(K_{\Sigma_0}K_0^{\Sigma}).$$
We also write 
$$H^i_{A,\Sigma}:=\varinjlim_{K_{\Sigma_0}}H^i(K_{\Sigma_0})_A, \quad \widehat{H}^i_{\Sigma}:=\varinjlim_{K_{\Sigma_0}}\widehat{H}^i(K_{\Sigma_0})$$
equipped with the actions of $G_{\Sigma}$ and $G_{\mathbb{Q}}$, 
where both the limits are taken over all compact open subgroups $K_{\Sigma_0}$ of 
$G_{\Sigma_0}$. We equip $\widehat{H}^i_{\Sigma}$ with the $\mathcal{O}$-linear inductive limit topology (i.e. a basis of open neighborhood of the zero in 
$\widehat{H}^i_{\Sigma}$ are all the sub $\mathcal{O}$-modules $U$ such that $U\cap \widehat{H}^i(K_{\Sigma_0})$ is 
$\varpi$-adically open in $\widehat{H}^i(K_{\Sigma_0})$ for every 
$K_{\Sigma_0}$). Then, $G_{\mathbb{Q}}\times G_p$ continuously acts on it. 



\subsubsection{Big Hecke algebras and universal deformation rings}
Let $K_{\Sigma}$ be a compact open subgroup of $G_{\Sigma}$. For any prime $l\not\in \Sigma$, we write $S_{l}$ and $T_{l}$ to denote the Hecke operators acting on 
$H^1(K_{\Sigma}K_0^{\Sigma})_{A}$ corresponding to the double cosets of $K_l$ in $G_l$ represented by $\begin{pmatrix}l & 0 \\ 0 & l\end{pmatrix}$ and 
$\begin{pmatrix} l & 0 \\ 0 & 1\end{pmatrix}\in G_l$ respectively. We let $\mathbb{T}(K_{\Sigma})_A$ denote the sub 
$A$-algebra 
of $\mathrm{End}_{A}(H^1(K_{\Sigma}K_0^{\Sigma})_{A})$ generated by the Hecke operators $S_{l}$ and $T_{l}$ for all the primes $l\not \in \Sigma$. 
We remark that the canonical map $\mathbb{T}(K_{\Sigma})_A\otimes_AA'\rightarrow \mathbb{T}(K_{\Sigma})_{A'}$ is isomorphism 
for any flat $A$-algebra $A'$. In particular, there exist a canonical isomorphism $\mathbb{T}(K_{\Sigma})_{\mathbb{Z}}\otimes_{\mathbb{Z}}\mathcal{O}
\isom \mathbb{T}(K_{\Sigma})_{\mathcal{O}}$ and inclusions
$\mathbb{T}(K_{\Sigma})_{\mathbb{Z}}\subset \mathbb{T}(K_{\Sigma})_{\mathbb{Z}}\otimes_{\mathbb{Z}}\mathbb{Q}
\isom \mathbb{T}(K_{\Sigma})_{\mathbb{Q}}$ and $\mathbb{T}(K_{\Sigma})_{\mathcal{O}}\subset \mathbb{T}(K_{\Sigma})_{\mathcal{O}}\otimes_{\mathcal{O}}E
\isom \mathbb{T}(K_{\Sigma})_{E}$ since $\mathbb{T}(K_{\Sigma})_{\mathbb{Z}}$ (resp. $\mathbb{T}(K_{\Sigma})_{\mathcal{O}}$) 
is a free $\mathbb{Z}$-module (resp. $\mathcal{O}$-module). Since $\mathbb{T}(K_{\Sigma})_{\mathbb{Q}}$ and $\mathbb{T}(K_{\Sigma})_{E}$ are reduced, $\mathbb{T}(K_{\Sigma})_{\mathbb{Z}}$
 and $\mathbb{T}(K_{\Sigma})_{\mathcal{O}}$ are also reduced rings. 

For $A=\mathcal{O}$, we simply write $\mathbb{T}(K_{\Sigma}):=\mathbb{T}(K_{\Sigma})_{\mathcal{O}}$ if there is no risk of confusing about $\mathcal{O}$. 
Let $K_{\Sigma_0}$ be a compact open subgroup of $G_{\Sigma_0}$. 
If $K_p'\subset K_p$ is any inclusion of compact open subgroups of $G_p$, then there is a natural surjection 
$\mathbb{T}(K'_pK_{\Sigma_0})\rightarrow \mathbb{T}(K_pK_{\Sigma_0})$. We define $\mathbb{T}(K_{\Sigma_0}):=
\varprojlim_{K_p}\mathbb{T}(K_pK_{\Sigma_0})$, and we equip 
$\mathbb{T}(K_{\Sigma_0})$ with its projective limit topology, each of the finite $\mathcal{O}$-algebras $\mathbb{T}(K_pK_{\Sigma_0})$ being equipped with its 
$\varpi$-adic topology. The $\mathcal{O}$-algebra $\mathbb{T}(K_{\Sigma_0})$ is then topologically generated by the Hecke 
operators $S_{l}$ and $T_{l}$ for all the primes $l\not \in \Sigma$, and acts faithfully on $\widehat{H}^1(K_{\Sigma_0})$. 
The topological complete reduced $\mathcal{O}$-algebra $\mathbb{T}(K_{\Sigma_0})$ is known to 
be decomposed into $\mathbb{T}(K_{\Sigma_0})=\prod_{\mathfrak{m}}\mathbb{T}(K_{\Sigma_0})_{\mathfrak{m}}$ 
the product of finitely many complete Noether local $\mathcal{O}$-algebras $\mathbb{T}(K_{\Sigma_0})_{\mathfrak{m}}$ with finite residue fields.
For any inclusion $K_{ \Sigma_0}\subset K'_{\Sigma_0}$, there is a continuous surjective morphism $\mathbb{T}(K_{ \Sigma_0})\rightarrow \mathbb{T}(K'_{ \Sigma_0})$ sending
$S_{l}$ and $T_{l}$ in the source to $S_{l}$ and $T_{l}$ in the target, which is compatible with respect to 
the natural inclusion $\widehat{H}^1(K'_{ \Sigma_0})\rightarrow \widehat{H}^1(K_{ \Sigma_0})$ and the action of $\mathbb{T}(K'_{ \Sigma_0})$ (resp. $\mathbb{T}(K_{ \Sigma_0})$) on its source (resp. target). 

We now fix  a continuous absolutely irreducible representation 
$$\overline{\rho}:G_{\mathbb{Q}}\rightarrow \mathrm{GL}_2(\mathbb{F}),$$ 
which is unramified outside $\Sigma\cup\{\infty\}$. Assume moreover that $\overline{\rho}$ is odd, hence modular 
by Serre conjecture proved by Khare-Wintenberger \cite{KW09a}, \cite{KW09b}, and Kisin \cite{Ki09a}.
We say that a compact open subgroup $K_{\Sigma_0}\subset G_{\Sigma_0}$ is an allowable level for $\overline{\rho}$ if there exists a maximal ideal $\mathfrak{m}$ of $\mathbb{T}(K_{\Sigma_0})$ with its residue field $\mathbb{F}$, such that 
$$\mathrm{tr}(\overline{\rho}(\mathrm{Frob}_{l}))=T_{l} \text{ mod } \mathfrak{m}, \quad \mathrm{det}(\overline{\rho}(\mathrm{Frob}_{l}))=l S_{l} \text{ mod }\mathfrak{m}  $$ 
for all the primes $l \not\in \Sigma$. Such $\mathfrak{m}$ is unique if it exists, and we 
write $\mathbb{T}(K_{\Sigma_0})_{\overline{\rho}}\coloneqq\mathbb{T}(K_{\Sigma_0})_{\mathfrak{m}}$.
Since our $\overline{\rho}$ is known to be modular, any sufficiently small $K_{\Sigma_0}\subset G_{\Sigma_0}$ is an allowable level for $\overline{\rho}$. 

Since $\overline{\rho}$ is assumed to be absolutely irreducible, it is known that there exists, for any allowable level $K_{\Sigma_0}$, a deformation 
$$\rho(K_{\Sigma_0}):G_{\mathbb{Q}}\rightarrow \mathrm{GL}_2(\mathbb{T}(K_{\Sigma_0})_{\overline{\rho}})$$
of $\overline{\rho}$ which is unramified out side $\Sigma\cup\{\infty\}$, 
uniquely determined (up to isomorphism) by the requirement 
$$\mathrm{tr}(\rho(K_{\Sigma_0})(\mathrm{Frob}_{l}))=T_{l},\quad \mathrm{det}(\rho(K_{\Sigma_0})(\mathrm{Frob}_{l}))=l S_{l}$$
for all the primes $l \not\in \Sigma$. If $K'_{\Sigma_0}\subset K_{\Sigma_0}$ is an inclusion of allowable levels, then the induced map 
$\mathbb{T}(K'_{\Sigma_0})_{\overline{\rho}}\rightarrow \mathbb{T}(K_{\Sigma_0})_{\overline{\rho}}$ is in general a surjection, and it is also an isomorphism for any sufficiently small allowable $K_{\Sigma_0}$ by Lemma 5.2.4 \cite{Em}. We write $\mathbb{T}_{\overline{\rho},\Sigma}:=\varprojlim_{K_{\Sigma_0}}\mathbb{T}(K_{\Sigma_0})_{\overline{\rho}}$, where the limit is taken 
over the allowable levels $K_{\Sigma_0}$ of $G_{\Sigma_0}$. For any $l\not\in \Sigma$, we write $T_{l}$ (resp. $S_l$) to denote the element of 
$\mathbb{T}_{\overline{\rho}, \Sigma}$ which  maps to the element $T_{l}$ (resp. $S_l$) in each of 
$\mathbb{T}(K_{\Sigma_0})_{\overline{\rho}}$. Then, there exists a unique deformation
$$\rho^{\mathfrak{m}}\colon G_{\mathbb{Q}}\rightarrow \mathrm{GL}_2(\mathbb{T}_{\overline{\rho}, \Sigma})$$ of 
$\overline{\rho}$ over $\mathbb{T}_{\overline{\rho},\Sigma}$ which is unramified outside $\Sigma \cup\{\infty\}$ characterized by the requirement
$$\mathrm{tr}(\rho^{\mathfrak{m}}(\mathrm{Frob}_l))=T_l,\quad \mathrm{det}(\rho^{\mathfrak{m}}(\mathrm{Frob}_l))=l S_l$$
for all $l\not\in \Sigma$. 

Since $\overline{\rho}$ is assumed to be absolutely irreducible, in particular, $\overline{\rho}$ has only trivial 
endomorphisms, there exists the universal deformation 
$$\rho^{u}:G_{\mathbb{Q},\Sigma}\rightarrow \mathrm{GL}_2(R_{\overline{\rho}, \Sigma})$$
over a complete Noetherian $\mathcal{O}$-algebra $R_{\overline{\rho}, \Sigma}$ with residue field $\mathbb{F}$, which is universal for 
deformations of $\overline{\rho}$ over objects in $\mathrm{Comp}(\mathcal{O})$ which are unramified outside $\Sigma\cup\{\infty\}$. By its universality, there exists a unique local $\mathcal{O}$-algebra homomorphism 
$$\phi_{\Sigma}:R_{\overline{\rho},\Sigma}\rightarrow \mathbb{T}_{\overline{\rho},\Sigma}$$
such that the composite $\phi_{\Sigma}\circ \rho^{u}$ is isomorphic to $\rho^{\mathfrak{m}}$. It is known that 
the map $\phi_{\Sigma}$ is surjective. 
B\"ockle \cite{Bo09} showed that the map $\phi_{\Sigma}$ is isomorphism in many cases as in the following theorem.

\begin{thm}\label{1.1}
Assume that the following conditions hold $:$ 
\begin{itemize}
\item[(1)]$p\geqq 3$. 
\item[(2)]$\overline{\rho}|_{G_{\mathbb{Q}(\zeta_p)}}$ is absolutely irreducible.
\item[(3)]$\overline{\rho}|_{G_{\mathbb{Q}_p}}$ is neither of the form $\chi\otimes 
\begin{pmatrix} 1 & * \\ 0 & 1\end{pmatrix}$ nor $\chi\otimes 
\begin{pmatrix} 1 & * \\ 0 & \overline{\varepsilon}\end{pmatrix}$, for any 
$\mathbb{F}$-valued character $\chi$ of $G_{\mathbb{Q}_p}$. 
\end{itemize}
Then, the map $\phi_{\Sigma}$ is isomorphism. 

\end{thm}
\begin{proof}
See Corollary 3.8 of \cite{Bo09} and the proof of Theorem 1.2.3 of \cite{Em}.
\end{proof}

In the main body of the article, we construct our zeta morphisms 
for $\rho^{\mathfrak{m}}$ over $\mathbb{T}_{\overline{\rho},\Sigma}$ with the similar properties (i), (ii), (iii)
as in Theorem \ref{0.1} (see Corollary \ref{3.11} and Theorem \ref{3.2}) instead of constructing those for $\rho^u$ over $R_{\overline{\rho},\Sigma}$. Theorem \ref{0.1} in Introduction is just a combination of 
Corollary \ref{3.11} and Theorem \ref{3.2} with Theorem \ref{1.1} above.

\subsubsection{$\overline{\rho}$-part of $\widehat{H}^1_{\Sigma}$}

For each allowable level $K_{\Sigma_0}\subset G_{\Sigma_0}$, 
we write 
$$\widehat{H}^1(K_{\Sigma_0})_{\overline{\rho}}:=\mathbb{T}(K_{\Sigma_0})_{\overline{\rho}}\otimes_{\mathbb{T}(K_{\Sigma_0})}\widehat{H}^1(K_{\Sigma_0}).$$ 
Since $\mathbb{T}(K_{\Sigma_0})_{\overline{\rho}}$ is a direct factor of $\mathbb{T}(K_{\Sigma_0})$,
$\widehat{H}^1(K_{\Sigma_0})_{ \overline{\rho}}$ 
 is also a direct summand of $\widehat{H}^1(K_{\Sigma_0})$, %
 and it 
is naturally equipped with the induced action of $G_{\mathbb{Q}}\times G_p$.  

If $K_{\Sigma_0}'\subset K_{\Sigma_0}$ is an inclusion of allowable levels, then the
 natural embedding $\widehat{H}^1(K_{\Sigma_0})\rightarrow \widehat{H}^1(K_{\Sigma_0}')$ restricts to an embedding 
 $\widehat{H}^1(K_{\Sigma_0})_{\overline{\rho}}\rightarrow \widehat{H}^1(K_{\Sigma_0}')_{\overline{\rho}}$. We write 
 $$\widehat{H}^1_{\overline{\rho}, \,\Sigma}\coloneqq\varinjlim_{K_{\Sigma_0}}\widehat{H}^1(K_{\Sigma_0})_{\overline{\rho}},$$ 
 where the limit is taken over all the allowable levels $K_{\Sigma_0}\subset G_{\Sigma_0}$. It is 
 a $G_{\mathbb{Q}}\times G_{\Sigma}$-stable direct summand of 
 $\widehat{H}^1_{\Sigma}$. The natural surjection $\mathbb{T}_{\overline{\rho}, \Sigma}\rightarrow \mathbb{T}(K_{\Sigma_0})_{\overline{\rho}}$ 
for each allowable level $K_{\Sigma_0}$
 allows us to regard $\widehat{H}^1(K_{\Sigma_0})_{\overline{\rho}}$ as a $\mathbb{T}_{\overline{\rho}, \Sigma}$-module, 
thus 
 $\widehat{H}^1_{\overline{\rho}, \Sigma}$ is naturally a $\mathbb{T}_{\overline{\rho}, \Sigma}$-module with a $\mathbb{T}_{\overline{\rho},\Sigma}$-linear continuous $G_{\mathbb{Q}}\times G_{\Sigma}$-action.

 \subsubsection{Completed cohomology with compact support}
 Replacing the usual cohomology $H^1(K_f)_A$ with 
 the cohomology $H^1_{c}(K_f)_A:=H^1_c(Y(K_f)(\mathbb{C}), A)$ with compact support, 
 one can define the groups $\widehat{H}^1_c(K^p)$, 
 $\widehat{H}^1_c(K_{\Sigma_0})$ and $\widehat{H}^1_{c,\,\Sigma}$ 
 in the same way as $\widehat{H}^1(K^p)$, 
 $\widehat{H}^1(K_{\Sigma_0})$ and $\widehat{H}^1_{\Sigma}$ respectively. 
 The natural map $H^1_{c}(K_f)_{\mathcal{O}}\rightarrow H^1(K_f)_{\mathcal{O}}$ forgetting support induces morphisms 
 $$\widehat{H}^1_c(K^p)\rightarrow \widehat{H}^1(K^p) , \quad \widehat{H}^1_c(K_{\Sigma_0})\rightarrow \widehat{H}^1(K_{\Sigma_0}) \quad \text{ 
 and }\quad 
 \widehat{H}^1_{c,\,\Sigma}\rightarrow \widehat{H}^1_{\Sigma},$$
 all of which are surjective by Proposition 4.39 of [Em06b]. 
 Using 
 $H^1_{c}(K_{\Sigma}K_0^{\Sigma}):=H^1_{c}(K_{\Sigma}K_0^{\Sigma})_{\mathcal{O}}$ and $\widehat{H}^1_c(K_{\Sigma_0})$ instead of $H^1(K_{\Sigma}K_0^{\Sigma})$ and $\widehat{H}^1(K_{\Sigma_0})$, one can also define the Hecke algebra $\mathbb{T}_c(K_{\Sigma})$ and the big Hecke algebra 
 $\mathbb{T}_c(K_{\Sigma_0})$ 
  in the same way as $\mathbb{T}(K_{\Sigma})$ and $\mathbb{T}(K_{\Sigma_0})$. 
 Then, it is proved that the surjection $\widehat{H}^1_c(K_{\Sigma_0})\rightarrow \widehat{H}^1(K_{\Sigma_0})$ above 
 induce an isomorphism of compact $\mathcal{O}$-algebras
 $$\mathbb{T}_c(K_{\Sigma_0})\isom  \mathbb{T}(K_{\Sigma_0})$$
 by Lemma 7.3.10 of [Em06b]. 
 By this isomorphism, 
 we regard $\widehat{H}^1_c(K_{\Sigma_0})$ as a $\mathbb{T}(K_{\Sigma_0})$-module. Therefore, one can similarly define 
 $$\widehat{H}^1_c(K_{\Sigma_0})_{\overline{\rho}}:=\mathbb{T}(K_{\Sigma_0})_{\overline{\rho}}\otimes_{\mathbb{T}(K_{\Sigma_0})}\widehat{H}^1_c(K_{\Sigma_0})$$ for any allowable level $K_{\Sigma_0}$, and 
 $$\widehat{H}^1_{c,\,\overline{\rho}, \Sigma}:=\varinjlim_{K_{\Sigma_0}}\widehat{H}^1_c(K_{\Sigma_0})_{\overline{\rho}}.$$ 
 Since we assume that $\overline{\rho}$ is absolutely irreducible,  Corollary 3.1.3 of \cite{BE10} implies that the surjections $\widehat{H}^1_c(K_{\Sigma_0})\rightarrow \widehat{H}^1(K_{\Sigma_0})$
 and 
 $\widehat{H}^1_{c,\,\Sigma}\rightarrow \widehat{H}^1_{\Sigma}$ induce isomorphisms 
 $$\widehat{H}^1_c(K_{\Sigma_0})_{\overline{\rho}}\isom \widehat{H}^1(K_{\Sigma_0})_{\overline{\rho}}\quad \text{ and }\quad
 \widehat{H}^1_{c,\,\overline{\rho}, \Sigma}\isom \widehat{H}^1_{\overline{\rho}, \Sigma}.$$

 
 \subsection{A refined local-global compatibility}
 In this subsection, we recall Emerton's refined local-global compatibility on the description of $\widehat{H}^1_{\overline{\rho}, \Sigma}$ in terms of 
  local Langlands correspondences for $G_l$ ($l\in \Sigma$) in families. 
 These local Langlands correspondences are recalled in Appendix B. In this section,
 we freely use the definitions and the notations recalled there.
 
 Fix an odd absolutely irreducible representation $\overline{\rho} : G_{\mathbb{Q}}\rightarrow \mathrm{GL}_2(\mathbb{F})$ unramified outside $\Sigma\cup\{\infty\}$. 
 In this subsection, we furthermore assume the following conditions : 
 \begin{itemize}
 \item[(a)]$\overline{\rho}_p:=\overline{\rho}|_{G_{\mathbb{Q}_p}}\not\isom  
 \chi\otimes \begin{pmatrix} 1 & * \\ 0 & \overline{\varepsilon}\end{pmatrix}$
 for any $\mathbb{F}$-valued character $\chi:G_{\mathbb{Q}_p}\rightarrow \mathbb{F}^{\times}$.
 \item[(b)]$p\geqq 3$.
 \item[(c)]$\mathrm{End}_{\mathbb{F}[G_{\mathbb{Q}_p}]}(\overline{\rho}_p)=\mathbb{F}$.
 \end{itemize}
 
 Then, Theorem \ref{5.18} implies that there is a unique (up to isomorphism) object $\overline{\pi}_p$ in 
 $\mathrm{Mod}_{G_p}^{\mathrm{adm}}(\mathbb{F})$ of finite length such that : 
  \begin{itemize}
 \item[(1)]$\mathrm{MF}(\overline{\pi}_p)\isom \overline{\rho}_p$. 
 \item[(2)]$\overline{\pi}_p$ has a central character equal to $\mathrm{det}(\overline{\rho}_p)\overline{\varepsilon}\circ \mathrm{rec}_p$. 
 Here, $\mathrm{rec}_p : \mathbb{Q}_p^{\times}\rightarrow G^{\mathrm{ab}}_{\mathbb{Q}_p}$ is the reciprocity map sending $p$ to 
 a lift of $\mathrm{Frob}_p$. 
 \item[(3)]$\overline{\pi}_p$ has no sub quotient which is finite dimensional.
 \end{itemize}

 Moreover, Theorem \ref{5.20} implies that the functor $\mathrm{MF}$ induces an isomorphism 
 $\mathrm{Def}(\overline{\pi}_p)\isom \mathrm{Def}(\overline{\rho}_p)$ of the deformation functors. In particular, 
 there exists a unique deformation $\pi^{\mathfrak{m}}_{p}$ of 
 $\overline{\pi}_p$ over $\mathbb{T}_{\overline{\rho}, \Sigma}$ such that 
 $$\mathrm{MF}(\pi^{\mathfrak{m}}_{p})\isom \rho^{\mathfrak{m}}|_{G_{\mathbb{Q}_p}}.$$
 
 We next consider $\rho^{\mathfrak{m}}|_{G_{\mathbb{Q}_l}} : G_{\mathbb{Q}_l}\rightarrow 
 \mathrm{GL}_2(\mathbb{T}_{\overline{\rho}, \Sigma})$ for the primes $l\in \Sigma_0$. 
 Since we assume that $\mathrm{End}_{\mathbb{F}[G_{\mathbb{Q}_p}]}(\overline{\rho}_p)=\mathbb{F}$, 
 (the proof of) Theorem 6.1.2 of \cite{Em} and Proposition 6.2.4 of \cite{EH14} (see also Theorem \ref{5.11} and Remark \ref{5.12} of this article) imply that there is a unique smooth 
 admissible $\mathbb{T}_{\overline{\rho},\Sigma}$-representation $\pi_l^{\mathfrak{m}}$ of $G_l$
  such that : 
 \begin{itemize}
 \item[(1)]$\pi_l^{\mathfrak{m}}$ is $\mathbb{T}_{\overline{\rho},\Sigma}$-torsion free. 
 \item[(2)]$\pi_l^{\mathfrak{m}}$ is a co-Whittaker 
 $\mathbb{T}_{\overline{\rho},\Sigma}[G_l]$-module.
 \item[(3)]for each minimal prime ideal $\mathfrak{a}$ of $\mathbb{T}_{\overline{\rho},\Sigma}$, the tensor product 
 $\kappa(\mathfrak{a})\otimes_{\mathbb{T}_{\overline{\rho},\Sigma}}\pi_l^{\mathfrak{m}}$ is 
 $\kappa(\mathfrak{a})[G_l]$-linearly isomorphic to 
 the smooth contragradient $\widetilde{\pi}(\rho_l(\mathfrak{a}))$ of $\pi(\rho_l(\mathfrak{a}))$, where 
 $\rho_l(\mathfrak{a}) : G_{\mathbb{Q}_l}\rightarrow \mathrm{GL}_2(\kappa(\mathfrak{a}))$ is the 
 base change of 
 $\rho^{\mathfrak{m}}|_{G_{\mathbb{Q}_l}}$
 by the canonical map $\mathbb{T}_{\overline{\rho},\Sigma}\rightarrow \mathbb{T}_{\overline{\rho},\Sigma}/\mathfrak{a}\hookrightarrow \kappa(\mathfrak{a})$. 
 
 \end{itemize}
 
 Finally, we write $\pi^{\mathfrak{m}}_{\Sigma_0}$ to denote the 
 maximal $\mathbb{T}_{\overline{\rho}, \Sigma}$-torsion free quotient 
 of the tensor product $\bigotimes_{l\in \Sigma_0}\pi^{\mathfrak{m}}_l$ over $\mathbb{T}_{\overline{\rho}, \Sigma}$. Here, for any ring $A$ and any 
 $A$-module $M$, 
 the maximal $A$-torsion free quotient is defined by the quotient of $M$ by the submodule $N$ consisting of all the elements $x\in M$ such that 
 $ax=0$ for some non-zero divisor $a\in A$. Then, $\pi^{\mathfrak{m}}_{\Sigma_0}$ is a smooth admissible 
 $\mathbb{T}_{\overline{\rho}, \Sigma}$-representation of $G_{\Sigma_0}$. 
 
 For any $\mathbb{Z}_p$- or $\mathbb{Q}$-algebra $A$ and any smooth $A$-representation $M$ of $P_{\Sigma_0}:=\prod_{l\in \Sigma_0}P_l\subset G_{\Sigma_0}$, we define 
 an $A$-module $\Phi_{\Sigma_0}(M)$ by 
 $$\Phi_{\Sigma_0}(M):=\Phi_{l_1}\circ \cdots \circ \Phi_{l_d}(M)$$
 for $\Sigma_0=\{l_1, \dots, l_d\}$. It is independent of the ordering of $\{l_1, \dots, l_d\}$, and is exact, and satisfies 
 $$\Phi_{\Sigma_0}(M\otimes_AN)\isom \Phi_{\Sigma_0}(M)\otimes_AN$$
 for any $A$-module $N$ by \S3.1 of \cite{EH14}. We remark that 
 one has a  $\mathbb{T}_{\overline{\rho},\Sigma}$-linear isomorphism
 $$\Phi_{\Sigma_0}(\pi^{\mathfrak{m}}_{\Sigma_0})\isom \mathbb{T}_{\overline{\rho},\Sigma}$$
 since $\pi_l^{\mathfrak{m}}$ is co-Whittaker 
 for each $l\in \Sigma_0$.

The following theorem of Emerton is the fundamental for our purpose 
(see Conjecture 6.1.6  and Theorem 6.4.16 of \cite{Em}). Let $\widetilde{\pi}_{\Sigma_0}^{\mathfrak{m}}$ be 
the smooth contragradient of $\pi_{\Sigma_0}^{\mathfrak{m}}$, which is a smooth coadmissible $\mathbb{T}_{\overline{\rho}, \Sigma}$-representation 
of $G_{\Sigma_0}$.

 \begin{thm}\label{1.2}
 There exists a $G_{\mathbb{Q}}\times G_{\Sigma}$-equivariant, $\mathbb{T}_{\overline{\rho},\Sigma}$-linear topological isomorphism 
 $$\rho^{\mathfrak{m}}\otimes_{\mathbb{T}_{\overline{\rho}, \Sigma}}\pi_{p}^{\mathfrak{m}}\overset{\curlywedge}{\otimes}_{\mathbb{T}_{\overline{\rho},\Sigma}}\widetilde{\pi}_{\Sigma_0}^{\mathfrak{m}}\isom \widehat{H}^1_{\overline{\rho}, \Sigma}.$$
 \end{thm}

 \subsection{Completed homology of modular curves}
 For our application to the construction of zeta elements, we need a dual version of 
 the isomorphism in Theorem \ref{1.2}. As a generalization of the duality between 
 cohomology and homology, the topological dual of the completed cohomology 
 $\widehat{H}^1_{\overline{\rho}, \Sigma}$ is described by a completed version of homology called 
 the completed homology. In this subsection, we recall the basic notions of completed homology
  and the duality therem between completed cohomology and homology following \cite{CE12}. 
  Then, we obtain a corollary (Corollary\ref{1.5}) which is a dual version of 
  Theorem \ref{1.2}.  We use this corollary in a crucial way in the main body of the article. 
 
  \subsubsection{Definition of the completed homology}
 Let $K^p$ be an open compact subgroup of $\mathrm{GL}_2(\mathbb{A}_f^p)$. 
 As a dual version of the completed cohomology $\widehat{H}^1(K^p)$, we consider the 
 following projective limit 
 $$\widetilde{H}^{BM}_1(K^p):=\varprojlim_{K_p}H^1(K^pK_p)(1),$$
 where $(1)$ is the Tate twist, and the limit is taken with respect to the 
 co-restriction maps
 $$\mathrm{cor} : H^1(K^pK'_p)(1)\rightarrow H^1(K^pK_p)(1)$$
 for all the open compact subgroups $K'_p\subseteq K_p$ of $G_p$. 
 \begin{rem}\label{1.3}
 Let $X(K_pK^p)(\mathbb{C})$ be the canonical compactification of $Y(K^pK_p)(\mathbb{C})$. 
 Then, Poincar\'e duality says that 
 $H^1(K^pK_p)(1)=H^1(Y(K^pK_p)(\mathbb{C}), \mathcal{O}(1))$ is canonically isomorphic to the relative homology 
 $$H_1(X(K_pK^p)(\mathbb{C}), \{\text{cusps}\}, \mathcal{O}),$$ 
 hence its limit $\widetilde{H}_1^{BM}(K^p)$ is also isomorphic to the 
 limit $$\varprojlim_{K_p}H_1(X(K_pK^p)(\mathbb{C}), \{\text{cusps}\}, \mathcal{O}).$$ 
 The homology group $H_1(X(K_pK^p)(\mathbb{C}), \{\text{cusps}\}, \mathcal{O})$ is also called Borel-Moore homology, 
 and the limit like $\varprojlim_{K_p}H_1(X(K_pK^p)(\mathbb{C}), \{\text{cusps}\}, \mathcal{O})$ is written 
 by using the notation $\widetilde{H}^{BM}_{\bullet}$ in \cite{CE12}. Since, it is convenient to use the cohomology $H^1(K^pK_p)(1)$
 instead of the relative homology $H_1(X(K_pK^p)(\mathbb{C}), \{\text{cusps}\}, \mathcal{O})$ for our purpose, we choose to define $\widetilde{H}^{BM}_1(K^p)$ using cohomology instead of 
 relative homology in this article. 
 \end{rem}

 We equip $\widetilde{H}^{BM}_1(K^p)$ with the projective limit topology with respect to the 
 $\varpi$-adic topology on each finite free $\mathcal{O}$-module $H^1(K^pK_p)(1)$. 
 $\widetilde{H}^{BM}_1(K^p)$ is also naturally equipped with $\mathcal{O}$-linear continuous left action of 
 $G_{\mathbb{Q}}\times G_p$. In particular, the restriction of this action to the subgroup $\mathrm{GL}_2(\mathbb{Z}_p)\subset G_p$ makes $\widetilde{H}^{BM}_1(K^p)$ a compact $\mathcal{O}[[\mathrm{GL}_2(\mathbb{Z}_p)]]$-module. 
 
 For any inclusion $K_1^p\subseteq K_2^p$ of open compact subgroups of of $\mathrm{GL}_2(\mathbb{A}_f^p)$, 
 the restriction map $$H^1(K_2^pK_p)(1)\rightarrow H^1(K_1^pK_p)(1)$$ 
 for each $K_p$ commutes with the co-restriction map
 $$H^1(K_1^pK'_p)(1)\rightarrow H^1(K_1^pK_p)(1)\quad \text{ and }\quad H^1(K_2^pK'_p)(1)\rightarrow H^1(K_2^pK_p)(1)$$ for any inclusion $K'_p\subseteq K_p$. Therefore, the restriction maps $H^1(K_2^pK_p)(1)\rightarrow H^1(K_1^pK_p)(1)$ for all $K_p$
 induce a continuous $\mathcal{O}[G_{\mathbb{Q}}\times G_p]$-linear map 
 $$\widetilde{H}^{BM}_1(K_2^p)\rightarrow \widetilde{H}^{BM}_1(K_1^p),$$
which we also call the restriction map. 
 We remark that this restriction map is an injection since every restriction map $H^1(K_2^pK_p)(1)\rightarrow H^1(K_1^pK_p)(1)$ 
 is so.

 We next fix a finite set $\Sigma_0$ of primes containing $p$, and set $\Sigma_0:=\Sigma\setminus\{p\}$. For any open compact subgroup $K_{\Sigma_0}$ of 
 $G_{\Sigma_0}$, we set
 $$\widetilde{H}^{BM}_{1}(K_{\Sigma_0}):=\widetilde{H}^{BM}_1(K_{\Sigma_0}K_0^{\Sigma})$$
 (recall that $K_0^{\Sigma}:=\prod_{l\not\in \Sigma}\mathrm{GL}_2(\mathbb{Z}_l)$), and 
 $$\widetilde{H}^{BM}_{1, \Sigma}:=\varinjlim_{K_{\Sigma_0}}\widetilde{H}^{BM}_{1}(K_{\Sigma_0}),$$
 where 
 the limit is taken with respect to the 
 restriction maps
 $$\widetilde{H}^{BM}_1(K'_{\Sigma_0})\rightarrow \widetilde{H}^{BM}_1(K_{\Sigma_0})$$ 
 for all inclusions $K'_{\Sigma_0}\subseteq K_{\Sigma_0}$ of open compact subgroups of $G_{\Sigma_0}$. The $\mathcal{O}[G_{\mathbb{Q}}\times G_p]$-module $\widetilde{H}^{BM}_{1,\Sigma}$ 
 is also naturally equipped with a smooth $\mathcal{O}[G_{\mathbb{Q}}\times G_p]$-linear action of 
 $G_{\Sigma_0}$. 
 

 
 \subsubsection{Duality between the completed cohomology and the completed homology}
For any sufficiently small compact open subgroup 
$K_f$ of $\mathrm{GL}_2(\mathbb{A}_f)$, Poincar\'e duality gives a canonical pairing 
$$\langle \, ,\, \rangle : H^1(K_f)(1)\times H^1_c(K_f)\rightarrow \mathcal{O}.$$
This pairing induces an $\mathcal{O}$-linear map 
$$h_{K_f} : H^1(K_f)(1)\rightarrow \mathrm{Hom}_{\mathcal{O}}(H^1_c(K_f), \mathcal{O})$$
which is $G_{\mathbb{Q}}$-equivariant (on right hand side, $G_{\mathbb{Q}}$ acts by $gf(x):=f(g^{-1}x)$). 

Since the pairing $\langle \, ,\, \rangle$ is compatible with the co-restriction maps on $H^1(K_f)(1)$ and 
  the restriction maps on $H^1_c(K_f)$, the maps $h_{K^pK_p}$ for all $K_p\subseteq G_p$ naturally induce an $\mathcal{O}$-linear map 
  $$h_{K^p} : \widetilde{H}^{BM}_1(K^p)\rightarrow \mathrm{Hom}_{\mathcal{O}}(H^1_c(K^p),\mathcal{O})=
  \mathrm{Hom}_{\mathcal{O}}(\widehat{H}^1_c(K^p),\mathcal{O})$$
  which is $G_{\mathbb{Q}}\times G_p$-equivariant.

 \begin{lemma}\label{1.4}
 The map $h_{K^p} : \widetilde{H}^{BM}_1(K^p)\rightarrow \mathrm{Hom}_{\mathcal{O}}(\widehat{H}^1_c(K^p),\mathcal{O})$ is topological isomorphism of 
 compact $\mathcal{O}$-modules. 
 \end{lemma}
 \begin{proof}
 Since $\widetilde{H}^{BM}_1(K^p)$ is isomorphic to the completed (Borel-Moore) homology 
 $\varprojlim_{K_p}H_1(X(K_pK^p)(\mathbb{C}), \{\text{cusps}\}, \mathcal{O})$ by Remark \ref{1.3},
 the map $h_{K^p} : \widetilde{H}^{BM}_1(K^p)\rightarrow 
 \mathrm{Hom}_{\mathcal{O}}(\widehat{H}^1_c(K^p),\mathcal{O})$ sits in the following short exact sequence 
 $$0\rightarrow \mathrm{Hom}_{\mathcal{O}}(\widehat{H}^2_c(K^p)[\varpi^{\infty}], L/\mathcal{O})\rightarrow 
 \widetilde{H}^{BM}_1(K^p)\xrightarrow{h_{K^p}} \mathrm{Hom}_{\mathcal{O}}(\widehat{H}^1_c(K^p),\mathcal{O})\rightarrow 0$$
 of compact $\mathcal{O}$-modules 
 by Theorem 1.1.3 of \cite{CE12}, where $\widehat{H}^2_c(K^p)[\varpi^{\infty}]$ is the $\varpi$-power torsion parts of $\widehat{H}^2_c(K^p)$. Since one has $\widehat{H}^2_c(K^p)=0$ by Proposition 4.3.6 of \cite{Em06b}, 
 the lemma follows from this exact sequence. 
 \end{proof}
 
 We next define a Hecke action on $\widetilde{H}^{BM}_1(K^p)$ such that the isomorphism $h_{K^p}$ is Hecke equivariant. 
 To define it, we first define another Hecke operators on the cohomology $H^1(K_f)$ for each $K_f\subset \mathrm{GL}_2(\mathbb{A}_f)$. 
  For any prime $l\not=p$ which is unramified in $K_f$, we define the following 
  Hecke operators of 
  $$T'_l:=T_lS_l^{-1}, \quad S'_l:=S_l^{-1}$$
  which act on $H^1(K_f)$. For each open compact subgroup $K_{\Sigma}$ of $G_{\Sigma}$, we also define another Hecke algebra 
  $$\mathbb{T}'(K_{\Sigma})\subset \mathrm{End}_{\mathcal{O}}(H^1(K_{\Sigma}K_0^{\Sigma}))$$
  which is the sub $\mathcal{O}$-algebra generated by all the Hecke operators $T'_l, S'_l$ for all $l\not\in \Sigma$. 
  For each open compact subgroup $K_{\Sigma_0}$ of $G_{\Sigma_0, }$, we also define another big Hecke algebra
  $$\mathbb{T}'(K_{\Sigma_0}):=\varprojlim_{K_p}\mathbb{T}'(K_pK_{\Sigma_0})$$
  which faithfully acts on $\widetilde{H}^{BM}_1(K_{\Sigma_0})$.

  We remark that the actions of $T'_l, S'_l$ on the source of the duality map
   $$h_{K_{\Sigma}K_0^{\Sigma}} : H^1(K_{\Sigma}K_0^{\Sigma})(1)\rightarrow \mathrm{Hom}_{\mathcal{O}}(H^1_c(K_{\Sigma}K_0^{\Sigma}), \mathcal{O})$$ are compatible with the actions of $T_l, S_l$ on the target. 
   Since the map $h_{K_{\Sigma}K_0^{\Sigma}}$ is isomorphism after inverting $p$, this implies that the map 
   $$\mathbb{T}'(K_{\Sigma})\rightarrow \mathbb{T}_c(K_{\Sigma}) : T'_l \ \ (\text{resp}. \,S'_l) \mapsto T_l\ \  (\text{resp}. \,S_l)$$
   is a well-defined isomorphism. Hence, we also obtain the following well-defined isomorphism 
   $$\mathbb{T}'(K_{\Sigma_0})\isom \mathbb{T}_c(K_{\Sigma_0}) : T'_l \ \ (\text{resp}. \, S'_l) \mapsto T_l\ \  (\text{resp}. \, S_l)$$
   Since we recall that one also has a natural isomorphism $\mathbb{T}_c(K_{\Sigma_0})\isom \mathbb{T}(K_{\Sigma_0})$, we freely identify 
   $\mathbb{T}'(K_{\Sigma_0})$ and $\mathbb{T}(K_{\Sigma_0})$ by the composite of these isomorphisms. As a consequence, one can equip 
   $\widetilde{H}^{BM}_1(K_{\Sigma_0})$ with the structure of $\mathbb{T}(K_{\Sigma_0})$-module on which $T_l$ and $S_l$ act
   by $T'_l$ and $S'_l$ respectively. Then, the isomorphism $$h_{K_{\Sigma_0}}:=h_{K_{\Sigma_0}K_0^{\Sigma}} : \widetilde{H}^{BM}_1(K_{\Sigma_0})\isom
 \mathrm{Hom}_{\mathcal{O}}(\widehat{H}^1_c(K_{\Sigma_0}),\mathcal{O})$$ is 
 $\mathbb{T}(K_{\Sigma_0})$-linear.


 
 We next generalize the isomorphism in Lemma \ref{1.4} to the completed homology 
 $\widetilde{H}^{BM}_{1, \Sigma}$. For this, we fix a compact open subgroup $K^0_{\Sigma_0}$ of $G_{\Sigma_0}$ whose pro-order is 
 prime to $p$ (e.g. $K^0_{\Sigma_0}=\prod_{l\in \Sigma_0}\left(1+l\mathrm{M}_2(\mathbb{Z}_l)\right)$ ), and fix the Haar measure $d\mu$ on $K^0_{\Sigma_0}$ normalized so that $\int_{K^0_{\Sigma_0}}d\mu(k)=1$. 
 Then, for any smooth $\mathcal{O}$-representation $M$ of $G_{\Sigma_0}$ and any open compact subgroup 
 $K_{\Sigma_0}$ of $K^0_{\Sigma_0}$, 
 one can define a projection 
 $$p_{K_{\Sigma_0}} : M\rightarrow M^{K_{\Sigma_0}}$$
  by $$p_{K_{\Sigma_0}}(v)=[K^0_{\Sigma_0} : K_{\Sigma_0}]\int_{K_{\Sigma_0}}kvd\mu(k)$$ for any $v\in M$. 
  Let $\widetilde{M}$ be the smooth contragradient of $M$. We remark that the map 
  $$g_{K_{\Sigma_0}} : \mathrm{Hom}_{\mathcal{O}}(M^{K_{\Sigma_0}}, \mathcal{O})\rightarrow (\widetilde{M})^{K_{\Sigma_0}} : f\mapsto f\circ p_{K_{\Sigma_0}}$$
  is  isomorphism, and it is the inverse of the isomorphism 
  $$ (\widetilde{M})^{K_{\Sigma_0}}\isom \mathrm{Hom}_{\mathcal{O}}(M^{K_{\Sigma_0}}, \mathcal{O}) : f\mapsto f|_{M^{K_{\Sigma_0}}}.$$
  For any inclusion $K'_{\Sigma_0}\subseteq K_{\Sigma_0}\subseteq K^0_{\Sigma_0}$, we define a map 
  $$g_{K_{\Sigma_0}, K'_{\Sigma_0}} : \mathrm{Hom}_{\mathcal{O}}(M^{K_{\Sigma_0}}, \mathcal{O})\rightarrow \mathrm{Hom}_{\mathcal{O}}(M^{K'_{\Sigma_0}}, \mathcal{O}) : 
  f\mapsto f\circ p_{K_{\Sigma_0}}.$$ Then, this map sits in the following commutative diagram
  \begin{equation*}
\begin{CD}
\mathrm{Hom}_{\mathcal{O}}(M^{K_{\Sigma_0}}, \mathcal{O})@>g_{K_{\Sigma_0}}>> (\widetilde{M})^{K_{\Sigma_0}} \\
@Vg_{K_{\Sigma_0}, K'_{\Sigma_0}}VV  @ VV \mathrm{inclusion} V \\
\mathrm{Hom}_{\mathcal{O}}(M^{K'_{\Sigma_0}}, \mathcal{O})@>g_{K'_{\Sigma_0}}>> (\widetilde{M})^{K'_{\Sigma_0}} 
\end{CD}
\end{equation*}
We apply this diagram to the completed cohomology $\widehat{H}^1_{c, \Sigma}$ with compact support. 
Since one has $(\widehat{H}^1_{c, \Sigma})^{K_{\Sigma_0}}=\widehat{H}^1_c(K_{\Sigma_0})$ for any $K_{\Sigma_0}$, and an equality 
$$p_{K_{\Sigma_0}}|_{\widehat{H}^1_{c}(K'_{\Sigma_0})}=\frac{1}{[K_{\Sigma_0} : K'_{\Sigma_0}]}\mathrm{cor} : 
\widehat{H}^1_c(K'_{\Sigma_0})
\rightarrow \widehat{H}^1_c(K_{\Sigma_0}) $$ for any inclusion
$K'_{\Sigma_0}\subseteq K_{\Sigma_0}$, one obtains the following commutative diagram
 \begin{equation*}
\begin{CD}
\widetilde{H}^{BM}_1(K_{\Sigma_0})@>\frac{1}{[K^0_{\Sigma_0} : K_{\Sigma_0}]}h_{K_{\Sigma_0}}>>  \mathrm{Hom}_{\mathcal{O}}(\widehat{H}^1_c(K_{\Sigma_0}),\mathcal{O}) \\
@V\mathrm{res} VV  @ VV g_{K_{\Sigma_0}, K'_{\Sigma_0}} V \\
\widetilde{H}^{BM}_1(K'_{\Sigma_0})@>\frac{1}{[K^0_{\Sigma_0} : K'_{\Sigma_0}]}h_{K'_{\Sigma_0}}>>  \mathrm{Hom}_{\mathcal{O}}(\widehat{H}^1_c(K_{\Sigma_0}),\mathcal{O}) 
\end{CD}
\end{equation*}
for any inclusion $K'_{\Sigma_0}\subseteq K_{\Sigma_0}$. Therefore, the composites
$$\widetilde{H}^{BM}_1(K_{\Sigma_0}) \isom  (\widetilde{\widehat{H}^1_{c, \Sigma}})^{K_{\Sigma_0}}$$
of $\frac{1}{[K^0_{\Sigma_0} : K_{\Sigma_0}]}h_{K_{\Sigma_0}}$ with $g_{\Sigma_0} : 
\mathrm{Hom}_{\mathcal{O}}(\widehat{H}^1_c(K_{\Sigma_0}), \mathcal{O})\isom 
 (\widetilde{\widehat{H}^1_{c, \Sigma}})^{K_{\Sigma_0}}$ for all $K_{\Sigma_0}\subseteq K^0_{\Sigma_0}$ induce an $\mathcal{O}$-linear topological isomorphism 
 $$\widetilde{H}^{BM}_{1, \Sigma}\isom \widetilde{\widehat{H}^1_{c, \Sigma}}$$
 which is Hecke and $G_{\mathbb{Q}}\times G_{\Sigma}$-equaivariant. 
 
 \subsubsection{$\overline{\rho}$-part of $\widetilde{H}^{BM}_{1, \Sigma}$}
 Let $\overline{\rho} : G_{\mathbb{Q}}\rightarrow \mathrm{GL}_2(\mathbb{F})$ be an odd absolutely irreducible  continuous representation as in \S 2.1.2. 
 We set 
 $$\widetilde{H}^{BM}_1(K_{\Sigma_0})_{\overline{\rho}}:=\widetilde{H}^{BM}_1(K_{\Sigma_0})\otimes_{\mathbb{T}(K_{\Sigma_0})}
 \mathbb{T}(K_{\Sigma_0})_{\overline{\rho}}$$ for any allowable level $K_{\Sigma_0}$ with respect to $\overline{\rho}$, and 
 $$\widetilde{H}^{BM}_{1,\,\overline{\rho}, \Sigma}:=\varinjlim_{K_{\Sigma_0}} \widetilde{H}^{BM}_1(K_{\Sigma_0})_{\overline{\rho}},$$
 where $K_{\Sigma_0}$ run through all the allowable ones. Then, $\widetilde{H}^{BM}_{1,\,\overline{\rho}, \Sigma}$ is a topological 
 $\mathbb{T}_{\overline{\rho}, \Sigma}$-module with a $\mathbb{T}_{\overline{\rho}, \Sigma}$-linear continuous action of $G_{\mathbb{Q}}\times G_p$ and 
 a $\mathbb{T}_{\overline{\rho}, \Sigma}$-linear smooth action of $G_{\Sigma_0}$. We recall that one has a canonical isomorphism 
 $$\widehat{H}^1_{c, \overline{\rho}, \Sigma}\isom \widehat{H}^1_{\overline{\rho}, \Sigma}.$$ 
 Then, the isomorphism
 $$\widetilde{H}^{BM}_{1, \Sigma}\isom \widetilde{\widehat{H}^1_{c, \Sigma}}$$ induces a $\mathbb{T}_{\overline{\rho}, \Sigma}$-linear topological isomorphism 
  $$\widetilde{H}^{BM}_{1,\overline{\rho},  \Sigma}\isom \widetilde{\widehat{H}^1_{c, \overline{\rho}, \Sigma}}\isom 
  \widetilde{\widehat{H}^1_{\overline{\rho}, \Sigma}}$$
  which is $G_{\mathbb{Q},\Sigma}\times G_{\Sigma}$-euivariant. 
  
  We now assume that all the assumptions (a), (b), (c) in \S 2.2 hold. Then, Theorem \ref{1.2} and Lemma \ref{5.16} and 
 the isomorphism $\widetilde{H}^{BM}_{1,\overline{\rho},  \Sigma}\isom \widetilde{\widehat{H}^1_{\overline{\rho}, \Sigma}}$ 
  immediately imply the following corollary, which is the foundation for our construction of zeta morphisms.

   \begin{corollary}\label{1.5}
 There exists a $\mathbb{T}_{\overline{\rho}, \Sigma}[G_{\mathbb{Q}}\times G_{\Sigma}]$-linear topological isomorphism
 $$\widetilde{H}^{BM}_{1, \overline{\rho}, \Sigma}\isom (\pi_p^{\mathfrak{m}})^*\otimes_{\mathbb{T}_{\overline{\rho}, \Sigma} 
 }(\rho^{\mathfrak{m}})^*\otimes_{\mathbb{T}_{\overline{\rho}, \Sigma}}\pi^{\mathfrak{m}}_{\Sigma_0}.$$

 \end{corollary}
 We next prove that the isomorphism in this corollary is unique 
up to $(\mathbb{T}_{\overline{\rho},\Sigma})^{\times}$. 
 
\begin{lemma}\label{1.7}
Let $A$ be an object in $\mathrm{Comp}(\mathcal{O})$. For any deformation 
$\rho$ $($resp. $\pi$$)$ of $\overline{\rho}_p$ $($resp. $\overline{\pi}_p$$)$ over $A$, 
and any co-Whittaker 
$A[G_{\Sigma_0}]$-module $V$, one has 
$$\mathrm{End}^{\mathrm{cont}}_{A[G_{\mathbb{Q}_p}\times G_{\Sigma}]}(\pi^*\otimes_A\rho^*\otimes_AV)=A.$$ 
Here, we say that an $A[G_{\mathbb{Q}_p}\times G_{\Sigma}]$-linear map 
$$f : \pi^*\otimes_A\rho^*\otimes_AV\rightarrow \pi^*\otimes_A\rho^*\otimes_AV$$ is continuous if its $K_{\Sigma_0}$-fixed part 
$$f^{K_{\Sigma_0}} :  \pi^*\otimes_A\rho^*\otimes_AV^{K_{\Sigma_0}} \rightarrow \pi^*\otimes_A\rho^*\otimes_AV^{K_{\Sigma_0}}$$ 
is continuous with respect to the canonical topology 
as compact $A$-module for any open compact subgroup $K_{\Sigma_0}$ of $G_{\Sigma_0}$. 
\end{lemma} 
\begin{proof}
Since $V$ is co-Whittaker, it is easy to see that the map 
$$\mathrm{End}^{\mathrm{cont}}_{A[G_{\mathbb{Q}_p}\times G_p]}(\pi^*\otimes_A\rho^*)\rightarrow \mathrm{End}^{\mathrm{cont}}_{A[G_{\mathbb{Q}_p}\times G_{\Sigma}]}(\pi^*\otimes_A\rho^*\otimes_AV) : f\rightarrow f\otimes \mathrm{id}_V$$ is isomorphism by Lemma  \ref{5.10}. 
Since we assume that $\mathrm{End}_{\mathbb{F}[G_{\mathbb{Q}_p}]}(\overline{\rho}_p)=\mathbb{F}$, one also has 
$\mathrm{End}_{\mathbb{F}[G_p]}(\overline{\pi}_p)=\mathbb{F}$ (see for example Remark 3.3.14 of \cite{Em}). Therefore, one also has $\mathrm{End}_{A[G_{\mathbb{Q}_p}]}(\rho^*)=A$ and 
$\mathrm{End}_{A[G_p]}(\pi)=A$ for any deformations $\rho$ and $\pi$ over $A$. 
Hence, one also has $\mathrm{End}^{\mathrm{cont}}_{A[G_p]}(\pi^*)=A$ since one has a canonical isomorphism 
$\pi\isom \mathrm{Hom}_A^{\mathrm{cont}}(\pi^*, A)$. Then, it is easy to see that the equalities $\mathrm{End}_{A[G_{\mathbb{Q}_p}]}(\rho^*)=A$ and $\mathrm{End}^{\mathrm{cont}}_{A[G_p]}(\pi^*)=A$ imply the equality $$\mathrm{End}^{\mathrm{cont}}_{A[G_{\mathbb{Q}_p}\times G_p]}(\pi^*\otimes_A\rho^*)=A,$$
which proves the lemma.
\end{proof}
By this Lemma and Corollary \ref{1.5}, we immediately obtain the following corollary. 
\begin{corollary}\label{1.8}
One has the following equalities
$$\mathrm{End}_{\mathbb{T}_{\overline{\rho}, \Sigma}[G_{\mathbb{Q}}\times G_{\Sigma}]}^{\mathrm{cont}}(\widetilde{H}^{BM}_{1, \overline{\rho}, \Sigma})=\mathrm{End}_{\mathbb{T}_{\overline{\rho}, \Sigma}[G_{\mathbb{Q}_p}\times G_{\Sigma}]}^{\mathrm{cont}}(\widetilde{H}^{BM}_{1, \overline{\rho}, \Sigma})=\mathbb{T}_{\overline{\rho}, \Sigma}.$$
\end{corollary}

 Next, we furthermore assume the following conditions : 
 \begin{itemize}
 \item[$(a)'$]$\overline{\rho}_p:=\overline{\rho}|_{G_{\mathbb{Q}_p}}\not\isom  
 \chi\otimes \begin{pmatrix} 1 & * \\ 0 & \overline{\varepsilon}^{\pm}\end{pmatrix}$
 for any $\mathbb{F}$-valued character $\chi:G_{\mathbb{Q}_p}\rightarrow \mathbb{F}^{\times}$.
 \item[$(b)'$]$p\geqq 5$.
 \item[$(c) $ ]$\mathrm{End}_{\mathbb{F}[G_{\mathbb{Q}_p}]}(\overline{\rho}_p)=\mathbb{F}$.
 \end{itemize}
 (Remark that these conditions are stronger than the conditions (a), (b), (c) in \S 2.2.)
 We write $R_p:=R_{\overline{\rho}_p}$ to denote the universal deformation ring of $\overline{\rho}_p$. 
 We also write  $\widetilde{P}\in \mathfrak{C}(\mathcal{O})$ to denote the projective envelope of the dual 
 $\overline{\pi}_{1,p}^{\vee}$ of the socle $\overline{\pi}_{1,p}\subset \overline{\pi}_p$ (see Appendix B.3.2 for its definition and properties). We recall that one has 
 a canonical isomorphism 
 $$R_p\isom \mathrm{End}_{\mathfrak{C}(\mathcal{O})}(\widetilde{P})$$ of 
 compact $\mathcal{O}$-algebras 
 by which $\widetilde{P}$ is a compact $R_p$-module with $R_p$-linear continuous $G_p$-action. Since 
 $\rho^{\mathfrak{m}}|_{G_{\mathbb{Q}_p}}$ is a deformation of $\overline{\rho}_p$ over $\mathbb{T}_{\overline{\rho}, \Sigma}$, the universality of 
 the ring $R_p$ naturally gives the ring $\mathbb{T}_{\overline{\rho}, \Sigma}$ a structure of $R_p$-algebra, by which we regard any $\mathbb{T}_{\overline{\rho},\Sigma}$-module also as $R_p$-module. In particular, we regard any compact $\mathbb{T}_{\overline{\rho},\Sigma}$-module (e.g. finite generated $\mathbb{T}_{\overline{\rho},\Sigma}$-module) also as compact $R_p$-module. 
 
 \begin{corollary}\label{1.6}
 There exists an $\mathbb{T}_{\overline{\rho}, \Sigma}[G_{\mathbb{Q}}\times G_{\Sigma}]$-linear topological isomorphism 
 $$\widetilde{H}^{BM}_{1, \overline{\rho}, \Sigma}\isom (\widetilde{P}\widehat{\otimes}_{R_p}
 (\rho^{\mathfrak{m}})^*)\otimes_{\mathbb{T}_{\overline{\rho}, \Sigma}}\pi^{\mathfrak{m}}_{\Sigma_0}.$$
 
 \end{corollary}
 \begin{proof}
 Let $\pi_p:=\pi_{\overline{\rho}_p}$ be the universal deformation of $\overline{\pi}_p$ over $R_p$. 
 Since one has 
 $$\pi_p\otimes_{R_p}\mathbb{T}_{\overline{\rho},\Sigma}\isom \pi_p^{\mathfrak{m}}$$
  by the universality of $\pi_p$, one also obtains an isomorphism
  $$(\pi_p^{*})\widehat{\otimes}_{R_p}\mathbb{T}_{\overline{\rho},\Sigma}\isom (\pi_p^{\mathfrak{m}})^*$$
  by taking its $\mathbb{T}_{\overline{\rho}, \Sigma}$-linear dual. Then, the corollary follows from 
  Proposition \ref{1.5} and the isomorphism $$(\pi_p^{*})\isom \widetilde{P}$$ proved in 
  Proposition \ref{5.26}.
 \end{proof}

\subsection{Completed homology with coefficients}
To compare our zeta elements with Kato's zeta elements associated to Hecke eigen new forms, we need to know 
how to recover the usual cohomologies of modular curves with coeffients from the completed homology. To do so, 
it is convenient to consider the completed homology with coefficients. In this subsection, we first recall basic 
definitions of the completed homology with coefficients, then prove some lemmas on the way how to recover the usual cohomologies with coeffients 
from the completed homology.

Let $M$ be a left $\mathbb{Z}[\mathrm{GL}_2(\mathbb{Z}_p)]$-module. For any open compact subgroup 
$K_f$ of $\mathrm{GL}_2(\mathbb{A}_f)$ whose image by the canonical projection $\mathrm{pr}_p : \mathrm{GL}_2(\mathbb{A}_f)\rightarrow G_p$ is contained in 
$\mathrm{GL}_2(\mathbb{Z}_p)$, we define a local system
$$\mathcal{V}_M:=\mathrm{GL}_2(\mathbb{Q})\backslash \mathcal{H}^{\pm} \times \mathrm{GL}_2(\mathbb{A}_f)\times M/K_f$$
  on $Y(K_f)(\mathbb{C})$, where $h\in \mathrm{GL}_2(\mathbb{Q})$ and $k\in K_f$ act on 
 $(\tau, g, x)\in \mathcal{H}^{\pm} \times \mathrm{GL}_2(\mathbb{A}_f)\times M$ by 
 $h(\tau,g,x)k=(h\tau, hgk, \mathrm{pr}_p(k)^{-1}x)$. 
 We remark that, if $M$ is written as a filtered projective limit $M=\varprojlim_{i\in I}M_i$ of $\mathrm{GL}_2(\mathbb{Z}_p)$-modules $M_i$ ($i\in I$), 
 one has $\mathcal{V}_M=\varprojlim_{i\in I}\mathcal{V}_{M_i}$. 
 
 We write its cohomology by 
  $$H^i(K_f, \mathcal{V}_M):=H^i(Y(K_f)(\mathbb{C}), \mathcal{V}_{M}).$$
  For any open compact subgroup $K^p$ of $\mathrm{GL}_2(\mathbb{A}_f^p)$, 
  we define 
  $$\widetilde{H}^{BM}_1(K^p, \mathcal{V}_M):=\varprojlim_{K_p}H^1(K^pK_p, \mathcal{V}_M)(1),$$
  where $K_p$ run through all the open subgroups of $\mathrm{GL}_2(\mathbb{Z}_p)$. For any 
  finite set $\Sigma$ of primes containing $p$, and any open compact subgroup $K_{\Sigma_0}$ of $G_{\Sigma_0}$, we also set 
  $$\widetilde{H}^{BM}_1(K_{\Sigma_0}, \mathcal{V}_M):=\widetilde{H}^{BM}_1(K_{\Sigma_0}K_0^{\Sigma}, \mathcal{V}_M).$$

  Both $H^i(K_f, \mathcal{V}_M)$ and $\widetilde{H}^{BM}_1(K^p, \mathcal{V}_M)$ are similarly equipped with Hecke actions 
  $T_l, S_l$ corresponding to the double cosets of $\begin{pmatrix}l & 0 \\ 0 & 1\end{pmatrix}, \begin{pmatrix}l & 0 \\ 0 & l\end{pmatrix}\in G_l$ and $T'_l, S'_l$ corresponding to those of $\begin{pmatrix}l^{-1} & 0 \\ 0 & 1\end{pmatrix}, \begin{pmatrix}l^{-1}& 0 \\ 0 & l^{-1}\end{pmatrix}\in G_l$ for any prime $l\not=p$ which is unramified in $K^p$, 
  and $\widetilde{H}^{BM}_1(K^p, \mathcal{V}_M)$ is naturally a $\mathbb{Z}[\mathrm{GL}_2(\mathbb{Z}_p)]$-module.
  

  If $M$ is a compact $\mathcal{O}[[\mathrm{GL}_2(\mathbb{Z}_p)]]$-module, then we can write $M=\varprojlim_{i\in I}M_i$ such that each $M_i$ is a 
 compact $\mathcal{O}[[\mathrm{GL}_2(\mathbb{Z}_p)]]$-module which is of finite length as $\mathcal{O}$-module. Then, one has 
 $$H^j(K_f, \mathcal{V}_M)=H^j(K_f, \varprojlim_{i\in I}\mathcal{V}_{M_i})\isom \varprojlim_{i\in I}H^j(K_f, \mathcal{V}_{M_i})$$
 since each $H^j(K_f, \mathcal{V}_{M_i})$ is an $\mathcal{O}$-module of finite length.  One also has
 $$\widetilde{H}^{BM}_1(K^p, \mathcal{V}_M)\isom \varprojlim_{K_p, \,i\in I}H^1(K^pK_p, \mathcal{V}_{M_i})(1)\isom \varprojlim_{ i\in I}
 \widetilde{H}^{BM}_1(K^p, \mathcal{V}_{M_i}).$$
 In particular, $\widetilde{H}^{BM}_1(K^p, \mathcal{V}_M)$ is naturally a compact $\mathcal{O}[[\mathrm{GL}_2(\mathbb{Z}_p)]]$-module. 
 
 For any compact $\mathcal{O}[[\mathrm{GL}_2(\mathbb{Z}_p)]]$-module $M$, $M\widehat{\otimes}_{\mathcal{O}}\widetilde{H}^{BM}_1(K^p)$ is
  the completed tensor product of compact $\mathcal{O}$-modules $M$ and $\widetilde{H}^{BM}_1(K^p)$. 
  $M\widehat{\otimes}_{\mathcal{O}}\widetilde{H}^{BM}_1(K^p)$ is equipped with
   the diagonal $\mathrm{GL}_2(\mathbb{Z}_p)$-action, and the Hecke actions defined by $U_l(u\widehat{\otimes}v)=u\widehat{\otimes}U_l(v)$ 
   for $U_l=T_l, S_l, T'_l, S'_l$.

  \begin{lemma}\label{1.9}
  For any compact $\mathcal{O}[[\mathrm{GL}_2(\mathbb{Z}_p)]]$-module $M$, 
 there exists a canonical Hecke equivariant isomorphism 
 $$M\widehat{\otimes}_{\mathcal{O}}\widetilde{H}^{BM}_1(K^p)\isom 
 \widetilde{H}^{BM}_1(K^p, \mathcal{V}_M)$$
  of compact $\mathcal{O}[[\mathrm{GL}_2(\mathbb{Z}_p)]]$-modules.
 
 \end{lemma}
 \begin{proof}We write $M=\varprojlim_{i\in I}M_i$ such that each $M_i$ is a compact 
 $\mathcal{O}[[\mathrm{GL}_2(\mathbb{Z}_p)]]$-module which is of finite length as $\mathcal{O}$-module. 
 Since one has canonical isomorphisms
 $$M\widehat{\otimes}_{\mathcal{O}}\widetilde{H}^{BM}_1(K^p)\isom 
 \varprojlim_{i\in I}M_i\otimes_{\mathcal{O}}\widetilde{H}^{BM}_1(K^p)
 \isom \varprojlim_{i\in I}M_i\widehat{\otimes}_{\mathcal{O}}\widetilde{H}^{BM}_1(K^p),$$
 and 
 $$\widetilde{H}^{BM}_1(K^p, \mathcal{V}_M)\isom \varprojlim_{i\in I}\widetilde{H}^{BM}_1(K^p, \mathcal{V}_{M_i}),$$
  it suffices to show the lemma for $M$ a compact $\mathcal{O}[[\mathrm{GL}_2(\mathbb{Z}_p)]]$-module which is of finite length as $\mathcal{O}$-module. 

 In this case, the action of $K_p$ on $M$ is trivial for any sufficiently small open subgroup $K_p$ of $\mathrm{GL}_2(\mathbb{Z}_p)$. For any such $K_p$, the canonical 
 map 
 $$M\otimes_{\mathcal{O}}H^1(K^pK_p)\rightarrow H^1(K^pK_p, \mathcal{V}_M)$$
 is isomorphism since $H^2(K^pK_p)$ is zero. 
 Taking its limit with respect to all such $K_p$, we obtain the topological isomorphism 
 $$M\widehat{\otimes}_{\mathcal{O}}\widetilde{H}^{BM}_1(K^p)\isom M\otimes_{\mathcal{O}}\widetilde{H}^{BM}_1(K^p)\isom 
 \widetilde{H}^{BM}_1(K^p, \mathcal{V}_M).$$
 The $\mathrm{GL}_2(\mathbb{Z}_p)$ and Hecke equivariance of the isomorphism is clear from the definition. 
 
 \end{proof}
 
 By the isomorphism in this lemma the Hecke actions $T'_l$ and $S'_l$ for all $l\not\in \Sigma$ on 
 $\widetilde{H}^{BM}_1(K_{\Sigma_0}, \mathcal{V}_M)$ equip this group with a structure of 
compact $\mathbb{T}'(K_{\Sigma_0})$-module such that the isomorphism in this lemma becomes $\mathbb{T}'(K_{\Sigma_0})$-linear. 
 
We mainly apply this lemma to the following special case. Let $W$ be an $E$-linear algebraic representation of $\mathrm{GL}_{2/\mathbb{Q}_p}$ 
on a finite dimensional $E$-vector space. 
Let 
$W_0$ be an $\mathcal{O}$-lattice of $W$ which is $\mathrm{GL}_2(\mathbb{Z}_p)$-stable. We remark that $W_0$ is naturally a compact 
$\mathcal{O}[[\mathrm{GL}_2(\mathbb{Z}_p)]]$-module with respect to the $\varpi$-adic topology on $W_0$, hence $W=\varinjlim_{n\geqq 0}\frac{1}{\varpi^n}W_0$ is also an $\mathcal{O}[[\mathrm{GL}_2(\mathbb{Z}_p)]]$-module.

\begin{lemma}\label{1.10}
 For any open subgroup $K_p$ of $\mathrm{GL}_2(\mathbb{Z}_p)$, there exists 
 a canonical $E$-linear isomorphism 
 $$W\otimes_{\mathcal{O}[[K_p]]}\widetilde{H}^{BM}_1(K^p)\isom H^1(K^pK_p, \mathcal{V}_{W})(1),$$
 where we regard $W$ as a right $\mathcal{O}[[K_p]]$-module by $v\cdot g:=g^{-1}v$ for $g\in K_p$ and $v\in W$. 
 \end{lemma}
 \begin{proof}
 By Lemma \ref{1.9}, one has a canonical map 
 $$W\otimes_{\mathcal{O}}\widetilde{H}^{BM}_1(K^p)\isom \widetilde{H}^{BM}_1(K^p, \mathcal{V}_{W_0})[1/p]
 \xrightarrow{\mathrm{can}} H^1(K^pK_p, \mathcal{V}_W)(1).$$
 Hence, it suffices to show that the canonical projection 
 $$\widetilde{H}^{BM}_1(K^p, \mathcal{V}_{W_0})[1/p]
 \xrightarrow{\mathrm{can}} H^1(K^pK_p, \mathcal{V}_W)(1)$$
  induces an $E$-linear isomorphism
 $$\widetilde{H}^{BM}_1(K^p, \mathcal{V}_{W_0})_{K_p}[1/p]\isom H^1(K^pK_p, \mathcal{V}_W)(1),$$
 where $\widetilde{H}^{BM}_1(K^p, \mathcal{V}_{W_0})_{K_p}$ is the quotient of $\widetilde{H}^{BM}_1(K^p, \mathcal{V}_{W_0})$ 
 by the closure of the sub $\mathcal{O}$-module generated by $\{(g-1)v\, |\, g\in K_p, v\in \widetilde{H}^{BM}_1(K^p, \mathcal{V}_{W_0})\}$. 

 For each $n\geqq 1$, we write $(W_0/\varpi^nW_0)^*$ (resp. $W^*_0$) to denote the $\mathcal{O}/\varpi^n$-linear dual of $W_0/\varpi^nW_0$ (resp. $\mathcal{O}$-linear 
 dual of $W_0$). We set 
 $$H^1_c(K^p, \mathcal{V}_{(W_0/\varpi^nW_0)^*}):=\varinjlim_{K_p}H_c^1(K^pK_p, \mathcal{V}_{(W_0/\varpi^nW_0)^*}),$$
 where the limit is taken over all the open compact subgroups $K_p$ of $\mathrm{GL}_2(\mathbb{Z}_p)$, and 
 $$\widehat{H}_c^1(K^p, \mathcal{V}_{W^*_0}):=\varprojlim_{n\geqq 1}H_c^1(K^p, \mathcal{V}_{(W_0/\varpi^nW_0)^*}).$$
 Since each $H_c^2(K^pK_p)$ is $\mathcal{O}$-torsion free, one obtains a canonical $\mathcal{O}[\mathrm{GL}_2(\mathbb{Z}_p)]$-linear isomorphism 
 $$W^*_0\otimes_{\mathcal{O}}\widehat{H}_c^1(K^p)
 \isom \widehat{H}_c^1(K^p,\mathcal{V}_{W_0^*})$$ in a similar way as in Lemma \ref{1.9}. 
 Therefore, the duality isomorphism $$\widetilde{H}^{BM}_1(K^p)\isom 
 \mathrm{Hom}_{\mathcal{O}}(\widehat{H}_c^1(K^p), \mathcal{O})$$ induces a topological $\mathcal{O}[[\mathrm{GL}_2(\mathbb{Z}_p)]]$-linear  isomorphism 
 $$\widetilde{H}^{BM}_1(K^p, \mathcal{V}_{W_0})\isom 
 \mathrm{Hom}_{\mathcal{O}}(\widehat{H}_c^1(K^p, \mathcal{V}_{W_0^*}), \mathcal{O})$$ as the composite 
 
 \begin{multline*}
 \widetilde{H}^{BM}_1(K^p, \mathcal{V}_{W_0})\isom W_0\otimes_{\mathcal{O}}\widetilde{H}^{BM}_1(K^p)
 \isom W_0\otimes_{\mathcal{O}} \mathrm{Hom}_{\mathcal{O}}(\widehat{H}_c^1(K^p), \mathcal{O})\\
 \isom 
 \mathrm{Hom}_{\mathcal{O}}(W_0^*\otimes_{\mathcal{O}} \widehat{H}_c^1(K^p), \mathcal{O})
 \isom  \mathrm{Hom}_{\mathcal{O}}(\widehat{H}_c^1(K^p,\mathcal{V}_{W_0^*}), \mathcal{O}).
\end{multline*}
Since $\widehat{H}_c^1(K^p,\mathcal{V}_{W_0^*})$ is $\mathcal{O}$-torsion free, $\varpi$-adically completed and separated 
$\mathcal{O}$-module, this $\mathcal{O}[[\mathrm{GL}_2(\mathbb{Z}_p)]]$-linear isomorphism induces an $\mathcal{O}$-linear isomorphism 
$$\widetilde{H}^{BM}_1(K^p, \mathcal{V}_{W_0})_{K_p}\isom \mathrm{Hom}_{\mathcal{O}}(\widehat{H}_c^1(K^p,\mathcal{V}_{W_0^*})^{K_p}, \mathcal{O})$$
for each $K_p$. Since one has a canonical isomorphism
$$H^1_c(K^pK_p, \mathcal{V}_{W^*})\isom \widehat{H}_c^1(K^p,\mathcal{V}_{W^*})^{K_p}$$
 by (4.3.7) of \cite{Em06a} for the $E$-linear dual $W^*$ of $W$, we obtain the following 
 isomorphim
 $$\widetilde{H}^{BM}_1(K^p, \mathcal{V}_{W_0})_{K_p}[1/p]\isom 
 \mathrm{Hom}_E(H^1_c(K^pK_p, \mathcal{V}_{W^*}), E)\isom H^1(K^pK_p, \mathcal{V}_{W})(1),$$
 where the last isomorphism follows from Poincar\'e duality. 
 \end{proof}

We next consider a left $E[\mathrm{GL}_2(\mathbb{Z}_p)]$-module $M$ of the form $M=\tau\otimes_EW$, where 
$W$ is an $E$-linear algebraic representation of $\mathrm{GL}_{2/\mathbb{Q}_p}$ on a finite dimensional $E$-vector space, and 
$\tau$ is a smooth $E$-representation of $\mathrm{GL}_2(\mathbb{Z}_p)$ which is finite dimensional over $E$. Let 
$W_0$ (resp. $\tau_0$) be a $\mathrm{GL}_2(\mathbb{Z}_p)$-stable $\mathcal{O}$-lattice of $W$ (resp. $\tau$). 
Then $M_0:=\tau_0\otimes_{\mathcal{O}}W_0$ is naturally a compact $\mathcal{O}[[\mathrm{GL}_2(\mathbb{Z}_p)]]$-module, and 
$M=\varinjlim_{n\geqq 0}\frac{1}{\varpi^n}M_0$ is also an $\mathcal{O}[[\mathrm{GL}_2(\mathbb{Z}_p)]]$-module.

\begin{lemma}\label{1.11}For any open subgroup $K_p$ of $\mathrm{GL}_2(\mathbb{Z}_p)$, 
there exists a canonical $E$-linear isomorphism 
$$M\otimes_{\mathcal{O}[[K_p]]}\widetilde{H}^{BM}_1(K^p)\isom H^1(K^pK_p, \mathcal{V}_{M})(1),$$
 where we regard $M$ as a right $\mathcal{O}[[K_p]]$-module by $v\cdot g:=g^{-1}v$ for $g\in K_p$ and $v\in M$. 
\end{lemma}
\begin{proof}
We take an open normal subgroup $K'_p$ of $K_p$ which trivially acts on $\tau$. 
Then, $\mathcal{V}_{\tau}$ is a constant local systme on $Y(K^pK'_p)(\mathbb{C})$, and one has 
\begin{multline*}M\otimes_{\mathcal{O}[[K'_p]]}\widetilde{H}^{BM}_1(K^p)
=\tau\otimes_E(W\otimes_{\mathcal{O}[[K'_p]]}\widetilde{H}^{BM}_1(K^p))\\
\isom \tau\otimes_E H^1(Y(K^pK'_p), \mathcal{V}_W)(1)
\isom H^1(Y(K^pK'_p), \mathcal{V}_M)(1)
\end{multline*}
by Lemma \ref{1.10}, Taking its $K_p/K'_p$-coinvariant, this isomorphism induces an isomorphism
$$M\otimes_{\mathcal{O}[[K_p]]}\widetilde{H}^{BM}_1(K^p)
\isom H^1(Y(K^pK'_p), \mathcal{V}_M)_{K_p/K'_p}(1)\isom H^1(Y(K^pK_p), \mathcal{V}_M)(1),$$
where the last isomorphism follows from the (dual of) Hochschild-Serre spectral sequence 
with respect to the Galois cover $Y(K^pK'_p)(\mathbb{C})\rightarrow Y(K^pK_p)(\mathbb{C})$ since 
one has $H_1(K_p/K'_p, N)=0$ for any $\mathbb{Q}[K_p/K'_p]$-module $N$.
\end{proof}


\section{Equivariant zeta morphisms for completed homology of modular curves}

\subsection{Kato's zeta elements and its zeta values formula}

\subsubsection{Modular curves over $\mathbb{Z}[1/N]$}
For any integer $N\geqq 1$, we define an open compact subgroup $K(N)$ of $\mathrm{GL}_2(\widehat{\mathbb{Z}})$ by 
$$K(N):=\mathrm{Ker}(\mathrm{GL}_2(\widehat{\mathbb{Z}})\xrightarrow{\mathrm{can}} \mathrm{GL}_2(\mathbb{Z}/N\mathbb{Z})),$$
and set 
$$Y(N)(\mathbb{C}):= Y(K(N))(\mathbb{C})=\mathrm{GL}_2(\mathbb{Q})\backslash \mathcal{H}^{\pm} \times (\mathrm{GL}_2(\mathbb{A}_f)/K(N)).$$ 
For each $N\geqq 3$, $Y(N)(\mathbb{C})$ canonically descends to the smooth curve $Y(N)$ over $\mathbb{Z}[1/N]$, which is the moduli scheme of the triples 
$(E, e_1,e_2)$ with $E$ an elliptic curve and $e_1,e_2$ sections of $E[N]$ 
such that $(\mathbb{Z}/N\mathbb{Z})^2\rightarrow E[N] : (a,b)\mapsto ae_1+be_2$ is 
isomorphism. We write $$f : E^{\mathrm{univ}}\rightarrow Y(N)$$ to denote the universal elliptic curve over $Y(N)$. 
One has the canonical compactification $X(N)$ of $Y(N)$ which is also defined over $\mathbb{Z}[1/N]$. 

For each $k\in \mathbb{Z}_{\geqq 2}$, we write $\mathcal{V}_{k/A}$ to denote the local system on 
$Y(N)(\mathbb{C})$ associated to the $\mathbb{Z}_p[\mathrm{GL}_2(\mathbb{Z}_p)]$-module 
$\mathrm{Sym}^{k-2}(A^2)$ for any $\mathbb{Z}_p$-algebra $A$, where $A^2$ is the standard representation of $\mathrm{GL}_2(\mathbb{Z}_p)$ over $A$. 
For $A=\mathcal{O}, E, \mathcal{O}/\varpi^n\mathcal{O}$, etc., 
the local 
system $\mathcal{V}_{k/A}$ naturally descends to the \'etale $p$-adic sheaf $\mathrm{Sym}^{k-2}(R^1f_*(A))$ on $$Y(N)\otimes_{\mathbb{Z}}
\mathbb{Z}[1/p]:=Y(N)\times_{\mathrm{Spec}(\mathbb{Z})}
\mathrm{Spec}(\mathbb{Z}[1/p])$$ 
 for the structure morphism $f : E^{\mathrm{univ}}\otimes_{\mathbb{Z}}
\mathbb{Z}[1/p]\rightarrow Y(N)\otimes_{\mathbb{Z}}
\mathbb{Z}[1/p]$, which we also denote by the same letter $\mathcal{V}_{k/A}$. 
 
 In our article, we mainly consider $\mathcal{V}_{k/A}$ and its $A$-linear dual $\mathcal{V}^*_{k/A}$, and write $H^i(Y(N), \mathcal{F})$
  to denote the cohomology $$H^i(K(N), \mathcal{F})=H^i(Y(N)(\mathbb{C}), \mathcal{F})
  \isom H^i_{\text{\'et}}(Y(N)_{\overline{\mathbb{Q}}}, \mathcal{F})$$ 
 for $\mathcal{F}=\mathcal{V}_{k/A}, \mathcal{V}^*_{k/A}$. 
  For $k=2$, then $\mathcal{V}_{k/A}=\mathcal{V}^*_{k/A}=A$, and we denote it by $$H^i(Y(N))_{A}:=H^i(Y(N), A)=H^i(K(N))_A$$ for simplicity.

We also 
write $H^i_{\text{\'et}}(Y(N), \mathcal{V}_{k/A})$ to denote the \'etale cohomology of $Y(N)$ over $\allowbreak\mathrm{Spce}(\mathbb{Z}[1/N])$, 
$K_i(Y(N))$ to denote the Qullen's $K$-group of $Y(N)$. 

We use similar notations $$H^i(X(N), j_*\mathcal{F}),\ \  H^i(X(N))_{A}\ \ \text{ and }\ \ H^i_{\text{\'et}}(X(N), j_*\mathcal{V}_{k/A})$$ for $X(N)$, where 
$j : Y(N)\hookrightarrow X(N)$ is the canonical inclusion. 

We set $\omega_{Y(N)}:=f_*\Omega^1_{E^{\mathrm{univ}}/Y(N)}$, which is a locally free $\mathcal{O}_{Y(N)}$-module of rank one. We also 
set $\omega_{X(N)}:=\overline{f}_*\Omega^1_{\overline{E}^{\mathrm{univ}}/X(N)}$, which is also a locally free $\mathcal{O}_{X(N)}$-module of rank one, where 
$\overline{f} : \overline{E}^{\mathrm{univ}}\rightarrow X(N)$ is the smooth N\'eron model of $E^{\mathrm{univ}}$ over $X(N)$. For each $k\in\mathbb{Z}_{\geqq2}$ and for any $\mathbb{Z}[1/N]$-algebra $A$, we write 
$$S_k(N)_{A}:=\Gamma(X(N)_{A}, \omega_{X(N)}^{\otimes(k-2)}\otimes_{\mathcal{O}_{X(N)}}\Omega^1_{X(N)_{A}/A})$$ to denote the space of cusp forms of wight $k$ on $X(N)_{A}$, 
 and write
$$M_k(N)_{A}:=\Gamma(X(N)_A,
\omega_{X(N)}^{\otimes (k-2)}\otimes_{\mathcal{O}_{X(N)}}\Omega^1_{X(N)_{A}/A}(\mathrm{log}\{\mathrm{cusps}\}))$$ to denote the space of modular forms of weight $k$ on $X(N)_{A}$,
where $$\Omega^1_{X(N)_{A}/A}(\mathrm{log}\{\mathrm{cusps}\})$$ is the sheaf of differential $1$-forms which may have logarithmic poles along the cusps $\{\text{cusps}\}:=X(N)\setminus Y(N)$. 

In \cite{Ka04} and \cite{FK}, they consider a left action of $\mathrm{GL}_2(\mathbb{Z}/N\mathbb{Z})$ on $Y(N)$ and $X(N)$. Precisely, on $Y(N)$, 
the isomorphism 
$g : Y(N)\rightarrow Y(N)$ induced by an element $g=\begin{pmatrix}a& b\\ c & d\end{pmatrix}\in \mathrm{GL}_2(\mathbb{Z}/N\mathbb{Z})$ sends the isomorphism class of 
$(E,e_1,e_2)$ to the class of $(E, e_1', e_2')$ where 
$$\begin{pmatrix}e_1'\\e_2'\end{pmatrix}=\begin{pmatrix}a & b \\ c & d\end{pmatrix}\begin{pmatrix}e_1\\e_2\end{pmatrix}.$$

For each divisors $M\geqq 3$ and $m\geqq 1$ of $N$, one has a finite \'etale morphism 
$$Y(N)\rightarrow Y(M)\otimes_{\mathbb{Z}}\mathbb{Z}[1/N]$$
induced by $(E, e_1, e_2)\mapsto (E, \frac{N}{M}e_1, \frac{N}{M}e_2),$
and has a morphism 
$$Y(N)\rightarrow \mu_{m/\mathbb{Z}[1/N]}^0$$ 
induced by
$(E, e_1, e_2)\mapsto \langle e_1, e_2\rangle^{N/m}$, where 
$\langle\, ,\,\rangle : E[N]\times E[N]\rightarrow \mu_{N}$ is the Weil pairing and 
$\mu_{m/\mathbb{Z}[1/N]}^0$ is the open and closed subscheme of $\mu_{m/\mathbb{Z}[1/N]}$ defined over $\mathbb{Z}[1/N]$ consisting of the primitive $m$-th roots of unity. As the product of these maps, one obtains a finite morphism 
$$Y(N)\rightarrow Y(M)\otimes_{\mathbb{Z}}\mu_{m/\mathbb{Z}[1/N]}^0:=Y(M)\times_{\mathrm{Spec}(\mathbb{Z})}\mu_{m/\mathbb{Z}[1/N]}^0$$ 
for each divisors $M\geqq 3$ and $m\geqq 1$ of $N$.

\subsubsection{Remarks on Hecke operators}In \cite{Ka04} and \cite{FK}, they consider two types of Hecke operators denoted by 
$$T(n) \ \ \text{ and } \ \ T'(n)=T^*(n)$$ for any integer $n\geqq 1$ such that 
$(n, N)=1$, which act on the cohomology groups $H^i(X(N), j_* \mathcal{V}_{k/A}),  H^i(Y(N),  \mathcal{V}_{k/A})$, the spaces of modular forms 
$S_k(N)_{A}, M_k(N)_{A}$, and the $K$-groups $K_i(X(N))$, $K_i(Y(N))$, etc. 

In \S2.1.2, for any prime $l$ such that $(l,N)=1$, we define another
 Hecke operators $T_l, T'_l, S_l, S'_l$ acting on $H^1(Y(N))_A$
 corresponding to the double cosets of $\begin{pmatrix}l & 0 \\ 0 & 1\end{pmatrix}, \begin{pmatrix}l^{-1}& 0 \\ 0 & 1\end{pmatrix}$
, $\begin{pmatrix}l & 0 \\ 0 & l\end{pmatrix}, \begin{pmatrix}l^{-1} & 0 \\ 0 & l^{-1}\end{pmatrix}\in G_l$ 
respectively.  As a generalization of these actions, we consider Hecke operators $T_l$ and $S_l$ (resp. $T'_l$ and $S'_l$) acting on 
$H^1(Y(N), \mathcal{V}_{k/A})$  (resp. $H^1(Y(N), \mathcal{V}^*_{k/A})$) for any $k\geqq 2$ corresponding to 
the double cosets of $\begin{pmatrix}l & 0 \\ 0 & 1\end{pmatrix}$ and $\begin{pmatrix}l& 0 \\ 0 & l\end{pmatrix}$ 
(resp. $\begin{pmatrix}l^{-1} & 0 \\ 0 & 1\end{pmatrix}$ and $\begin{pmatrix}l^{-1} & 0 \\ 0 & l^{-1}\end{pmatrix}$) in $G_l$ 
respectively. 

 For our purpose, we need to know the precise relation between these Hecke operators. 
 We first remark that one has the following equalities of the actions on the cohomology $H^1(Y(N), \mathcal{V}_{k/A})$ : 
\begin{equation}\label{aa}T_l=T(l)\begin{pmatrix}l & 0 \\ 0 & 1\end{pmatrix}^*=T^*(l)\begin{pmatrix}1 & 0 \\ 0 & l\end{pmatrix}^*,
\end{equation}
where the action $\begin{pmatrix}l & 0 \\ 0 & 1\end{pmatrix}^*$ is the right action on $H^i(Y(N), \mathcal{V}_{k/A})$ induced by Kato's and Fukaya-Kato's left action 
of $\begin{pmatrix}l & 0 \\ 0 & 1\end{pmatrix}\in \mathrm{GL}_2(\mathbb{Z}/N\mathbb{Z})$ 
on the space $Y(N)$ (see the sentence before (4.9.3) of \cite{Ka04}),
 and 
\begin{equation}\label{bb}S_l=l^{k-2}\begin{pmatrix}l & 0 \\ 0 & l\end{pmatrix}^*.
\end{equation}

We furthermore need to compare these actions with the actions of $T'_l$ and $S'_l$ on $H^1(Y(N), \mathcal{V}^*_{k/A})$. 
We set $e_1=\begin{pmatrix}1\\ 0\end{pmatrix}, e_2=\begin{pmatrix}0 \\ 1\end{pmatrix}\in A^2$, and 
$\mathrm{det}_{/A}:=A$ on which $\mathrm{GL}_2(\mathbb{Z}_p)$ acts by the determinant. Then, 
one has a $\mathrm{GL}_2(\mathbb{Z}_p)$-equivariant pairing 
 $$(\, ,\, ) : A^2\times A^2\xrightarrow{(x, y)\mapsto x\wedge y} Ae_1\wedge e_2\isom \mathrm{det}_{/A},$$ 
 where the isomorphism is defined by $e_1\wedge e_2\mapsto 1$. For each $k\geqq 2$, it naturally induces a pairing 
 $$\mathrm{Sym}^{k-2}(A^2)\times \mathrm{Sym}^{k-2}(A^2)\rightarrow \mathrm{det}^{\otimes (k-2)}.$$
 This induces an
 isomorphism $$\mathrm{Sym}^{k-2}(A^2)\isom \mathrm{Sym}^{k-2}(A^2)^*\otimes \mathrm{det}^{\otimes (k-2)}$$ 
 of $A[\mathrm{GL}_2(\mathbb{Z}_p)]$-modules, 
 which also induces an isomorphism 
 \begin{equation}\label{aa1}
 \mathcal{V}_{k/A}\isom \mathcal{V}^*_{k/A}(2-k)
 \end{equation}
 since the character $\mathrm{det}_{/A}$ corresponds to the Tate twist $A(-1)$, by which we identify the both sides.
This isomorphism induces an isomorphism 
\begin{equation}\label{aa2}
H^1(Y(N), \mathcal{V}_{k/A})\isom H^1(Y(N), \mathcal{V}^*_{k/A})(2-k).
\end{equation}
Using this isomorphism, we equip actions of $T'_l$ and $S'_l$ on $H^1(Y(N), \mathcal{V}_{k/A})$ by 
$$T'_l(v\otimes e_{2-k}):=T'_l(v)\otimes e_{2-k},\ \ S'_l(v\otimes e_{2-k}):=S'_l(v)\otimes e_{2-k},$$
here $e_{2-k}$ is a base of $A(2-k)$.

\begin{lemma}\label{2.01}
On has equalities 
$$T'_l=T^*(l)\begin{pmatrix}l^{-1}& 0 \\ 0 & 1\end{pmatrix}^*=l^{k-2}T_lS_l^{-1}$$
and 
$$S'_l=l^{k-2}\begin{pmatrix}l^{-1}& 0 \\ 0 & l^{-1}\end{pmatrix}^*=l^{2(k-2)}S_l^{-1}$$
as operators acting on $H^1(Y(N), \mathcal{V}_{k/A})$. 
\end{lemma}
\begin{proof}
The equalities $$T^*(l)\begin{pmatrix}l^{-1}& 0 \\ 0 & 1\end{pmatrix}^*=l^{(k-2)}T_lS_l^{-1}$$ and 
$$l^{k-2}\begin{pmatrix}l^{-1}& 0 \\ 0 & l^{-1}\end{pmatrix}^*=l^{2(k-2)}S_l^{-1}$$ immediately follows from 
the equalities (\ref{aa}) and (\ref{bb}). The equalities $T'_l=l^{(k-2)}T_lS_l^{-1}$ and 
$S'_l=l^{2(k-2)}S_l^{-1}$ follows from the fact that one has 
$$g(v\otimes e_{2-k})=|\mathrm{det}(g)|_l^{(2-k)}g(v)\otimes e_{2-k}$$
for every $v\in H^1(Y(N), \mathcal{V}^*_{k/A})$ and $g\in G_l$, where 
$|-|_l : \mathbb{Q}_l^{\times}\rightarrow \mathbb{Q}_{>0}$ is the absolute value such that $|l|_l=1/l$. 

\end{proof}

We similarly define actions of $T_l, S_l$  and $T'_l, S'_l$ on $S_k(N)_A$ and $M_k(N)_A$ by the same formula : 
$$T_l:=T^*(l)\begin{pmatrix}1 & 0 \\ 0 & l\end{pmatrix}^*, \ \ S_l=l^{k-2}\begin{pmatrix}l& 0 \\ 0 & l\end{pmatrix}^*$$
and 
$$T'_l:=T^*(l)\begin{pmatrix}l^{-1}& 0 \\ 0 & 1\end{pmatrix}^*, \ \ S'_l=l^{k-2}\begin{pmatrix}l^{-1}& 0 \\ 0 & l^{-1}\end{pmatrix}^*.$$

We write $\mathbb{T}_k(N)_A$ (resp. $\mathbb{T}'_k(N)_A$) to denote the sub $A$-algebra of  
$$\mathrm{End}_A(H^1(Y(N), \mathcal{V}_{k/A}))$$ generated by 
$T_l$ and $S_l$ (resp. $T'_l$ and $S'_l$) for all $l\not|Np$. By definition of $T'_l$ and $S'_l$ on $H^1(Y(N), \mathcal{V}_{k/A})$, we remark that one can also define $\mathbb{T}'_k(N)_A$ as a sub $A$-algebra 
of $\mathrm{End}_A(H^1(Y(N), \mathcal{V}^*_{k/A}))$ generated by $T'_l$ and $S'_l$ for all $l\not|Np$.

\subsubsection{Zeta elements in $K_2$}
For each integer $c\geqq 2$ such that $(6,c)=1$, Kato  (\cite{Ka04}, Proposition 1.3) defined an element 
$${}_c\,\theta_E\in \Gamma(E\setminus E[c], \mathcal{O})^{\times}$$ for any 
elliptic curve $E$ which is uniquely characterized by the equalities 
$$\mathrm{div}({}_c\,\theta_E)=E[c]-c^2[0]$$
and $$N_a({}_c\,\theta_E|_{E\setminus E[ac]})={}_c\,\theta_E$$ for any integer $a$ which is prime to $c$, where the map 
$$N_a : \Gamma(E\setminus E[ac], \mathcal{O})^{\times}\rightarrow \Gamma(E\setminus E[c], \mathcal{O})^{\times}$$ is the norm map 
with respect to the multiplication map $[a] : E\setminus E[ac]\rightarrow E\setminus E[c]$. 

For each $N\geqq 3$, $(\alpha, \beta)\in \left(\frac{1}{N}\mathbb{Z}/\mathbb{Z}\right)^2\setminus \{(\overline{0}, \overline{0})\}$, and $c\geqq 2$ such that $(c, 6N)=1$, 
we define an element $${}_c g_{\alpha, \beta}:=\iota^*_{\alpha, \beta}({}_c\, \theta_{E^{\mathrm{univ}}})\in \Gamma(Y(N), \mathcal{O})^{\times}$$ using 
the element ${}_c\, \theta_{E^{\mathrm{univ}}}$ for the universal elliptic curve $E^{\mathrm{univ}}$ over $Y(N)$, where $\iota_{\alpha, \beta}$ 
is the section
$$\iota_{\alpha, \beta}:=(N\alpha) e_1+(N\beta) e_2 : Y(N)\hookrightarrow E^{\mathrm{univ}}\setminus E^{\mathrm{univ}}[c].$$
The element ${}_c g_{\alpha, \beta}$ satisfies the equality 
$$\sigma^*({}_c g_{\alpha, \beta})={}_c g_{\alpha', \beta'}$$ for any 
$\sigma=\begin{pmatrix}a& b\\c& d\end{pmatrix}\in \mathrm{GL}_2(\mathbb{Z}/N\mathbb{Z})$, where $(\alpha', \beta')$ is defined by 
$$(\alpha', \beta')=(\alpha,\beta)\begin{pmatrix}a& b\\c& d\end{pmatrix}.$$ 

For each $N\geqq 3$, $c, d\geqq 2$ such that $(cd, 6N)=1$, we define an element 
$${}_{c,d}z_N:=\{{}_cg_{1/N, 0}, {}_dg_{0,1/N}\}\in K_2(Y(N))$$ as the image of the element 
$$({}_cg_{1/N, 0}, {}_dg_{0,1/N})\in (\Gamma(Y(N), \mathcal{O})^{\times})^2$$ by the symbol map 
$\{-,-\} : (\Gamma(Y(N), \mathcal{O})^{\times})^2\rightarrow K_2(Y(N))$. This element ${}_{c,d}z_N$ is denoted by the notation ${}_{c,d}z_{N,N}$ in \cite{Ka04}. 
In our article, we use these elements only for $c=d$, hence we write $${}_cz_N:={}_{c,c}z_N$$ for simplicity. 

For each $M\geqq 3$, $c\geqq 2$ such that $(c, 6M)=1$, and any divisor $N\geqq 3$ of $M$, we recall the formula
 (\cite{Ka04}, Proposition 2.3, 2.4) of the image
$$\mathrm{Norm}({}_cz_{M})\in K_2(Y(N)\otimes_{\mathbb{Z}}\mathbb{Z}[1/M])$$
of the element ${}_cz_M$ by the norm map $$\mathrm{Norm} : K_2(Y(M))\rightarrow K_2(Y(N)\otimes_{\mathbb{Z}}\mathbb{Z}[1/M])$$ 
with respect to the finite flat map $Y(M)\rightarrow Y(N)\otimes_{\mathbb{Z}}\mathbb{Z}[1/M]$ over $\mathrm{Spec}(\mathbb{Z}[1/M])$. 
We write $\mathrm{prime}(M)$ to denote the set of prime divisors of $M$. 

If $\mathrm{prime}(M)=\mathrm{prime}(N)$, one has the equality 
\begin{equation}\label{e1}
\mathrm{Norm}({}_cz_{M})={}_cz_{N}.
\end{equation}

If $M=Nl$ for a prime $l$ such that $(l, N)=1$, one has the equality 
\begin{equation}\label{e2}
\mathrm{Norm}({}_cz_{Nl})=\left(1-T^*(l)\begin{pmatrix}1/l& 0 \\ 0 & 1\end{pmatrix}^*+l\begin{pmatrix}1/l& 0 \\ 0 & 1/l\end{pmatrix}^*\right){}_cz_{N}.
\end{equation}

\subsubsection{Zeta elements in Galois cohomology}
Let $N\geqq 3$, $c\geqq 2$ be integers such that $(c, 6N)=1$. Here, we write 
$$H^1(Y(N)):=H^1(Y(N))_{\mathbb{Z}_p}$$ to simplify the notation (i.e. we assume that $\mathcal{O}=\mathbb{Z}_p$). 
We first define an element 
$${}_cz^{(p)}_{N}\in H^1(\mathbb{Z}[1/Np], H^1(Y(N))(2))$$
as the image of ${}_cz_{N_0, p^{\infty}}$ by the following composites
\begin{multline*}
K_2(Y(N))\xrightarrow{(a)} H^2_{\text{\'et}}(Y(N), \mathbb{Z}_p(2))
\xrightarrow{(b)} H^1(\mathbb{Z}[1/Np], H^1(Y(N))_{\mathbb{Z}_p}(2)), 
\end{multline*}
where the map $(a)$ is the Chern class map, and
the map $(b)$ is the map
obtained from the Hochschild-Serre spectral sequence 
(remark that one has $$H^1(\mathbb{Z}[1/Np], -)=H^1_{\text{\'et}}(\mathrm{Spec}(\mathbb{Z}[1/Np]), -)$$ and 
$H^2(Y(N))=0$ since $Y(N)_{\overline{\mathbb{Q}}}$ is a smooth affine curve).
We remark that both the maps $(a), (b)$ are equivariant for 
the Hecke actions of $T(l), T^*(l)$ for any prime $l \not|N$, and Kato's right action on
$\mathrm{GL}_2(\mathbb{Z}/N\mathbb{Z})$. Since this composite is also compatible with the norm map in the left hand side 
and the co-restriction map in the right hand side, for any prime $l$, the norm relation for ${}_cz_N$ implies the similar relations 
\begin{equation}\label{cc}
\mathrm{Cor}({}_cz^{(p)}_{Nl})
=\begin{cases} {}_cz^{(p)}_{N} & \text{ if } l|N\\
\left(1-T^*(l)\begin{pmatrix}1/l& 0 \\ 0 & 1\end{pmatrix}^*+l\begin{pmatrix}1/l& 0 \\ 0 & 1/l\end{pmatrix}^*\right) {}_cz^{(p)}_{N} & \text{ if } (l,N)=1\end{cases}
\end{equation}
in the space $H^1(\mathbb{Z}[1/Nlp], H^1(Y(N))_{\mathbb{Z}_p}(2)),$
where the map 
$$\mathrm{Cor} : H^1(\mathbb{Z}[1/Nlp], H^1(Y(Nl))_{\mathbb{Z}_p}(2))\rightarrow 
H^1(\mathbb{Z}[1/Nlp], H^1(Y(N))_{\mathbb{Z}_p}(2))$$ is the map induced by the co-restriction map 
$$\mathrm{Cor} : H^1(Y(Nl))\rightarrow H^1(Y(N))$$ with respect to the canonical map $Y(Nl)_{\overline{\mathbb{Q}}}\rightarrow Y(N)_{\overline{\mathbb{Q}}}$.  

We remark that the latter equality in (\ref{cc}) can be simply written by 
$$\mathrm{Cor}({}_cz^{(p)}_{Nl})=(1-T'_l+lS'_l) {}_cz^{(p)}_{N}$$
 by Lemma \ref{2.01}. 

Let $N_0\geqq 1, c\geqq 2$ be integers such that $(p, N_0)=1$ and $(c, 6N_0p)=1$. We recall that we write 
$$\widetilde{H}^{BM}_1(K^p(N_0)):=\varprojlim_{k\geqq 1}H^1(Y(N_0p^k))(1)$$
 to denote the completed (Borel-Moore) homology, where we set 
 $$K^p(N_0):=\mathrm{Ker}\left(\mathrm{GL}_2(\widehat{\mathbb{Z}}^{(p)})\xrightarrow{\mathrm{can}}\mathrm{GL}_2(\mathbb{Z}/N_0\mathbb{Z})\right).$$
 
 We set $\Sigma:=\mathrm{prime}(N_0)\cup\{p\}$.  
 We remark that 
 the Mittag-Leffler condition implies that the natural map 
 $$H^1(\mathbb{Z}[1/\Sigma], \widetilde{H}^{BM}_1(K^p(N_0))(m))
 \isom \varprojlim_{k\geqq 1}H^1(\mathbb{Z}[1/\Sigma], H^1(Y(N_0p^k))(m+1))$$
  is isomorphism for every $m\in \mathbb{Z}$ since the groups $$H^1(\mathbb{Z}[1/\Sigma], H^1(Y(N_0p^k))_{\mathbb{Z}/p^s\mathbb{Z}}(m+1))$$ are
 of finite order for all integrs $k,s\geqq 1$. We freely identify the both sides by this isomorphism. Then, one can define an element 
 $${}_cz^{(p)}_{N_0p^{\infty}}:=({}_cz^{(p)}_{N_0p^k})_{k\geqq 1}\in H^1(\mathbb{Z}[1/\Sigma], \widetilde{H}^{BM}_1(K^p(N_0))(1))$$
 by the norm relation (\ref{cc}) for ${}_cz^{(p)}_{N_0p^k}$ ($k\geqq 1$). 

Let $l$ be a prime different from $p$. We write $\Sigma_l:=\Sigma\cup\{l\}$. 
the norm relation (\ref{cc}) for ${}_cz^{(p)}_{N}$ implies the similar equality 
\begin{equation}\label{c}
\mathrm{Cor}({}_cz^{(p)}_{N_0lp^{\infty}})
=\begin{cases} {}_cz^{(p)}_{N_0p^{\infty}} & \text{ if } l|N_0\\(1-T'_l+lS'_l)
{}_cz^{(p)}_{N_0p^{\infty}} & \text{ if } (l,N_0p)=1\end{cases}
\end{equation}
in the space $H^1(\mathbb{Z}[1/\Sigma_l], \widetilde{H}^{BM}_1(K^p(N_0))(1)),$
where the map 
$$\mathrm{Cor} : H^1(\mathbb{Z}[1/\Sigma_l], \widetilde{H}^{BM}_1(K^p(N_0l))(1))\rightarrow H^1(\mathbb{Z}[1/\Sigma_l], \widetilde{H}^{BM}_1(K^p(N_0))(1))$$ is the map induced by the map 
$$\mathrm{Cor} : \widetilde{H}^{BM}_1(K^p(N_0l))(1)\rightarrow \widetilde{H}^{BM}_1(K^p(N_0))(1)$$
which is the limit of the co-restriction maps $$\mathrm{Cor} : H^1(Y(N_0lp^k))(1)\rightarrow H^1(Y(N_0p^k))(1)$$ 
 for all $k\geqq 1$. 



We next define zeta elements in the Galois cohomology of $H^1(Y(N), \mathcal{V}^*_{k/\mathbb{Z}_p})$ for each $N\geqq 3$ and $k\geqq 2$. 
For $m\in \mathbb{Z}_{\geqq 0}$, we set $K_0:=\mathrm{GL}_2(\mathbb{Z}_p)$ if $m=0$ and $K_m:=1+p^m\mathrm{M}_{2}(\mathbb{Z}_p)$ if $m\geqq 1$. 
For $N=p^mN_0$ with $(N_0, p)=1$ and $k\in \mathbb{Z}_{\geqq 2}$, we consider 
 the following composite : 
\begin{multline*}
f : \mathrm{Sym}^{k-2}(\mathbb{Z}_p^2)^*\otimes_{\mathbb{Z}_p[[K_m]]}
H^1(\mathbb{Z}[1/\Sigma], \widetilde{H}^{BM}_1(K^p(N_0))(1))\\
\xrightarrow{(a)}
H^1(\mathbb{Z}[1/\Sigma], \mathrm{Sym}^{k-2}(\mathbb{Z}_p^2)^*\otimes_{\mathbb{Z}_p[[K_m]]}
 \widetilde{H}^{BM}_1(K^p(N_0))(1))\\
 \xrightarrow{(b)}
H^1(\mathbb{Z}[1/\Sigma], H^1(Y(N), \mathcal{V}^*_{k/\mathbb{Z}_p})(2)),
\end{multline*}
where the map $(a)$ is the canonical base change map, and the map $(b)$ is the map induced by the canonical map
$$\mathrm{Sym}^{k-2}(\mathbb{Z}_p^2)^*\otimes_{\mathbb{Z}_p[[K_m]]}
 \widetilde{H}^{BM}_1(K^p(N_0))(1)\rightarrow H^1(Y(N), \mathcal{V}^*_{k/\mathbb{Z}_p})(2)$$ defined in (the proof of) Lemma \ref{1.10}  for 
 $W=\mathrm{Sym}^{k-2}(\mathbb{Z}_p^2)^*$. 
 We remark that this composite is 
 Hecke (i.e. $T'_l, S'_l$) equivariant.
 For each integer $c\geqq 2$ such that $(c, 6N_0p)=1$, using this map $f$ and the element 
 $${}_cz^{(p)}_{N_0p^{\infty}}\in H^1(\mathbb{Z}[1/\Sigma], \widetilde{H}^{BM}_1(K^p(N_0))(1)),$$ we define the following $\mathbb{Z}_p$-linear map 
$${}_cz^{(p)}_{N}(k,-) : \mathrm{Sym}^{k-2}(\mathbb{Z}_p^2)^*\rightarrow H^1(\mathbb{Z}[1/\Sigma], H^1(Y(N), \mathcal{V}^*_{k/\mathbb{Z}_p})(2))$$
by $${}_cz^{(p)}_{N}(k, v):=f(v\otimes {}_cz^{(p)}_{N_0, p^{\infty}})$$ for every 
$v\in  \mathrm{Sym}^{k-2}(\mathbb{Z}_p^2)^*$. 

By definition and the norm relation (\ref{c}) for ${}_cz^{(p)}_{N_0p^{\infty}}$, one also has
\begin{equation}\label{ccc}
\mathrm{Cor}\circ {}_cz^{(p)}_{Nl}(k,-)=\begin{cases}{}_cz^{(p)}_{N}(k,-)& \text{ if } l|N, \\
(1-T'_l+lS'_l) {}_cz^{(p)}_{N}(k,-)& \text{ if } (l,Np)=1\end{cases}
\end{equation}
in the group $H^1(\mathbb{Z}[1/\Sigma_l], H^1(Y(N), \mathcal{V}^*_{k/\mathbb{Z}_p})(2))$
for any prime $l$ (we don't consider the case where $l=p$ and $(p, N)=1$), 
where the map 
$$\mathrm{Cor} : H^1(\mathbb{Z}[1/\Sigma_l], H^1(Y(Nl), \mathcal{V}^*_{k/\mathbb{Z}_p})(2))\rightarrow H^1(\mathbb{Z}[1/\Sigma_l], H^1(Y(N), \mathcal{V}^*_{k/\mathbb{Z}_p})(2))$$ 
is the map induced by the co-restriction
$$\mathrm{Cor} : H^1(Y(Nl), \mathcal{V}^*_{k/\mathbb{Z}_p})(2)\rightarrow H^1(Y(N), \mathcal{V}^*_{k/\mathbb{Z}_p})(2)$$ 
with respect to the canonical finite map $Y(Nl)_{\overline{\mathbb{Q}}}\rightarrow Y(N)_{\overline{\mathbb{Q}}}$. 

Let $n\geqq 1$ and $m\geqq 0$ be integers such that $(n, Np)=1$. 
We set $\Sigma_n:=\Sigma\cup\mathrm{prime}(n)$ the union of $\Sigma$ and the set of prime divisors $\mathrm{prime}(n)$ of $n$. 
For each $c\geqq 2$ such that $(c, 6Nnp)=1$, 
we define the following $\mathbb{Z}_p$-linear map 
$${}_cz^{(p)}_{N, np^m}(k, -) : \mathrm{Sym}^{k-2}(\mathbb{Z}_p^2)^*\rightarrow 
H^1(\mathbb{Z}[1/\Sigma_n, \zeta_{np^m}], H^1(Y(N),\mathcal{V}^*_{k/\mathbb{Z}_p})(2))$$
as the following composite :
\begin{multline*}
 \mathrm{Sym}^{k-2}(\mathbb{Z}_p^2)^*\xrightarrow{{}_cz^{(p)}_{M}(k,-)} H^1(\mathbb{Z}[1/\Sigma_n], H^1(Y(M), \mathcal{V}^*_{k/\mathbb{Z}_p})(2))
 \\\xrightarrow{(a)}H^1(\mathbb{Z}[1/\Sigma_n], H^1(Y(N)\otimes_{\mathbb{Z}}\mu^0_{np^m/\mathbb{Z}[1/np]}, \mathcal{V}^*_{k/\mathbb{Z}_p})(2))\\
 \xrightarrow{(b)}H^1(\mathbb{Z}[1/\Sigma_n, \zeta_{np^m}], H^1(Y(N), \mathcal{V}^*_{k/\mathbb{Z}_p})(2))
\end{multline*} for any $M\geqq 1$ such that $Nn|M$ and $p^m|M$ and $\mathrm{prime}(M)=\Sigma_n$, where 
the map $(a)$ is induced by the push forward map $$H^1(Y(M), \mathcal{V}^*_{k/\mathbb{Z}_p})(2)
\rightarrow H^1(Y(N)\otimes_{\mathbb{Z}}\mu^0_{np^m/\mathbb{Z}[1/np]}, \mathcal{V}^*_{k/\mathbb{Z}_p})(2)$$ with respect to the canonical finite map 
$$Y(M)\rightarrow Y(N)\otimes_{\mathbb{Z}}\mu^0_{np^m/\mathbb{Z}[1/np]},$$ and the map $(b)$ is the isomorphism induced from Shapiro's lemma. We remark that, by the norm relation (\ref{ccc}) for ${}_cz^{(p)}_{M}(k,-)$, the definition of ${}_cz^{(p)}_{N, np^m}(k, -)$ does not depend on the choice of $M$, and 
the map ${}_cz^{(p)}_{N, np^m}(k, -)$ satisfies the following norm relation : 
\begin{equation}\label{cccc}
\mathrm{Cor}\circ {}_cz^{(p)}_{N, np^{m+1}}(k, -)={}_cz^{(p)}_{N, np^m}(k, -)
\end{equation}
in the group $H^1(\mathbb{Z}[1/\Sigma_{n}, \zeta_{np^m}], H^1(Y(N), \mathcal{V}^*_{k/\mathbb{Z}_p})(2))$
for every $m\geqq 1$, and 
\begin{equation}\label{ccccc}\mathrm{Cor}\circ {}_cz^{(p)}_{N, nlp^{m}}(k, -)=\begin{cases}{}_cz^{(p)}_{N, np^m}(k, -)& \text{ if } l|n,
\\ (1-T'_l\otimes\sigma_l^{-1}+lS'_l\otimes \sigma_l^{-2}){}_cz^{(p)}_{N, np^m}(k, -)& \text{ if } (l,n)=1\end{cases}
\end{equation}
in the group $H^1(\mathbb{Z}[1/\Sigma_{nl}, \zeta_{np^m}], H^1(Y(N), \mathcal{V}^*_{k/\mathbb{Z}_p})(2))$
for every prime $l\not\in \Sigma$, where the map 
$$\mathrm{Cor} : H^1(\mathbb{Z}[1/\Sigma_{nl}, \zeta_{nlp^m}], H^1(Y(N), \mathcal{V}^*_{k/\mathbb{Z}_p})(2))\rightarrow 
H^1(\mathbb{Z}[1/\Sigma_{nl}, \zeta_{np^m}], H^1(Y(N), \mathcal{V}^*_{k/\mathbb{Z}_p})(2))$$ 
is the co-restriction map with respect to the canonical map 
$$\mathrm{Spec}(\mathbb{Z}[\zeta_{nlp^m}, 1/Nnlp])\rightarrow \mathrm{Spec}(\mathbb{Z}[\zeta_{np^m}, 1/Nnlp]),$$ 
and the element $\sigma_l\in \mathrm{Gal}(\mathbb{Q}(\zeta_{np^m}/\mathbb{Q}))$ corresponds to 
$l\in (\mathbb{Z}/np^m\mathbb{Z})^{\times}$ by the cyclotomic character. 

In particular, for every $N, n, c\geqq 1$ as above, one can also define the following map
\begin{multline*}
{}_cz^{\mathrm{Iw}}_{N, n}(k,-):=\varprojlim_{m\geqq 1}
{}_cz^{(p)}_{N, np^m}(k,-) : \mathrm{Sym}^{k-2}(\mathbb{Z}_p^2)^*
\\
\rightarrow 
\varprojlim_{m\geqq 1} H^1(\mathbb{Z}[1/\Sigma_n, \zeta_{np^m}], H^1(Y(N), \mathcal{V}^*_{k/\mathbb{Z}_p})(2))=:
H^1_{\mathrm{Iw}}(\mathbb{Z}[1/\Sigma_n, \zeta_n], H^1(Y(N), \mathcal{V}^*_{k/\mathbb{Z}_p})(2)), 
\end{multline*}and it also satisfies the similar norm relation : 
\begin{equation}\label{c1}
\mathrm{Cor}\circ {}_cz^{\mathrm{Iw}}_{N, nl}(k, -)=\begin{cases}{}_cz^{\mathrm{Iw}}_{N, n}(k, -)& \text{ if } l|n,
\\ (1-T'_l\otimes\sigma_l^{-1}+lS'_l\otimes \sigma_l^{-2}){}_cz^{\mathrm{Iw}}_{N, n}(k, -)& \text{ if } (l,n)=1\end{cases}
\end{equation} 
for any prime $l\not\in \Sigma$, where the element $\sigma_l\in \mathrm{Gal}(\mathbb{Q}(\zeta_{np^{\infty}}/\mathbb{Q}))$ corresponds to 
$l\in \varprojlim_{m\geqq 1}(\mathbb{Z}/np^m\mathbb{Z})^{\times}$ by the cyclotomic character.

\subsubsection{Zeta values formula}
Here, we recall Kato's formula which relates zeta elements in Galois cohomology with its associated zeta values. This formula is the most important theorem for Kato's work \cite{Ka04} and also for our construction of zeta elements. 

For each integer $N\geqq 3$, one has the canonical period map
$$\mathrm{per} : M_k(N)_{\mathbb{Q}}\rightarrow H^1(Y(N)(\mathbb{C}), \mathcal{V}_{k/\mathbb{C}})$$
for the scheme$Y(N)_{\mathbb{Q}}$. 
Since one has \begin{multline*}
H^0(X(N)\otimes_{\mathbb{Z}}\mu_{n/\mathbb{Q}}^0, \omega^{\otimes (k-2)}\otimes\Omega^1(\mathrm{log}\{\text{cusps}\}))\\
=H^0(X(N)_{\mathbb{Q}}, \omega^{\otimes (k-2)}\otimes\Omega^1(\mathrm{log}\{\text{cusps}\}))\otimes_{\mathbb{Q}}\mathbb{Q}(\zeta_n)
=M_k(N)_{\mathbb{Q}}\otimes_{\mathbb{Q}}\mathbb{Q}(\zeta_n),\end{multline*}
 and $$H^1((Y(N)\otimes_{\mathbb{Z}}\mu_{n/\mathbb{Q}}^0)(\mathbb{C}), \mathcal{V}_{k/\mathbb{C}})
\isom H^1(Y(N)(\mathbb{C}),  \mathcal{V}_{k/\mathbb{C}})\otimes_{\mathbb{C}}
\mathbb{C}[\mathrm{Gal}(\mathbb{Q}(\zeta_n)/\mathbb{Q})]$$ by Shapiro's lemma for each $n\geqq 1$,  one also has the period map
$$\mathrm{per} : M_k(N)_{\mathbb{Q}}\otimes_{\mathbb{Q}}\mathbb{Q}(\zeta_n) \rightarrow H^1(Y(N)(\mathbb{C}), \mathcal{V}_{k/\mathbb{C}})\otimes_{\mathbb{C}}
\mathbb{C}[\mathrm{Gal}(\mathbb{Q}(\zeta_n)/\mathbb{Q})]$$
for the scheme
$Y(N)\otimes_{\mathbb{Z}}\mu_{n/\mathbb{Q}}^0$.


For a finite dimensional $\mathbb{Q}_p$-vector space $V$ with a continuous $\mathbb{Q}_p$-linear action of $G_{\mathbb{Q}_p}$, we write 
$$\mathrm{exp}^* : H^1(\mathbb{Q}_p, V)\rightarrow \mathrm{Fil}^0D_{\mathrm{dR}}(V):=(\bold{B}^+_{\mathrm{dR}}\otimes_{\mathbb{Q}_p}V)^{G_{\mathbb{Q}_p}}$$
 to denote the Bloch-Kato dual exponential. For $V=H^1(Y(N), \mathcal{V}_{k/\mathbb{Q}_p})(1)|_{G_{\mathbb{Q}_p}}$, one has the following canonical isomorphisms
 \begin{multline*}
 \mathrm{Fil}^0D_{\mathrm{dR}}(V)=\mathrm{Fil}^1D_{\mathrm{dR}}(H^1(Y(N), \mathcal{V}_{k/\mathbb{Q}_p}))\otimes_{\mathbb{Q}_p}D_{\mathrm{dR}}(\mathbb{Q}_p(1))\\
 \isom \mathrm{Fil}^1D_{\mathrm{dR}}(H^1(Y(N), \mathcal{V}_{k/\mathbb{Q}_p}))
 \isom H^0(X(N)_{\mathbb{Q}_p}, \omega^{\otimes (k-2)}\otimes\Omega^1(\mathrm{log}\{\text{cusps}\}))\\
 =M_k(N)_{\mathbb{Q}}\otimes_{\mathbb{Q}}\mathbb{Q}_p,
 \end{multline*}
 where the first isomorphism is induced by the 
 canonical isomorphism $D_{\mathrm{dR}}(\mathbb{Q}_p(1))\isom \mathbb{Q}_p,$ and the second one is from the comparison theorem of $p$-adic Hodge theory (see \S 11 of \cite{Ka04}). 
 Therefore, the dual exponential map 
 for $H^1(Y(N), \mathcal{V}_{k/\mathbb{Q}_p})(1)|_{G_{\mathbb{Q}_p}}$ becomes the following map 
 $$\mathrm{exp}^* : H^1(\mathbb{Q}_p, H^1(Y(N), \mathcal{V}_{k/\mathbb{Q}_p})(1))\rightarrow M_k(N)_{\mathbb{Q}}\otimes_{\mathbb{Q}}\mathbb{Q}_p.$$ 
 Similarly, since one has canonical isomorphisms $$\mathrm{Fil}^0D_{\mathrm{dR}}(H^1(Y(N)\otimes_{\mathbb{Z}}\mu_{n/\mathbb{Q}}^0, \mathcal{V}_{k/\mathbb{Q}_p})(1))
 \isom M_k(N)_{\mathbb{Q}}\otimes_{\mathbb{Q}}\mathbb{Q}(\zeta_{n})\otimes_{\mathbb{Q}}\mathbb{Q}_p,$$ and 
 $$H^1(\mathbb{Q}_p, H^1(Y(N)\otimes_{\mathbb{Z}}\mu_{n/\mathbb{Q}}^0, \mathcal{V}_{k/\mathbb{Q}_p})(1))\isom 
 H^1(\mathbb{Q}(\zeta_{n})\otimes_{\mathbb{Q}}\mathbb{Q}_p, H^1(Y(N), \mathcal{V}_{k/\mathbb{Q}_p})(1))$$ by Shapiro's lemma, the dual exponential for 
 $H^1(Y(N)\otimes_{\mathbb{Z}}\mu_{n/\mathbb{Q}}^0, \mathcal{V}_{k/\mathbb{Q}_p})(1)|_{G_{\mathbb{Q}_p}}$ becomes the following map 
 $$\mathrm{exp}^* : H^1(\mathbb{Q}(\zeta_{n})\otimes_{\mathbb{Q}}\mathbb{Q}_p, H^1(Y(N), \mathcal{V}_{k/\mathbb{Q}_p})(1))
  \rightarrow M_k(N)_{\mathbb{Q}}\otimes_{\mathbb{Q}}\mathbb{Q}(\zeta_{n})\otimes_{\mathbb{Q}}\mathbb{Q}_p.$$
 
Let $\Sigma$ be a finite set of primes. In \S 4.5 of \cite{Ka04}, Kato considers the following operator valued functions
$$Z_{\Sigma}(s)=
 \prod_{l \not\in\Sigma}P_l(l^{-s})^{-1}$$
 acting on the cohomology $H^1(Y(N)(\mathbb{C}), \mathcal{V}_{k/\mathbb{C}})$ for every $N\geqq 3$ such that 
 $\mathrm{prime}(N)\subseteq \Sigma$ and $k\geqq 2$, 
  where one defines
 $$P_{l}(u)=1-T^*_l\begin{pmatrix}l^{-1}& 0 \\ 0 & 1\end{pmatrix}^*u+l^{k-1}\begin{pmatrix}l^{-1}& 0 
 \\ 0 & l^{-1}\end{pmatrix}u^2$$
 for any prime $l\not\in \Sigma$.  We remark that one can also write
 $$P_{l}(u)=1-T'_lu+lS'_lu^2$$
 by Lemma \ref{2.01}.

 For each $n\geqq 1$, we set $\Sigma_n:=\Sigma\cup\mathrm{prime}(n)$.
  We also need the 
  following operator valued zeta function 
  $$Z_{\Sigma, n}(s):=\prod_{l \not\in \Sigma_n}P_l(l^{-s}\otimes\sigma_l^{-1})^{-1}$$
  acting on the cohomology $H^1(Y(N)(\mathbb{C}), \mathcal{V}_{k/\mathbb{C}})\otimes_{\mathbb{C}}
  \mathbb{C}[\mathrm{Gal}(\mathbb{Q}(\zeta_n)/\mathbb{Q})]$ for every $N$ such that $\mathrm{prime}(N)\subset \Sigma_n$. 
  
    It is known that the zeta functions $Z_{\Sigma}(s)$ and $Z_{\Sigma, n}(s)$ converge when $\mathrm{Re}(s)>k$, and have meromorphic analytic continuation to the whole complex plane $\mathbb{C}$ and 
 are holomorphic on $\mathbb{C}$ except at $k=2$ (see \S2.5 and \S4.5 of \cite{Ka04}). 
 
 We remark that, for any $N\geqq 3$ such that $\mathrm{prime}(N)\subset \Sigma$, and any divisors $M\geqq 3$ and $n\geqq 1$ of $N$ (then $\Sigma_n=\Sigma$), one has the equality 
 \begin{equation}\label{d1}
 f_*\circ Z_{\Sigma}(s)=Z_{\Sigma, n}(s)\circ f_*
 \end{equation}
 for arbitrary $s\not=2$, where the map
 \begin{multline*}
 f_* : H^1(Y(N)(\mathbb{C}), \mathcal{V}_{k/\mathbb{C}})\rightarrow H^1((Y(M)\otimes_{\mathbb{Z}}\mu_{n/\mathbb{Q}}^0)(\mathbb{C}), \mathcal{V}_{k/\mathbb{C}})\\
 \isom H^1(Y(M)(\mathbb{C}), \mathcal{V}_{k/\mathbb{C}})\otimes_{\mathbb{C}}\mathbb{C}[\mathrm{Gal}(\mathbb{Q}(\zeta_n)/\mathbb{Q})]
 \end{multline*}
 is the push-forward map with respect to the canonical finite map $f : Y(N)_{\mathbb{Q}}\rightarrow Y(M)\otimes_{\mathbb{Z}}\mu_{n/\mathbb{Q}}^0$.

 We next define local systems $\mathcal{H}^1$ and $\mathcal{H}_1$ on $Y(N)(\mathbb{C})$ by 
 $$\mathcal{H}^1:=R^1f_*(\underline{\mathbb{Z}})\text{ and } \mathcal{H}_1:=\mathcal{H}om(\mathcal{H}^1, 
 \underline{\mathbb{Z}}), $$
 where $f : E^{\mathrm{univ}}(\mathbb{C})\rightarrow Y(N)(\mathbb{C})$ is the canonical structure map, and 
$ \underline{\mathbb{Z}}$ is the constant local system on each spaces 
$E^{\mathrm{univ}}(\mathbb{C})$ and $Y(N)(\mathbb{C})$. 
For each integer $k\geqq 2$, we write  
$$\mathcal{V}_{k/\mathbb{Z}}:=\mathrm{Sym}^{k-2}(\mathcal{H}^1),$$
then its $\mathbb{Z}$-dual is 
$$\mathcal{V}^*_{k/\mathbb{Z}}=\mathrm{Sym}^{k-2}(\mathcal{H}_1).$$



 For each $N\geqq 3$, we define a continuous map 
 $$\varphi : (0,\infty)\rightarrow Y(N)(\mathbb{C})=\mathrm{GL}_2(\mathbb{Q})\backslash \mathcal{H}^{\pm} \times (\mathrm{GL}_2(\mathbb{A}_f)/K(N)) : 
 y \mapsto [(iy, e)]$$ 
 ($e$ is the unit). We remark that the stalk of $\varphi^{-1}(\mathcal{H}_1)$ at $y\in (0,\infty)$ is identified with 
 $H_1(\mathbb{C}/\mathbb{Z}yi+\mathbb{Z})=\mathbb{Z}yi+\mathbb{Z}$, and the group $\Gamma((0,\infty), \varphi^{-1}(\mathcal{H}_1))$ 
 is a free $\mathbb{Z}$-module of rank two with basis $f_1$ and $f_2$ such that the stalk of $f_1$ (resp. $f_2$) at any $y\in (0,\infty)$ 
 is $yi$ (resp. $1$) in $\mathbb{Z}yi+\mathbb{Z}$. We write $e_1:=x_1^*, e_2:=x_2^*\in (\mathbb{Z}^2)^*$ to denote the dual basis with respect to the canonical basis
 $x_1:=\begin{pmatrix}1\\ 0\end{pmatrix}, x_2:=\begin{pmatrix}0\\1\end{pmatrix}$ 
 of $\mathbb{Z}^2$. We define an isomorphism 
 $$(\mathbb{Z}^2)^*\isom \Gamma((0,\infty), \varphi^{-1}(\mathcal{H}_1)) : e_1\mapsto f_1, e_2\mapsto f_2,$$
 then this also induces an isomorphism 
 $$\mathrm{Sym}^{k-2}(\mathbb{Z})^*\isom  \Gamma((0,\infty), \varphi^{-1}(\mathcal{V}^*_{k/\mathbb{Z}})) : e_1^ie_2^{k-2-i}\mapsto f_1^{i}f_2^{k-2-i}$$
 for each $k\geqq 2$, by  which we identify the both sides. 
 Using this isomorphism, we define a map 
 $$\delta_N(k,-) : \mathrm{Sym}^{k-2}(\mathbb{Z}^2)^* \rightarrow H^1(Y(N)(\mathbb{C}), \mathcal{V}^*_{k/\mathbb{Z}})(1)$$
 as the following composite : 
 \begin{multline*}
 \mathrm{Sym}^{k-2}(\mathbb{Z}^2)^*\isom \Gamma((0,\infty), \varphi^{-1}(\mathcal{V}^*_{k/\mathbb{Z}}))
 \xrightarrow{(a)} H_1([0,\infty], \{0,\infty\},\varphi^{-1}(\mathcal{V}^*_{k/\mathbb{Z}}))\\
 \xrightarrow{\varphi_*} H_1(X(N)(\mathbb{C}), \{\text{cusps}\}, \mathcal{V}^*_{k/\mathbb{Z}})\xrightarrow{(b)} H^1(Y(N)(\mathbb{C}), \mathcal{V}^*_{k/\mathbb{Z}})(1),
 \end{multline*}
 where the map $(a)$ is the canonical isomorphism, and $(b)$ is the isomorphism induced from Poincar\'e duality (see \S 4.7 of \cite{Ka04}). 
 
We remark that, for any divisors $M\geqq 3$ and $n\geqq 1$ of $N$, one has an equality
 $$f_*\circ\delta_N(k,-)=\delta_{M}(k,-)\otimes 1$$
 as the map from $\mathrm{Sym}^{k-2}(\mathbb{Z}^2)^*$ to $
 H^1(Y(M)(\mathbb{C}), \mathcal{V}^*_{k/\mathbb{Z}})(1)\otimes_{\mathbb{Z}}\mathbb{Z}[\mathrm{Gal}(\mathbb{Q}(\zeta_n)/\mathbb{Q})],$ where the map 
 \begin{multline*}
 f_* : H^1(Y(N)(\mathbb{C}), \mathcal{V}^*_{k/\mathbb{Z}})(1)\rightarrow H^1((Y(M)\otimes_{\mathbb{Z}}\mu_n^0)(\mathbb{C}), \mathcal{V}^*_{k/\mathbb{Z}})(1)\\
 \isom H^1((Y(M)(\mathbb{C}), \mathcal{V}^*_{k/\mathbb{Z}})(1)\otimes_{\mathbb{Z}}\mathbb{Z}[\mathrm{Gal}(\mathbb{Q}(\zeta_n)/\mathbb{Q})]
 \end{multline*}
 is the push-forward map with respect to the canonical finite morphism $f : Y(N)_{\mathbb{Q}}\rightarrow Y(M)\otimes_{\mathbb{Z}}\mu_{n/\mathbb{Q}}^0$. 
 
 For any $N\geqq 3$ and $c\geqq 2$ such that $(c, 6Np)=1$, we define a $\mathbb{Z}$-linear map 
 $${}_c\delta_N(k, -):\mathrm{Sym}^{k-2}(\mathbb{Z}^2)^*\rightarrow H^1(Y(N)(\mathbb{C}), \mathcal{V}^*_{k/\mathbb{Z}})(1)$$
 by 
 $${}_c\delta_N(k, e_1^{j-1}e_2^{k-j-1}):=
 \left(c^2-c^{k-j}\begin{pmatrix}c& 0 \\ 0 & 1\end{pmatrix}^*\right)\left(c^2-c^{j}\begin{pmatrix}1& 0 \\ 0 & c\end{pmatrix}^*\right)
 \delta_N(k, e_1^{j-1}e_2^{k-j-1})
$$ for every $1\leqq j\leqq k-1$, here we regard $\begin{pmatrix}c& 0\\0 & 1\end{pmatrix}$ and $\begin{pmatrix}1& 0 \\ 0 & c\end{pmatrix}$ as elements in 
$\mathrm{GL}_2(\mathbb{Z}/N\mathbb{Z})$. 
 
 For any $N\geqq 3$, $n\geqq 1$ and $c\geqq 2$ such that $(c, 6Nnp)=1$, we similarly define a $\mathbb{Z}$-linear map 
 $$ {}_c\delta_{N, n}(k,-) : \mathrm{Sym}^{k-2}(\mathbb{Z}^2)^*\rightarrow H^1(Y(N)(\mathbb{C}), \mathcal{V}^*_{k/\mathbb{Z}})(1)\otimes_{\mathbb{Z}}\mathbb{Z}[\mathrm{Gal}(\mathbb{Q}(\zeta_n)/\mathbb{Q})]$$
 by 
 \begin{multline*}
 {}_c\delta_{N, n}(k, e_1^{j-1}e_2^{k-j-1})\\
 :=
 \left(c^2-c^{k-j}\begin{pmatrix}c& 0 \\ 0 & 1\end{pmatrix}^*\otimes\sigma_c\right)\left(c^2-c^{j}\begin{pmatrix}1& 0 \\ 0 & c\end{pmatrix}^*\otimes\sigma_c\right)
 \left(\delta_N(k, e_1^{j-1}e_2^{k-j-1})\otimes 1\right)
 \end{multline*}
 for every $1\leqq j\leqq k-1$.

 We remark that, for every divisors $M\geqq3$ and $n\geqq 1$ of $N$, one has the equality 
 \begin{equation}\label{d2}
 f_*\circ {}_c\delta_N(k,-)={}_c\delta_{M, n}(k,-)
 \end{equation}
 as the map from $\mathrm{Sym}^{k-2}(\mathbb{Z}^2)^*$ to $H^1(Y(M)(\mathbb{C}), \mathcal{V}^*_{k/\mathbb{Z}})(1)\otimes_{\mathbb{Z}}\mathbb{Z}[\mathrm{Gal}(\mathbb{Q}(\zeta_n)/\mathbb{Q})]$, 
 where the map $$f_* : H^1(Y(N)(\mathbb{C}), \mathcal{V}^*_{k/\mathbb{Z}})(1)\rightarrow H^1(Y(M)(\mathbb{C}), \mathcal{V}^*_{k/\mathbb{Z}})(1)\otimes_{\mathbb{Z}}\mathbb{Z}[\mathrm{Gal}(\mathbb{Q}(\zeta_n)/\mathbb{Q})]$$ is the push-forward map with respect to the canonical finite morphism
 $Y(N)_{\mathbb{Q}}\rightarrow Y(M)\otimes_{\mathbb{Z}}\mu_{n/\mathbb{Q}}^0$. 
 
 We remark that $H^1(Y(N)(\mathbb{C}), \mathcal{V}^*_{k/\mathbb{Z}})$ and $H^1(Y(N)(\mathbb{C}), \mathcal{V}_{k/\mathbb{Z}})$ have natural actions of the Galois group 
 $\mathrm{Gal}(\mathbb{C}/\mathbb{R})=\{\tau, e\}\isom \{\pm 1\}$ induced by its action on $Y(N)(\mathbb{C})$.  It also acts on the ring $\mathbb{Z}[\mathrm{Gal}(\mathbb{Q}(\zeta_{n})/\mathbb{Q})]$ by 
 $\tau\sigma_c=\sigma_{-c}$ for any $c\in (\mathbb{Z}/n\mathbb{Z})^{\times}$. For any $\mathrm{Gal}(\mathbb{C}/\mathbb{R})$-module  $M$, 
 we write $M^+:=M^{\tau=1}$. 
 
 We now fix a finite set of primes $\Sigma$ containing $p$. 
 We can now state Kato's zeta values formula.

 \begin{thm}\label{2.1}
 Let $N\geqq3, n\geqq 1$ be any integers such that $\mathrm{prime}(Np)=\Sigma$ and $(n, Np)=1$. 
 For every element $v\in \mathrm{Sym}^{k-2}(\mathbb{Z}^2)^*\subseteq \mathrm{Sym}^{k-2}(\mathbb{Z}_p)^*$ and $m\geqq 0$, 
 the image $${}_c\omega_{N, np^m}(k,v)\in M_k(N)_{\mathbb{Q}}\otimes_{\mathbb{Q}}\mathbb{Q}(\zeta_{np^m})\otimes_{\mathbb{Q}}\mathbb{Q}_p$$ of 
 the element
 $${}_cz^{\mathrm{Iw}}_{N, n}(k,v)\in 
 H^1_{\mathrm{Iw}}(\mathbb{Z}[1/\Sigma_n, \zeta_n], H^1(Y(N), \mathcal{V}^*_{k/\mathbb{Z}_p})(2))$$ by the composite
 
 \begin{multline*}
 H_{\mathrm{Iw}}^1(\mathbb{Z}[1/\Sigma_n, \zeta_n], H^1(Y(N), \mathcal{V}^*_{k/\mathbb{Z}_p})(2))\xrightarrow{(a)}
 H_{\mathrm{Iw}}^1(\mathbb{Z}[1/\Sigma_n, \zeta_n], H^1(Y(N), \mathcal{V}_{k/\mathbb{Z}_p})(1))
  \\\xrightarrow{\mathrm{can}}H^1(\mathbb{Z}[1/\Sigma_n, \zeta_{np^m}], H^1(Y(N),\mathcal{V}_{k/\mathbb{Z}_p})(1))
 \xrightarrow{\mathrm{loc}_p}H^1(\mathbb{Q}(\zeta_{np^m})\otimes_{\mathbb{Q}}\mathbb{Q}_p, H^1(Y(N),\mathcal{V}_{k/\mathbb{Z}_p})(1))\\
 \xrightarrow{\mathrm{exp}^*}M_k(N)_{\mathbb{Q}}\otimes_{\mathbb{Q}}\mathbb{Q}(\zeta_{np^m})\otimes_{\mathbb{Q}}\mathbb{Q}_p,
 \end{multline*}
 where the map $(a)$ is the twist by $((\zeta_{p^m})_{m\geqq 1})^{\otimes (1-k)}$ in Iwasawa cohomology, is contained in the subspace $M_k(N)_{\mathbb{Q}}\otimes_{\mathbb{Q}}\mathbb{Q}(\zeta_{np^m})
 $, and its image by the period map $$\mathrm{per} : M_k(N)_{\mathbb{Q}}\otimes_{\mathbb{Q}}\mathbb{Q}(\zeta_{np^m})
 \rightarrow H^1(Y(N)(\mathbb{C}), \mathcal{V}_{k/\mathbb{C}})\otimes_{\mathbb{C}}\mathbb{C}[\mathrm{Gal}(\mathbb{Q}(\zeta_{np^m})/\mathbb{Q})]$$ satisfies the equality 
 \begin{equation}\label{zeta}
 \mathrm{per}({}_c\omega_{N, np^m}(k,v))=\frac{1}{(2\pi i)^{k-1}}Z_{\Sigma, np^m}(k-1)\cdot {}_c\delta_{N, np^m}(k,v)
 \end{equation}
 in the quotient $$H^1(Y(N)(\mathbb{C}), \mathcal{V}_{k/\mathbb{C}})\otimes_{\mathbb{C}[\{\pm1\}]}\mathbb{C}[\mathrm{Gal}(\mathbb{Q}(\zeta_{np^m})/\mathbb{Q})]$$ of 
 $H^1(Y(N)(\mathbb{C}), \mathcal{V}_{k/\mathbb{C}})\otimes_{\mathbb{C}}\mathbb{C}[\mathrm{Gal}(\mathbb{Q}(\zeta_{np^m})/\mathbb{Q})]$.

 \end{thm}
 \begin{rem}
 We remark that one has a canonical isomorphism 
 $$\mathcal{V}_{k/\mathbb{Z}}\isom \mathcal{V}^*_{k/\mathbb{Z}}(2-k)$$
defined in the same way as in (\ref{aa1}). It induces a canonical isomorphism 
 $$H^1(Y(N)(\mathbb{C}), \mathcal{V}^*_{k/\mathbb{Z}})(2-k)\isom H^1(Y(N)(\mathbb{C}), \mathcal{V}_{k/\mathbb{Z}})$$
 as in (\ref{aa2}), and its $\mathbb{C}$-linearization induces an isomorphism
 $$H^1(Y(N)(\mathbb{C}), \mathcal{V}^*_{k/\mathbb{C}})\isom H^1(Y(N)(\mathbb{C}), \mathcal{V}_{k/\mathbb{C}}),$$
 by which we identify the both side. Hence, the equation (\ref{zeta}) makes sense. 
 \end{rem}
\begin{proof}
By the equations (\ref{d1}) and (\ref{d2}), it suffices to show the theorem for $n=1$ and $m\geqq1$, and for $v=e_1^{j-1}e_2^{k-j-1}$ ($1\leqq j\leqq k-1$). 
By definition of ${}_cz^{\mathrm{Iw}}_{N}(k,-)$, we remark that the image in 
$H^1(\mathbb{Z}[1/\Sigma, \zeta_{p^m}], H^1(Y(N),\mathcal{V}_{k/\mathbb{Z}_p})(1))$ of the element ${}_cz^{\mathrm{Iw}}_{N}(k, e_1^{j-1}e_2^{k-j-1})$
by the composite $\mathrm{can}\circ (a)$ coincides with the image of Kato's element 
$${}_{c,c}z^{(p)}_{Np^m, Np^m}(k,k-1,j)\in H^1(\mathbb{Z}[1/\Sigma], H^1(Y(Np^m),\mathcal{V}_{k/\mathbb{Z}_p})(1))$$
defined in \S 8.4 of \cite{Ka04}, by the map 
$$H^1(\mathbb{Z}[1/\Sigma], H^1(Y(Np^m),\mathcal{V}_{k/\mathbb{Z}_p})(1))\rightarrow 
H^1(\mathbb{Z}[1/\Sigma, \zeta_{p^m}], H^1(Y(N),\mathcal{V}_{k/\mathbb{Z}_p})(1))$$ induced by the canonical finite map $Y(Np^m)_{\mathbb{Q}}\rightarrow Y(N)\otimes \mu^0_{p^m/\mathbb{Q}}$ 
and Shapiro's lemma. By Theorem 9.5 of \cite{Ka04}, the image of ${}_{c,c}z^{(p)}_{Np^m, Np^m}(k,k-1,j)$ by the map 
$$\mathrm{exp}^*\circ \mathrm{loc}_p : H^1(\mathbb{Z}[1/\Sigma], H^1(Y(Np^m),\mathcal{V}_{k/\mathbb{Z}_p})(1))
\rightarrow M_k(Np^m)_{\mathbb{Q}}\otimes_{\mathbb{Q}}\mathbb{Q}_p$$ is in $M_k(Np^m)_{\mathbb{Q}}$, and coincides 
with the element ${}_{c,c}z_{Np^m, Np^m}(k,k-1,j)\in M_k(Np^m)_{\mathbb{Q}}$ defined in \S 4.2 of \cite{Ka04}. 
By Theorem 4.6 of \cite{Ka04}, one has 
$$(\mathrm{per}({}_{c,c}z_{Np^m, Np^m}(k,k-1,j)))^+=Z_{\Sigma}(k-1)\cdot \gamma^+$$
for $$\gamma=\left(c^2-c^{k-j}\begin{pmatrix}c & 0 \\ 0 & 1\end{pmatrix}^*\right)\left(c^2-c^j\begin{pmatrix}1 & 0 \\ 0 & c\end{pmatrix}^*\right)\delta_{Np^m, Np^m}(k,j),$$
where 
the element $\delta_{Np^m, Np^m}(k,j)\in H^1(Y(Np^m),\mathcal{V}_{k/\mathbb{Z}})$ is defined in \S 4.7 of \cite{Ka04}. 
Since one has 
$$\delta_{Np^m, Np^m}(k,j)=\frac{1}{(2\pi i)^{k-1}}\delta_{Np^m}(k, e_1^{j-1}e_2^{k-j-1})$$
under the canonical isomorphism $H^1(Y(Np^m),\mathcal{V}_{k/\mathbb{Z}})\isom H^1(Y(Np^m),\mathcal{V}^*_{k/\mathbb{Z}}(2-k))$, 
the element $Z_{\Sigma}( k-1)\cdot \gamma$ coincides with 
$$\frac{1}{(2\pi i)^{k-1}}Z_{\Sigma}(k-1)\cdot {}_c\delta_{Np^m}(k, e_1^{j-1}e_2^{k-j-1}).$$
Hence we obtain the equality 
$$(\mathrm{per}({}_{c,c}z_{Np^m, Np^m}(k,k-1,j)))^+=\left(\frac{1}{(2\pi i)^{k-1}}Z_{\Sigma}(k-1)\cdot {}_c\delta_{Np^m}(k, e_1^{j-1}e_2^{k-j-1})\right)^+.$$
This equality implies the desired equality 
$$\mathrm{per}({}_c\omega_{N, p^m}(k, e_1^{j-1}e_2^{k-j-1}))=\frac{1}{(2\pi i)^{k-1}} Z_{\Sigma, p^m}(k-1)\cdot {}_c\delta_{N, p^m}(k, e_1^{j-1}e_2^{k-j-1})$$ in the quotient $H^1(Y(N)(\mathbb{C}), \mathcal{V}_{k/\mathbb{C}})\otimes_{\mathbb{C}[\{\pm1\}]}\mathbb{C}[\mathrm{Gal}(\mathbb{Q}(\zeta_{np^m})/\mathbb{Q})]$.

\end{proof}

 \subsection{Zeta morphisms for the parabolic cohomology}
 
 Fix $N\geqq 3$ such that $\mathrm{prime}(Np)=\Sigma$ and $k\geqq 2$. 
 We recall that, by a theorem of Drinfeld-Manin, one has a unique $\mathbb{T}_k(N)_{\mathbb{Q}}$-linear splitting
 $$s_N : H^1(Y(N)(\mathbb{C}), \mathcal{V}_{k/\mathbb{Q}})\rightarrow H^1(X(N)(\mathbb{C}), j_*\mathcal{V}_{k/\mathbb{Q}})$$
 of the canonical injection $H^1(X(N)(\mathbb{C}), j_*\mathcal{V}_{k/\mathbb{Q}})\hookrightarrow 
 H^1(Y(N)(\mathbb{C}), \mathcal{V}_{k/\mathbb{Q}})$ (over $\mathbb{Q}$, not over $\mathbb{Z}$), 
 which we call Drinfeld-Manin splitting (see for example \cite{El90}). It is also $\mathbb{T}'_k(N)_{\mathbb{Q}}$-linear by Lemma \ref{2.01}. 
 Since it is defined using Hecke actions, 
 it is $\mathrm{Gal}(\mathbb{C}/\mathbb{R})$-equivariant, and its base change 
  $$s_N: H^1(Y(N), \mathcal{V}_{k/\mathbb{Q}_p})\rightarrow H^1(X(N), j_*\mathcal{V}_{k/\mathbb{Q}_p})$$ to $\mathbb{Q}_p$ is $G_{\mathbb{Q}}$-equivariant. 

 We recall that one has  Hecke ($T_l, S_l$ and $T'_l, S'_l$) equivariant isomorphisms 
 \begin{equation}\label{ES1}
 M_k(N)_{\mathbb{C}}\oplus S_k(N)_{\mathbb{C}}\isom H^1(Y(N)(\mathbb{C}),  \mathcal{V}_{k/\mathbb{C}}) : (x, y)\mapsto \tau(\mathrm{per}(x))+\mathrm{per}(y)
 \end{equation}
 and 
 \begin{equation}\label{ES2}
 S_k(N)_{\mathbb{C}}\oplus S_k(N)_{\mathbb{C}}\isom H^1(X(N)(\mathbb{C}),  j_*\mathcal{V}_{k/\mathbb{C}}) : (x, y)\mapsto \tau(\mathrm{per}(x))+\mathrm{per}(y)
 \end{equation}
defined by the $\mathbb{C}$-linearization of the period map 
 $$\mathrm{per} : M_k(N)_{\mathbb{Q}}\rightarrow H^1(Y(N)(\mathbb{C}),  \mathcal{V}_{k/
 \mathbb{C}}).$$
 
  The existence of these isomorphisms implies that $M_k(N)_{\mathbb{Q}}$ and $S_k(N)_{\mathbb{Q}}$ are both 
 $\mathbb{T}_k(N)_{\mathbb{Q}}$ and $\mathbb{T}'_k(N)_{\mathbb{Q}}$-modules, and 
 the splitting $s_N$ also induces a $\mathbb{T}_k(N)_{\mathbb{Q}}$ and $\mathbb{T}'_k(N)_{\mathbb{Q}}$-linear splitting 
 $$s_N : M_k(N)_{\mathbb{Q}}\rightarrow S_k(N)_{\mathbb{Q}}$$
  of the canonical injection $S_k(N)_{\mathbb{Q}}\hookrightarrow M_k(N)_{\mathbb{Q}}$. 

 Using the canonical isomorphisms 
  $$H^1(Y(N), \mathcal{V}^*_{k/\mathbb{Q}_p})\isom H^1(Y(N), \mathcal{V}_{k/\mathbb{Q}_p})(2-k)$$
   and 
   $$H^1(X(N), j_*\mathcal{V}^*_{k/\mathbb{Q}_p})\isom H^1(X(N), j_*\mathcal{V}_{k/\mathbb{Q}_p})(2-k),$$
   we also regard the splitting $s_N$ as a splitting 
   $$s_N : H^1(Y(N)(\mathbb{C}), \mathcal{V}^*_{k/\mathbb{Q}})\rightarrow H^1(X(N)(\mathbb{C}), j_*\mathcal{V}^*_{k/\mathbb{Q}}).$$

    For each $n\geqq 1$ such that $(n, p)=1$, we write $\Gamma_n:=\mathrm{Gal}(\mathbb{Q}(\zeta_{np^{\infty}})/\mathbb{Q})$, 
   and write $\Lambda_n:=\mathbb{Z}_p[[\Gamma_n]]$ to denote the Iwasawa algebra of $\Gamma_n$ with coefficients in $\mathbb{Z}_p$. 
  For each $c\in \mathbb{Z}$ such that $(c, np)=1$, we write $\sigma_c$ to denote the element in $\Gamma_n$ corresponding to 
  $c\in \varprojlim_{m\geqq 1}(\mathbb{Z}/np^m\mathbb{Z})^{\times}$ by the cyclotomic character. 
   We simply write $\Gamma:=\Gamma_1$ and $\Lambda:=\Lambda_1$ for $n=1$. 
   

   We fix a finite set of primes $\Sigma$ containing $p$. For any $N\geqq 3$ and $n\geqq 1$ such that $\mathrm{prime}(Np)=\Sigma$ and 
   $(n, \Sigma)=1$, the Iwasawa cohomology $$H^1_{\mathrm{Iw}}(\mathbb{Z}[1/\Sigma_n, \zeta_n], 
   H^1(X(N), j_*\mathcal{V}_{k/\mathbb{Z}_p}(i)))$$ ($i\in \mathbb{Z}$) 
   and the limit $\varprojlim_{m\geqq 1}S_k(N)\otimes_{\mathbb{Q}}\mathbb{Q}(\zeta_{np^m})\otimes_{\mathbb{Q}}\mathbb{Q}_p$ 
   are naturally equipped with the structures of $\Lambda_n$-modules, where 
   the limit is taken with respect to the transition map induced by the trace $\mathbb{Q}(\zeta_{np^{m+1}})\rightarrow \mathbb{Q}(\zeta_{np^m})$. 
   Then, the map 
   $$\mathrm{exp}^* : H^1_{\mathrm{Iw}}(\mathbb{Z}[1/\Sigma_n, \zeta_n], 
   H^1(X(N), j_*\mathcal{V}_{k/\mathbb{Z}_p}(1)))\rightarrow \varprojlim_{m\geqq 0}S_k(N)\otimes_{\mathbb{Q}}\mathbb{Q}(\zeta_{np^m})\otimes_{\mathbb{Q}}\mathbb{Q}_p$$
   which is defined as the projective limit of the composites
   \begin{multline*}
   H^1_{\mathrm{Iw}}(\mathbb{Z}[1/\Sigma_n, \zeta_n],  H^1(X(N), j_*\mathcal{V}_{k/\mathbb{Z}_p}(1)))\xrightarrow{\mathrm{can}}
   H^1( \mathbb{Z}[1/\Sigma_n, \zeta_{np^m}],  H^1(X(N), j_*\mathcal{V}_{k/\mathbb{Z}_p}(1)))
   \\\xrightarrow{\mathrm{loc}_p}
    H^1(\mathbb{Q}(\zeta_{np^m})\otimes_{\mathbb{Q}}\mathbb{Q}_p,  H^1(X(N), j_*\mathcal{V}_{k/\mathbb{Z}_p}(1)))
    \xrightarrow{\mathrm{exp}^*}
    S_k(N)\otimes_{\mathbb{Q}}\mathbb{Q}(\zeta_{np^m})\otimes_{\mathbb{Q}}\mathbb{Q}_p
    \end{multline*}
    for all $m\geqq 0$, is $\Lambda_n$-linear. 
    
    For any ring $A$ and any $A$-module $M$, we say that 
    $M$ is torsion free if $ax=0$ implies $x=0$ for any non-zero divisor $a\in A$. 
    
   \begin{lemma}\label{2.2}
   
   \begin{itemize}
   \item[]
   \item[(1)]For any $i\in \mathbb{Z}$, the $\Lambda_n$-module $H^1_{\mathrm{Iw}}(\mathbb{Z}[1/\Sigma_n, \zeta_n], 
   H^1(X(N), j_*\mathcal{V}_{k/\mathbb{Z}_p}(i)))$  is torsion free. 
   \item[(2)]The map $$\mathrm{exp}^* : H^1_{\mathrm{Iw}}(\mathbb{Z}[1/\Sigma_n, \zeta_n], 
   H^1(X(N), j_*\mathcal{V}_{k/\mathbb{Z}_p}(1)))\rightarrow \varprojlim_{m\geqq 1}S_k(N)\otimes_{\mathbb{Q}}\mathbb{Q}(\zeta_{np^m})\otimes_{\mathbb{Q}}\mathbb{Q}_p$$ 
   is an injection.
   
   \end{itemize}
   \end{lemma}
   \begin{proof}(1) is 
   Proposition 3.1.3 (1) of \cite{FK}, and (2) follows from Lemma 3.1.4. of \cite{FK}. 
   Precisely, they consider $X_1(N)$ and $k=2$ in \cite{FK}, but the same argument works for $X(N)$ and any  $k\geqq 2$.  
   \end{proof}
    \begin{rem}\label{2.3}
   Both the statements (1) and (2) are not true if we replace $X(N)$ with $Y(N)$ 
   (see Proposition 3.1.3 and Lemma 3.1.4 of \cite{FK}). 
   This fact is one of the reasons why we need to consider $X(N)$ instead of $Y(N)$. 
   \end{rem}

   \begin{corollary}\label{2.4}
   Let $N\geqq 3, c\geqq 2,  n\geqq 1$ be integers such that $\mathrm{prime}(Np)=\Sigma$ and $(n, Np)=1$ and $(c, 6Nnp)=1$. 
 For every element $v\in \mathrm{Sym}^{k-2}(\mathbb{Z}^2)^*\subseteq \mathrm{Sym}^{k-2}(\mathbb{Z}_p)^*$, 
 the element
 $$s_N\left({}_cz^{\mathrm{Iw}}_{N, n}(k,v)\right)\in H^1_{\mathrm{Iw}}(\mathbb{Z}[1/\Sigma_n, \zeta_n], H^1(X(N), j_*\mathcal{V}^*_{k/\mathbb{Q}_p}(2)))$$
  is a unique element $z_v$ in $H^1_{\mathrm{Iw}}(\mathbb{Z}[1/\Sigma_n, \zeta_n], H^1(X(N), j_*\mathcal{V}^*_{k/\mathbb{Q}_p}(2)))$ satisfying the following property $:$ 
  for every $m\geqq 0$, 
 the image $\omega_m(z_v)\in S_k(N)_{\mathbb{Q}}\otimes_{\mathbb{Q}}\mathbb{Q}(\zeta_{np^m})\otimes_{\mathbb{Q}}\mathbb{Q}_p$ of 
 the element
 $z_v$ by the following composite
 
 \begin{multline*}
 H_{\mathrm{Iw}}^1(\mathbb{Z}[1/\Sigma_n, \zeta_n], H^1(X(N), j_*\mathcal{V}^*_{k/\mathbb{Q}_p}(2)))\xrightarrow{(a)}
 H_{\mathrm{Iw}}^1(\mathbb{Z}[1/\Sigma_n, \zeta_n], H^1(X(N), j_*\mathcal{V}_{k/\mathbb{Q}_p}(1)))
  \\\xrightarrow{\mathrm{can}}H^1(\mathbb{Z}[1/\Sigma_n, \zeta_{np^m}], H^1(X(N),j _*\mathcal{V}_{k/\mathbb{Q}_p}(1)))\\
 \xrightarrow{\mathrm{loc}_p}
 H^1(\mathbb{Q}(\zeta_{np^m})\otimes_{\mathbb{Q}}\mathbb{Q}_p, H^1(X(N), j_*\mathcal{V}_{k/\mathbb{Q}_p}(1)))
 \xrightarrow{\mathrm{exp}^*}S_k(N)_{\mathbb{Q}}\otimes_{\mathbb{Q}}\mathbb{Q}(\zeta_{np^m})\otimes_{\mathbb{Q}}\mathbb{Q}_p,
 \end{multline*}
 where the map $(a)$ is the twist by $((\zeta_{p^m})_{m\geqq 1})^{\otimes (1-k)}$ in Iwasawa cohomology, is contained in the subspace $S_k(N)_{\mathbb{Q}}\otimes_{\mathbb{Q}}\mathbb{Q}(\zeta_{np^m})
 $, and its image by the period map $$\mathrm{per} : S_k(N)_{\mathbb{Q}}\otimes_{\mathbb{Q}}\mathbb{Q}(\zeta_{np^m})
 \rightarrow H^1(X(N)(\mathbb{C}), j_*\mathcal{V}_{k/\mathbb{C}})\otimes_{\mathbb{C}}\mathbb{C}[\mathrm{Gal}(\mathbb{Q}(\zeta_{np^m})/\mathbb{Q})]$$ satisfies the equality 
 \begin{equation}
 \mathrm{per}(\omega_m(z_v))=\frac{1}{(2\pi i)^{k-1}} Z_{\Sigma, np^m}(k-1)\cdot s_N({}_c\delta_{N, np^m}(k, v))
 \end{equation}
 in the quotient $$H^1(X(N)(\mathbb{C}), j_*\mathcal{V}_{k/\mathbb{C}})\otimes_{\mathbb{C}[\{\pm1\}]}\mathbb{C}[\mathrm{Gal}(\mathbb{Q}(\zeta_{np^m})/\mathbb{Q})]$$ of 
 $H^1(X(N)(\mathbb{C}), j_*\mathcal{V}_{k/\mathbb{C}})\otimes_{\mathbb{C}}\mathbb{C}[\mathrm{Gal}(\mathbb{Q}(\zeta_{np^m})/\mathbb{Q})]$.
 \end{corollary}
 \begin{proof}
 This follows from Theorem \ref{2.1} and Lemma \ref{2.2} (1) since the composite 
 \begin{multline*}
 S_k(N)_{\mathbb{Q}}\otimes_{\mathbb{Q}}\mathbb{Q}(\zeta_{np^m})\xrightarrow{\mathrm{per}}H^1(X(N)(\mathbb{C}), j_*\mathcal{V}_{k/\mathbb{C}})\otimes_{\mathbb{C}}\mathbb{C}[\mathrm{Gal}(\mathbb{Q}(\zeta_{np^m})/\mathbb{Q})]\\
 \xrightarrow{\mathrm{can}}H^1(X(N)(\mathbb{C}), j_*\mathcal{V}_{k/\mathbb{C}})\otimes_{\mathbb{C}[\{\pm1\}]}\mathbb{C}[\mathrm{Gal}(\mathbb{Q}
 (\zeta_{np^m})/\mathbb{Q})]
 \end{multline*}
 is an injection, which follows from the isomorphism (\ref{ES2}). 
 \end{proof}
 \begin{rem}
 The author doesn't know whether the similar injectivity of the compsoite  $\mathrm{can}\circ \mathrm{per}$ in the proof above 
 holds if we replace $S_k(N)$ (resp. $X(N)$) with $M_k(N)$ (resp. $Y(N)$). 
 This is another reason why we consider parabolic cohomology instead of $H^1(Y(N), \mathcal{V}^*_{k/\mathbb{Z}_p})$. 
 \end{rem}

   For any integers $k\geqq 2$ and $n\geqq 1$, we write 
   $Q_{k, n}$ to denote the multiplicative set of $\Lambda_n$ generated by the elements of the form
   $(c^2-c^{2-j}\sigma_c)$ ($1\leqq j\leqq k-1$) for some $c\in \mathbb{Z}_{\geqq 2}$ such that $(c, np)=1$. 
   For any $\Lambda_n$-module $M$, we write $M_{Q_{k,n}}:=M\otimes_{\Lambda_n}\Lambda_n[1/Q_{k,n}]$. 
   Since every elements $(c^2-c^j\sigma_c)$ are  non-zero divisors in $\Lambda_n$,  Lemma \ref{2.2} (1) implies that 
   the canonical map 
   $$H_{\mathrm{Iw}}^1(\mathbb{Z}[1/\Sigma_n, \zeta_n], H^1(X(N), j_*\mathcal{V}^*_{k/\mathbb{Z}_p}(2)))\rightarrow 
   H_{\mathrm{Iw}}^1(\mathbb{Z}[1/\Sigma_n, \zeta_n], H^1(X(N), j_*\mathcal{V}^*_{k/\mathbb{Z}_p}(2)))_{Q_{k,n}}$$
   ia an injection.

   \begin{prop}\label{2.6}
   For each integers $N\geqq 3$ and $n\geqq 1$ such that $\mathrm{prime}(Np)=\Sigma$ and $(n, Np)=1$, there exists a unique $\mathbb{Z}_p[\mathrm{GL}_2(\mathbb{Z}/N\mathbb{Z})]$-linear map 
   $$z^{\mathrm{Iw}}_{N, n}(k,-) : H^1(Y(N), \mathcal{V}^*_{k/\mathbb{Z}_p})(1)\rightarrow 
   H^1_{\mathrm{Iw}}(\mathbb{Z}[1/\Sigma_n, \zeta_n], 
   H^1(X(N), j_*\mathcal{V}^*_{k/\mathbb{Q}_p}(2)))_{Q_{k,n}}$$ satisfying the following $:$ 
   \begin{itemize}
      \item[(1)]For every $c\geqq 2$ such that $c\equiv 1$ $(\bmod N)$ and $(c, 6pn)=1$, and $1\leqq j\leqq k-1$, one has
      \begin{multline*}
   (c^2-c^{2-j}\sigma_c)(c^2-c^{j-k+2}\sigma_c)\cdot z^{\mathrm{Iw}}_{N, n}(k, \delta_N(k, e_1^{j-1}e_2^{k-1-j}))\\
   =s_N\left({}_cz^{\mathrm{Iw}}_{N, n}(k, e_1^{j-1}e_2^{k-1-j})\right).
   \end{multline*}
     \end{itemize}
   
   Moreover, the map $z^{\mathrm{Iw}}_{N, n}(k,-)$ satisfies the following : 
   \begin{itemize}
   \item[(2)]For every $c\geqq 2$ such that $(c, 6Nnp)=1$ and $1\leqq j\leqq k-1$, one has the following equality
   \begin{multline*}
   \left(c^2-c^{2-j}\begin{pmatrix}c& 0 \\ 0 & 1\end{pmatrix}^*\otimes\sigma_c\right)\left(c^2-c^{j-k+2}\begin{pmatrix}1& 0 \\ 0 & c\end{pmatrix}^*\otimes\sigma_c\right)\cdot z^{\mathrm{Iw}}_{N, n}(k, \delta_N(k, e_1^{j-1}e_2^{k-1-j}))\\=s_N\left({}_cz^{\mathrm{Iw}}_{N, n}(k, e_1^{j-1}e_2^{k-1-j})\right).
   \end{multline*}
\item[(3)]For every $v\in H^1(Y(N), \mathcal{V}^*_{k/\mathbb{Z}_p})(1)$, one has 

     $$z^{\mathrm{Iw}}_{N, n}(k,\tau(v))=\sigma_{-1}(z^{\mathrm{Iw}}_{N, n}(k,v)).$$
     \end{itemize}

   \end{prop}
   \begin{proof}(of proposition)
   The proof of this proposition is essentially the same as that of Theorem 3.1.5 of \cite{FK}.  The uniqueness follows from the fact that 
   the elements $\delta_N(k, e_1^{j-1}e_2^{k-1-j})$ ($1\leqq j\leqq k-1$) are generators of the $\mathbb{Z}[\mathrm{GL}_2(\mathbb{Z}/N\mathbb{Z})]$-module 
   $H^1(Y(N)(\mathbb{C}), \mathcal{V}^*_{k/\mathbb{Z}})(1)$ by \cite{AS86}. 
   
   We first show the condition (2) assuming the existence of $z^{\mathrm{Iw}}_{N, n}(k, -)$ satisfying the condition (1). For any $c\geqq 2$ such that $(c, 6Nnp)=1$. We take another $c_1$ such that 
   $c_1\equiv 1$ ($\bmod$ $N$) and $(c_1, 6np)=1$. For each $1\leqq j\leqq k-1$, we set 
   $$d:=\left(c^2-c^{2-j}\begin{pmatrix}c& 0 \\ 0 & 1\end{pmatrix}^*\otimes\sigma_c\right)\left(c^2-c^{j-k+2}\begin{pmatrix}1& 0 \\ 0 & c\end{pmatrix}^*\otimes\sigma_c\right)$$
   and 
   $d_1:= (c_1^2-c_1^{2-j}\sigma_{c_1})(c_1^2-c_1^{j-k+2}\sigma_{c_1})\in Q_{k,n}$. Then, one has 
     $$d_1d\cdot z^{\mathrm{Iw}}_{N, n}(k, \delta_N(k, e_1^{j-1}e_2^{k-1-j}))=d\cdot s_N\left({}_{c_1}z^{\mathrm{Iw}}_{N, n}(k, e_1^{j-1}e_2^{k-1-j})\right)$$
     since we assume the condition (1). 
   Since one also has  equality 
  
   $$d\cdot s_N\left({}_{c_1}z^{\mathrm{Iw}}_{N, n}(k, e_1^{j-1}e_2^{k-1-j})\right)
   =d_1\cdot s_N\left({}_{c}z^{\mathrm{Iw}}_{N, n}(k, e_1^{j-1}e_2^{k-1-j})\right)$$
    by (the same proof as) Corollary \ref{2.4}, one obtains the equality 
    $$d\cdot z^{\mathrm{Iw}}_{N, n}(k, \delta_N(k, e_1^{j-1}e_2^{k-1-j}))=
   s_N\left({}_{c}z^{\mathrm{Iw}}_{N, n}(k, e_1^{j-1}e_2^{k-1-j})\right)$$
    since we have $d_1\in Q_{k,n}$, which shows the condition (2). 
    
    We next define the map $z^{\mathrm{Iw}}_{N, n}(k,-)$. We take  $c_1\geqq 2$ as above. For every $1\leqq j\leqq k-1$, we define the element 
    $$z^{\mathrm{Iw}}_{N, n}(k, \delta_N(k, e_1^{j-1}e_2^{k-1-j})):=d_1^{-1}\cdot s_N\left({}_{c_1}z^{\mathrm{Iw}}_{N, n}(k, e_1^{j-1}e_2^{k-1-j})\right)$$
    in $H^1_{\mathrm{Iw}}(\mathbb{Z}[1/\Sigma_n, \zeta_n], 
   H^1(X(N), j_*\mathcal{V}^*_{k/\mathbb{Q}_p}(2)))_{Q_{k,n}}$. By the same argument as above, we can show that this element does not depend on the choice of such $c_1$, and satisfies the equality 
   $$(c^2-c^{2-j}\sigma_c)(c^2-c^{j-k+2}\sigma_c)\cdot z^{\mathrm{Iw}}_{N, n}(k, \delta_N(k, e_1^{j-1}e_2^{k-1-j}))=s_N\left({}_cz^{\mathrm{Iw}}_{N, n}(k, e_1^{j-1}e_2^{k-1-j})\right)$$
   for every $c\geqq 2$ such that $c\equiv 1$ ($\bmod$ $N$) and $(c, 6np)=1$. For every $v=\sum_{j=1}^{k-1}a_je_1^{j-1}e_2^{k-1-j}\in \mathrm{Sym}^{k-2}(\mathbb{Z}^2)^*$ ($a_j\in \mathbb{Z}$), we define 
   $$z^{\mathrm{Iw}}_{N, n}(k, \delta_N(k, v)):=\sum_{j=1}^{k-1}a_j\cdot z^{\mathrm{Iw}}_{N, n}(k, \delta_N(k, e_1^{j-1}e_2^{k-j-1})),$$ 
   and, for every $v=\sum_{i=1}^dg_i\cdot\delta_N(k,v_i)\in H^1(Y(N), \mathcal{V}^*_{k/\mathbb{Z}})(1)$ ($g_i\in \mathrm{GL}_2(\mathbb{Z}/N\mathbb{Z}))$, we define
   \begin{equation}
   z^{\mathrm{Iw}}_{N, n}(k, v):=\sum_{i=1}^dg_i\cdot z^{\mathrm{Iw}}_{N, n}(k, \delta_N(k, v_i)).
   \end{equation}
   
   We show the well-definedness of $z^{\mathrm{Iw}}_{N, n}(k, v)$. For $c_1$ as above, we set $$d_2:=\prod_{j=1}^{k-1}(c_1^2-c_1^{2-j}\sigma_{c_1})(c_1^2-c_1^{j-k+2}\sigma_{c_1})\in \Lambda_n.$$ By definition, 
   the element $$d_2\cdot\left(\sum_{i=1}^dg_i\cdot z^{\mathrm{Iw}}_{N, n}(k, \delta_N(k, v_i))\right)$$ is in 
   $H^1_{\mathrm{Iw}}(\mathbb{Z}[1/\Sigma_n, \zeta_n], 
   H^1(X(N), j_*\mathcal{V}^*_{k/\mathbb{Q}_p}(2)))$, and, for every $m\geqq 0$, 
 its image $\omega_m\in S_k(N)_{\mathbb{Q}}\otimes_{\mathbb{Q}}\mathbb{Q}(\zeta_{np^m})\otimes_{\mathbb{Q}}\mathbb{Q}_p$ 
 by the composites
 
 \begin{multline*}
 H_{\mathrm{Iw}}^1(\mathbb{Z}[1/\Sigma_n, \zeta_n], H^1(X(N), j_*\mathcal{V}^*_{k/\mathbb{Q}_p}(2)))\xrightarrow{(a)}
 H_{\mathrm{Iw}}^1(\mathbb{Z}[1/\Sigma_n, \zeta_n], H^1(X(N), j_*\mathcal{V}_{k/\mathbb{Q}_p}(1)))
  \\\xrightarrow{\mathrm{can}}H^1(\mathbb{Z}[1/\Sigma_n, \zeta_{np^m}], H^1(X(N), j_*\mathcal{V}_{k/\mathbb{Q}_p}(1)))\\
 \xrightarrow{\mathrm{loc}_p}H^1(\mathbb{Q}(\zeta_{np^m})\otimes_{\mathbb{Q}}\mathbb{Q}_p, H^1(X(N), j_*\mathcal{V}_{k/\mathbb{Q}_p}(1)))
 \xrightarrow{\mathrm{exp}^*}S_k(N)_{\mathbb{Q}}\otimes_{\mathbb{Q}}\mathbb{Q}(\zeta_{np^m})\otimes_{\mathbb{Q}}\mathbb{Q}_p,
 \end{multline*}
 where the map $(a)$ is the twist by $((\zeta_{p^m})_{m\geqq 1})^{\otimes (1-k)}$ in Iwasawa cohomology, is contained in the subspace $S_k(N)_{\mathbb{Q}}\otimes_{\mathbb{Q}}\mathbb{Q}(\zeta_{np^m})
 $, and its image by the period map $$\mathrm{per} : S_k(N)_{\mathbb{Q}}\otimes_{\mathbb{Q}}\mathbb{Q}(\zeta_{np^m})
 \rightarrow H^1(X(N)(\mathbb{C}), j_*\mathcal{V}_{k/\mathbb{C}})\otimes_{\mathbb{C}}\mathbb{C}[\mathrm{Gal}(\mathbb{Q}(\zeta_{np^m})/\mathbb{Q})]$$ satisfies the equality 
 \begin{equation}
 \mathrm{per}(\omega_m)=\frac{1}{(2\pi i)^{k-1}} Z_{\Sigma, np^m}(k-1)(d_3\cdot s_N(v))
 \end{equation}
 for $d_3:=\prod_{j=1}^{k-1}(c_1^2-c_1^{k-j}\sigma_{c_1})(c_1^2-c_1^{j}\sigma_{c_1})\in \Lambda_n$
 in the quotient $$H^1(X(N)(\mathbb{C}), j_*\mathcal{V}_{k/\mathbb{C}})\otimes_{\mathbb{C}[\{\pm1\}]}\mathbb{C}[\mathrm{Gal}(\mathbb{Q}(\zeta_{np^m})/\mathbb{Q})]$$ of 
 $H^1(X(N)(\mathbb{C}), j_*\mathcal{V}_{k/\mathbb{C}})\otimes_{\mathbb{C}}\mathbb{C}[\mathrm{Gal}(\mathbb{Q}(\zeta_{np^m})/\mathbb{Q})]$. 
By the same argument as in the proof of Corollary \ref{2.4}, this shows that the element $z^{\mathrm{Iw}}_{N, n}(k, v)$ is well-defined. 
 
 By definition, it is clear that 
  the map 
  $$z^{\mathrm{Iw}}_{N, n}(k, -) : H^1(Y(N)(\mathbb{C}), \mathcal{V}^*_{k/\mathbb{Z}})(1)\rightarrow 
  H^1_{\mathrm{Iw}}(\mathbb{Z}[1/\Sigma_n, \zeta_n], 
   H^1(X(N), j_*\mathcal{V}^*_{k/\mathbb{Q}_p}(2)))_{Q_{k,n}}$$
   is $\mathrm{GL}_2(\mathbb{Z}/N\mathbb{Z})$-equivariant, and satisfies the condition (1) in the proposition. 
   
   To prove the proposition, it remains to show the equality
   \begin{equation}
   z^{\mathrm{Iw}}_{N, n}(k, \tau(v))=\sigma_{-1}\left(z^{\mathrm{Iw}}_{N, n}(k, v)\right)
   \end{equation}
   for any $v\in H^1(Y(N)(\mathbb{C}), \mathcal{V}^*_{k/\mathbb{Z}})(1)$.  Since  the elements $\delta_N(k, e_1^{j-1}e_2^{k-j-1})$ ($1\leqq j\leqq k-1$) are generators 
   of $\mathbb{Z}[\mathrm{GL}_2(\mathbb{Z}/N\mathbb{Z})]$-module $H^1(Y(N)(\mathbb{C}), \mathcal{V}^*_{k/\mathbb{Z}})(1)$, it suffices to show the 
   equality for $v=\delta_N(k, e_1^{j-1}e_2^{k-j-1})$ ($1\leqq j\leqq k-1$). Since one has 
   $$\tau(\delta_N(k, e_1^{j-1}e_2^{k-j-1}))=(-1)^{k-j-1}\begin{pmatrix}1& 0 \\ 0 & -1\end{pmatrix}^*\cdot \delta_N(k, e_1^{j-1}e_2^{k-j-1})$$
   for $\begin{pmatrix}1& 0 \\ 0 & -1\end{pmatrix}\in \mathrm{GL}_2(\mathbb{Z}/N\mathbb{Z})$ by (the proof of) Lemma 7.19 of \cite{Ka04}, it suffices to show, for every $1\leqq j\leqq k-1$, the 
   equality
   \begin{equation}
   \left(\begin{pmatrix}1& 0 \\ 0 & -1\end{pmatrix}^*\otimes \sigma_{-1}\right)\cdot {}_cz^{\mathrm{Iw}}_{N, n}(k, e_1^{j-1}e_2^{k-1-j})
   =(-1)^{k-j-1}{}_cz^{\mathrm{Iw}}_{N, n}(k, e_1^{j-1}e_2^{k-1-j})
\end{equation}
 for any $c\geqq 2$ such that $c\equiv 1$ ($\bmod$ $N$) and $(c, 6np)=1$. By definition of ${}_cz^{\mathrm{Iw}}_{N, n}(k, e_1^{j-1}e_2^{k-1-j})$, 
it suffices to show the equality 
\begin{equation}
\begin{pmatrix}1& 0 \\ 0 & -1\end{pmatrix}^*\left(e_1^{j-1}e_2^{k-1-j}\otimes {}_cz^{(p)}_{N_0p^{\infty}}\right)
   =(-1)^{k-j-1}e_1^{j-1}e_2^{k-1-j}\otimes {}_cz^{(p)}_{N_0p^{\infty}}
   \end{equation}
   for $\begin{pmatrix}1& 0 \\ 0 & -1\end{pmatrix}\in \varprojlim_{m\geqq 0}\mathrm{GL}_2(\mathbb{Z}/N_0p^m\mathbb{Z})$ in the module 
   $$\mathrm{Sym}^{k-2}(\mathbb{Z}_p^2)^*\otimes_{\mathbb{Z}_p}
   H^1(\mathbb{Z}[1/\Sigma_n], \widetilde{H}^{BM}_1(K^p(N_0)))$$ 
   for $N_0$ such that $nN=N_0p^m$ and $(N_0,p)=1$. Since we have 
   $$\begin{pmatrix}1& 0 \\ 0 & -1\end{pmatrix}^*\left(e_1^{j-1}e_2^{k-1-j}\otimes {}_cz^{(p)}_{N_0p^{\infty}}\right)
   =(-1)^{k-1-j}e_1^{j-1}e_2^{k-1-j}\otimes \begin{pmatrix}1& 0 \\ 0 & -1\end{pmatrix}^*{}_cz^{(p)}_{N_0p^{\infty}},$$
   it suffices to show the equality
   $$\begin{pmatrix}1& 0 \\ 0 & -1\end{pmatrix}^*{}_cz^{(p)}_{N_0p^{\infty}}={}_cz^{(p)}_{N_0p^{\infty}}.$$
   
   By definition of ${}_cz^{(p)}_{N_0p^{\infty}}$, it suffices to show the equality 
   $$\begin{pmatrix}1& 0 \\ 0 & -1\end{pmatrix}^*{}_cz_{N}={}_cz_{N}$$
   in $K_2(Y(N))$ for every $N\geqq 3$. Finally, this equality follows from the lemma below. 
     \end{proof}
     \begin{lemma}\label{2.8}
     For every $N\geqq 3$ and $c\geqq 2$ such that $(c, 6N)=1$, one has an equality 
      $$\tau_0\left({}_cz_{N}\right)={}_cz_{N}$$
      in $K_2(Y(N))$ for $\tau_0=\begin{pmatrix}1 & 0 \\0 &-1\end{pmatrix}\in \mathrm{GL}_2(\mathbb{Z}/N\mathbb{Z})$. 
     \end{lemma}
     \begin{proof}Since one has 
     $\tau_0({}_cg_{1/N, 0})={}_cg_{1/N, 0}$ and 
     $$\tau_0({}_cg_{0, 1/N})={}_cg_{0, -1/N}=\begin{pmatrix}-1 & 0 \\0 &-1\end{pmatrix}^*{}_cg_{0, 1/N}
     ={}_cg_{0, 1/N},$$ we obtain the equality
     $\tau_0\left({}_cz_{N}\right)=\{\tau_0({}_cg_{1/N, 0}), \tau_0({}_cg_{0, 1/N})\}=\{{}_cg_{1/N, 0}, {}_cg_{0,1/N}\}={}_cz_N$. 
     \end{proof}
     
     We set 
     $K_{\Sigma}(N):=\mathrm{Ker}(\prod_{l\in\Sigma} \mathrm{GL}_2(\mathbb{Z}_l)\rightarrow \mathrm{GL}_2(\mathbb{Z}/N\mathbb{Z}))$. For any 
     field $F$ of characteristic zero, we 
     write $\mathcal{H}(N)_{F}$ to denote the $F$-vector space
     of $F$-valued locally constant functions on $G_{\Sigma}$ with compact support which are left and right $K_{\Sigma}(N)$-invariant. If we choose 
     the $\mathbb{Q}$-valued Haar measure on $G_{\Sigma}$ such that $\mathrm{vol}(K_{\Sigma}(N))=1$, then the convolution product with respect to this measure makes 
     $\mathcal{H}(N)_{F}$ a $F$-algebra with its unit $e_{N}$ corresponding to the 
     characteristic function on $K_{\Sigma}(N)$. Since the $F$-vector space $H^1(Y(N), \mathcal{V}^*_{k/F})$ is equal to 
     the $K_{\Sigma}(N)$-fixed part of the $G_{\Sigma}$-representation $\varinjlim_{N'}H^1(Y(N'), \mathcal{V}^*_{k/F})$ where 
     $N'$ run through all $N'\geqq 3$ such that $\mathrm{prime}(N')=\Sigma$, 
     $H^1(Y(N), \mathcal{V}^*_{k/F})$ is naturally equipped with a structure of 
      $\mathcal{H}(N)_{F}$-module, and $H^1(X(N), j_*\mathcal{V}^*_{k/F})$ and $S_k(N)_{F}$ and $M_k(N)_{F}$ are similarly equipped with structures of 
      $\mathcal{H}(N)_{F}$-modules.

   \begin{corollary}\label{2.9}
   The map $z^{\mathrm{Iw}}_{N, n}(k,-)$ is $\mathbb{T}'_k(N)_{\mathbb{Z}_p}$-linear, and its base change 
   $$z^{\mathrm{Iw}}_{N, n}(k,-) : H^1(Y(N), \mathcal{V}^*_{k/\mathbb{Q}_p})(1)\rightarrow 
   H^1_{\mathrm{Iw}}(\mathbb{Z}[1/\Sigma_n, \zeta_n], 
   H^1(X(N), j_*\mathcal{V}^*_{k/\mathbb{Q}_p}(2)))_{Q_{k,n}}$$
   to $\mathbb{Q}_p$ is $\mathcal{H}_k(N)_{\mathbb{Q}_p}$-linear. 
   \end{corollary}
   \begin{proof}
   Since both the maps
   $$H^1_{\mathrm{Iw}}(\mathbb{Z}[1/\Sigma_n, \zeta_n], 
   H^1(X(N), j_*\mathcal{V}^*_{k/\mathbb{Q}_p}(2)))\xrightarrow{\mathrm{exp}^*}\varprojlim_{m\geqq 0}S_k(N)\otimes_{\mathbb{Q}}\mathbb{Q}(\zeta_{np^m})\otimes_{\mathbb{Q}}\mathbb{Q}_p$$
   and $$\mathrm{per} : S_k(N)\otimes_{\mathbb{Q}}\mathbb{Q}(\zeta_{np^m})\rightarrow H^1(X(N)(\mathbb{C}),  j_*\mathcal{V}_{k/\mathbb{C}})\otimes_{\mathbb{C}}\mathbb{C}
  [\mathrm{Gal}(\mathbb{Q}(\zeta_{np^m})/\mathbb{Q})]$$ 
  ($m\geqq 0$) are $\mathbb{T}'_k(N)_{\mathbb{Z}}$ and $\mathcal{H}_k(N)_{\mathbb{Q}}$-linear, the corollary immediately follows from the 
  characterization property of the map $z^{\mathrm{Iw}}_{N, n}(k,-)$ in Proposition \ref{2.6}.


   \end{proof}
   By the norm relation (\ref{c1}) for ${}_cz^{\mathrm{Iw}}_{N, n}(k,-)$, we immediately obtain the following corollary. 
   \begin{corollary}\label{2.10}
   For each integer $n\geqq 1$ such that $(n,\Sigma)=1$ and prime $l\not\in \Sigma$, one has the following commutative 
   diagram $:$ 
   \begin{equation*}
\begin{CD}
H^1(Y(N), \mathcal{V}^*_{k/\mathbb{Z}_p})(1)@>z^{\mathrm{Iw}}_{N, nl}(k,-)>> 
H^1_{\mathrm{Iw}}(\mathbb{Z}[1/\Sigma_{nl}, \zeta_{nl}], 
   H^1(X(N), j_*\mathcal{V}^*_{k/\mathbb{Q}_p}(2)))_{Q_{k,nl}} \\
@V \mathrm{id}VV  @ VV \mathrm{Cor} V \\
H^1(Y(N), \mathcal{V}^*_{k/\mathbb{Z}_p})(1)@> (*)>> 
H^1_{\mathrm{Iw}}(\mathbb{Z}[1/\Sigma_{nl}, \zeta_n], 
   H^1(X(N), j_*\mathcal{V}^*_{k/\mathbb{Q}_p}(2)))_{Q_{k,n}},
\end{CD}
\end{equation*}
where the map $(*)$ is 
\begin{equation*}
(*)=\begin{cases}z^{\mathrm{Iw}}_{N, n}(k,-) & \text{ if } l|n, 
\\ P_l(\sigma_l^{-1})z^{\mathrm{Iw}}_{N, n}(k,-)& \text{ if } (l, n)=1 .
\end{cases}
\end{equation*}
\end{corollary}

   \begin{lemma}\label{2.11}
There exist the following commutative diagrams $:$
\begin{itemize}
\item[(1)]For each $N|N'$ and $n\geqq 1$ such that $\mathrm{prime}(Np)=\mathrm{prime}(N'p)=\Sigma$ and $(n, \Sigma)=1$, one has the following commutative diagram $:$
\begin{equation*}
\begin{CD}
H^1(Y(N'), \mathcal{V}^*_{k/\mathbb{Z}_p})(1)@>z^{\mathrm{Iw}}_{N', n}(k,-)>> 
H^1_{\mathrm{Iw}}(\mathbb{Z}[1/\Sigma_n, \zeta_n], 
   H^1(X(N'), j_*\mathcal{V}^*_{k/\mathbb{Q}_p}(2)))_{Q_{k,n}} \\
@V \mathrm{Cor}VV  @ VV\mathrm{Cor} V \\
H^1(Y(N), \mathcal{V}^*_{k/\mathbb{Z}_p})(1)@>z^{\mathrm{Iw}}_{N, n}(k,-)>> 
H^1_{\mathrm{Iw}}(\mathbb{Z}[1/\Sigma_n, \zeta_n], 
   H^1(X(N), j_*\mathcal{V}^*_{k/\mathbb{Q}_p}(2)))_{Q_{k,n}}
\end{CD}
\end{equation*}
\item[(2)] For each $N|N'$ such that $\mathrm{prime}(N)=\mathrm{prime}(N')$ and $\mathrm{ord}_p(N)=\mathrm{ord}_p(N')$, one has the following commutative diagram $:$
\begin{equation*}
\begin{CD}
H^1(Y(N),\mathcal{V}^*_{k/\mathbb{Z}_p})(1)@>z^{\mathrm{Iw}}_{N, n}(k,-)>> 
H^1_{\mathrm{Iw}}(\mathbb{Z}[1/\Sigma_n, \zeta_n], 
   H^1(X(N), j_*\mathcal{V}^*_{k/\mathbb{Q}_p}(2)))_{Q_{k,n}}\\
@V \mathrm{Res} VV  @ VV\mathrm{Res} V \\
H^1(Y(N'),\mathcal{V}^*_{k/\mathbb{Z}_p})(1)@>z^{\mathrm{Iw}}_{N', n}(k,-)>> 
H^1_{\mathrm{Iw}}(\mathbb{Z}[1/\Sigma_n, \zeta_n], 
   H^1(X(N'), j_*\mathcal{V}^*_{k/\mathbb{Q}_p}(2)))_{Q_{k,n}}
\end{CD}
\end{equation*}
\end{itemize}

\end{lemma}
\begin{proof}
It suffices to show, for every $v\in \mathrm{Sym}^{k-2}(\mathbb{Z}_p^2)^*$,  the equalities 
\begin{equation}\label{a}
z^{\mathrm{Iw}}_{N, n}(\mathrm{Cor}(\delta_{N'}(k,v)))=\mathrm{Cor}\left(z^{\mathrm{Iw}}_{N', n}(\delta_{N'}(k,v))\right)
\end{equation}
for (1), and 
\begin{equation}\label{b}z^{\mathrm{Iw}}_{N', n}(\mathrm{Res}(\delta_{N}(k,v)))=\mathrm{Res}\left(z^{\mathrm{Iw}}_{N, n}(\delta_{N}(k,v))\right)
\end{equation}
 for (2). 

For the equation (\ref{a}),  since one has 
$\mathrm{Cor}(\delta_{N'}(k,v))=\delta_N(k,v)$, the equality follows from 
the equality 
$$\mathrm{Cor}({}_{c}z^{\mathrm{Iw}}_{N', n}(k,v))={}_{c}z^{\mathrm{Iw}}_{N, n}(k,v)$$
in $H^1_{\mathrm{Iw}}(\mathbb{Z}[1/\Sigma_n, \zeta_n], 
   H^1(Y(N), \mathcal{V}^*_{k/\mathbb{Z}_p})(2))$ for any $c\geqq 2$ such that $c\equiv 1$ (mod $N'$) and $(c, 6np)=1$. 

For the equation (\ref{b}), since one has 
$$\mathrm{Res}(\delta_N(k,v))=\sum_{\gamma\in H}\gamma\cdot\delta_{N'}(k,v)$$
 for $H:=\mathrm{Ker}(\mathrm{GL}_2(\mathbb{Z}/N'\mathbb{Z})\xrightarrow{\mathrm{can}} \mathrm{GL}_2(\mathbb{Z}/N\mathbb{Z}))$, 
 it suffices to show the equality
 \begin{equation}
 \mathrm{Res}({}_{c}z^{\mathrm{Iw}}_{N, n}(k,v))=\sum_{\gamma\in H}\gamma\cdot 
 {}_{c}z^{\mathrm{Iw}}_{N', n}(k,v)
 \end{equation}
 in $H^1_{\mathrm{Iw}}(\mathbb{Z}[1/\Sigma_n, \zeta_n], 
   H^1(Y(N'),\mathcal{V}^*_{k/\mathbb{Z}_p})(2))$
 for any $c\geqq 2$ such that $c\equiv 1$ (mod $N'$) and $(c, 6np)=1$. 
 By the assumption that $m_0:=v_p(N)=v_p(N')$, one has $H\isom \mathrm{Ker}(\mathrm{GL}_2(\mathbb{Z}/N'_0n\mathbb{Z})\rightarrow 
 \mathrm{GL}_2(\mathbb{Z}/N_0n\mathbb{Z}))$ for $N=N_0p^{m_0}, N'=N_0'p^{m_0}$. Hence, by definition of ${}_{c}z^{\mathrm{Iw}}_{N, n}(k,v)$, it suffices to show the 
 equality 
 \begin{equation}
 \mathrm{Res}(v\otimes {}_cz^{(p)}_{N_0np^{\infty}})=\sum_{\gamma\in H}\gamma\cdot (v\otimes {}_cz^{(p)}_{N'_0np^{\infty}})
 \end{equation}
 in $\mathrm{Sym}^{k-2}(\mathbb{Z}_p)^*\otimes_{\mathbb{Z}_p[[K_{m}]]}H^1(\mathbb{Z}[1/\Sigma_n], \widetilde{H}^{BM}_1(K^p(N'_0n))(1))$ 
 for every $m\geqq1$. 
 Since one has $$\mathrm{Res}(v\otimes {}_cz^{(p)}_{N_0np^{\infty}})=v\otimes \mathrm{Res}({}_cz^{(p)}_{N_0np^{\infty}})$$ 
 and $$\gamma\cdot (v\otimes {}_cz^{(p)}_{N'_0np^{\infty}})=v\otimes (\gamma\cdot {}_cz^{(p)}_{N'_0np^{\infty}})$$ for $\gamma\in H$, it suffices to show the equality 
 \begin{equation}
 \mathrm{Res}({}_cz^{(p)}_{N_0np^{\infty}})=\sum_{\gamma\in H}\gamma\cdot {}_cz^{(p)}_{N'_0np^{\infty}}
 \end{equation}
 in $H^1(\mathbb{Z}[1/\Sigma_n], \widetilde{H}^{BM}_1(K^p(N'_0n))(1))$. 
 Therefore, by definition of ${}_{c}z^{(p)}_{N_0np^{\infty}}$, it suffices to show the equality 
\begin{equation}
\mathrm{Res}({}_{c}z_{N_0np^m})=\sum_{\gamma\in H}\gamma\cdot 
{}_{c}z_{N'_0np^m}
\end{equation}
in $K_2(Y(N_0'np^m))$ for every $m\geqq 1$. Since one has
$$\mathrm{Norm}({}_{c}z_{N_0'np^m})={}_{c}z_{N_0np^m}$$
by the norm relation (\ref{e1}), we obtain the equality
$$\mathrm{Res}({}_{c}z_{N_0np^m})=\mathrm{Res}\circ \mathrm{Norm}({}_{c}z_{N_0'np^m})=\sum_{\gamma\in H}\gamma\cdot 
{}_{c}z_{N_0'np^m}.$$

 

\end{proof}
   
  We next fix an integer $N_0\geqq 1$ such that $\mathrm{prime}(N_0)=\Sigma\setminus \{p\}=:\Sigma_0$. 
  We set 
  $$K_{\Sigma_0}(N_0):=\mathrm{Ker}(\prod_{l\in \Sigma_0}\mathrm{GL}_2(\mathbb{Z}_l)\xrightarrow{\mathrm{can}}\mathrm{GL}_2(\mathbb{Z}/N_0\mathbb{Z})).$$
  By Lemma \ref{2.11} (1), 
  the projective limit of maps $z^{\mathrm{Iw}}_{N_0p^m, n}(2,-)$ for $k=2$ and $m\geqq 1$ gives the following map 
  \begin{equation}
  z^{\mathrm{Iw}}_{N_0p^{\infty}, n} : \widetilde{H}^{BM}_1(K^p(N_0)) \rightarrow 
  \varprojlim_{m\geqq 1}H^1_{\mathrm{Iw}}(\mathbb{Z}[1/\Sigma_n, \zeta_n], H^1(X(N_0p^m))_{\mathbb{Q}_p}(2))_{Q_{2,n}}
 \end{equation}
  which is $\mathbb{T}'(K_{\Sigma_0}(N_0))[\mathrm{GL}_2(\mathbb{Z}_p)]$-linear. We remark that 
  both the left and the right hand sides of the map are naturally equipped with an action of $G_p$. 
  Moreover, by Lemma \ref{2.11} (2), the injective limit of the maps $z^{\mathrm{Iw}}_{N_0p^{\infty}, n}$ 
  gives the following map 
  \begin{equation}
  z^{\mathrm{Iw}}_{\Sigma, n} : \widetilde{H}^{BM}_{1,\Sigma}\\
  \rightarrow \varinjlim_{N_0}\left(\varprojlim_{m\geqq 1}H^1_{\mathrm{Iw}}(\mathbb{Z}[1/\Sigma_n, \zeta_n], H^1(X(N_0p^m))_{\mathbb{Q}_p}(2)))_{Q_{2,n}}\right),
  \end{equation}
  where $N_0$ run through all $N_0$ such that $\mathrm{prime}(N_0)=\Sigma\setminus \{p\}$ and the transition maps
  are the restriction maps for every such $N_0|N_0'$. We remark that 
  both the left and the right hand sides of the map are naturally equipped with the action of $\mathrm{G}_{\Sigma}$. 
  By the characterization property of each maps $z^{\mathrm{Iw}}_{N_0p^m, n}(2,-)$, we immediately obtain the following lemma. 
    
  \begin{lemma}\label{2.12}
  \begin{itemize}
  \item[(1)]For any element $$v\in \varprojlim_{m\geqq 1}H^1(Y(N_0p^m)(\mathbb{C}), \mathbb{Z}(1)) \subseteq 
  \widetilde{H}_1^{BM}(K^p(N_0))$$ and $g\in G_p$, one has the equality 
  $$z^{\mathrm{Iw}}_{N_0p^{\infty}, n}(gv)=g\cdot z^{\mathrm{Iw}}_{N_0p^{\infty}, n}(v).$$
  \item[(2)]For any element $$v\in \varinjlim_{N_0}\left(\varprojlim_{m\geqq 1}H^1(Y(N_0p^m)(\mathbb{C}), \mathbb{Z}(1))\right) \subseteq 
  \widetilde{H}_{1,\Sigma}^{BM}$$ and $g\in \mathrm{G}_{\Sigma}$, one has the equality 
  $$z^{\mathrm{Iw}}_{\Sigma, n}(gv)=g\cdot z^{\mathrm{Iw}}_{\Sigma, n}(v).$$
  \end{itemize}
  \end{lemma}

 \subsection{Zeta morphisms for the \'etale cohomology localized at non-Eisenstein ideals}
 We fix $\Sigma$, $\mathcal{O}$, and $\overline{\rho} : G_{\mathbb{Q}}\rightarrow \mathrm{GL}_2(\mathbb{F})$ as in 
 \S 1.2. We set $\Sigma_0:=\Sigma\setminus \{p\}$. 
 
 For each $N_0\geqq 1$ such that $\mathrm{prime}(N_0)=\Sigma_0$, 
 and for each $N=N_0p^m\geqq 3$ and $m\geqq 0$, we recall that one has a canonical Hecke equivariant isomorphism 
 $$\mathrm{Sym}^{k-2}(\mathcal{O}^2)^*\otimes_{\mathcal{O}[[K_m]]}
 \widetilde{H}^{BM}_1(K^p(N_0))[1/p]\isom H^1(Y(N), \mathcal{V}^*_{k/E})(1)$$
 by Lemma \ref{1.10}. 
 Hence, one has a well-defined surjective $\mathcal{O}$-algebra homomorphism $\mathbb{T}'(K_{\Sigma_0}(N_0))\rightarrow \mathbb{T}'_k(N)_{\mathcal{O}} : 
 T'_l\, (\text{resp}. S'_l)\mapsto T'_l\, (\text{resp}. S'_l)$. Since we also have a canonical isomorphism 
 $$\mathbb{T}'(K_{\Sigma_0}(N_0))\isom \mathbb{T}(K_{\Sigma_0}(N_0)) : T'_l\, (\text{resp}. S'_l)\mapsto T_l\, (\text{resp}. S_l)$$ 
 in \S 2.3.2, we regard the ring $\mathbb{T}'_k(N)_{\mathcal{O}}$ as a $\mathbb{T}(K_{\Sigma_0}(N_0))$-algebra by this isomorphism. 
 
 We say that $N_0$ as above is allowable for $\overline{\rho}$ if the subgroup $K_{\Sigma_0}(N_0)$
 is allowable with respect to $\overline{\rho}$ in the sense of \S 2.1.2. Then, we recall that there exists a unique maximal ideal $\mathfrak{m}$ of 
 $\mathbb{T}(K_{\Sigma_0}(N_0))$ corresponding to $\overline{\rho}$. Hence, for any allowable $N_0$ and $m\geqq 0$ and $k\geqq 2$, and for any 
 $\mathbb{T}'_k(N_0p^m)_{\mathcal{O}}$-module $M$, we can define its localization at $\mathfrak{m}$ which we denote by 
 $M_{\overline{\rho}}:=M\otimes_{\mathbb{T}(K_{\Sigma_0}(N_0))}\mathbb{T}(K_{\Sigma_0}(N_0))_{\overline{\rho}}$.

 We recall that we assume that $\overline{\rho}$ is absolutely irreducible. Then, for any $N=N_0p^m$ as above, the localization 
 $$\mathrm{can} : H^1(X(N), j_*\mathcal{V}^*_{k/A})_{\overline{\rho}}\rightarrow H^1(Y(N),\mathcal{V}^*_{k/A})_{\overline{\rho}}$$
 (which we also denote by the same letter $\mathrm{can}$) of the canonical inclusion 
 $$\mathrm{can} : H^1(X(N), j_*\mathcal{V}^*_{k/A})\hookrightarrow H^1(Y(N),\mathcal{V}^*_{k/A})$$ is isomorphism for $A=\mathcal{O}, E$, 
 and its inverse for $A=E$ is the localization 
 $$s_N : H^1(Y(N),\mathcal{V}^*_{k/E})_{\overline{\rho}}\isom H^1(X(N), j_*\mathcal{V}^*_{k/E})_{\overline{\rho}}$$ 
 (which we also denote by $s_N$) of the splitting $$s_N : H^1(Y(N),\mathcal{V}^*_{k/E})\rightarrow H^1(X(N), j_*\mathcal{V}^*_{k/E}).$$ 
 
 Since the map $z^{\mathrm{Iw}}_{N, n}(k,-)$ is $\mathbb{T}'_k(N)_{\mathbb{Z}_p}$-linear, we can consider the localization 
 (of its base change from $\mathbb{Z}_p$ to $\mathcal{O}$)
\begin{equation}z^{\mathrm{Iw}}_{N, n, \overline{\rho}}(k,-) : H^1(Y(N),\mathcal{V}^*_{k/\mathcal{O}})_{\overline{\rho}}(1)\rightarrow 
   H^1_{\mathrm{Iw}}(\mathbb{Z}[1/\Sigma_n, \zeta_n], 
   H^1(X(N), j_*\mathcal{V}^*_{k/E})_{\overline{\rho}}(2))_{Q_{k,n}}.
   \end{equation}
   By definition of the map $z^{\mathrm{Iw}}_{N, n}(k,-)$, we immediately obtain the following corollary. 
   \begin{corollary}\label{2.13}
   The image of the composite 
   \begin{multline*}
   H^1(Y(N),\mathcal{V}^*_{k/\mathcal{O}})_{\overline{\rho}}(1)\xrightarrow{z^{\mathrm{Iw}}_{N, n, \overline{\rho}}(k,-)}
   H^1_{\mathrm{Iw}}(\mathbb{Z}[1/\Sigma_n, \zeta_n], 
   H^1(X(N), j_*\mathcal{V}^*_{k/E})_{\overline{\rho}}(2))_{Q_{k,n}}\\
   \xrightarrow{s_N^{-1}} H^1_{\mathrm{Iw}}(\mathbb{Z}[1/\Sigma_n, \zeta_n], 
   H^1(Y(N),\mathcal{V}^*_{k/E})_{\overline{\rho}}(2))_{Q_{k,n}}
   \end{multline*}
   is contained in the submodule $H^1_{\mathrm{Iw}}(\mathbb{Z}[1/\Sigma_n, \zeta_n], 
   H^1(Y(N),\mathcal{V}^*_{k/\mathcal{O}})_{\overline{\rho}}(2))_{Q_{k,n}}.$
 
   \end{corollary}     
   By this corollary, we obtain the map 
   \begin{multline}
   z^{\mathrm{Iw}}_{N, n, \overline{\rho}}(k,-):=s_N^{-1}\circ z^{\mathrm{Iw}}_{N, n, \overline{\rho}}(k,-) : 
   H^1(Y(N),\mathcal{V}^*_{k/\mathcal{O}})_{\overline{\rho}}(1)\\
   \rightarrow 
   H^1_{\mathrm{Iw}}(\mathbb{Z}[1/\Sigma_n, \zeta_n], 
   H^1(Y(N),\mathcal{V}^*_{k/\mathcal{O}})_{\overline{\rho}}(2))_{Q_{k,n}},
   \end{multline}
   which we also denote by the same letter $z^{\mathrm{Iw}}_{N, n, \overline{\rho}}(k,-)$. 
   
   The following proposition concerning the integrality of the map $z^{\mathrm{Iw}}_{N, n, \overline{\rho}}(k,-)$ is crucial 
   for our purpose. 
    \begin{prop}\label{2.14}
   The image of the map 
   $$z^{\mathrm{Iw}}_{N, n, \overline{\rho}}(k,-) : H^1(Y(N),\mathcal{V}^*_{k/\mathcal{O}})_{\overline{\rho}}(1)\rightarrow 
   H^1_{\mathrm{Iw}}(\mathbb{Z}[1/\Sigma_n, \zeta_n], 
   H^1(Y(N),\mathcal{V}^*_{k/\mathcal{O}})_{\overline{\rho}}(2))_{Q_{k,n}}$$
   is contained in the submodule $H^1_{\mathrm{Iw}}(\mathbb{Z}[1/\Sigma_n, \zeta_n], 
   H^1(Y(N), \mathcal{V}^*_{k/\mathcal{O}})_{ \overline{\rho}}(2)).$
   \end{prop}
   \begin{proof}We prove this proposition by reducing it to the integrality of Kato's zeta morphisms 
   proved in Theorem 12.5 \cite{Ka04} and Corollary \ref{4.5} of this article. 
   
   We write $\Gamma_0:=\mathrm{Gal}(\mathbb{Q}(\zeta_{p^{\infty}})/\mathbb{Q}(\zeta_{p}))$, 
   and $\Lambda_0:=\mathbb{Z}_p[[\Gamma_0]]$ which is a sub $\mathbb{Z}_p$-algebra of $\Lambda$ non-canonically isomorphic to $\mathbb{Z}_p[[T]]$, in particular, it is a regular local ring with its residue field $\mathbb{F}_p$. 
    We remark that the canonical decomposition 
   $$\mathrm{Gal}(\mathbb{Q}(\zeta_{np^{\infty}})/\mathbb{Q})\isom \mathrm{Gal}(\mathbb{Q}(\zeta_{n})/\mathbb{Q})\times 
   \mathrm{Gal}(\mathbb{Q}(\zeta_{p^{\infty}})/\mathbb{Q}):\sigma \mapsto (\sigma|_{\mathbb{Q}(\zeta_n)}, \sigma|_{\mathbb{Q}(\zeta_{p^{\infty}})})$$ induces a  
   canonical isomorphism $$\Lambda_n\isom \mathbb{Z}_p[\mathrm{Gal}(\mathbb{Q}(\zeta_n)/\mathbb{Q})]\otimes_{\mathbb{Z}_p}\Lambda.$$ We 
   regard the rings $\Lambda$ and $\Lambda_0$ as sub $\mathbb{Z}_p$-algebras of $\Lambda_n$ by the map 
   $$\Lambda \rightarrow \mathbb{Z}_p[\mathrm{Gal}(\mathbb{Q}(\zeta_n)/\mathbb{Q})]\otimes_{\mathbb{Z}_p}\Lambda : a\mapsto 1\otimes a,$$ by which we regard  $\Lambda_n$-modules also as $\Lambda_0$-modules. In particular, for any $c\in \mathbb{Z}$ such that $c\equiv 1$ ($\bmod$ $pn$),  
   we can regard
  the element $\sigma_c\in  \Lambda_n$ ($\sigma_c\in \Gamma_n$) 
  as an element in $\Lambda_0$ by this inclusion. 
  
  Since the base change 
   $$\bold{Dfm}_n(H^1(Y(N),\mathcal{V}^*_{k/\mathcal{O}})_{\overline{\rho}}(2))\otimes_{\Lambda_0}\mathbb{F}_p
   =H^1(Y(N),\mathcal{V}^*_{k/\mathcal{O}})_{\overline{\rho}}(2)\otimes_{\mathbb{Z}_p}\widetilde{\Lambda}_n\otimes_{\Lambda_0}\mathbb{F}_p$$ can be written 
    as $\mathcal{O}[G_{\mathbb{Q}}]$-module by a successive extension of the representations of the form
   $\overline{\rho}(\eta)$ for some characters 
   $\eta : \mathrm{Gal}(\mathbb{Q}(\zeta_{np})/\mathbb{Q})\rightarrow \mathbb{F}^{\times}$, 
   one has $$\mathrm{H}^0(\mathbb{Z}[1/\Sigma_n], \bold{Dfm}_n(H^1(Y(N),\mathcal{V}^*_{k/\mathcal{O}})_{\overline{\rho}})\otimes_{\Lambda_0}\mathbb{F}_p)=0$$ 
   because  $\overline{\rho}$ is absolutely irreducible. 
   Since the ring $\Lambda_0$ is regular local with its residue field $\mathbb{F}_p$, 
   this vanishing implies that $H^1_{\mathrm{Iw}}(\mathbb{Z}[1/\Sigma_n, \zeta_n], 
   H^1(Y(N),\mathcal{V}^*_{k/\mathcal{O}})_{\overline{\rho}}(2))$ is a finite free $\Lambda_0$-module by the same argument as in the proof 
   of Theorem 12.4 (3) of \cite{Ka04}. 
   
   By definition of the map $z^{\mathrm{Iw}}_{N, n, \overline{\rho}}(k,-)$, one has an inclusion 
   $$d\cdot \mathrm{Im}(z^{\mathrm{Iw}}_{N, n, \overline{\rho}}(k,-))
   \subseteq H^1_{\mathrm{Iw}}(\mathbb{Z}[1/\Sigma_n, \zeta_n], 
   H^1(Y(N), \mathcal{V}^*_{k/\mathcal{O}})_{\overline{\rho}})(2))$$ 
   for $d:=\prod_{j=1}^{k-1}(c^2-c^{2-j}\sigma_c)(c^2-c^{j-k+2}\sigma_c)\in \Lambda_n$ defined for 
   any integer $c\geqq 2$ such that $c\equiv 1$ (mod $pnN$) and $(c,6)=1$. 
   We remark that such $d\in \Lambda_n$ is 
   an element in $\Lambda_0$ by the remark above
   and can not be divided by $p$ in the ring $\Lambda_0$. By these facts, it suffices to show the inclusion after inverting $p$, i.e. it suffices to show that the image of the map 
   $$z^{\mathrm{Iw}}_{N, n, \overline{\rho}}(k,-) : H^1(Y(N),\mathcal{V}^*_{k/E})_{\overline{\rho}}(1)\rightarrow 
   H^1_{\mathrm{Iw}}(\mathbb{Z}[1/\Sigma_n, \zeta_n], 
   H^1(Y(N),\mathcal{V}^*_{k/E})_{\overline{\rho}}(2))_{Q_{k.n}}$$
   is contained in the subspace $H^1_{\mathrm{Iw}}(\mathbb{Z}[1/\Sigma_n, \zeta_n], 
   H^1(Y(N), \mathcal{V}^*_{k/E})_{\overline{\rho}}(2)).$ We reduce this claim to Theorem 12.5 \cite{Ka04} and Corollary \ref{4.5} of this article as follows. 
   Since replacing $E$ to any finite extension does not affect the validity of the claim, we freely enlarge $E$ below. 
   
   To show the claim, we first remark that one has a well-defined $\mathcal{O}$-algebra morphism 
   $$\mathbb{T}(K_{\Sigma_0}(N_0))\rightarrow \mathbb{T}_k(N)_{\mathcal{O}} : T_l \ \ (\text{resp.} S_l)\mapsto T_l \ \ (\text{resp.} S_l)$$
   by $(4.3.4)$ of \cite{Em06a}. By this morphism, we regard any $\mathbb{T}_k(N)_{\mathcal{O}}$-module $M$ also as a $\mathbb{T}(K_{\Sigma_0}(N_0))$-module, and 
   we set $M_{\overline{\rho}}:=M\otimes_{\mathbb{T}(K_{\Sigma_0}(N_0))}\mathbb{T}(K_{\Sigma_0}(N_0))_{\overline{\rho}}$. 
     Then,  the natural map 
   $$H^1(X(N), j_*\mathcal{V}_{k/E})_{\overline{\rho}}\rightarrow H^1(Y(N),\mathcal{V}_{k/E})_{\overline{\rho}}$$
    is isomorphism since $\overline{\rho}$ is absolutely irreducible. 
    Hence, by the compatibility of the global and the classical local Langlands correspondence (\cite{La73}, \cite{De71}, \cite{Ca86}, \cite{Sc90}, \cite{Sa97}), 
    one has 
   a decomposition 
   \begin{equation}
   H^1(Y(N),\mathcal{V}_{k/E})_{\overline{\rho}}=\bigoplus_{f}\rho_{f}\otimes_E (\pi(f))^{K(N)},
   \end{equation}
   (after replacing $E$ to a suitably large finite extension) where $f=\sum_{n=1}^{\infty}a_n(f)q^n$ run through all the normalized Hecke eigen cusp new forms  with weight $k$ and level $N_f$ dividing $N$ 
   such that the reduction $\overline{\rho}_{f} : G_{\mathbb{Q}}\rightarrow \mathrm{GL}_2(\mathbb{F})$ of its associated representation $\rho_{f,E} : G_{\mathbb{Q}}\rightarrow \mathrm{GL}_2(E)$ (see Appendix A for the definitions of $\rho_{f}$ and $\overline{\rho}_{f}$) is isomorphic to $\overline{\rho}$. 
   Here, $\pi(f)$ is an absolutely irreducible smooth admissible representation of $\mathrm{GL}_2(\mathbb{A}_f)$ defined over $E$ which is isomorphic to 
   the restricted tensor product of $\otimes'_{l}\pi_{l}(f)$ with $\pi_{l}(f)$ corresponding to $\rho_{f}|_{G_{\mathbb{Q}_l}}$ via Tate's normalized local 
   Langlands correspondence. 
   
   Since $H^1(Y(N),\mathcal{V}^*_{k/E})_{\overline{\rho}}(1)$ is the dual of 
   $H^1(Y(N),\mathcal{V}_{k/E})_{\overline{\rho}}$ by Poincar\'e duality, one obtains a similar decomposition 
   \begin{equation}
   H^1(Y(N),\mathcal{V}^*_{k/E})_{\overline{\rho}}(1)=\bigoplus_{f}(\rho_{f})^*\otimes_E(\widetilde{\pi}(f))^{K(N)},
   \end{equation}
   where $\widetilde{\pi}(f)$ is the smooth contragradient of $\pi(f)$. 
   
   Take any such $f$. By the strong multiplicity one theorem, the sub $E$-vector space 
   $$H^1(Y(N),\mathcal{V}^*_{k/E})_{\overline{\rho}}(1)[f]$$ consisting of elements $v\in H^1(Y(N),\mathcal{V}^*_{k/E})_{\overline{\rho}}(1)$ satisfying 
    $$T'_lv=a_l(f)v\ \ \text{ and }\ \  S'_lv=l^{k-2}\varepsilon_f(l)v$$ for all the  primes $l \not\in \Sigma$ is isomorphic to $(\rho_{f})^*
    \otimes_E(\widetilde{\pi}(f))^{K(N)}$. We also remark that 
    one has $$H^1(Y(N),\mathcal{V}^*_{k/E})(1)[f]=H^1(Y(N),\mathcal{V}^*_{k/E})_{\overline{\rho}}(1)[f].$$ 
   
    Since the map $z^{\mathrm{Iw}}_{N, n, \overline{\rho}}(k,-)$ is $\mathbb{T}'_k(N)_{\mathbb{Z}_p}$-linear, 
    it naturally induces a map 
    \begin{equation}
    z^{\mathrm{Iw}}_{N, n}(k,-)[f] : H^1(Y(N),\mathcal{V}^*_{k/E})(1)[f]\rightarrow 
   H^1_{\mathrm{Iw}}(\mathbb{Z}[1/\Sigma_n, \zeta_n], 
   H^1(Y(N),\mathcal{V}_{k/E}^*(2)[f])_{Q_{k.n}}
   \end{equation}
   for every such $f$. Therefore, it suffices to show the inclusion 
    \begin{equation}
    \mathrm{Im}(z^{\mathrm{Iw}}_{N, n}(k,-)[f])\subset H^1_{\mathrm{Iw}}(\mathbb{Z}[1/\Sigma_n, \zeta_n], 
   H^1(Y(N),\mathcal{V}_{k/E}^*)(2)[f])
   \end{equation}
   for every such $f$. 
   
   Let $N_f\geqq 1$ be the level of $f$. Then, we remark that the fixed part 
    \begin{equation}
    V'_2(f)_E:=\left(H^1(Y(N),\mathcal{V}^*_{k/E})(1)[f]\right)^{K_1(N_f)}=H^1(Y_1(N_f),\mathcal{V}^*_{k/E})(1)[f]
    \end{equation}
     by the subgroup $$K_1(N_f):=\left\{\left. g\in \mathrm{GL}_2(\widehat{\mathbb{Z}})\right| g\equiv \begin{pmatrix}* & * \\ 0 & 1\end{pmatrix} \bmod N_f\right\}$$
   is isomorphic to $(\rho_{f})^*$ by the theory of new vectors, and the map $z^{\mathrm{Iw}}_{N, n}(k,-)[f]$ 
   naturally induces the following map 
   \begin{equation}(z^{\mathrm{Iw}}_{N, n}(k,-)[f])^{K_1(N_f)} : V'_2(f)_E\rightarrow 
   H^1_{\mathrm{Iw}}(\mathbb{Z}[1/\Sigma_n, \zeta_n], V'_2(f)_E(1))_{Q_{k.n}}
   \end{equation}
   by the $\mathcal{H}_k(N)_{E}$-equivariance of $z^{\mathrm{Iw}}_{N, n}(k,-)$. 
   Since $(\pi(f))^{K(N)}$ is irreducible as $\mathcal{H}_k(N)_{E}$-module, it suffices to show that 
   the image of the map $(z^{\mathrm{Iw}}_{N, n}(k,-)[f])^{K_1(N_f)}$ is 
   contained in $H^1_{\mathrm{Iw}}(\mathbb{Z}[1/\Sigma_n, \zeta_n], V'_p(f)(1))$. 
   
   To show this claim, we similarly define $V'_2(f)_F:=H^1(Y_1(N_f),\mathcal{V}^*_{k/F})(1)[f]$ for $F=E\cap \overline{\mathbb{Q}}$, which is a $\mathrm{Gal}(\mathbb{C}/\mathbb{R})$-stable $F$-lattice of $V'_2(f)_E$, and $V_2(f^*)_A:=H^1(Y_1(N_f),\mathcal{V}_{k/A})[f]$,  $S_2(f^*)_A:=S_k(\Gamma_1(N_f))_A[f]$ for 
   $A=F, E, \mathbb{C}$ using Hecke actions of $T'_l$ and $S'_l$. By the same proof as that of Lemma \ref{4.3}, we can show that 
   one has an equality 
   $$V_2(f^*)_A=\{v\in H^1(Y_1(N_f),\mathcal{V}_{k/A})\mid T_lv=\overline{a_l}v, S_lv=l^{k-2}\overline{\varepsilon_f(l)}v \text{ for any } l\not\in \Sigma\},$$
   and a similar equality for $S_2(f^*)_A$.

   By definition of the map $z^{\mathrm{Iw}}_{N, n}(k,-)$, the induced map $z_f:=(z^{\mathrm{Iw}}_{N, n}(k,-)[f])^{K_1(N_f)}$ 
   is an $E$-linear map satisfying the following property : for each $m\geqq 0$, $\gamma\in V'_2(f)_F\subset V'_2(f)_E$, and  $c\geqq 2$ such that $c\equiv 1$ (mod $N$) and $(c, 6pn)=1$, one has 
   $d\cdot z_f(\gamma)\in H^1_{\mathrm{Iw}}(\mathbb{Z}[1/\Sigma_n, \zeta_n], V'_2(f)_E(1))$ for 
   $d=\prod_{1\leqq j\leqq k-1}(c^2-c^{2-j}\sigma_c)(c-c^{j-k+2}\sigma_c)\in Q_{k,n}$, and 
  the image $\omega_{\gamma, m, d}$ of $d\cdot z_f(\gamma)$ 
  by the following composite
  \begin{multline*}
  H^1_{\mathrm{Iw}}(\mathbb{Z}[1/\Sigma_{n}, \zeta_n], V'_2(f)_E(1))\xrightarrow{v\mapsto v\otimes (\zeta_{p^{n}})_{n\geqq 1}^{\otimes (-k)}} 
    H^1_{\mathrm{Iw}}(\mathbb{Z}[1/\Sigma_{n}, \zeta_n], V_2(f^*)_E)\\
    \xrightarrow{\mathrm{exp}^*_{m,1}}S_2(f^*)_E\otimes_{\mathbb{Q}}\mathbb{Q}(\zeta_{np^m})
  \end{multline*}
  belongs to $S_2(f^*)_F\otimes_{\mathbb{Q}}\mathbb{Q}(\zeta_{np^m})$, and the map 
  $$S_2(f^*)_F\otimes_{\mathbb{Q}}\mathbb{Q}(\zeta_{np^m})\rightarrow V_2(f^*)_{\mathbb{C}}^{\pm} : u\otimes v\mapsto \sum_{\sigma\in \mathrm{Gal}(\mathbb{Q}(\zeta_{np^m})/\mathbb{Q})}\chi(\sigma)\sigma(v)\mathrm{per}(u)^{\pm},$$where 
  $\chi : \mathrm{Gal}(\mathbb{Q}(\zeta_{np^m})/\mathbb{Q})\rightarrow \mathbb{C}^{\times}$ is any character and $\pm 1=\chi(-1)$, sends the element 
  $\omega_{\gamma, m, d}$ to 
  $$L_{\Sigma, n}(f,\chi, k-1)\cdot (d'\cdot \gamma)^{\pm}$$
  for $d'=\prod_{j=1}^{k-1}(c^2-c^{k-j}\sigma_c)(c^2-c^j\sigma_c)\in Q_{k,n}$. 
  
   In Appendix A, we define similar objects 
   $V'_1(f)_A, V_1(f^*)_A$ and $S_1(f)_A$ as quotients of $H^1(Y_1(N), \mathcal{V}^*_{k/A}), H^1(Y_1(N), \mathcal{V}^*_{k/A})$ and 
   $S_k(\Gamma_1(N_f))_A$ respectively, which are also equipped with a motivic structure (e.g. $\mathrm{Gal}(\mathbb{C}/\mathbb{R})$-action, period map, comparison 
   isomorphism of $p$-adic Hodge theory, etc.). Then, the composite
   $V'_2(f)_A\rightarrow V'_1(f)_A$ of the inclusion $V'_2(f)_A\subseteq H^1(Y_1(N), \mathcal{V}^*_{k/A})$ and the canonical 
   surjection $H^1(Y_1(N), \mathcal{V}^*_{k/A})\rightarrow V'_1(f)_A$, and the similar composites 
   $V_2(f^*)_A\rightarrow V_1(f^*)_A$ and 
   $S_2(f^*)_A\rightarrow S_1(f^*)_A$ are all isomorphism which are compatible with the corresponding motivic structure, by which 
   we identify these objects. 
   
   By the uniqueness of the map $\bold{z}_{n}(f)$ in Theorem \ref{4.5}, these facts imply that one has 
   equality 
   $$z_f=\prod_{l\in \Sigma_0}P_{f,l}(\sigma_l^{-1})\cdot \bold{z}_{n}(f)$$ 
   under the above identification. In particular, the image of $z_f$ is contained in 
   $H^1_{\mathrm{Iw}}(\mathbb{Z}[1/\Sigma_n, \zeta_n], V'_p(f)(1))$, which 
   shows the claim, hence finally proves the proposition.

    \end{proof}
    
    \subsection{Equivariant zeta morphisms for the completed homology}
   We fix an allowable integer $N_0\geqq 1$ as in \S 3.3 and $n\geqq 1$ such that $(n, \Sigma)=1$. 
    We recall that one has a canonical topological isomorphism 
    $$H^1_{\mathrm{Iw}}(\mathbb{Z}[1/\Sigma_n, \zeta_n], \widetilde{H}_1^{BM}(K^p(N_0))_{\overline{\rho}}(1))
    \isom \varprojlim_{m\geqq 1}H^1_{\mathrm{Iw}}(\mathbb{Z}[1/\Sigma_n, \zeta_n], H^1(Y(N_0p^m))_{\overline{\rho}}(2))$$
    of compact $\Lambda_n$-modules.
    By Lemma \ref{2.11} (1) and Proposition \ref{2.14}, one can take the projective limit of the maps 
 $$z^{\mathrm{Iw}}_{N_0p^m, n,\overline{\rho}}(2,-) : H^1(Y(N_0p^m))_{\overline{\rho}}(1)\rightarrow 
   H^1_{\mathrm{Iw}}(\mathbb{Z}[1/\Sigma_n, \zeta_n], 
   H^1(Y(N_0p^m))_{\overline{\rho}}(2))$$
    for all $m\geqq 1$, which we denote by  
    \begin{equation}
    z^{\mathrm{Iw}}_{N_0p^{\infty}, n, \overline{\rho}} : \widetilde{H}_1^{BM}(K^p(N_0))_{\overline{\rho}}
    \rightarrow H^1_{\mathrm{Iw}}(\mathbb{Z}[1/\Sigma_n, \zeta_n], \widetilde{H}_1^{BM}(K^p(N_0))_{\overline{\rho}}(1)).
    \end{equation}
    
    By Lemma \ref{2.11} (2), one can take the inductive limit of the map $z^{\mathrm{Iw}}_{N_0p^{\infty}, n, \overline{\rho}}$ 
    for all such $N_0$ to obtain the following map 
    \begin{equation}
    z^{\mathrm{Iw}}_{\Sigma, n, \overline{\rho}} : \widetilde{H}^{BM}_{1, \overline{\rho},\Sigma}
    \rightarrow H^1_{\mathrm{Iw}}(\mathbb{Z}[1/\Sigma_n, \zeta_n], \widetilde{H}^{BM}_{1, \overline{\rho},\Sigma}(1)).
    \end{equation}
    
    \begin{thm}\label{2.15}
    For each intger $n\geqq 1$ such that $(n, \Sigma)=1$, the map 
    $z^{\mathrm{Iw}}_{\Sigma, n, \overline{\rho}}$ is continuous  and 
    $\mathbb{T}_{\overline{\rho},\Sigma}[G_{\Sigma}]$-linear, and satisfies 
    $$z^{\mathrm{Iw}}_{\Sigma, n, \overline{\rho}}(\tau(v))=\sigma_{-1}(z^{\mathrm{Iw}}_{\Sigma, n, \overline{\rho}}(v))$$ for every 
    $v\in \widetilde{H}^{BM}_{1,  \overline{\rho},\Sigma}$. Moreover, for each 
    prime $l\not\in \Sigma$, one has the following commutative diagram $:$ 
     \begin{equation*}
\begin{CD}
\widetilde{H}^{BM}_{1, \overline{\rho},\Sigma}@>z^{\mathrm{Iw}}_{\Sigma, nl, \overline{\rho}} >> 
H^1_{\mathrm{Iw}}(\mathbb{Z}[1/\Sigma_{nl}, \zeta_{nl}], 
   \widetilde{H}^{BM}_{1, \overline{\rho},\Sigma}(1)) \\
@V \mathrm{id} VV  @ VV \mathrm{Cor} V \\
\widetilde{H}^{BM}_{1, \overline{\rho},\Sigma}@> (*)>> 
H^1_{\mathrm{Iw}}(\mathbb{Z}[1/\Sigma_{nl}, \zeta_n], 
   \widetilde{H}^{BM}_{1,  \overline{\rho},\Sigma}(1)),
\end{CD}
\end{equation*}
where the map $(*)$ is 
\begin{equation*}
(*)=\begin{cases}z^{\mathrm{Iw}}_{\Sigma, n, \overline{\rho}}  & \text{ if } \quad l|n , 
\\ P_l(\sigma_l^{-1})z^{\mathrm{Iw}}_{\Sigma, n, \overline{\rho}} & \text{ if } \quad (l, n)=1 .
\end{cases}
\end{equation*}
 \end{thm}
 \begin{proof}
 By definition, the map $z^{\mathrm{Iw}}_{\Sigma, n, \overline{\rho}}$ is clearly continuous and $\mathbb{T}_{\overline{\rho},\Sigma}$-linear and satisfies the commutative diagram above. We note that, for each allowable $N_0\geqq 1$ such that 
 $\mathrm{prime}(N_0)=\Sigma\setminus\{p\}$, 
 the image of the module $\varprojlim_{m\geqq 1}H^1(Y(N_0p^m)(\mathbb{C}), \mathbb{Z}(1))$ in 
 $\widetilde{H}^{BM}_1(K^p(N_0))_{\overline{\rho}}$ topologically generates the latter group as $\mathbb{T}_{\overline{\rho}, \Sigma}$-module. Hence, $G_{\Sigma}$-equivariance of the map 
 $z^{\mathrm{Iw}}_{\Sigma, n, \overline{\rho}}$ follows from Lemma \ref{2.12} because the map $z^{\mathrm{Iw}}_{\Sigma, n, \overline{\rho}}$ is continuous. 

 \end{proof}
 
 We remark that, for each $k\geqq 2$ and $m_1\geqq 0$, one has a canonical isomorphism 
 $$\mathrm{Sym}^{k-2}(\mathcal{O}^2)^*\otimes_{\mathcal{O}[[K_{m_1}]]}\widetilde{H}^{BM}_{1}(K^p(N_0))_{\overline{\rho}}[1/p]
 \isom H^1(Y(N), \mathcal{V}^*_{k/E})_{\overline{\rho}}(1)$$
 for $N=N_0p^{m_1}$ by Lemma \ref{1.10}. 
  By this isomorphism, we regard the 
 map
 \begin{multline*}
 \mathrm{Sym}^{k-2}(\mathcal{O}^2)^*\otimes_{\mathcal{O}[[K_{m_1}]]}\widetilde{H}^{BM}_{1}(K^p(N_0))_{\overline{\rho}}[1/p]
 \xrightarrow{\mathrm{id}\otimes z^{\mathrm{Iw}}_{N_0p^{\infty}, n,\overline{\rho}}}\\
  \mathrm{Sym}^{k-2}(\mathcal{O}^2)^*\otimes_{\mathcal{O}[[K_{m_1}]]}
  H^1_{\mathrm{Iw}}(\mathbb{Z}[1/\Sigma_n, \zeta_n], \widetilde{H}^{BM}_{1}(K^p(N_0))_{\overline{\rho}}(1))[1/p]\\
\xrightarrow{\mathrm{can}}H^1_{\mathrm{Iw}}(\mathbb{Z}[1/\Sigma_n, \zeta_n], 
   \mathrm{Sym}^{k-2}(\mathcal{O}^2)^*\otimes_{\mathcal{O}[[K_{m_1}]]}\widetilde{H}^{BM}_{1}(K^p(N_0))_{ \overline{\rho}}(1)[1/p])
 \end{multline*} as the following map 
 \begin{equation}
 \mathrm{id}\otimes z^{\mathrm{Iw}}_{N_0p^{\infty}, n,\overline{\rho}} : 
 H^1(Y(N), \mathcal{V}^*_{k/E})_{\overline{\rho}}(1)\rightarrow H^1_{\mathrm{Iw}}(\mathbb{Z}[1/\Sigma_n, \zeta_n], H^1(Y(N), \mathcal{V}^*_{k/E})_{ \overline{\rho}}(2)).
 \end{equation}
 \begin{prop}\label{2.16}
 Under the identification above, one has 
 $$\mathrm{id}\otimes z^{\mathrm{Iw}}_{N_0p^{\infty}, n,\overline{\rho}}=z^{\mathrm{Iw}}_{N, n, \overline{\rho}}(k,-).$$
  
 \end{prop}
 \begin{proof}
 For each $N\geqq 3$, we set $\delta_N:=\delta_N(2,x)|_{x=1}\in H^1(Y(N)(\mathbb{C}), \mathbb{Z}(1))$ for $x=1\in \mathbb{Z}=\mathrm{Sym}^{2-2}(\mathbb{Z})^*$.
 Then one has $(\delta_{N_0p^m})_{m\geqq 0} \in \widetilde{H}^{BM}_{1}(K^p(N_0))$.
 By definition of the map $\delta_N(k,-) : \mathrm{Sym}^{k-2}(\mathbb{Z}^2)^*\rightarrow H^1(Y(N)(\mathbb{C}), \mathcal{V}^*_{k/\mathbb{Z}})(1)$ and the isomorphism
 $$f : \mathrm{Sym}^{k-2}(\mathcal{O}^2)^*\otimes_{\mathcal{O}[[K_{m_1}]]}\widetilde{H}^{BM}_{1}(K^p(N_0))[1/p]
 \isom H^1(Y(N), \mathcal{V}^*_{k/E})(1)$$
 which we denote by $f$, one can easily check the equality
 $$\delta_N(k, v)=f(v\otimes (\delta_{N_0p^m})_{m\geqq 0})$$
 for any $v\in \mathrm{Sym}^{k-2}(\mathbb{Z}^2)^*$. Therefore, it suffices to check that, for every $1\leqq j\leqq k-1$, the image 
 of $e_1^{j-1}e_2^{k-j-1}\otimes  (\delta_{N_0p^m})_{m\geqq 0}$ by the following composite which we denote by $g$
 \begin{multline*}
  g : \mathrm{Sym}^{k-2}(\mathcal{O}^2)^*\otimes_{\mathcal{O}[[K_{m_1}]]}\widetilde{H}^{BM}_{1}(K^p(N_0))_{\overline{\rho}}[1/p]
 \xrightarrow{\mathrm{id}\otimes z^{\mathrm{Iw}}_{N_0p^{\infty}, n,\overline{\rho}}}\\
  \mathrm{Sym}^{k-2}(\mathcal{O}^2)^*\otimes_{\mathcal{O}[[K_{m_1}]]}
  H^1_{\mathrm{Iw}}(\mathbb{Z}[1/\Sigma_n, \zeta_n], \widetilde{H}^{BM}_{1}(K^p(N_0))_{\overline{\rho}}(1))[1/p]\\
  \xrightarrow{\mathrm{can}}
  H^1_{\mathrm{Iw}}(\mathbb{Z}[1/\Sigma_n, \zeta_n], 
   \mathrm{Sym}^{k-2}(\mathcal{O}^2)^*\otimes_{\mathcal{O}[[K_{m_1}]]}\widetilde{H}^{BM}_{1}(K^p(N_0))_{\overline{\rho}}(1)[1/p])\\
   \xrightarrow{f}H^1_{\mathrm{Iw}}(\mathbb{Z}[1/\Sigma_n, \zeta_n], H^1(Y(N), \mathcal{V}^*_{k/E})_{\overline{\rho}}(2))
 \end{multline*}
 is equal to $z^{\mathrm{Iw}}_{N, n, \overline{\rho}}(k,\delta_N(k, e_1^{j-1}e_2^{k-j-1}))$ because both the maps $\mathrm{id}\otimes z^{\mathrm{Iw}}_{N_0p^{\infty}, 
 n,\overline{\rho}}$ and $z^{\mathrm{Iw}}_{N, n, \overline{\rho}}(k,-)$ are 
 $\mathbb{T}_{\overline{\rho}, \Sigma}[\mathrm{GL}_2(\mathbb{Z}/N\mathbb{Z})]$-linear and the elements $\delta_N(k, e_1^{j-1}e_2^{k-j-1})$ generate $H^1(Y(N), \mathcal{V}^*_{k/E})_{\overline{\rho}}(1)$. 
 
 To prove this claim, we take an integer $c\geqq 2$ such that $c\equiv 1$ (mod $N$) and $(c,6np)=1$, and 
 set $d_j:=(c^2-c^{2-j}\sigma_c)(c^2-c^{j-k+2}\sigma_c)\in \Lambda_n$. Then, it suffices to show the equality 
 \begin{equation}d_j\cdot g(e_1^{j-1}e_2^{k-j-1}\otimes  (\delta_{N_0p^m})_{m\geqq 0})=d_j\cdot z^{\mathrm{Iw}}_{N, n, \overline{\rho}}(k,\delta_N(k, e_1^{j-1}e_2^{k-j-1}))
 \end{equation}
 since $d_j$ is a non zero divisor of $\Lambda_n$ and 
 $H^1_{\mathrm{Iw}}(\mathbb{Z}[1/\Sigma_n, \zeta_n], H^1(Y(N), \mathcal{V}^*_{k/E})_{ \overline{\rho}}(2))$ is a torsion free $\Lambda_n$ module by 
 Lemma \ref{2.2} (1). 
 
 By definition of $z^{\mathrm{Iw}}_{N, n, \overline{\rho}}(k,-)$, one has an equality 
 \begin{equation}
 d_j\cdot z^{\mathrm{Iw}}_{N, n, \overline{\rho}}(k,\delta_N(k, e_1^{j-1}e_2^{k-j-1}))={}_cz^{\mathrm{Iw}}_{N, n}(k, e_1^{j-1}e_2^{k-j-1})
 \end{equation}
  in $H^1_{\mathrm{Iw}}(\mathbb{Z}[1/\Sigma_n, \zeta_n], H^1(Y(N), \mathcal{V}^*_{k/E})_{\overline{\rho}}(2))$. 
  
  We next compute the left hand side 
  $d_j\cdot g(e_1^{j-1}e_2^{k-j-1}\otimes  (\delta_{N_0p^m})_{m\geqq 0})$. Since the elements 
  $\begin{pmatrix}c& 0\\0 & 1\end{pmatrix}, \begin{pmatrix}1& 0\\0 & c\end{pmatrix}\in \varprojlim_{m\geqq 0}\mathrm{GL}_2(\mathbb{Z}/p^{m}N_0\mathbb{Z})=:G_0$ are in the subgroup 
  $$\mathrm{Ker}(G_0\rightarrow \mathrm{GL}_2(\mathbb{Z}/N\mathbb{Z}))\isom K_{m_1},$$ the element
  $$d_j\cdot(e_1^{j-1}e_2^{k-j-1}\otimes z^{\mathrm{Iw}}_{N_0p^{\infty}, n,\overline{\rho}}((\delta_{N_0p^m})_{m\geqq 0}))$$ in 
  $\mathrm{Sym}^{k-2}(\mathcal{O}^2)^*\otimes_{\mathcal{O}[[K_{m_1}]]}
  H^1_{\mathrm{Iw}}(\mathbb{Z}[1/\Sigma_n, \zeta_n], \widetilde{H}^{BM}_{1}(K^p(N_0))_{\overline{\rho}}(1))$ is equal to 
  \begin{multline*}
  \left(c^2-c^{2-j}\begin{pmatrix}c& 0\\0 & 1\end{pmatrix}\sigma_c\right)\left(c^2-c^{j-k+2} \begin{pmatrix}1& 0\\0 & c\end{pmatrix}\sigma_c\right)\cdot(e_1^{j-1}e_2^{k-j-1}\otimes z^{\mathrm{Iw}}_{N_0p^{\infty}, n,\overline{\rho}}((\delta_{N_0p^m})_{m\geqq 0}))\\
  =e_1^{j-1}e_2^{k-j-1}\otimes d^0_j\cdot z^{\mathrm{Iw}}_{N_0p^{\infty}, n,\overline{\rho}}((\delta_{N_0p^m})_{m\geqq 0}))
  \end{multline*}
  for $d^0_j= \left(c^2-c\begin{pmatrix}c& 0\\0 & 1\end{pmatrix}\sigma_c\right)\left(c^2-c \begin{pmatrix}1& 0\\0 & c\end{pmatrix}\sigma_c\right)$. 
  Since one has 
  $$d^0_j\cdot z^{\mathrm{Iw}}_{N_0p^{\infty}, n,\overline{\rho}}((\delta_{N_0p^m})_{m\geqq 0}))=({}_cz^{\mathrm{Iw}}_{N_0p^m, n})_{m\geqq 0}=
  {}_cz^{\mathrm{Iw}}_{N_0p^{\infty}, n}$$ by definition of $z^{\mathrm{Iw}}_{N_0p^{\infty}, n,\overline{\rho}}$ and Proposition \ref{2.6}, we obtain 
  $$d_j\cdot(e_1^{j-1}e_2^{k-j-1}\otimes z^{\mathrm{Iw}}_{N_0p^{\infty}, n,\overline{\rho}}((\delta_{N_0p^m})_{m\geqq 0}))
  ={}_cz^{\mathrm{Iw}}_{N, n}(k, e_1^{j-1}e_2^{k-j-1})$$
  by definition of ${}_cz^{\mathrm{Iw}}_{N, n}(k,-)$.

 \end{proof}

    \section{Zeta morphisms for rank two universal deformations}
    In this section except in \S 4.3, we assume the following $:$ 
\begin{itemize}
\item[(1)]$p\geqq 5$.
\item[(2)]$\mathrm{End}_{\mathbb{F}[G_{\mathbb{Q}_p}]}(\overline{\rho}_p)=\mathbb{F}$.
\item[(3)]$\overline{\rho}_p$ is not of the form $\begin{pmatrix}1& *\\ 0 & \overline{\varepsilon}^{\pm 1}\end{pmatrix}
\otimes \chi\,  $ for any character $\chi : G_{\mathbb{Q}_p}\rightarrow \mathbb{F}^{\times}$.
\end{itemize}
We need these condition when we use theorems on $p$-adic local Langlands correspondence \cite{Co10}, \cite{Em}, and \cite{Pas13}, \cite{Pas15}.

  \subsection{Definition of zeta morphisms for rank two universal deformations}
  
  We combine the results in previous two sections. We first fix an isomorphism 
  \begin{equation}
  \phi_1 : \widetilde{H}^{BM}_{1,\overline{\rho},\Sigma}\isom (\pi^{\mathfrak{m}}_p)^*\otimes_{\mathbb{T}_{\overline{\rho},\Sigma}}
  (\rho^{\mathfrak{m}})^*\otimes_{\mathbb{T}_{\overline{\rho} ,\Sigma}}\widetilde{\pi}^{\mathfrak{m}}_{\Sigma_0}
  \end{equation}
  in Proposition \ref{1.5}. 
  
  Let $N_0\geqq 1$ be an integer such that $\mathrm{prime}(N_0)=\Sigma_0$ which is allowable with respect to $\overline{\rho}$. Since 
  $(\widetilde{\pi}^{\mathfrak{m}}_{\Sigma_0})^{K^p(N_0)}$ is a finite generated $\mathbb{T}_{\overline{\rho},\Sigma}$-module, 
   the Iwasawa cohomology $H^1_{\mathrm{Iw}}(\mathbb{Z}[1/\Sigma_n, \zeta_n], (\rho^{\mathfrak{m}})^*(1)
  \otimes_{\mathbb{T}_{\overline{\rho} ,\Sigma}}(\widetilde{\pi}^{\mathfrak{m}}_{\Sigma_0})^{K^p(N_0)})$ is a finite generated 
  $\mathbb{T}_{\overline{\rho},\Sigma}\widehat{\otimes}_{\mathbb{Z}_p}\Lambda_n$-module for each $n\geqq 1$ such that $(n,\Sigma)=1$ by 
  the standard finiteness results in the theory of Galois cohomology. In particular, it is naturally a compact $\mathbb{T}_{\overline{\rho},\Sigma}$-module. 
  Then, we write 
  \begin{multline*}
  (\pi_p^{\mathfrak{m}})^*\widehat{\otimes}_{\mathbb{T}_{\overline{\rho},\Sigma}}
  H^1_{\mathrm{Iw}}(\mathbb{Z}[1/\Sigma_n, \zeta_n], 
  (\rho^{\mathfrak{m}})^*(1)
  \otimes_{\mathbb{T}_{\overline{\rho} ,\Sigma}}\widetilde{\pi}^{\mathfrak{m}}_{\Sigma_0})\\
  :=\varinjlim_{N_0}
 (\pi_p^{\mathfrak{m}})^*\widehat{\otimes}_{\mathbb{T}_{\overline{\rho},\Sigma}}
 H^1_{\mathrm{Iw}}(\mathbb{Z}[1/\Sigma_n, \zeta_n], 
  (\rho^{\mathfrak{m}})^*(1)
  \otimes_{\mathbb{T}_{\overline{\rho} ,\Sigma}}(\widetilde{\pi}^{\mathfrak{m}}_{\Sigma_0})^{K^p(N_0)}).
  \end{multline*}


   \begin{lemma}\label{3.1}
   For each $n\geqq 1$ such that $(n,\Sigma)=1$, 
   the isomorphism $\phi_1$ naturally induces a $\mathbb{T}_{\overline{\rho},\Sigma}\widehat{\otimes}_{\mathbb{Z}_p}\Lambda_n[G_{\Sigma}]$-linear isomorphism
   \begin{equation}
   \widetilde{\phi}_1 : H^1_{\mathrm{Iw}}(\mathbb{Z}[1/\Sigma_n, \zeta_n], \widetilde{H}^{BM}_{1, \overline{\rho},\Sigma}(1))\isom 
   (\pi_p^{\mathfrak{m}})^*\widehat{\otimes}_{\mathbb{T}_{\overline{\rho},\Sigma}}
   H^1_{\mathrm{Iw}}(\mathbb{Z}[1/\Sigma_n, \zeta_n],(\rho^{\mathfrak{m}})^*(1)
  \otimes_{\mathbb{T}_{\overline{\rho} ,\Sigma}}\widetilde{\pi}_{\Sigma_0}^{\mathfrak{m}})
  \end{equation}
   \end{lemma}
    \begin{proof}
    By the standard finiteness result in the theory of Galois cohomology, there exists a perfect complex $N^{\bullet}$ of 
    $\mathbb{T}_{\overline{\rho},\Sigma}\widehat{\otimes}_{\mathbb{Z}_p}\Lambda_n$-module, and 
 a $\mathbb{T}_{\overline{\rho},\Sigma}\widehat{\otimes}_{\mathbb{Z}_p}\Lambda_n$-linear quasi-isomorphism 
 $$\psi : C_{\mathrm{Iw}}^{\bullet}(\mathbb{Z}[1/\Sigma_n, \zeta_n], (\rho^{\mathfrak{m}})^*(1))\rightarrow N^{\bullet}.$$ 
 Applying Corollary \ref{5.32} to $A=\mathbb{T}_{\overline{\rho},\Sigma}\widehat{\otimes}_{\mathbb{Z}_p}\Lambda_n$, 
 $\rho=\bold{Dfm}_n((\rho^{\mathfrak{m}})^*(1))$, $P= (\pi_p^{\mathfrak{m}})^*\widehat{\otimes}_{\mathbb{Z}_p}\Lambda_n$ or
 $P=\mathbb{T}_{\overline{\rho},\Sigma}\widehat{\otimes}_{\mathbb{Z}_p}\Lambda_n$, 
 and $M=\widetilde{\pi}_{\Sigma_0}^{\mathfrak{m}}\otimes_{\mathbb{T}_{\overline{\rho},\Sigma}}
 (\mathbb{T}_{\overline{\rho},\Sigma}\widehat{\otimes}_{\mathbb{Z}_p}\Lambda_n)$, $\psi$ naturally induces quasi-isomorphisms
 
 $$\psi_1 : C_{\mathrm{Iw}}^{\bullet}(\mathbb{Z}[1/\Sigma_n, \zeta_n],  (\pi_p^{\mathfrak{m}})^*\otimes_{\mathbb{T}_{\overline{\rho},\Sigma}}(\rho^{\mathfrak{m}})^*(1)
 \otimes_{\mathbb{T}_{\overline{\rho}, \Sigma}}\widetilde{\pi}_{\Sigma_0}^{\mathfrak{m}} )\rightarrow 
  (\pi_p^{\mathfrak{m}})^*\widehat{\otimes}_{\mathbb{T}_{\overline{\rho},\Sigma}}N^{\bullet}\otimes_{\mathbb{T}_{\overline{\rho}, \Sigma}}\widetilde{\pi}_{\Sigma_0}^{\mathfrak{m}}$$
  and 
   $$\psi_2 : C_{\mathrm{Iw}}^{\bullet}(\mathbb{Z}[1/\Sigma_n, \zeta_n], (\rho^{\mathfrak{m}})^*(1)
 \otimes_{\mathbb{T}_{\overline{\rho}, \Sigma}}\widetilde{\pi}_{\Sigma_0}^{\mathfrak{m}} )\rightarrow 
  N^{\bullet}\otimes_{\mathbb{T}_{\overline{\rho}, \Sigma}}\widetilde{\pi}_{\Sigma_0}^{\mathfrak{m}}.$$
  Therefore, $\phi_1$, $\psi_1$ and $\psi_2$ induce an isomorphism 
  \begin{multline*}
  H^1_{\mathrm{Iw}}(\mathbb{Z}[1/\Sigma_n, \zeta_n], \widetilde{H}^{BM}_{1, \overline{\rho},\Sigma}(1))\isom 
  H^1_{\mathrm{Iw}}(\mathbb{Z}[1/\Sigma_n, \zeta_n], (\pi_p^{\mathfrak{m}})^*\otimes_{\mathbb{T}_{\overline{\rho},\Sigma}}(\rho^{\mathfrak{m}})^*(1)
 \otimes_{\mathbb{T}_{\overline{\rho}, \Sigma}}\widetilde{\pi}_{\Sigma_0}^{\mathfrak{m}} )\\
 \isom H^1( (\pi_p^{\mathfrak{m}})^*\widehat{\otimes}_{\mathbb{T}_{\overline{\rho},\Sigma}}N^{\bullet}\otimes_{\mathbb{T}_{\overline{\rho}, \Sigma}}\widetilde{\pi}_{\Sigma_0}^{\mathfrak{m}})
 \isom (\pi_p^{\mathfrak{m}})^*\widehat{\otimes}_{\mathbb{T}_{\overline{\rho},\Sigma}}H^1( N^{\bullet}\otimes_{\mathbb{T}_{\overline{\rho}, \Sigma}}\widetilde{\pi}_{\Sigma_0}^{\mathfrak{m}})\\
 \isom  (\pi_p^{\mathfrak{m}})^*\widehat{\otimes}_{\mathbb{T}_{\overline{\rho},\Sigma}}
  H^1_{\mathrm{Iw}}(\mathbb{Z}[1/\Sigma_n, \zeta_n], (\rho^{\mathfrak{m}})^*(1)
 \otimes_{\mathbb{T}_{\overline{\rho}, \Sigma}}\widetilde{\pi}_{\Sigma_0}^{\mathfrak{m}} )
\end{multline*}
where the third isomorphism follows from the pro-freeness of $(\pi_p^{\mathfrak{m}})^*$.

   \end{proof}
   By the isomorphisms $\phi_1$ and $\widetilde{\phi}_1$, the homomorphism 
     $$z^{\mathrm{Iw}}_{\Sigma, n, \overline{\rho}} : \widetilde{H}^{BM}_{1, \overline{\rho},\Sigma}
    \rightarrow H^1_{\mathrm{Iw}}(\mathbb{Z}[1/\Sigma_n, \zeta_n], \widetilde{H}^{BM}_{1, \overline{\rho},\Sigma}(1))$$ can be regarded as 
    a continuous $\mathbb{T}_{\overline{\rho},\Sigma}[G_{\Sigma}]$-linear homomorphism 
   \begin{multline}z^{\mathrm{Iw}}_{\Sigma, n, \overline{\rho}} :(\pi^{\mathfrak{m}}_p)^*\otimes_{\mathbb{T}_{\overline{\rho},\Sigma}}
  (\rho^{\mathfrak{m}})^*\otimes_{\mathbb{T}_{\overline{\rho} ,\Sigma}}\widetilde{\pi}^{\mathfrak{m}}_{\Sigma_0}\\
  \rightarrow  (\pi_p^{\mathfrak{m}})^*\widehat{\otimes}_{\mathbb{T}_{\overline{\rho},\Sigma}}
  H^1_{\mathrm{Iw}}(\mathbb{Z}[1/\Sigma_n, \zeta_n], (\rho^{\mathfrak{m}})^*(1)
 \otimes_{\mathbb{T}_{\overline{\rho}, \Sigma}}\widetilde{\pi}_{\Sigma_0}^{\mathfrak{m}} ).
 \end{multline}
 We remark that this homomorphism is independent of the choice of $\phi_1$ by Lemma \ref{1.7}. 
 
 Since one has 
 $$(\pi^{\mathfrak{m}}_p)^*\otimes_{\mathbb{T}_{\overline{\rho},\Sigma}}
  (\rho^{\mathfrak{m}})^*\otimes_{\mathbb{T}_{\overline{\rho} ,\Sigma}}\widetilde{\pi}^{\mathfrak{m}}_{\Sigma_0}
  =(\widetilde{P}\widehat{\otimes}_{R_p} (\rho^{\mathfrak{m}})^*)\otimes_{\mathbb{T}_{\overline{\rho} ,\Sigma}}\widetilde{\pi}^{\mathfrak{m}}_{\Sigma_0}$$
  and 
  \begin{multline*}
   (\pi_p^{\mathfrak{m}})^*\widehat{\otimes}_{\mathbb{T}_{\overline{\rho},\Sigma}}
  H^1_{\mathrm{Iw}}(\mathbb{Z}[1/\Sigma_n, \zeta_n], (\rho^{\mathfrak{m}})^*(1)
 \otimes_{\mathbb{T}_{\overline{\rho}, \Sigma}}\widetilde{\pi}_{\Sigma_0}^{\mathfrak{m}} )\\
 =\varinjlim_{N_0}\widetilde{P}\widehat{\otimes}_{R_p}
  H^1_{\mathrm{Iw}}(\mathbb{Z}[1/\Sigma_n, \zeta_n], (\rho^{\mathfrak{m}})^*(1)
 \otimes_{\mathbb{T}_{\overline{\rho}, \Sigma}}(\widetilde{\pi}_{\Sigma_0}^{\mathfrak{m}})^{K^p(N_0)} ),
 \end{multline*}
 both of which are  inductive limits of the objects in $\mathfrak{C}(\mathcal{O})$, Corollary \ref{5.28} implies that 
 there exists a unique continuous $\mathbb{T}_{\overline{\rho},\Sigma}[G_{\Sigma_0}]$-linear homomorphism 
 \begin{equation}
 \phi_2 :  (\rho^{\mathfrak{m}})^*\otimes_{\mathbb{T}_{\overline{\rho} ,\Sigma}}\widetilde{\pi}^{\mathfrak{m}}_{\Sigma_0}
 \rightarrow H^1_{\mathrm{Iw}}(\mathbb{Z}[1/\Sigma_n, \zeta_n], (\rho^{\mathfrak{m}})^*(1)
 \otimes_{\mathbb{T}_{\overline{\rho}, \Sigma}}\widetilde{\pi}_{\Sigma_0}^{\mathfrak{m}} )
 \end{equation}
  satisfying the equality
  \begin{equation}
  \mathrm{id}_{  (\pi_p^{\mathfrak{m}})^*}\otimes \phi_2=z^{\mathrm{Iw}}_{\Sigma, n, \overline{\rho}}.
  \end{equation}
  Applying the functor $\Psi_{\Sigma_0}$ to the map $\phi_2$, we obtain a $\mathbb{T}_{\overline{\rho}, \Sigma}$-linear homomorphism 
  \begin{equation}
    \phi_3:=\Psi_{\Sigma_0}(\phi_2) :  (\rho^{\mathfrak{m}})^*\otimes_{\mathbb{T}_{\overline{\rho} ,\Sigma}}
  \Psi_{\Sigma_0}(\widetilde{\pi}^{\mathfrak{m}}_{\Sigma_0})\rightarrow 
  H^1_{\mathrm{Iw}}(\mathbb{Z}[1/\Sigma_n, \zeta_n], (\rho^{\mathfrak{m}})^*(1)
 \otimes_{\mathbb{T}_{\overline{\rho}, \Sigma}}\Psi_{\Sigma_0}(\widetilde{\pi}_{\Sigma_0}^{\mathfrak{m}}) )
 \end{equation}
 since one has canonical isomorphisms
 $$ \Psi_{\Sigma_0}((\rho^{\mathfrak{m}})^*\otimes_{\mathbb{T}_{\overline{\rho} ,\Sigma}}\widetilde{\pi}^{\mathfrak{m}}_{\Sigma_0})
 \isom (\rho^{\mathfrak{m}})^*\otimes_{\mathbb{T}_{\overline{\rho} ,\Sigma}}
  \Psi_{\Sigma_0}(\widetilde{\pi}^{\mathfrak{m}}_{\Sigma_0})$$
  and 
  \begin{multline*}\Psi_{\Sigma_0}(H^1_{\mathrm{Iw}}(\mathbb{Z}[1/\Sigma_n, \zeta_n], (\rho^{\mathfrak{m}})^*(1)
 \otimes_{\mathbb{T}_{\overline{\rho}, \Sigma}}\widetilde{\pi}_{\Sigma_0}^{\mathfrak{m}} )
 \isom H^1_{\mathrm{Iw}}(\mathbb{Z}[1/\Sigma_n, \zeta_n], \Psi_{\Sigma_0}((\rho^{\mathfrak{m}})^*(1)
 \otimes_{\mathbb{T}_{\overline{\rho}, \Sigma}}\widetilde{\pi}_{\Sigma_0}^{\mathfrak{m}}) )\\
\isom H^1_{\mathrm{Iw}}(\mathbb{Z}[1/\Sigma_n, \zeta_n], (\rho^{\mathfrak{m}})^*(1)
 \otimes_{\mathbb{T}_{\overline{\rho}, \Sigma}}\Psi_{\Sigma_0}(\widetilde{\pi}_{\Sigma_0}^{\mathfrak{m}}) ).
 \end{multline*}
 
We recall that one has a $\mathbb{T}_{\overline{\rho}, \Sigma}$-linear isomorphism 
$\Psi_{\Sigma_0}(\widetilde{\pi}_{\Sigma_0}^{\mathfrak{m}})\isom \mathbb{T}_{\overline{\rho}, \Sigma}$ which we fix, 
 then the map $\phi_3$ can be regarded as a $\mathbb{T}_{\overline{\rho}, \Sigma}$-linear homomorphism which we denote by 
 \begin{equation}
 \bold{z}_{\Sigma, n}(\rho^{\mathfrak{m}}) :  (\rho^{\mathfrak{m}})^*
 \rightarrow H^1_{\mathrm{Iw}}(\mathbb{Z}[1/\Sigma_n, \zeta_n], (\rho^{\mathfrak{m}})^*(1)).
 \end{equation}
  It is clear that this map is independent of the choice of the isomorphism $\Psi_{\Sigma_0}(\widetilde{\pi}_{\Sigma_0}^{\mathfrak{m}})\isom \mathbb{T}_{\overline{\rho}, \Sigma}$. 
  
  \begin{corollary}\label{3.11}
  For every integer $n\geqq 1$ such that $(n, \Sigma)=1$, the map $ \bold{z}_{\Sigma, n}(\rho^{\mathfrak{m}})$ satisfies the equality 
  $$ \bold{z}_{\Sigma, n}(\rho^{\mathfrak{m}})(\tau(v))=\sigma_{-1} (\bold{z}_{\Sigma, n}(\rho^{\mathfrak{m}})(v))$$
  for every $v\in  (\rho^{\mathfrak{m}})^*$. Moreover, for every prime $l\not\in \Sigma$, one has 
  $$\mathrm{Cor}\circ  \bold{z}_{\Sigma, nl}(\rho^{\mathfrak{m}})=\begin{cases} \bold{z}_{\Sigma, n}(\rho^{\mathfrak{m}})& \text{ if } l|n\\ 
  P_l(\sigma_l^{-1}) \bold{z}_{\Sigma, n}(\rho^{\mathfrak{m}})& \text{ if } (l,n)=1\end{cases}$$
  for the corestriction map 
  $$\mathrm{Cor} : H^1_{\mathrm{Iw}}(\mathbb{Z}[1/\Sigma_{nl}, \zeta_{nl}], (\rho^{\mathfrak{m}})^*(1))\rightarrow 
  H^1_{\mathrm{Iw}}(\mathbb{Z}[1/\Sigma_{nl}, \zeta_n], (\rho^{\mathfrak{m}})^*(1)).$$
  
  \end{corollary}
  \begin{proof}
  By definition of the map $\bold{z}_{\Sigma, n}(\rho^{\mathfrak{m}})$, the corollary immediately follows from Theorem 
  \ref{2.15}.
  \end{proof}
  

  \subsection{Statement of the main theorem}
  Let $x_f\in \mathrm{Spec}(\mathbb{T}_{\overline{\rho},\Sigma})(E)$ be an $E$-valued point which is modular, i.e. 
  the base change $\rho_{f}:=\rho^{\mathfrak{m}}\otimes_{\mathbb{T}_{\overline{\rho}, \Sigma}, x_f^*}E$ by the map 
  $x_f^* : \mathbb{T}_{\overline{\rho},\Sigma}\rightarrow E$ corresponding to $x_f$ is the representation associated to a normalized Hecke eigen cusp new form 
  $f(\tau)=\sum_{n=1}^{\infty}a_nq^n\in S_k(\Gamma_1(N_f))^{\mathrm{new}}$ of weight $k\in \mathbb{Z}_{\geqq 2}$ and level $N_f$ with neben character 
  $\varepsilon_f : (\mathbb{Z}/N_f\mathbb{Z})^{\times}\rightarrow \mathbb{C}^{\times}$ (see Appendix A). Then, as the base change of the map $\bold{z}_{\Sigma, n}(\rho^{\mathfrak{m}})$ with respect to the map $x_f^*$, we obtain an $E$-linear map 
 \begin{equation}
 \bold{z}_{\Sigma, n}(f) : 
 \rho_{f}^*\rightarrow H^1_{\mathrm{Iw}}(\mathbb{Z}[1/\Sigma_n, \zeta_n], \rho_f^*(1))
 \end{equation}
 for every $n\geqq 1$ such that $(n, \Sigma)=1$.

   In Appendix A, we define $V'_1(f)_A, V_1(f^*)_A$ and $S_1(f^*)_A$ for $A=F, E, \mathbb{C}$ with a motivic structure, and 
   an $E$-linear map 
   $$\bold{z}_{n}(f) : V'_1(f)_E\rightarrow H^1_{\mathrm{Iw}}(\mathbb{Z}[1/\Sigma_{f,n}, \zeta_n], V'_1(f)_E(1)) $$
    for every $n\geqq 1$ such that $(n, \Sigma_f)=1$ for $\Sigma_f:=\mathrm{prime}(N_f)\cup\{p\}$ (we remark that one has 
 $\Sigma_f \subseteq \Sigma$). We recall that one has an $E[G_{\mathbb{Q}}]$-linear isomorphism 
 \begin{equation}\label{2a}
 V'_1(f)_E\isom \rho_f^*.
 \end{equation}
 For every prime $l\not=p$, we set 
 $$P_{f,l}(u):=\mathrm{det}(1-\mathrm{Frob}_l \cdot u\mid (\rho_f)^{I_l})=\begin{cases}1-a_{l}(f)u+l^{k-1}\varepsilon_f(l)u^2& \text{ if } l\not\in \Sigma_f \\
 1-a_l(f)u & \text{ if } l\in \Sigma_f\end{cases}$$ 
 
 The main result of this article is the following.
 \begin{thm}\label{3.2}
For each integer $n\geqq 1$ such that $(n, \Sigma)=1$, 
 one has an equality 
 $$ \bold{z}_{\Sigma, n}(f)=\prod_{l\in \Sigma_0}
 P_{f, l}(\sigma_l^{-1})\cdot  \bold{z}_{n}(f)$$
 under the isomorphism $(\ref{2a})$
 as a map from $V'_1(f)_E$ to $H^1_{\mathrm{Iw}}(\mathbb{Z}[1/\Sigma_n, \zeta_n], V'_1(f)_E(1))$ 
 
 \end{thm}
 \begin{rem}
 Since one has $\mathrm{End}_{E[G_{\mathbb{Q}}]}(\rho_f)=E$, the validity of the equality in this theorem 
 does not depend on the choice of the isomorphism $(\ref{2a})$. 
 \end{rem}
  To prove this theorem, we need Pa\v{s}k\={u}nas' another result, which we recall in the next subsection.
  
  

  \subsection{Pa\v{s}k\={u}nas' result on potentially semi-stable deformation rings}
  Let $k\geqq 2$ be an integer, and fix an inertia type $\tau : I_p\rightarrow\mathrm{GL}_2(E)$, i.e. a continuous homomorphism with an open kernel. For a $2$-dimensional potentially semi-stable $E$-
  representation $V$ of $G_{\mathbb{Q}_p}$, we let $WD(V)$ be the associated $2$-dimensional $E$-representation of the 
  Weil-Deligne group of $\mathbb{Q}_p$. 
  We say that $V$ is of Galois type $\tau$ if $\tau\isom WD(V)|_{I_p}$. 
  \begin{thm}\label{3.3}
  There is a unique $($possibly trivial$)$ quotient $R^{\mathrm{st}}_{p}(k,\tau)$ $($resp. $R^{\mathrm{cr}}_{p}(k,\tau)$$)$ of $R_{p}$ with the following properties $:$ 
  \begin{itemize}
  \item[(1)]$R^{\mathrm{st}}_{p}(k,\tau)$ $($resp. $R^{\mathrm{cr}}_{p}(k,\tau)$$)$ is $p$-torsion free,
  $R^{\mathrm{st}}_{p}(k,\tau)[1/p]$ $($resp. $R^{\mathrm{cr}}_{p}(k,\tau)[1/p]$$)$ is 
  reduced.
  \item[(2)]if $E'$ is a finite extension of $E$, then an $\mathcal{O}$-algebra homomorphism $x : R_{p}\rightarrow E'$ 
  factors through $R^{\mathrm{st}}_{p}(k,\tau)$ $($resp. $R^{\mathrm{cr}}_{p}(k,\tau)$$)$ if and only if the corresponding 
  $E'$-representation $V_x$ is potentially semi-stable $($resp. potentially crystalline$)$ of Galois type $\tau$ with Hodge-Tate weights $\{-(k-1), 0\}$. 
  \end{itemize}
    
  
  \end{thm}
  \begin{proof}
  This is proved by Kisin in \cite{Ki08}. Precisely, in \cite{Ki08}, he treated the deformations 
  with a fixed determinant. However, we can easily obtain the same result for general deformation 
  without any determinant condition. See, for example, \S 6.1 of \cite{CEG${}^{+}$18}.
  \end{proof}
  \begin{rem}\label{3.4}
  By the characterization, $R^{\mathrm{cr}}_{p}(k,\tau)$ is a quotient of $R^{\mathrm{st}}_{p}(k,\tau)$. It is isomorphism except when 
  $\tau=\chi\oplus \chi$ for some character $\chi$. It is known that both $R^{\mathrm{cr}}_{p}(k,\tau)[1/p]$ and the complement 
  $\mathrm{Spec}(R^{\mathrm{st}}_{p}(k,\tau)[1/p])\setminus \mathrm{Spec}(R^{\mathrm{cr}}_{p}(k,\tau)[1/p])$ are regular by 
  Theorem 3.3.8 of \cite{Ki08} for $R^{\mathrm{cr}}_{p}(k,\tau)[1/p]$ and (the 
  proof of) Theorem A.2 of \cite{Ki09b} for $\mathrm{Spec}(R^{\mathrm{st}}_{p}(k,\tau)[1/p])\setminus \mathrm{Spec}(R^{\mathrm{cr}}_{p}(k,\tau)[1/p])$. 
  
  \end{rem}
  
  For any inertial type $\tau : I_p\rightarrow\mathrm{GL}_2(E)$, Henniart \cite{He02} shows that there exists a unique finite dimensional absolutely irreducible smooth $E$-representation
  $\sigma^{\mathrm{st}}(\tau)$ (resp. $\sigma^{\mathrm{cr}}(\tau)$) of $K_0:=\mathrm{GL}_2(\mathbb{Z}_p)$ satisfying the following property : 
   if $W$ is any Frobenius semi-simple, 
  $2$-dimensional $E$-representation of the Weil-Deligne group of $\mathbb{Q}_p$, and 
  $\pi(W)$ is the smooth $E$-representation of $G_p$ associated to $W$ by Tate's normalized
  local Langlands correspondence, then $\pi(W)|_{K_0}$
  contains $\sigma^{\mathrm{st}}(\tau)$ (resp. $\sigma^{\mathrm{cr}}(\tau)$) if and only if 
  $W|_{I_p}\isom \tau$ (resp. $W|_{I_p}\isom \tau$ and the monodromy on $W$ is zero). Moreover, it is proved that, for all $W$, one has 
  \begin{equation}
  \mathrm{dim}_{E}\mathrm{Hom}_{E[K_0]}(
  \sigma^{\diamond}(\tau), \pi(W)|_{K_0})\leqq 1
  \end{equation} for each $\diamond\in \{\mathrm{st}, \mathrm{cr}\}$.
  
  In \cite{Pas15}, Pa\v{s}k\={u}nas defined the quotients $R^{\mathrm{cr}}_{p}(k,\tau)$ and $R^{\mathrm{st}}_{p}(k,\tau)$ using the projective envelope 
  $\widetilde{P}\in \mathfrak{C}(\mathcal{O})$. For each $\diamond\in \{\mathrm{st}, \mathrm{cr}\}$, we fix a $K_0$-stable 
  $\mathcal{O}$-lattice $\sigma^{\diamond}_0(\tau)$ of $\sigma^{\diamond}(\tau)$. Then, we set 
  $$\sigma^{\diamond}_0(k, \tau):=\mathrm{Sym}^{k-2}(\mathcal{O}^2)^*\otimes_{\mathcal{O}}\sigma_0^{\diamond}(\tau)$$ which is naturally 
  a compact $\mathcal{O}[[K_0]]$-module with respect to the $\varpi$-adic topology on $\sigma^{\diamond}_0(k, \tau)$, 
  and $\sigma^{\diamond}(k, \tau):=\mathrm{Sym}^{k-2}(E^2)^*\otimes_{E}\sigma^{\diamond}(\tau)$. We set 
  $$M_0^{\diamond}(k,\tau):=\sigma_0^{\diamond}(k,\tau)\otimes_{\mathcal{O}[[K_0]]}\widetilde{P}$$
  on which the ring $R_p$ acts by $a(u\otimes v):=u\otimes av$ ($a\in R_p, u\in \sigma_0^{\diamond}(k,\tau), v\in \widetilde{P})$, where 
  we regard $\sigma_0^{\diamond}(k,\tau)$ as a right $\mathcal{O}[[K_0]]$-module by 
  $v\cdot g:=g^{-1}\cdot v$ ($g\in K_0, v\in \sigma_0^{\diamond}(k,\tau)$). 
  Then, $M_0^{\diamond}(k,\tau)$ is a compact $R_p$-module. We also 
  set $M^{\diamond}(k,\tau):=\sigma^{\diamond}(k,\tau)\otimes_{\mathcal{O}[[K_0]]}\widetilde{P}$.
  The crucial result for the proof of our main theorem is the following theorem. 
  
  \begin{thm}\label{3.5}
  Assume that $p\geqq 5$, $\mathrm{End}_{\mathbb{F}[G_{\mathbb{Q}_p}]}(\overline{\rho}_p)=\mathbb{F},$ and 
  $$\overline{\rho}_p\not\isom \begin{pmatrix} \overline{\varepsilon} & *\\ 0 & 1\end{pmatrix}\otimes \eta$$ for any charatcter $\eta$. 
  Then, as a quotient of $R_p$, the quotient $$R_p/\mathrm{Ann}_{R_p}(M_0^{\diamond}(k,\tau))$$ is equal to 
  the quotient $R_p^{\diamond}(k,\tau)$ for each $\diamond\in \{\mathrm{st}, \mathrm{cr}\}$. 
  
  Moreover, $M^{\mathrm{cr}}(k,\tau)$ is an 
  $R_p^{\mathrm{cr}}(k,t)[1/p]$-module which is locally free of rank one, and 
  $M^{\mathrm{st}}(k,\tau)$ is an $R^{\mathrm{st}}_p(k,t)[1/p]$-module which locally free of rank one over 
  $\mathrm{Spec}(R^{\mathrm{st}}_p(k,t)[1/p])\setminus \mathrm{Spec}(R^{\mathrm{cr}}_p(k,t)[1/p])$.
  \end{thm}
  \begin{proof}
  We first remark that one has a natural $R_p$-linear topological isomorphism 
  $$M_0^{\diamond}(k,\tau)\isom \mathrm{Hom}_{\mathcal{O}[[K_0]]}^{\mathrm{cont}}(\widetilde{P}, \sigma_0^{\diamond}(k,\tau)^*)^*$$
  by Remark 5.1.7 of \cite{GN16}. Then, the equality 
  $$R_p/\mathrm{Ann}_{R_p}(M_0^{\diamond}(k,\tau))=R_p^{\diamond}(k,\tau)$$
  for each $\diamond  \in \{\mathrm{st}, \mathrm{cr}\}$ follows from Corollary 6.5 of \cite{Pas15}. 
  Precisely, this corollary is also proved for the deformation rings with a fixed determinant condition. The general case 
  follows from this special case by Proposition 6.12 of \cite{CEG${}^{+}$18} 
  (precisely, this proposition is also proved when $\tau$ is trivial, but the proof is applied to any $\tau$). 
  
  The latter statement in the theorem is proved in Proposition 6.14 \cite{CEG${}^{+}$18} when $\tau$ is trivial. This case follows from 
  Proposition 4.14 and Corollary 6.5 of \cite{Pas15} and the fact that the ring $R_p^{\mathrm{cr}}(k,\tau)[1/p]$ is regular. Therefore, 
  the general case also follows from Remark \ref{3.4}. 
  
  \end{proof}
    
  \subsection{Proof of the main theorem}
  
  Let $\sigma$ be a smooth $E$-representation of $K_0$ on a finite dimensional $E$-vector space, and 
  $\mathcal{V}_{\sigma}$ be the local system on $Y(N)(\mathbb{C})$ (and $Y(N)_{\overline{\mathbb{Q}}}$) corresponding to $\sigma$ for each $N\geqq 3$. 
  We set $\mathcal{V}^*_{k/E}(\sigma):=\mathcal{V}^*_{k/E}\otimes_E\mathcal{V}_{\sigma}$, and 
  \begin{equation}
  H^1(K^p(N_0), \mathcal{V}^*_{k/E}):=\varinjlim_{m\geqq 0}H^1(Y(N_0p^m), \mathcal{V}^*_{k/E})
  \end{equation}
   for every $N_0\geqq 1$ such that $\mathrm{prime}(N_0)=\Sigma_0$. 
  
  We take a sufficiently large $m\geqq 1$ such that 
  $K_m$ acts trivially on $\sigma$. Then, we remark that one has a canonical isomorphism 
  \begin{multline}
  H^1(Y(N_0), \mathcal{V}^*_{k/E}(\sigma))\isom \sigma\otimes_{E[K_0/K_m]}H^1(Y(N_0p^m), \mathcal{V}^*_{k/E})\\
  \isom \mathrm{Hom}_{E[K_0]}(\sigma^*, H^1(K^p(N_0), \mathcal{V}^*_{k/E}))
  \end{multline}
 for every such $N_0\geqq 3$, by which we identify the both sides. If such $N_0$ is equal to $1$ or $2$, then we also 
 define 
 \begin{multline}
 H^1(Y(N_0), \mathcal{V}^*_{k/E}(\sigma)):=  \sigma\otimes_{E[K_0/K_m]}H^1(Y(N_0p^m), \mathcal{V}^*_{k/E})\\
 \isom 
 \mathrm{Hom}_{E[K_0]}(\sigma^*, H^1(K^p(N_0), \mathcal{V}^*_{k/E}))
 \end{multline}
 for some $m\geqq1$ such that $N^0p^m\geqq 3$. 
 The module $H^1(Y(N_0), \mathcal{V}^*_{k/E}(\sigma))$ is naturally a 
 $\mathbb{T}'(K_{\Sigma_0}(N_0))$-module, and we set 
 $$H^1(Y(N_0), \mathcal{V}^*_{k/E}(\sigma))_{\overline{\rho}}:=H^1(Y(N_0), \mathcal{V}^*_{k/E}(\sigma))\otimes_{\mathbb{T}'(K_{\Sigma_0}(N_0))}
 \mathbb{T}'(K_{\Sigma_0}(N_0))_{\overline{\rho}}.$$
  for each allowable $N_0$ with respect to  $\overline{\rho}$. Then, we define an $E$-linear map 
  \begin{equation}
  z^{\mathrm{Iw}}_{N_0, n, \overline{\rho}}(k, \sigma, -) : 
  H^1(Y(N_0), \mathcal{V}^*_{k/E}(\sigma))_{\overline{\rho}}(1)
  \rightarrow H^1_{\mathrm{Iw}}(\mathbb{Z}[1/\Sigma_n, \zeta_n], 
  H^1(Y(N_0), \mathcal{V}^*_{k/E}(\sigma))_{\overline{\rho}}(2))
  \end{equation}
  for each $n\geqq 1$ such that $(n, \Sigma)=1$ as the following composite : 
  \begin{multline}
  \sigma\otimes_{E[K_0/K_m]}H^1(Y(N_0p^m), \mathcal{V}^*_{k/E})_{\overline{\rho}}(1)\\
  \xrightarrow{\mathrm{id}_{\sigma}\otimes z^{\mathrm{Iw}}_{N_0p^m, n, \overline{\rho}}(k, -)}
  \sigma\otimes_{E[K_0/K_m]}H^1_{\mathrm{Iw}}(\mathbb{Z}[1/\Sigma_n, \zeta_n], 
  H^1(Y(N_0p^m), \mathcal{V}^*_{k/E})_{\overline{\rho}}(2))\\
  \xrightarrow{\mathrm{can}}
  H^1_{\mathrm{Iw}}(\mathbb{Z}[1/\Sigma_n, \zeta_n], 
   \sigma\otimes_{E[K_0/K_m]}H^1(Y(N_0p^m), \mathcal{V}^*_{k/E})_{\overline{\rho}}(2)).
  \end{multline}
  We remark that it is independent of the choice of $m$ by Lemma \ref{2.11} (1).
  
  We next set $H^1_{\Sigma}(\mathcal{V}^*_{k/E}):=\varinjlim_{N_0}H^1(K^p(N_0), \mathcal{V}^*_{k/E})$, and 
  $$H^1_{\Sigma_0}(\mathcal{V}^*_{k/E}(\sigma)):=\varinjlim_{N_0}H^1(Y(N_0), \mathcal{V}^*_{k/E}(\sigma))=\mathrm{Hom}_{E[K_0]}(\sigma\*, H^1_{\Sigma}(\mathcal{V}^*_{k/E}))
,$$
  where $N_0$ run through all the allowable ones with respect to  $\overline{\rho}$ such that $\mathrm{prime}(N_0)=\Sigma_0$. We also set 
  $$H^1_{\Sigma_0}(\mathcal{V}^*_{k/E}(\sigma))_{\overline{\rho}}:=\varinjlim_{N_0}H^1(Y(N_0), \mathcal{V}^*_{k/E}(\sigma))_{\overline{\rho}}
  =\mathrm{Hom}_{E[K_0]}(\sigma\*, H^1_{\Sigma}(\mathcal{V}^*_{k/E})_{\overline{\rho}})
,$$
  which is a $\mathbb{T}_{\overline{\rho}, \Sigma}[1/p]$-module with a smooth 
  $\mathbb{T}_{\overline{\rho}, \Sigma}$-linear $G_{\Sigma_0}$-action. By Lemma \ref{2.11} (2), we can take the inductive limit of the map 
  $z^{\mathrm{Iw}}_{N_0, n, \overline{\rho}}(k, \sigma, -)$ to obtain a map 
  \begin{equation}
  z^{\mathrm{Iw}}_{\Sigma_0, n, \overline{\rho}}(k, \sigma, -) : 
  H^1_{\Sigma_0}(\mathcal{V}^*_{k/E}(\sigma))_{\overline{\rho}}(1)
 \rightarrow H^1_{\mathrm{Iw}}(\mathbb{Z}[1/\Sigma_n, \zeta_n], H^1_{\Sigma_0}(\mathcal{V}^*_{k/E}(\sigma))_{\overline{\rho}}(2)).
 \end{equation}
  We can show that this map is $\mathbb{T}_{\overline{\rho}, \Sigma}[G_{\Sigma_0}]$-linear by the similar proof as that of Theorem \ref{2.15}. 
  In particular, we can apply the functor $\Psi_{\Sigma_0}$ to this map to obtain the following $\mathbb{T}_{\overline{\rho},\Sigma}$-linear map 
  \begin{multline}
  \Psi_{\Sigma_0}(z^{\mathrm{Iw}}_{\Sigma_0, n, \overline{\rho}}(k, \sigma, -)) : \Psi_{\Sigma_0}(H^1_{\Sigma_0}(\mathcal{V}^*_{k/E}(\sigma))_{\overline{\rho}})(1)\\
  \rightarrow \Psi_{\Sigma_0}(H^1_{\mathrm{Iw}}(\mathbb{Z}[1/\Sigma_n, \zeta_n], H^1_{\Sigma_0}(\mathcal{V}^*_{k/E}(\sigma))_{\overline{\rho}}(2)))
  \isom H^1_{\mathrm{Iw}}(\mathbb{Z}[1/\Sigma_n, \zeta_n],  \Psi_{\Sigma_0}(H^1_{\Sigma_0}(\mathcal{V}^*_{k/E}(\sigma))_{\overline{\rho}})(2)). 
  \end{multline}
  
  Now, let $\tau : I_p\rightarrow \mathrm{GL}_2(E)$ be an inertia type. We 
  apply the definitions above to the $E$-representation $\sigma^{\diamond}(\tau)$ of $K_0$ for $\diamond\in \{\mathrm{st},\mathrm{crys}\}$. 
  As a corollary of Pa\v{s}k\={u}nas' result which we recall in the preceding subsection, we obtain the following corollary. 
  For each such $\tau$, $k\in \mathbb{Z}_{\geqq 2}$ and $\diamond\in \{\mathrm{st},\mathrm{crys}\}$, we set 
  $$\mathbb{T}^{\diamond}_{\overline{\rho}, \Sigma}(k, \tau):=\mathbb{T}_{\overline{\rho}, \Sigma}\otimes_{R_p}R^{\diamond}_p(k, \tau)[1/p].$$
  \begin{corollary}\label{3.6}
  For each allowable $N_0$ with respect to $\overline{\rho}$ such that $\mathrm{prime}(N_0)=\Sigma_0$, the action of $\mathbb{T}_{\overline{\rho}, \Sigma}[1/p]$ on 
  $H^1(Y(N_0), \mathcal{V}^*_{k/E}(\sigma^{\diamond}(\tau)))_{\overline{\rho}}$ factors through the quotient 
  $\mathbb{T}^{\diamond}_{\overline{\rho}, \Sigma}(k, \tau)$, 
  and there is a canonical 
  $\mathbb{T}^{\diamond}_{\overline{\rho}, \Sigma}(k, \tau)[G_{\mathbb{Q}}]$-linear isomorphism 
  \begin{equation}
  \sigma^{\diamond}(k,\tau)\otimes_{\mathcal{O}[[K_0]]}\widetilde{H}^{BM}_{1}(K^p(N_0))_{\overline{\rho}}\isom H^1(Y(N_0), \mathcal{V}^*_{k/E}(\sigma^{\diamond}(\tau)))_{\overline{\rho}}(1).
  \end{equation}
  Moreover, the inductive limit of this isomorphism with respect to all the allowable $N_0$ induces a 
  $\mathbb{T}^{\diamond}_{\overline{\rho}, \Sigma}(k, \tau)[G_{\mathbb{Q}}\times G_{\Sigma_0}]$-linear isomorphism 
  \begin{equation}
  \sigma^{\diamond}(k,\tau)\otimes_{\mathcal{O}[[K_0]]}\widetilde{H}^{BM}_{1, \overline{\rho}, \Sigma}
  \isom H^1_{\Sigma_0}(\mathcal{V}^*_{k/E}(\sigma^{\diamond}(\tau)))_{\overline{\rho}}(1).
  \end{equation}
  
  \end{corollary}
  \begin{proof}
  The existence of a canonical isomorphism follows from Lemma \ref{1.11}, whose $G_{\mathbb{Q}}$ and 
  $G_{\Sigma_0}$ equivariance also immediately follows from the definition of the isomorphism in Lemma \ref{1.11}. 
  Hence,  it remains to show that 
  $\sigma^{\diamond}(k,\tau)\otimes_{\mathcal{O}[[K_0]]}\widetilde{H}^{BM}_{1}(K^p(N_0))_{\overline{\rho}}$ is 
  $\mathbb{T}^{\diamond}_{\overline{\rho}, \Sigma}(k, \tau)$-module. 
  As a base change of a $\mathbb{T}_{\overline{\rho}, \Sigma}[G_{\mathbb{Q}}\times G_p]$-linear isomorphism 
  $$\widetilde{H}^{BM}_{1}(K^p(N_0))_{\overline{\rho}}\isom \widetilde{P}\widehat{\otimes}_{R_p}(\rho^{\mathfrak{m}})^*
  \otimes_{\mathbb{T}_{\overline{\rho}, \Sigma}}(\widetilde{\pi}^{\mathfrak{m}}_{\Sigma_0})^{K^p(N_0)}$$ which is obtained by 
  Corollary \ref{1.6}, we obtain a $\mathbb{T}_{\overline{\rho}, \Sigma}[1/p][G_{\mathbb{Q}}]$-linear isomorphism 
  \begin{multline*}
  \sigma^{\diamond}(k,\tau)\otimes_{\mathcal{O}[[K_0]]}\widetilde{H}^{BM}_{1}(K^p(N_0))_{\overline{\rho}}\isom 
  \sigma^{\diamond}(k,\tau)\otimes_{\mathcal{O}[[K_0]]}\widetilde{P}\widehat{\otimes}_{R_p}(\rho^{\mathfrak{m}})^*
  \otimes_{\mathbb{T}_{\overline{\rho}, \Sigma}}(\widetilde{\pi}^{\mathfrak{m}}_{\Sigma_0})^{K^p(N_0)}\\
  \isom M^{\diamond}(k,\tau)\otimes_{R_p}(\rho^{\mathfrak{m}})^*
  \otimes_{\mathbb{T}_{\overline{\rho}, \Sigma}}(\widetilde{\pi}^{\mathfrak{m}}_{\Sigma_0})^{K^p(N_0)}.
  \end{multline*}
  The last term is $\mathbb{T}^{\diamond}_{\overline{\rho}, \Sigma}(k,\tau)$-module by Theorem \ref{3.5}, 
  which proves the corollary.

  \end{proof}
  
  \begin{prop}\label{3.7}
  There exists a $\mathbb{T}^{\diamond}_{\overline{\rho}, \Sigma}(k,\tau)[G_{\mathbb{Q}}]$-linear isomorphism 
  $$M^{\diamond}(k,\tau)\otimes_{R_p}(\rho^{\mathfrak{m}})^*\isom 
   \Psi_{\Sigma_0}(H^1_{\Sigma_0}(\mathcal{V}^*_{k/E}(\sigma^{\diamond}(\tau)))_{\overline{\rho}})(1)$$ such that, under this isomorphism, the map 
  $\Psi_{\Sigma_0}(z^{\mathrm{Iw}}_{\Sigma_0, n, \overline{\rho}}(k, \sigma^{\diamond}(\tau), -))$ 
  corresponds to the following composite $:$ 
  \begin{multline}
  \mathrm{id}_{M^{\diamond}(k,\tau)}\otimes \bold{z}_{\Sigma, n}(\rho^{\mathfrak{m}}) : 
  M^{\diamond}(k,\tau)\otimes_{R_p}(\rho^{\mathfrak{m}})^* \\
   \xrightarrow{\mathrm{id}_{M^{\diamond}(k,\tau)}\otimes \bold{z}_{\Sigma, n}(\rho^{\mathfrak{m}})} M^{\diamond}(k,\tau)\otimes_{R_p}
   H^1_{\mathrm{Iw}}(\mathbb{Z}[1/\Sigma_n, \zeta_n], (\rho^{\mathfrak{m}})^*(1))\\
   \xrightarrow{\mathrm{can}}
   H^1_{\mathrm{Iw}}(\mathbb{Z}[1/\Sigma_n, \zeta_n], M^{\diamond}(k,\tau)\otimes_{R_p}(\rho^{\mathfrak{m}})^*(1)),
   \end{multline}
   which we also denote by  the same letter $\mathrm{id}_{M^{\diamond}(k,\tau)}\otimes \bold{z}_{\Sigma, n}(\rho^{\mathfrak{m}})$. 
  \end{prop}
  \begin{proof}
  We first define an isomorphism $$M^{\diamond}(k,\tau)\otimes_{R_p}(\rho^{\mathfrak{m}})^*\isom 
   \Psi_{\Sigma_0}(H^1_{\Sigma_0}(\mathcal{V}^*_{k/E}(\sigma^{\diamond}(\tau)))_{\overline{\rho}})(1)$$ as the composite of the following isomorphisms : 
   \begin{multline*}
   M^{\diamond}(k,\tau)\otimes_{R_p}(\rho^{\mathfrak{m}})^*\xrightarrow{(a)}  M^{\diamond}(k,\tau)\otimes_{R_p}(\rho^{\mathfrak{m}})^*\otimes_{\mathbb{T}_{\overline{\rho}, \Sigma}} \Psi_{\Sigma_0}(\widetilde{\pi}^{\mathfrak{m}}_{\Sigma_0})
   \\ \xrightarrow{(b)}\Psi_{\Sigma_0}(M^{\diamond}(k,\tau)\otimes_{R_p}(\rho^{\mathfrak{m}})^*\otimes_{\mathbb{T}_{\overline{\rho}, \Sigma}}
   \widetilde{\pi}^{\mathfrak{m}}_{\Sigma_0})
   =\Psi_{\Sigma_0}((\sigma^{\diamond}(k, \tau)
   \otimes_{\mathcal{O}[[K_0]]}\widetilde{P})\otimes_{R_p}(\rho^{\mathfrak{m}})^*\otimes_{\mathbb{T}_{\overline{\rho}, \Sigma}}
   \widetilde{\pi}^{\mathfrak{m}}_{\Sigma_0})\\
   =\Psi_{\Sigma_0}(\sigma^{\diamond}(k,\tau)
   \otimes_{\mathcal{O}[[K_0]]}(\widetilde{P}\widehat{\otimes}_{R_p}(\rho^{\mathfrak{m}})^*)\otimes_{\mathbb{T}_{\overline{\rho}, \Sigma}}
   \widetilde{\pi}^{\mathfrak{m}}_{\Sigma_0})
   \xrightarrow{(c)}
   \Psi_{\Sigma_0}(\sigma^{\diamond}(k,\tau)\otimes_{\mathcal{O}[[K_0]]}\widetilde{H}^{BM}_{1,\overline{\rho}, \Sigma})\\
   \xrightarrow{(d)}\Psi_{\Sigma_0}(H^1_{\Sigma_0}(\mathcal{V}^*_{k/E}(\sigma^{\diamond}(\tau)))_{\overline{\rho}}(1)),
   \end{multline*}
   where the isomorphism $(a)$ is induced by a $\mathbb{T}_{\overline{\rho}, \Sigma}$-linear isomorphism 
   $\mathbb{T}_{\overline{\rho}, \Sigma}\isom \Psi_{\Sigma_0}(\widetilde{\pi}^{\mathfrak{m}}_{\Sigma_0})$, $(b)$ is induced by 
   the canonical base change isomorphism for the functor $\Psi_{\Sigma_0}$, $(c)$ is induced by 
   a $\mathbb{T}_{\overline{\rho}, \Sigma}[G_{\mathbb{Q}}\times G_{\Sigma}]$-linear topological isomorphism 
   $$\widetilde{H}^{BM}_{1,\overline{\rho}, \Sigma}\isom (\widetilde{P}\widehat{\otimes}_{R_p}(\rho^{\mathfrak{m}})^*)\otimes_{\mathbb{T}_{\overline{\rho}, \Sigma}}
   \widetilde{\pi}^{\mathfrak{m}}_{\Sigma_0},$$ and $(d)$ is induced by Corollary \ref{3.6}. 
   
   
  By definition of the map $\bold{z}_{\Sigma, n}(\rho^{\mathfrak{m}})$, 
  under the composite of $(a), (b)$ and $(c)$, the map 
  $\mathrm{id}_{M^{\diamond}(k,\tau)}\otimes \bold{z}_{\Sigma, n}(\rho^{\mathfrak{m}})$ corresponds to the following composite : 
  \begin{multline*}
 \Psi_{\Sigma_0}(\sigma^{\diamond}(k,\tau)\otimes_{\mathcal{O}[[K_0]]}\widetilde{H}^{BM}_{1,\overline{\rho}, \Sigma})\\
   \xrightarrow{\Psi_{\Sigma_0}(\mathrm{id}_{\sigma^{\diamond}(k,\tau)}\otimes z^{\mathrm{Iw}}_{\Sigma, n, \overline{\rho}})}
  \Psi_{\Sigma_0}(\sigma^{\diamond}(k,\tau)\otimes_{\mathcal{O}[[K_0]]}
  H^1_{\mathrm{Iw}}(\mathbb{Z}[1/\Sigma_n, \zeta_n], \widetilde{H}^{BM}_{1,\overline{\rho}, \Sigma})(1))\\
  \xrightarrow{\mathrm{can}} H^1_{\mathrm{Iw}}(\mathbb{Z}[1/\Sigma_n, \zeta_n], \Psi_{\Sigma_0}
  (\sigma^{\diamond}(k,\tau)\otimes_{\mathcal{O}[[K_0]]}
  \widetilde{H}^{BM}_{1,\overline{\rho}, \Sigma})(1)). 
  \end{multline*}
  
  Finally, by Proposition \ref{2.16}, the map $z^{\mathrm{Iw}}_{\Sigma_0, n, \overline{\rho}}(k, \sigma^{\diamond}(\tau), -)$ corresponds to the following composite : 
  \begin{multline*}
   \sigma^{\diamond}(k,\tau)\otimes_{\mathcal{O}[[K_0]]}\widetilde{H}^{BM}_{1,\overline{\rho}, \Sigma}\xrightarrow{\mathrm{id}_{\sigma^{\diamond}(k,\tau)}\otimes z^{\mathrm{Iw}}_{\Sigma, n, \overline{\rho}}} 
    \sigma^{\diamond}(k,\tau)\otimes_{\mathcal{O}[[K_0]]} H^1_{\mathrm{Iw}}
    (\mathbb{Z}[1/\Sigma_n, \zeta_n], \widetilde{H}^{BM}_{1,\overline{\rho}, \Sigma})(1))\\
    \xrightarrow{\mathrm{can}}
    H^1_{\mathrm{Iw}}
    (\mathbb{Z}[1/\Sigma_n, \zeta_n], \sigma^{\diamond}(k,\tau)\otimes_{\mathcal{O}[[K_0]]} \widetilde{H}^{BM}_{1,\overline{\rho}, \Sigma})(1))
    \end{multline*}
    under the canonical isomorphism 
    $\sigma^{\diamond}(k,\tau)\otimes_{\mathcal{O}[[K_0]]}
  \widetilde{H}^{BM}_{1,\overline{\rho}, \Sigma}\isom H^1_{\Sigma_0}(\mathcal{V}^*_{k/E}(\sigma^{\diamond}(\tau)))_{\overline{\rho}}(1)$, which 
  proves the proposition.

  \end{proof}
  We finally prove our main theorem. 
  \begin{proof} (of Theorem \ref{3.2})
  Let $x_f\in \mathrm{Spec}(\mathbb{T}_{\overline{\rho},\Sigma})(E)$ be a modular point as in \S 4.2.1. 
  We write $\tau : I_p\rightarrow \mathrm{GL}_2(E)$ to denote the Galois type of $\rho_f|_{G_{\mathbb{Q}_p}}$. 
  We take $\diamond=\mathrm{cr}$ if $\rho_f|_{G_{\mathbb{Q}_p}}$ is potentially crystalline, and 
  $\diamond=\mathrm{st}$ if not. Then, the map 
  $x_f^* : \mathbb{T}_{\overline{\rho},\Sigma}[1/p]\rightarrow E$ corresponding to $x_f$ factors through 
  the quotient $\mathbb{T}^{\diamond}_{\overline{\rho},\Sigma}(k,\tau)$ by Theorem \ref{3.3}. We write 
  $\mathfrak{m}_f\subset \mathbb{T}_{\overline{\rho},\Sigma}[1/p]$ to denote the kernel 
  of the induced map $x^*_f$. Then, the localization $(M^{\diamond}(k,\tau)\otimes_{R_p}
  \mathbb{T}^{\diamond}_{\overline{\rho}, \Sigma})_{\mathfrak{m}_f}$ of $M^{\diamond}(k,\tau)\otimes_{R_p}\mathbb{T}^{\diamond}_{\overline{\rho}, \Sigma}$ 
  at $\mathfrak{m}_f$ is a free $\mathbb{T}^{\diamond}_{\overline{\rho},\Sigma}(k,\tau)_{\mathfrak{m}_f}$-module of rank one 
  by Theorem \ref{3.5}, and we fix a $\mathbb{T}^{\diamond}_{\overline{\rho},\Sigma}(k,\tau)_{\mathfrak{m}_f}$-linear isomorphism 
  $$(M^{\diamond}(k,\tau)\otimes_{R_p}
  \mathbb{T}^{\diamond}_{\overline{\rho}, \Sigma})_{\mathfrak{m}_f}\isom \mathbb{T}^{\diamond}_{\overline{\rho},\Sigma}(k,\tau)_{\mathfrak{m}_f}.$$ 
  
  By this isomorphism, the map 
  $$\bold{z}_{\Sigma, n}(f) : 
 \rho_f^*\rightarrow H^1_{\mathrm{Iw}}(\mathbb{Z}[1/\Sigma_n, \zeta_n], \rho_f^*(1))$$ corresponds 
  to the base change by the map $x_f^* : \mathbb{T}_{\overline{\rho}, \Sigma}[1/p]\rightarrow E$ 
  of the map 
 \begin{equation}\label{3a}  \mathrm{id}_{M^{\diamond}(k,\tau)}\otimes \bold{z}_{\Sigma, n}(\rho^{\mathfrak{m}}) : 
  M^{\diamond}(k,\tau)\otimes_{R_p}(\rho^{\mathfrak{m}})^*\rightarrow   H^1_{\mathrm{Iw}}(\mathbb{Z}[1/\Sigma_n, \zeta_n], M^{\diamond}(k,\tau)\otimes_{R_p}(\rho^{\mathfrak{m}})^*(1)).
  \end{equation}
  
  By Proposition \ref{3.7}, the map (\ref{3a}) also corresponds to the base change by $x^*_f$
  of the following map 
  \begin{multline}
  \Psi_{\Sigma_0}(z^{\mathrm{Iw}}_{\Sigma_0, n, \overline{\rho}}(k, \sigma^{\diamond}(\tau), -)) :  \Psi_{\Sigma_0}(H^1_{\Sigma_0}(\mathcal{V}^*_{k/E}(\sigma^{\diamond}(\tau)))_{\overline{\rho}})(1)\\
  \rightarrow H^1_{\mathrm{Iw}}(\mathbb{Z}[1/\Sigma_n, \zeta_n],  \Psi_{\Sigma_0}(H^1_{\Sigma_0}(\mathcal{V}^*_{k/E}(\sigma^{\diamond}(\tau)))_{\overline{\rho}})(2)). 
   \end{multline}
   
   Therefore, it suffices to compare this base change with Kato's zeta morphism 
   $$\prod_{l\in \Sigma_0}
 P_{f, l}(\sigma_l^{-1})\cdot \bold{z}_{n}(f) : V'_1(f)_E\rightarrow H^1_{\mathrm{Iw}}(\mathbb{Z}[1/\Sigma_{f,n}, \zeta_n], V'_1(f)_E(1)).$$
 
 For this purpose, 
 we first remark that the composite 
 $$ H^1_{\Sigma_0}(\mathcal{V}^*_{k/E}(\sigma^{\diamond}(\tau)))[f]\rightarrow  
 H^1_{\Sigma_0}(\mathcal{V}^*_{k/E}(\sigma^{\diamond}(\tau)))_{\overline{\rho}}\otimes_{\mathbb{T}_{\overline{\rho}, \Sigma}, x^*_f}E$$
 of the inclusion $H^1_{\Sigma_0}(\mathcal{V}^*_{k/E}(\sigma^{\diamond}(\tau)))[f]\hookrightarrow H^1_{\Sigma_0}(\mathcal{V}^*_{k/E}(\sigma^{\diamond}(\tau)))$
  with the quotient map $H^1_{\Sigma_0}(\mathcal{V}^*_{k/E}(\sigma^{\diamond}(\tau)))\rightarrow 
  H^1_{\Sigma_0}(\mathcal{V}^*_{k/E}(\sigma^{\diamond}(\tau)))_{\overline{\rho}}\otimes_{\mathbb{T}_{\overline{\rho}, \Sigma}, x^*_f}E$ 
  is isomorphism, 
  where we set 
 \begin{multline}
 H^1_{\Sigma_0}(\mathcal{V}^*_{k/E}(\sigma^{\diamond}(\tau)))[f]:=\{v\in H^1_{\Sigma_0}(\mathcal{V}^*_{k/E}(\sigma^{\diamond}(\tau))) \mid \\
 T'_lv=a_lv, \ \ S'_lx=l^{k-2}\varepsilon(l)v \text{ for any } l\not\in \Sigma\}.
 \end{multline}
Applying the functor $\Psi_{\Sigma_0}$, we obtain the following isomorphism 
 \begin{multline}
\Psi_{\Sigma_0}(H^1_{\Sigma_0}(\mathcal{V}^*_{k/E}(\sigma^{\diamond}(\tau)))[f])
 \isom  \Psi_{\Sigma_0}(H^1_{\Sigma_0}(\mathcal{V}^*_{k/E}(\sigma^{\diamond}(\tau)))_{\overline{\rho}}\otimes_{\mathbb{T}_{\overline{\rho}, \Sigma}, x^*_f}E)\\
\isom \Psi_{\Sigma_0}(H^1_{\Sigma_0}(\mathcal{V}^*_{k/E}(\sigma^{\diamond}(\tau)))_{\overline{\rho}})\otimes_{\mathbb{T}_{\overline{\rho}, \Sigma}, x^*_f}E.
\end{multline}

By this isomorphism, the base change of  $\Psi_{\Sigma_0}(z^{\mathrm{Iw}}_{\Sigma_0, n, \overline{\rho}}(k, \sigma^{\diamond}(\tau), -))$ by the map 
$x_f^*$ corresponds to the following map : 
\begin{multline}
 \Psi_{\Sigma_0}(z^{\mathrm{Iw}}_{\Sigma_0, n, \overline{\rho}}(k, \sigma^{\diamond}(\tau), -)[f]) : \Psi_{\Sigma_0}(H^1_{\Sigma_0}(\mathcal{V}^*_{k/E}(\sigma^{\diamond}(\tau)))[f])(1)\\
\rightarrow H^1_{\mathrm{Iw}}(\mathbb{Z}[1/\Sigma_n, \zeta_n], \Psi_{\Sigma_0}(H^1_{\Sigma_0}(\mathcal{V}^*_{k/E}(\sigma^{\diamond}(\tau)))[f])(2)).
\end{multline}

We finally equip $\Psi_{\Sigma_0}(H^1_{\Sigma_0}(\mathcal{V}^*_{k/E}(\sigma^{\diamond}(\tau)))[f])$ with a motivic structure, then 
compare the map $ \Psi_{\Sigma_0}(z^{\mathrm{Iw}}_{\Sigma_0, n, \overline{\rho}}(k, \sigma^{\diamond}(\tau), -)[f]) $ with Kato's zeta morphism.

For this, we take a sufficiently large $m\geqq 1$ such that $K_m$ acts trivially on $\sigma^{\diamond}(\tau)$, then 
$\sigma^{\diamond}(\tau)$ is an aboslutely irreducible $E$-representation of the finite group $\mathrm{GL}_2(\mathbb{Z}/p^m\mathbb{Z})$. 
Hence, we can take a $\mathrm{GL}_2(\mathbb{Z}/p^m\mathbb{Z})$-stable $F$-lattice $\sigma^{\diamond}(\tau)_F$ of $\sigma^{\diamond}(\tau)$ 
for $F=E\cap \overline{\mathbb{Q}}$. 

We set 
\begin{equation}
H^1_{\Sigma_0}(\mathcal{V}_A(\sigma^{\diamond}(\tau))):=\mathrm{Hom}_{F[K_0]}(\sigma^{\diamond}(\tau)_F, 
H^1_{\Sigma}(\mathcal{V}_A))
\end{equation}
for $A=F, E, \mathbb{C}$ and $\mathcal{V}_A=\mathcal{V}^*_{k/A}, \mathcal{V}_{k/A}$. 
We also set $S_{k,\Sigma, A}:=\varinjlim_{N}S_k(N)_A$, where $N$ run through all the $N\geqq 3$ such that $\mathrm{prime}(N)=\Sigma$, and 
$$S_{k, \Sigma_0}(\sigma^{\diamond}(\tau))_A:=\mathrm{Hom}_{F[K_0]}(\sigma^{\diamond}(\tau)_F, S_{k,\Sigma, A}).$$
We remark that these are naturally equipped with an action of $G_{\Sigma_0}$ and Hecke actions of $T'_l, S'_l$ for any prime 
$l\not\in \Sigma$. Then, we define 
\begin{equation}
V_3'(f)_A:=\Psi_{\Sigma_0}(H^1_{\Sigma_0}(\mathcal{V}^*_{k/A}(\sigma^{\diamond}(\tau)))[f])(1), \ \ 
 V_3(f^*)_A:=\Psi_{\Sigma_0}(H^1_{\Sigma_0}(\mathcal{V}_{k/A}(\sigma^{\diamond}(\tau)))[f]),
 \end{equation}
 and 
 \begin{equation}
 S_3(f^*)_A:=\Psi_{\Sigma_0}(S_{k, \Sigma_0}(\sigma^{\diamond}(\tau))_A[f]). 
 \end{equation}
 
 By (the proof of) Lemma \ref{4.3}, $H^1_{\Sigma_0}(\mathcal{V}_{k/A}(\sigma^{\diamond}(\tau)))[f]$ and $S_{k, \Sigma_0}(\sigma^{\diamond}(\tau))_A[f]$ 
 are also the subspaces of $H^1_{\Sigma_0}(\mathcal{V}_{k/A}(\sigma^{\diamond}(\tau)))$ and $S_{k, \Sigma_0}(\sigma^{\diamond}(\tau))_A[f]$ respectively 
 consisting of elements $v$ on which $T_l$ and $S_l$ act by $\overline{a_l}$ and $l^{k-2}\overline{\varepsilon_f(l)}$ for any 
 prime $l\not\in \Sigma$.

As in the proof of Proposition \ref{2.14}, the compatibility of the global and the classical local Langlands correspondence (\cite{La73}, \cite{De71}, \cite{Ca86}, \cite{Sc90}, \cite{Sa97}) and 
 the strong multiplicity one theorem imply that there exists an isomorphism 
$$H^1_{\Sigma_0}(\mathcal{V}^*_{k/E}(\sigma^{\diamond}(\tau)))(1)[f]\isom (\rho_f)^*\otimes_E
\widetilde{\pi}_{\Sigma_0}(f)\otimes_E\mathrm{Hom}_{E[K_0]}(\sigma^{\diamond}(\tau)^*, \widetilde{\pi}_{p}(f)) $$
of $E[G_{\mathbb{Q}}\times G_{\Sigma_0}]$-modules, where $\widetilde{\pi}_{\Sigma_0}(f)$ is the smooth contragradient of $\pi_{\Sigma_0}(f):=\otimes_{l\in \Sigma_0}\pi_l(f)$. Therefore, one also has an isomorphism 
\begin{equation}
V'_3(f)_E
\isom(\rho_f)^*\otimes_E
\Psi_{\Sigma_0}(\widetilde{\pi}_{\Sigma_0}(f))\otimes_E\mathrm{Hom}_{E[K_0]}(\sigma^{\diamond}(\tau)^*, \widetilde{\pi}_{p}(f))\isom (\rho_f)^*
\end{equation}
of $E[G_{\mathbb{Q}}]$-modules since one has 
$$\mathrm{dim}_E\left(\Psi_{\Sigma_0}(\widetilde{\pi}_{\Sigma_0}(f))\right)=1$$ and 
$$\mathrm{dim}_E\left(\mathrm{Hom}_{E[K_0]}(\sigma^{\diamond}(\tau)^*, \widetilde{\pi}_{p}(f))\right)= \mathrm{dim}_E\left(\mathrm{Hom}_{E[K_0]}(\sigma^{\diamond}(\tau), \pi_{p}(f))\right)=1.$$ 
Similarly, one can show the equalities
$$\mathrm{dim}_A(V'_3(f)_A)=\mathrm{dim}_A(V_3(f^*)_A)=2$$
and 
$$\mathrm{dim}_A(S_3(f^*)_A)=1$$
for any $A=E, F, \mathbb{C}$. 
The $\mathrm{Gal}(\mathbb{C}/\mathbb{R})$-action on $H^1(Y(N), \mathcal{V}^*_{k/F})$ for each $N$ induces its action on 
 $V'_3(f)_F$, and the map $\mathrm{per} : S_k(N)_F\rightarrow H^1(Y(N), \mathcal{V}_{k/\mathbb{C}})$ induces a map 
 $\mathrm{per} : V_3(f^*)_F\rightarrow S_3(f^*)_{\mathbb{C}},$ etc. Therefore, one can naturally equip $V'_3(f)_E$ with a motivic structure. 
 
 We next define an isomorphism $V'_1(f)_E\isom V'_3(f)_E$ compatible with its motivic structures. Since one has such an isomorphism 
 $V'_1(f)_E\isom V'_2(f)_E$ defined in the proof of Proposition \ref{2.14}, it suffices to define such an isomorphism 
 $V'_2(f)_E\isom V'_3(f)_E$. Since we can freely enlarge $E$, it suffices to define such an isomorphism when 
 $E=\overline{\mathbb{Q}}_p$ (and $F=\overline{\mathbb{Q}}$).  
 
 Let $\pi_{\Sigma_0}(f)_F$ (resp. $\pi_p(f)_F$) be an $G_{\Sigma_0}$ (resp. $G_p$) stable $F$-lattice of $\pi_{\Sigma_0}(f)$ (resp. $\pi_{p}(f)$). 
 Then, one has $\mathrm{dim}_F\mathrm{Hom}_{F[K_0]}(\sigma^{\diamond}(\tau)^*_F, \widetilde{\pi}_p(f)_F)=1$, 
 and we take a non-zero element $\varphi_0\in \mathrm{Hom}_{F[K_0]}(\sigma^{\diamond}(\tau)^*_F, \widetilde{\pi}_p(f)_F)$ and 
 $y_1\in \sigma^{\diamond}(\tau)^*_F$ such that $y_2:=\varphi_0(y_1)\not=0$. Since $\widetilde{\pi}_p(f)_F$ is irreducible, 
 there exists some $a\in F[G_p]$ such that $y_2=ay_3$ for some new vector $y_3\in \widetilde{\pi}_p(f)_F$. 
 We next remark that one has $\mathrm{dim}_F(\Psi_{\Sigma_0}(\widetilde{\pi}_{\Sigma_0}(f)_F))=1$ and 
 that $\Psi_{\Sigma_0}(\widetilde{\pi}_{\Sigma_0}(f)_F)$ is a quotient of $\widetilde{\pi}_{\Sigma_0}(f)_F$ 
 (since we assume $F=\overline{\mathbb{Q}}$). Then, we take some $w_1\in \widetilde{\pi}_{\Sigma_0}(f)_F$ such that 
 $\overline{w_1}\in \Psi_{\Sigma_0}(\widetilde{\pi}_{\Sigma_0}(f)_F)$ is non-zero, and some $b\in F[G_{\Sigma_0}]$ such that 
 $w_1=b w_2$ for some product $w_2=\otimes_{l\in \Sigma_0}v_l$ of new vectors $v_l\in \widetilde{\pi}_l(f)_F$. 
 Then, we write $V_4'(f)_A\subset \Psi_{\Sigma_0}(H^1_{\Sigma}(\mathcal{V}^*_{k/A})(1)[f])$ 
 to denote the image of $ab\cdot V_2'(f)_A\subseteq H^1_{\Sigma}(\mathcal{V}^*_{k/A})(1)[f]$ by 
 the quotient map $H^1_{\Sigma}(\mathcal{V}^*_{k/A})(1)[f]\rightarrow \Psi_{\Sigma_0}(H^1_{\Sigma}(\mathcal{V}^*_{k/A})(1)[f])$, and similarly 
 define submodules $V_4(f^*)_A\subset \Psi_{\Sigma_0}(H^1_{\Sigma}(\mathcal{V}_{k/A})[f])$ and 
 $S_4(f^*)_A\subset  \Psi_{\Sigma_0}(S_{k,\Sigma,A}[f])$. Then, $V'_4(f)_E$ is equipped with a motivic structure 
 naturally defined using $V_4'(f)_A$, $V_4(f^*)_A$ and $S_4(f^*)_A$, and the maps
 $H^1_{\Sigma}(\mathcal{V})(1)[f]\rightarrow \Psi_{\Sigma_0}(H^1_{\Sigma}(\mathcal{V})(1)[f])$ for 
 $\mathcal{V}=\mathcal{V}^*_{k/A}, \mathcal{V}_{k/A}$, and 
 $S_{k,\Sigma, A}[f]\rightarrow \Psi_{\Sigma_0}(S_{k,\Sigma, A}[f])$ defined by 
 $x\mapsto \overline{abx}$ induce isomorphisms 
 $V'_2(f)_A\isom V'_4(f)_A$, $V_2(f^*)_A\isom V_4(f^*)_A$ and $S_2(f^*)_A\isom 
 S_4(f^*)_A$ compatible with motivic structures. Moreover, the map
 \begin{multline}
 \Psi_{\Sigma_0}(H^1_{\Sigma_0}(\mathcal{V}(\sigma^{\diamond}(\tau))[f])=\Psi_{\Sigma_0}(\mathrm{Hom}_{A[K_0]}(\sigma^{\diamond}(\tau)_F^*, 
 H^1_{\Sigma}(\mathcal{V})[f]))\\
 \xrightarrow{\varphi\mapsto \varphi(y_1)}\Psi_{\Sigma_0}(H^1_{\Sigma}(\mathcal{V})[f])
 \end{multline}
 for $\mathcal{V}=\mathcal{V}^*_{k/A}, \mathcal{V}_{k/A}$, and the similar map 
 $$ \Psi_{\Sigma_0}(S_{k, \Sigma_0}(\sigma^{\diamond}(\tau))_A[f])\rightarrow \Psi_{\Sigma_0}(S_{k, \Sigma, A}[f])$$ 
 induce an isomorphisms 
 $V'_3(f)_A\isom V'_4(f)_A$, $V_3(f^*)_A\isom V_4(f^*)_A$ and $S_3(f^*)_A\isom 
 S_4(f^*)_A$ compatible with motivic structures. Therefore, one obtains an isomorphism 
 $V'_2(f)_E\isom V'_3(f)_E$ compatible with motivic structures. 
 
 Finally, by definition (especially, by Proposition \ref{2.6}), the map 
 $$z_3:=\Psi_{\Sigma_0}(z^{\mathrm{Iw}}_{\Sigma_0, n, \overline{\rho}}(k, \sigma^{\diamond}(\tau), -)[f]) : V'_3(f)_E\rightarrow 
 H^1_{\mathrm{Iw}}(\mathbb{Z}[1/\Sigma_n, \zeta_n], V'_3(f)(1))$$
    is an $E$-linear map satisfying the following property : for each $m\geqq 0$, $\gamma\in V'_3(f)_F\subset V'_3(f)_E$, and  $c\geqq 2$ such that $c\equiv 1$ (mod $N$) and $(c, 6pn)=1$, one has 
   $d\cdot z_3(\gamma)\in H^1_{\mathrm{Iw}}(\mathbb{Z}[1/\Sigma_n, \zeta_n], V'_3(f)_E(1))$ for 
   $d=\prod_{1\leqq j\leqq k-1}(c^2-c^{2-j}\sigma_c)(c-c^{j-k+2}\sigma_c)\in Q_{k,n}$, and 
  the image $\omega_{\gamma, m, d}$ of $d\cdot z_f(\gamma)$ 
  by the following composite
  \begin{multline*}
  H^1_{\mathrm{Iw}}(\mathbb{Z}[1/\Sigma_{f,n}, \zeta_n], V'_3(f)_E(1))\xrightarrow{x\mapsto x\otimes x\otimes (\zeta_{p^{n}})_{n\geqq 1}^{\otimes (-k)}} 
    H^1_{\mathrm{Iw}}(\mathbb{Z}[1/\Sigma_{f,n}, \zeta_n], V_3(f^*)_E)\\
    \xrightarrow{\mathrm{exp}^*_{m,1}}S_3(f^*)_E\otimes_{\mathbb{Q}}\mathbb{Q}(\zeta_{np^m})
  \end{multline*}
  belongs to $S_3(f^*)_F\otimes_{\mathbb{Q}}\mathbb{Q}(\zeta_{np^m})$, and the map 
  $$S_3(f^*)_F\otimes_{\mathbb{Q}}\mathbb{Q}(\zeta_{np^m})\rightarrow V_3(f^*)_{\mathbb{C}}^{\pm} : x\otimes y\mapsto \sum_{\sigma\in \mathrm{Gal}(\mathbb{Q}(\zeta_{np^m})/\mathbb{Q})}\chi(\sigma)\sigma(y)\mathrm{per}(x)^{\pm},$$where 
  $\chi : \mathrm{Gal}(\mathbb{Q}(\zeta_{np^m})/\mathbb{Q})\rightarrow \mathbb{C}^{\times}$ is any character and $\pm 1=\chi(-1)$, sends the element 
  $\omega_{\gamma, m, d}$ to 
  $$L_{\Sigma, n}(f,\chi, k-1)\cdot (d'\cdot \gamma)^{\pm}$$
  for $d'=\prod_{j=1}^{k-1}(c^2-c^{k-j}\sigma_c)(c^2-c^j\sigma_c)\in Q_{k,n}$. 

 By all these arguments, we finally obtain the desired equality 
  $$\bold{z}_{\Sigma, n}(f)=\prod_{l\in \Sigma_0}
 P_{f, l}(\sigma_l^{-1})\cdot  \bold{z}_{n}(f)$$
 by Corollary \ref{4.5}.


  \end{proof}
  \section{An application to Kato's main conjecture}
  
  \subsection{Recall of Kato's main conjecture}
  Let $f=\sum_{n=1}^{\infty}a_nq^n$ be a normalized Hecke eigen cusp newform. We take a $G_{\mathbb{Q}}$-stable $\mathcal{O}$-lattice $\rho^*_f$
  of $V'_1(f)$. 
  We set 
  $$H^2_{\mathrm{Iw}}(\mathbb{Z}[1/p], \rho_f^*(1)):=\mathrm{Ker}(H^2_{\mathrm{Iw}}(\mathbb{Z}[1/p], \rho_f^*(1))
  \rightarrow \oplus_{l\in \Sigma_{f,0}}H^2_{\mathrm{Iw}}(\mathbb{Q}_l, \rho_f^*(1))),$$
  and $\Lambda_{n, A}:=\Lambda_n\otimes_{\mathbb{Z}_p}A$ for $A=\mathbb{F}, \mathcal{O}, E$. As a consequence of Kato's Euler system, he showed in Theorem 12.4 of \cite{Ka04} that 
  $H^2_{\mathrm{Iw}}(\mathbb{Z}[1/p], \rho_f^*(1))$ is a finite generated torsion $\Lambda_{\mathcal{O}}$-module and 
  $H^1_{\mathrm{Iw}}(\mathbb{Z}[1/\Sigma_f], \rho_f^*(1))$ is a torsion free $\Lambda_{\mathcal{O}}$-module generically of rank one. 
  
  From now on, we assume for simplicity that $\overline{\rho}_f$ is absolutely irreduceble, and $p$ is odd. Then, $\bold{z}_n(f)$ preserved integral structure, i.e.  $\bold{z}_n(f)$ induces 
    an $\mathcal{O}$-linear map 
    $$\bold{z}_n(f) : \rho_f^*\rightarrow \mathrm{H}^1_{\mathrm{Iw}}(\mathbb{Z}[1/\Sigma_{f,n}, \zeta_n], \rho_f^*(1)) $$
    by Remark (\ref{integer}). By Remark \ref{pm1}, this map $\bold{z}_n(f)$ naturally induces a $\Lambda_{n, \mathcal{O}}$-linear map 
    $$\bold{z}_n(f) : \rho_f^*\otimes_{\mathcal{O}[\{\pm 1\}]}\Lambda_{n, \mathcal{O}}\rightarrow \mathrm{H}^1_{\mathrm{Iw}}(\mathbb{Z}[1/\Sigma_{f,n}, \zeta_n], \rho_f^*(1)) $$ which we denote by the same letter. 
    
    As a version of Iwasawa main conjecture, Kato proposed in Conjecture 12.10 of \cite{Ka04} the following conjecture.
    \begin{conjecture}\label{Katoconj}
    One has an equality 
    $$\mathrm{Char}_{\Lambda_{\mathcal{O}}}(H^1_{\mathrm{Iw}}(\mathbb{Z}[1/\Sigma_f], \rho_f^*(1))/\mathrm{Im}(\bold{z}_1(f)))=
    \mathrm{Char}_{\Lambda_{\mathcal{O}}}(H^2_{\mathrm{Iw}}(\mathbb{Z}[1/p], \rho_f^*(1)))$$
    of characteristic ideals of torsion $\Lambda_{\mathcal{O}}$-modules.
    \end{conjecture}
   In this article, we call this conjecture Kato's main conjecture. 
   Concerning this conjecture, Kato (essentially) proved the following.
   \begin{thm}\label{Kato}
  Assume  furthermore the following $:$ 
  \begin{itemize}
  \item[(a)] there exists an element 
   $\sigma\in \mathrm{Gal}(\overline{\mathbb{Q}}/\mathbb{Q}(\zeta_{p^{\infty}}))$ such that 
   $\rho_{f}/(\sigma-1)\rho_f$ is a free $\mathcal{O}$-module of rank one. 
   \end{itemize}
   Then, one has an inclusion 
   $$\mathrm{Char}_{\Lambda_{\mathcal{O}}}(H^1_{\mathrm{Iw}}(\mathbb{Z}[1/\Sigma_f], \rho_f^*(1))/\mathrm{Im}(\bold{z}_1(f)))
   \subset 
    \mathrm{Char}_{\Lambda_{\mathcal{O}}}(H^2_{\mathrm{Iw}}(\mathbb{Z}[1/p], \rho_f^*(1)))$$
   \end{thm}
   \begin{proof}
   Kato proved this theorem in Theorem 12.4 (4) of \cite{Ka04} under the stronger assumption that 
   $\mathrm{Im}(\rho_f)$ contains a conjugate of $\mathrm{SL}_2(\mathbb{Z}_p)$. However, Skinner pointed out in a comment after Theorem 2.5.4 of \cite{Sk16} 
   that the weaker assumption (a)  (plus the absolutely irreducibility of $\overline{\rho}_f$) is enough for the inclusion. 
   \end{proof}
   
   We set $\Delta:=\mathrm{Gal}(\mathbb{Q}(\zeta_p)/\mathbb{Q})$. Let $\eta : \Delta\rightarrow \mathbb{Z}_p^{\times}$ be a character. 
   For a $\mathbb{Z}_p[\Delta]$-module $M$, we write $M^{\eta}$ to denote the $\eta$-component of $M$, i.e. the sub module of $M$ consisting of 
   elements on which $\Delta$ acts by $\eta$. Since we assume that $p$ is odd, every such $M$ is written as a sum of the $\eta$-components 
   $M=\oplus_{\eta\in \widehat{\Delta}}M^{\eta}$. By this decomposition, Conjecture \ref{Kato} is also 
   decomposed into the $\eta$-component for each character $\eta$, which we call the $\eta$-part of Kato's main conjecture.
   \begin{conjecture}\label{eta}
    $$\mathrm{Char}_{\Lambda^{\eta}_{\mathcal{O}}}(H^1_{\mathrm{Iw}}(\mathbb{Z}[1/\Sigma_f], \rho_f^*(1))^{\eta}/\mathrm{Im}(\bold{z}_1(f)^{\eta}))=
    \mathrm{Char}_{\Lambda^{\eta}_{\mathcal{O}}}(H^2_{\mathrm{Iw}}(\mathbb{Z}[1/p], \rho_f^*(1))^{\eta})$$
    \end{conjecture}

 Since the ring $\Lambda^{\eta}_{\mathcal{O}}$ is (non-canonically) isomorphic to $\mathcal{O}[[T]]$, one can define the $\mu$- and $\lambda$- invariants for each finite generated torsion $\Lambda^{\eta}_{\mathcal{O}}$-modules, which we recall now. Each element $f(T)\in \mathcal{O}[[T]]$ can be written of the form $f(T)=\varpi^mg(T)h(T)$ with $m\in \mathbb{Z}_{\geqq 0}$, $g(T)=T^n+c_1T^{n-1}+\cdots +c_n$ ($c_i\in (\varpi)$)
   and $h(T)\in \mathcal{O}[[T]]^{\times}$. Then, we set $\mu(f(T)):=m$ and $\lambda(f(T)):=n$. For each finite generated torsion 
   $\Lambda_{\mathcal{O}}^{\eta}$-module $M$, we define $\mu(M):=\mu(f_M)$ and $\lambda(M):=\lambda(f_M)$ using any generator $f_M$ of the characteristic ideal $\mathrm{Char}_{\Lambda_{\mathcal{O}}^{\eta}}(M)$ of $M$. 
      
   By Theorem \ref{Kato}, we can rephrase the $\eta$-part of Kato's main conjecture using $\lambda$- and $\mu$- invariants. 
   \begin{corollary}\label{mu} Let $\eta : \Delta\rightarrow \mathbb{Z}_p^{\times}$ be a character. 
   Assume that $f$ satisfies the assumption $(a)$ in Theorem $\ref{Kato}$. Then the following conditions $(1)$ and $(2)$ are equivalent. 
   \begin{itemize}
   \item[(1)]Conjecture $\ref{eta}$ holds for $f$ and $\eta$. 
   \item[(2)]One has $$\mu(H^1_{\mathrm{Iw}}(\mathbb{Z}[1/\Sigma_f], \rho_f^*(1))^{\eta}/\mathrm{Im}(\bold{z}_1(f)^{\eta}))=
   \mu(H^2_{\mathrm{Iw}}(\mathbb{Z}[1/p], \rho_f^*(1))^{\eta})$$ and 
   $$\lambda(H^1_{\mathrm{Iw}}(\mathbb{Z}[1/\Sigma_f], \rho_f^*(1))^{\eta}/\mathrm{Im}(\bold{z}_1(f)^{\eta}))=
   \lambda(H^2_{\mathrm{Iw}}(\mathbb{Z}[1/p], \rho_f^*(1))^{\eta}).$$
   
   \end{itemize}
   \end{corollary}

   
   \subsection{Congruences of zeta morphisms}
  Let $f=\sum_{n=1}^{\infty}a_nq^n$ be a Hecke eigen cusp newform such that $\overline{\rho}_f$ is absolutely irreducible. 
  For each integer $m\geqq 1$, one can consider the mod $\varpi^m$-reduction of the map $\bold{z}_n(f) : \rho_f^*\rightarrow \mathrm{H}^1_{\mathrm{Iw}}(\mathbb{Z}[1/\Sigma_{f,n}, \zeta_n], \rho_f^*(1)),$ which we denote by
  $$\bold{z}_n(f) \bmod \varpi^m : \rho_{f, m}^*\rightarrow 
  H^1_{\mathrm{Iw}}(\mathbb{Z}[1/\Sigma_{f,n}, \zeta_n], \rho_{f, m}^*(1))$$
  for $\rho_{f, m}:=\rho_f\otimes_{\mathcal{O}}\mathcal{O}/\varpi^m$. 
  Since  $\overline{\rho}_f$ is absolutely irreducible, all the $G_{\mathbb{Q}}$-stable $\mathcal{O}$-lattices of 
  $V'_1(f)$ are of the form $\varpi^k\rho_f^*$ for some $k\in \mathbb{Z}$. Hence, we note that $\rho_{f,m}$ and $\bold{z}_n(f) \bmod \varpi^m$ (up to isomorphism) are independent of the choice of $\rho_f$. 
  
  We recall that we fixed embeddings $\iota_{\infty} : \overline{\mathbb{Q}}\hookrightarrow \mathbb{C}$ and 
  $\iota_p : \overline{\mathbb{Q}}\hookrightarrow \overline{\mathbb{Q}}_p$, by which we regard the Fourier coefficients $a_n$ ($n\geqq 1$) of $f$ 
  as elements in $E$, in fact in its integer ring $\mathcal{O}$. Let $f=\sum_{n=1}^{\infty}a_nq^n$ and $g=\sum_{n=1}^{\infty}b_nq^n$ are normalized Hecke eigen cusp new forms. 
  We take a finite extension $E\subset \overline{\mathbb{Q}}_p$ of $\mathbb{Q}_p$ containing all $a_n, b_n$ ($n\geqq 1$). In this situation, for $m\geqq1$, 
  we say that $f$ and $g$ are congruent modulo $\varpi^m$ if $a_l\equiv b_l\bmod \varpi^m$ for all but finitely many primes $l$. 
  \begin{lemma}\label{7.1}
  Assume that $\overline{\rho}_f$ is absolutely irreducible. Then the following are equivalent :
  \begin{itemize}
  \item[(1)]$f$ and $g$ are congruent module $\varpi^m$. 
  \item[(2)]$\rho_{f,m}\isom \rho_{g,m}$.
  \end{itemize}
  \end{lemma}
  \begin{proof}
  Since we assume that $\overline{\rho}_f$ is absolutely irreducible, the isomorphism class of $\rho_{f, m}$ is determined by its trace by 
  the theory of pseudo representations by \cite{Ny96}. Then the equivalence of (1) and (2) follows from Chevotarev density theorem since 
  one has $a_l=\mathrm{tr}(\rho_f(\mathrm{Frob}_l))$ and $b_l=\mathrm{tr}(\rho_g(\mathrm{Frob}_l))$ for all but finitely many primes $l$. 
  \end{proof}
  
  As an immediate corollary of our main theorem, we obtain the following congruence of zeta morphisms $\bold{z}_n(f)$. 
  \begin{corollary}\label{cong}
  Let $f=\sum_{n=1}^{\infty}a_nq^n$ and $g=\sum_{n=1}^{\infty}b_nq^n$ be normalized Hecke eigen cusp new forms such that 
  $a_n, b_n\in E$ for all $n\geqq 1$. Assume that $\overline{\rho}_f$ is absolutely irreducible, and all the assumptions $(1), (2), (3)$ in the beginning of $\S 4$
  are satisfied for $\overline{\rho}:=\overline{\rho}_f$. 
  If $f$ and $g$ are congruent module $\varpi^m$, 
  then one has 
  $$\prod_{l\in \Sigma_0}P_{f, l}(\sigma_l^{-1})\cdot \bold{z}_n(f) \bmod \varpi^m=\prod_{l\in \Sigma_0}P_{g, l}(\sigma_l^{-1})\cdot\bold{z}_n(g) \bmod \varpi^m$$
  under the isomorphism $\rho_{f,m}\isom \rho_{g,m}$ above for $\Sigma_0:=\Sigma_f\cup \Sigma_g\setminus\{p\}$.
  \end{corollary}
 \begin{proof}Since both $\rho_f$ and $\rho_g$ are deformations of $\overline{\rho}=\overline{\rho}_f\isom \overline{\rho}_g$, 
one has $\mathcal{O}$-valued points 
 $$x_f, x_g\in \mathrm{Spec}(\mathbb{T}_{\overline{\rho}, \Sigma})(\mathcal{O})$$ corresponding to $\rho_f$ and $\rho_g$ respectively 
 for $\Sigma:=\Sigma_f\cup \Sigma_g$. By
 Theorem \ref{3.2}, one has 
 $$\bold{z}_{\Sigma, n}(f)=\prod_{l\in \Sigma_0}P_{f, l}(\sigma_l^{-1})\cdot \bold{z}_n(f)$$
 and 
  $$\bold{z}_{\Sigma, n}(g)=\prod_{l\in \Sigma_0}P_{g, l}(\sigma_l^{-1})\cdot \bold{z}_n(g).$$
  The existence of isomorphism $\rho_{f,m}\isom \rho_{g,m}$ implies the equality 
  $$x_f\bmod \varpi^m=x_g\bmod \varpi^n : \mathbb{T}_{\overline{\rho},\Sigma}\rightarrow \mathcal{O}/\varpi^m.$$
  Therefore, one obtains an equality
  $$\bold{z}_{\Sigma, n}(f)\bmod \varpi^m=\bold{z}_{\Sigma, n}(g) \bmod \varpi^m$$
  under the isomorphism $\rho_{f,m}\isom \rho_{g,m}$, 
  which shows the corollary.

 \end{proof}
 \subsection{An application to Kato's main conjecture}
  Let $\overline{\rho} : G_{\mathbb{Q}}\rightarrow \mathrm{GL}_2(\mathbb{F})$  be as in \S 4. We write $\Sigma_{\overline{\rho}}$ to denote the union of $\{p\}$ and the set of primes on which $\overline{\rho}$ is ramified. Then, for each finite set of primes $\Sigma$ containing $\Sigma_{\overline{\rho}}$, and $n\geqq 1$ such that $(n, \Sigma)=1$, one has the zeta morphism
 $$\bold{z}_{\Sigma, n}(\rho^{\mathfrak{m}}_{\Sigma}) : (\rho^{\mathfrak{m}}_{\Sigma})^*\rightarrow H^1_{\mathrm{Iw}}(\mathbb{Z}[1/\Sigma_n, \zeta_n], 
 (\rho^{\mathfrak{m}}_{\Sigma})^*(1))$$ for the universal promodular deformation 
 $ \rho^{\mathfrak{m}}_{\Sigma}$ over $\mathbb{T}_{\overline{\rho}, \Sigma}$, which is simply denoted by $\rho^{\mathfrak{m}}$ in the previous section since 
 we fix $\Sigma$ there. For each finite sets of primes $\Sigma_{\overline{\rho}}\subset\Sigma\subset \Sigma'$, one has 
 a canonical ring homomorphism $\mathbb{T}_{\overline{\rho}, \Sigma'}\rightarrow \mathbb{T}_{\overline{\rho}, \Sigma}$ sending $T_l$ and $S_l$ for $l\not\in \Sigma$ in the source to $T_l$ and $S_l$ in the target, and this induces an isomorphism 
 $$\rho^{\mathfrak{m}}_{\Sigma'}\otimes_{\mathbb{T}_{\overline{\rho}, \Sigma'}}\mathbb{T}_{\overline{\rho}, \Sigma}\isom \rho^{\mathfrak{m}}_{\Sigma}.$$
 By definition of the map $\bold{z}_{\Sigma, n}(\rho^{\mathfrak{m}}_{\Sigma})$, or by Theorem \ref{3.2} and the density of modular points in $\mathbb{T}_{\overline{\rho}, \Sigma}$, one has 
 \begin{equation}\label{77}\bold{z}_{\Sigma', n}(\rho^{\mathfrak{m}}_{\Sigma'})\otimes \mathrm{id}_{\mathbb{T}_{\overline{\rho}, \Sigma'}}
 =\prod_{l\in \Sigma'\setminus \Sigma}P_l(\sigma_l^{-1})\cdot \bold{z}_{\Sigma, n}(\rho^{\mathfrak{m}}_{\Sigma})
 \end{equation} under the isomorphism 
 $\rho^{\mathfrak{m}}_{\Sigma'}\otimes_{\mathbb{T}_{\overline{\rho}, \Sigma'}}\mathbb{T}_{\overline{\rho}, \Sigma}\isom \rho^{\mathfrak{m}}_{\Sigma}$ above.

 For each $\Sigma$ and $n\geqq 1$ as above, we write 
 $$\bold{z}_{\Sigma, n}(\overline{\rho}) : \overline{\rho}^*\rightarrow H^1_{\mathrm{Iw}}(\mathbb{Z}[1/\Sigma_n, \zeta_n], 
 \overline{\rho}^*(1))$$ 
 to denote the base change of $\bold{z}_{\Sigma, n}(\rho^{\mathfrak{m}}_{\Sigma})$ by the canonical quotient map
 $\mathbb{T}_{\overline{\rho}, \Sigma}\rightarrow \mathbb{T}_{\overline{\rho}, \Sigma}/\mathfrak{m}=\mathbb{F}$, which we call the zeta morphism 
 for $\overline{\rho}$. By the equality (\ref{77}), one has 
 $$\bold{z}_{\Sigma', n}(\overline{\rho})=\prod_{l\in \Sigma'\setminus \Sigma}P_l(\sigma_l^{-1})\cdot\bold{z}_{\Sigma, n}(\overline{\rho})$$ for each 
 finite set of primes $\Sigma_{\overline{\rho}}\subset\Sigma\subset \Sigma'$. We remark that the map $\bold{z}_{\Sigma, n}(\overline{\rho})$ naturally induces a
 $\Lambda_{n, \mathbb{F}}$-linear map 
  $$\bold{z}_{\Sigma, n}(\overline{\rho}) : \overline{\rho}^*\otimes_{\mathbb{F}[\{\pm 1\}]}\Lambda_{n, \mathbb{F}}\rightarrow H^1_{\mathrm{Iw}}(\mathbb{Z}[1/\Sigma_n, \zeta_n], 
 \overline{\rho}^*(1))$$ which we denote by the same letter.

 We set $\mathfrak{X}(\overline{\rho}):=\bigcup_{\Sigma_{\overline{\rho}}\subset \Sigma}\mathrm{Spec}(\mathbb{T}_{\overline{\rho}, \Sigma})(\overline{\mathbb{Z}}_p)$. For any point $x\in \mathrm{Spec}(\mathbb{T}_{\overline{\rho}, \Sigma})(\mathcal{O}')\subset\mathfrak{X}(\overline{\rho})$, we write $\rho_x$  to denote 
the specialization of $\rho^{\mathfrak{m}}_{\Sigma}$ at $x$, and write 
$$\bold{z}_{\Sigma, n}(\rho_x) : \rho_x^*\otimes_{\mathcal{O}'[\{\pm 1\}]}\Lambda_{n, \mathcal{O}'}\rightarrow H^1_{\mathrm{Iw}}(\mathbb{Z}[1/\Sigma_n,\zeta_n], 
\rho_x^*)$$
to denote the specialization of the map $\bold{z}_{\Sigma, n}(\rho^{\mathfrak{m}}_{\Sigma})$ at $x$. To simplify the notation, we always assume that such a point $x$ is contained in $\mathrm{Spec}(\mathbb{T}_{\overline{\rho}, \Sigma})(\mathcal{O})$ by freely extending scalar $E$ suitably.

 We write $\mathfrak{X}(\overline{\rho})^{\mathrm{mod}}\subset \mathfrak{X}(\overline{\rho})$ to denote the subset consisting of 
 all the modular points. 
  

As a consequence of (the proof of) Corollary \ref{cong} (for $m=n=1$) and the proof of Theorem 2.1 of \cite{KLP19}, we immediately obtain the following theorems. 
The author would like to Chan-Ho Kim for sending him this article, and discussing applications of our main theorem to Iwasawa main conjecture. 
We give the proof of the theorems after Corollary \ref{coro} below. 

We fix a character $\eta : \Delta\rightarrow \mathbb{Z}_p^{\times}$.
\begin{thm}\label{thmA}
The following conditions $(1), (2), (3)$ are equivalent.
\begin{itemize}
\item[(1)]The $\Lambda_{\mathbb{F}}^{\eta}$-linear map $\bold{z}_{\Sigma_{\overline{\rho}}, 1}(\overline{\rho})^{\eta}$ is non zero. 
\item[(2)]One has
$$\mu(H^1_{\mathrm{Iw}}(\mathbb{Z}[1/\Sigma_f], \rho_f^*(1))^{\eta}/\mathrm{Im}(\bold{z}_1(f)^{\eta}))=0$$ 
for some $x_f\in \mathfrak{X}(\overline{\rho})^{\mathrm{mod}}$. 
\item[(3)]
One has
$$\mu(H^1_{\mathrm{Iw}}(\mathbb{Z}[1/\Sigma_f], \rho_f^*(1))^{\eta}/\mathrm{Im}(\bold{z}_1(f)^{\eta}))=0$$ 
for all $x_f\in \mathfrak{X}(\overline{\rho})^{\mathrm{mod}}$. 
\end{itemize}
\end{thm}

 We write $\mathfrak{X}(\overline{\rho})_a^{\mathrm{mod}}\subset \mathfrak{X}^{\mathrm{mod}}(\overline{\rho})$ to denote the subset consisting of all the points $x_f$ such that $\rho_f$ satisfies the condition (a) in Theorem \ref{Kato}. 

From now on, we assume that $\mathfrak{X}(\overline{\rho})_a^{\mathrm{mod}}\not=\phi$ and the equivalent conditions (1), (2), (3) in Theorem \ref{thmA} are satisfied for a fixed $\eta : \Delta\rightarrow \mathbb{Z}_p^{\times}$. 
\begin{thm}\label{thmB}One has the following.
\begin{itemize}
\item[(1)]$H^1_{\mathrm{Iw}}(\mathbb{Z}[1/\Sigma_{\overline{\rho}}], \overline{\rho}^*(1))^{\eta}$ is a free $\Lambda_{\mathbb{F}}^{\eta}$-module of rank one.
\item[(2)] Both $H^1_{\mathrm{Iw}}(\mathbb{Z}[1/\Sigma_{\overline{\rho}}], \overline{\rho}^*(1))^{\eta}/\mathrm{Im}(\bold{z}_{\Sigma_{\overline{\rho}}, 1}(\overline{\rho})^{\eta})$ and 
$H^2_{\mathrm{Iw}}(\mathbb{Z}[1/\Sigma_{\overline{\rho}}], \overline{\rho}^*(1))^{\eta}$ are finite dimensional $\mathbb{F}$-vector spaces. 
\end{itemize}
\end{thm}
By this theorem, we can formulate the following conjecture, which we call the $\eta$-part of Kato's main conjecture for $\overline{\rho}$. 
\begin{conjecture}\label{conjmodp}
$$\mathrm{dim}_{\mathbb{F}}\left(H^1_{\mathrm{Iw}}(\mathbb{Z}[1/\Sigma_{\overline{\rho}}], \overline{\rho}^*(1))^{\eta}/\mathrm{Im}(\bold{z}_{\Sigma_{\overline{\rho}}, 1}(\overline{\rho})^{\eta})\right)
=\mathrm{dim}_{\mathbb{F}}(H^2_{\mathrm{Iw}}(\mathbb{Z}[1/\Sigma_{\overline{\rho}}], \overline{\rho}^*(1))^{\eta})$$
\end{conjecture}
\begin{thm}\label{thmC}
For any point $x\in \mathrm{Spec}(\mathbb{T}_{\overline{\rho}, \Sigma})(\mathcal{O})\subset\mathfrak{X}(\overline{\rho})$, 
one has the following.
\begin{itemize}
\item[(1)]
$H^1_{\mathrm{Iw}}(\mathbb{Z}[1/\Sigma], \rho_x^*(1))^{\eta}$ is a free $\Lambda_{\mathcal{O}}^{\eta}$-module of rank one.
\item[(2)]Both 
$H^1_{\mathrm{Iw}}(\mathbb{Z}[1/\Sigma], \rho_x^*(1))^{\eta}/\mathrm{Im}(\bold{z}_{\Sigma, 1}(\rho_x)^{\eta})$ and 
$H^2_{\mathrm{Iw}}(\mathbb{Z}[1/\Sigma], \rho_x^*(1))^{\eta}$ are torsion $\Lambda_{\mathcal{O}}^{\eta}$-modules.
\item[(3)]One has $$\mu(H^1_{\mathrm{Iw}}(\mathbb{Z}[1/\Sigma], \rho_x^*(1))^{\eta}/\mathrm{Im}(\bold{z}_{\Sigma,1}(\rho_x)^{\eta}))=\mu(H^2_{\mathrm{Iw}}(\mathbb{Z}[1/\Sigma], \rho_x^*(1))^{\eta})=0.$$ 
\item[(4)]If $x$ is in $\mathfrak{X}(\overline{\rho})^{\mathrm{mod}}$, or if $\rho_x$ satisfies the following condition 
\begin{itemize}
\item[(b)]There exists an element $\tau\in \mathrm{Gal}(\overline{\mathbb{Q}}/\mathbb{Q}(\zeta_{p^{\infty}}))$ such that 
$$\left(\rho_x/(\tau-1)\rho_x\right)[1/p]$$ is one dimensional $E$-vector space. 
\end{itemize}
Then, one has an inclusion 
$$\mathrm{Char}_{\Lambda^{\eta}_{\mathcal{O}}}(H^1_{\mathrm{Iw}}(\mathbb{Z}[1/\Sigma], \rho_x^*(1))^{\eta}/\mathrm{Im}(\bold{z}_{\Sigma,1}(\rho_x)^{\eta}))
\subset 
    \mathrm{Char}_{\Lambda^{\eta}_{\mathcal{O}}}(H^2_{\mathrm{Iw}}(\mathbb{Z}[1/\Sigma], \rho_x^*(1))^{\eta})$$

\end{itemize}

\end{thm}
By this theorem, one can also formulate the following conjecture, which we call the $\eta$-part of Kato's main conjecture for $\rho_x$.
\begin{conjecture} \label{conjx}
$$\mathrm{Char}_{\Lambda^{\eta}_{\mathcal{O}}}(H^1_{\mathrm{Iw}}(\mathbb{Z}[1/\Sigma], \rho_x^*(1))^{\eta}/\mathrm{Im}(\bold{z}_{\Sigma,1}(\rho_x)^{\eta}))=
    \mathrm{Char}_{\Lambda^{\eta}_{\mathcal{O}}}(H^2_{\mathrm{Iw}}(\mathbb{Z}[1/\Sigma], \rho_x^*(1))^{\eta}).$$
    \end{conjecture}
      We will show in the proof below that the validity of this equality for $x$ is independent of the choice of $\Sigma$, and 
    equivalent to the $\eta$-part of Kato's main conjecture for $f$ when $x=x_f$. 

    We also formulate a weak version of the conjecture, which we call the $\eta$-part of weak Kato's main conjecture for $\rho_x$. 
    \begin{conjecture} \label{wconjx}
   $$ \lambda(H^1_{\mathrm{Iw}}(\mathbb{Z}[1/\Sigma], \rho_x^*(1))^{\eta}/\mathrm{Im}(\bold{z}_{\Sigma,1}(\rho_x)^{\eta}))=
\lambda(H^2_{\mathrm{Iw}}(\mathbb{Z}[1/\Sigma], \rho_x^*(1))^{\eta}.$$
    \end{conjecture} 
    
    We write $\mathfrak{X}(\overline{\rho})_b$ to denote the subset of $\mathfrak{X}(\overline{\rho})$ consisting of all the points $x$ satisfying the 
    condition (b). By Theorem \ref{thmC}, Conjecture \ref{conjx} and Conjecture \ref{wconjx} are equivalent for $\rho_x$ such that 
    $x\in \mathfrak{X}(\overline{\rho})^{\mathrm{mod}}\cup \mathfrak{X}(\overline{\rho})_b$.

    Concerning the relation between these conjectures, we also prove the following.
    \begin{thm}\label{thmD}For any point $x\in \mathrm{Spec}(\mathbb{T}_{\overline{\rho}, \Sigma})(\mathcal{O})\subset\mathfrak{X}(\overline{\rho})$, one has 
    \begin{multline*}
\lambda(H^1_{\mathrm{Iw}}(\mathbb{Z}[1/\Sigma], \rho_x^*(1))^{\eta}/\mathrm{Im}(\bold{z}_{\Sigma,1}(\rho_x)^{\eta}))
-\lambda(H^2_{\mathrm{Iw}}(\mathbb{Z}[1/\Sigma], \rho_x^*(1))^{\eta})\\
=\mathrm{dim}_{\mathbb{F}}(H^1_{\mathrm{Iw}}(\mathbb{Z}[1/\Sigma_{\overline{\rho}}], 
\overline{\rho}^*(1))^{\eta}/\mathrm{Im}(\bold{z}_{\Sigma_{\overline{\rho}}, 1}(\overline{\rho})^{\eta}))
-\mathrm{dim}_{\mathbb{F}}H^2_{\mathrm{Iw}}(\mathbb{Z}[1/\Sigma_{\overline{\rho}}], 
\overline{\rho}^*(1))^{\eta}).
\end{multline*}
In particular, it is independent of $x\in \mathfrak{X}(\overline{\rho})$. 
    \end{thm}
    As a corollary of this theorem, we immediately obtain the following corollary. 
\begin{corollary}\label{coro}
The following conditions $(1), (2), (3)$ are equivalent.
\begin{itemize}
\item[(1)]The $\eta$-part of Kato's main conjecture holds for $\overline{\rho}$. 
\item[(2)]The $\eta$-part of weak Kato's main conjecture for $\rho_x$ holds for some 
$x\in \mathfrak{X}(\overline{\rho})$.
\item[(3)]The $\eta$-part of weak Kato's main conjecture for $\rho_x$ holds for all
$x\in \mathfrak{X}(\overline{\rho})$.

\end{itemize}

\end{corollary}

\begin{proof}(of Theorem \ref{thmA}, Theorem \ref{thmB}, Theorem \ref{thmC}, and Theorem \ref{thmD})
In fact, all these theorems easily follow from Corollary \ref{cong} and the proof of Theorem 2.1 of \cite{KLP19} and some standard calculations of 
Galois cohomology and the general theory of Euler systems. For the convenience of readers, we give a proof by recalling their arguments in \cite{KLP19}. 

We first remark that the restriction map
\begin{equation}
H^1_{\mathrm{Iw}}(\mathbb{Z}[1/\Sigma], \rho^*(1))\isom H^1_{\mathrm{Iw}}(\mathbb{Z}[1/\Sigma'], \rho^*(1))
\end{equation}
is isomorphism, and there exists a short exact sequence
\begin{equation}
0\rightarrow H^2_{\mathrm{Iw}}(\mathbb{Z}[1/\Sigma], \rho^*(1))\rightarrow H^2_{\mathrm{Iw}}(\mathbb{Z}[1/\Sigma'], \rho^*(1))
\rightarrow \oplus_{l\in \Sigma'\setminus \Sigma}H^2_{\mathrm{Iw}}(\mathbb{Q}_l, \rho^*(1))\rightarrow 0
\end{equation}
 both for $\rho=\rho_x$ ($x\in \mathfrak{X}(\overline{\rho})$) and $\rho=\overline{\rho}$, and $p\in \Sigma\subset \Sigma'$ 
 (see for example Proposition 5.1 \cite{KLP19}, where it is proved for $\rho_x$ but the same proof works for $\overline{\rho}$). 
 
 We also remark that 
 $H^1_{\mathrm{Iw}}(\mathbb{Z}[1/\Sigma], \rho^*(1))$ is a free $\Lambda_{\mathcal{O}}$-module (resp. $\Lambda_{\mathbb{F}}$-module) 
 for $\rho=\rho_x$ ($x\in \mathfrak{X}(\overline{\rho})$) (resp. $\rho=\overline{\rho}$) since $\overline{\rho}$ is absolutely irreducible 
 (see for example 13.8 of \cite{Ka04}). 
 
 Moreover, for each prime $l\not=p$, $H^2_{\mathrm{Iw}}(\mathbb{Q}_l, \rho^*(1))$ is a torsion $\Lambda_{\mathcal{O}}$-module 
 (resp. $\Lambda_{\mathbb{F}}$-module) with 
 \begin{equation}
 \mathrm{Char}_{\Lambda_{\mathcal{O}}}(H^2_{\mathrm{Iw}}(\mathbb{Q}_l, \rho_x^*(1)))=(P_{\rho_x, l}(\sigma_l^{-1}))\quad (\text{resp}. 
 \mathrm{Char}_{\Lambda_{\mathbb{F}}}(H^2_{\mathrm{Iw}}(\mathbb{Q}_l, \overline{\rho}^*(1)))=(P_{\overline{\rho}, l}(\sigma_l^{-1})))
 \end{equation}
 for $\rho=\rho_x$ (resp. $\rho=\overline{\rho}$) by (the proof of) Corollary 3.7 of \cite{KLP19}, where we set $P_{\rho, l}(u):=\mathrm{det}(1-\mathrm{Frob}_l\cdot u\mid 
 \rho^{I_l})$.  
 
 By these remarks, the condition that $H^2_{\mathrm{Iw}}(\mathbb{Z}[1/\Sigma], \rho_x^*(1))^{\eta}$ is a torsion $\Lambda_{\mathcal{O}}^{\eta}$-module is independent of $\Sigma$, and 
 one has 
 \begin{equation}
 \lambda(H^2_{\mathrm{Iw}}(\mathbb{Z}[1/\Sigma'], \rho_x^*(1))^{\eta})
 =\lambda(H^2_{\mathrm{Iw}}(\mathbb{Z}[1/\Sigma], \rho_x^*(1))^{\eta})
 +\sum_{l\in \Sigma'\setminus \Sigma}\lambda(P_{\rho_x, l}(\sigma_l^{-1})^{\eta})
 \end{equation}
 and 
 \begin{equation}\mu(H^2_{\mathrm{Iw}}(\mathbb{Z}[1/\Sigma'], \rho_x^*(1))^{\eta})
 =\mu(H^2_{\mathrm{Iw}}(\mathbb{Z}[1/\Sigma], \rho_x^*(1))^{\eta})
 \end{equation}
  for any $\Sigma\subset \Sigma'$ if $H^2_{\mathrm{Iw}}(\mathbb{Z}[1/\Sigma], \rho_x^*(1))^{\eta}$ is a torsion $\Lambda_{\mathcal{O}}^{\eta}$-module. 
  For each $f\in \Lambda_{\mathbb{F}}^{\eta}\isom \mathbb{F}[[T]]$, we write its $T$-adic valuation by 
  $\lambda_{\mathbb{F}}(f)\in \mathbb{Z}_{\geqq 0}$. For each finite generated torsion $\Lambda_{\mathbb{F}}^{\eta}$-module $M$ (which is of course equivalent to that $M$ is a finite dimensional over $\mathbb{F}$), we set $\lambda_{\mathbb{F}}(M):=\mathrm{dim}_{\mathbb{F}}(M)$, which is equal 
  to $\lambda_{\mathbb{F}}(f_M)$ for any generator $f_M$ of $\mathrm{Char}_{\Lambda_{\mathbb{F}}}(M)$. 
   Then, the condition that $H^2_{\mathrm{Iw}}(\mathbb{Z}[1/\Sigma], \overline{\rho}^*(1))$ is finite dimensional over $\mathbb{F}$ is independent of $\Sigma$, and 
   one has 
   \begin{equation}\lambda_{\mathbb{F}}(H^2_{\mathrm{Iw}}(\mathbb{Z}[1/\Sigma'], \overline{\rho}^*(1))^{\eta})
 =\lambda_{\mathbb{F}}(H^2_{\mathrm{Iw}}(\mathbb{Z}[1/\Sigma], \overline{\rho}^*(1))^{\eta})
  +\sum_{l\in \Sigma'\setminus \Sigma}\lambda_{\mathbb{F}}(P_{\overline{\rho}, l}(\sigma_l^{-1})^{\eta})
  \end{equation}
 for any $\Sigma\subset \Sigma'$ if $H^2_{\mathrm{Iw}}(\mathbb{Z}[1/\Sigma], \overline{\rho}^*(1))^{\eta}$ is finite dimensional.




We next remark that 
$((\rho^{\mathfrak{m}}_{\Sigma})^*\widehat{\otimes}_{\mathbb{Z}_p[\{\pm 1\}]}\Lambda)^{\eta}$ is a free $\mathbb{T}_{\overline{\rho}, \Sigma}\widehat{\otimes}_{\mathbb{Z}_p}\Lambda_{\mathcal{O}}^{\eta}$-module of rank one since $(\rho^{\mathfrak{m}}_{\Sigma})^{*, \pm}$ is free of rank one over $\mathbb{T}_{\overline{\rho}, \Sigma}$. Concretely, if we fix a $\mathbb{T}_{\overline{\rho}, \Sigma}$-base
$e_{\pm}$ of $(\rho^{\mathfrak{m}}_{\Sigma})^{*, \pm}$, then $e_{\pm}\otimes e_{\eta}$ is a $\mathbb{T}_{\overline{\rho}, \Sigma}\widehat{\otimes}_{\mathbb{Z}_p}\Lambda_{\mathcal{O}}^{\eta}$-base of $((\rho^{\mathfrak{m}}_{\Sigma})^*\widehat{\otimes}_{\mathbb{Z}_p[\{\pm 1\}]}\Lambda)^{\eta}$, where we set
$e_{\eta}:=\frac{1}{p-1}\sum_{\sigma\in \Delta}\eta(\sigma)^{-1}[\sigma]\in \Lambda$ and we choose $e_{\pm}$ such that 
$\eta(\sigma_{-1})=\pm 1$. We write $e_{\pm}(x)\in (\rho_x^*)^{\pm}$ for any $x\in \mathrm{Spec}(\mathbb{T}_{\overline{\rho}, \Sigma})(\mathcal{O})\subset \mathfrak{X}(\overline{\rho})$, and 
$\overline{e}_{\pm}\in (\overline{\rho}^*)^{\pm}$ to denote the base changes of $e_{\pm}$. To simplify the notation, we assume that such a point $x$ is in $\mathrm{Spec}(\mathbb{T}_{\overline{\rho}, \Sigma})(\mathcal{O})$ by freely extending the scalar $E$ suitably. Since taking $\pm$-parts commutes with any base change of coefficients (remark $p$ is odd), then $e_{\pm}(x)\otimes e_{\eta}$ (resp. $\overline{e}_{\pm}\otimes e_{\eta}$) is a $\Lambda^{\eta}_{\mathcal{O}}$-base (resp. $\Lambda^{\eta}_{\mathbb{F}}$-base) of $(\rho_x^*\otimes_{\mathbb{Z}_p[\{\pm 1\}]}\Lambda)^{\eta}$ (resp. $(\overline{\rho}^*\otimes_{\mathbb{Z}_p[\{\pm 1\}]}\Lambda)^{\eta}$). Therefore, if we write 
$$z_{\Sigma}(\rho_x)^{\eta}:=\bold{z}_{\Sigma, 1}(\rho_x)(e_{\pm}(x)\otimes e_{\eta}), \ \ \overline{z}_{\Sigma}^{\eta}:=\bold{z}_{\Sigma, 1}(\overline{\rho})(\overline{e}_{\pm}(x)\otimes e_{\eta})$$
and 
$$z(f)^{\eta}:=\bold{z}_1(f)(e_{\pm}(x_f)\otimes e_{\eta}), \quad z_{\Sigma}(f)^{\eta}:=\bold{z}_{\Sigma, 1}(f)(e_{\pm}(x_f)\otimes e_{\eta})$$ then one has 
 \begin{equation}\mathrm{Im}(\bold{z}_{\Sigma, 1}(\rho_x)^{\eta})=\Lambda_{\mathcal{O}}\cdot z_{\Sigma}(\rho_x)^{\eta}, \quad \quad \mathrm{Im}(\bold{z}_{\Sigma, 1}(\overline{\rho})^{\eta})=\Lambda_{\mathbb{F}}\cdot \overline{z}_{\Sigma}^{\eta}\end{equation}
and  \begin{equation}\mathrm{Im}(\bold{z}_{1}(f)^{\eta})=\Lambda_{\mathcal{O}}\cdot z(f)^{\eta}, \quad 
\mathrm{Im}(\bold{z}_{\Sigma, 1}(f)^{\eta})=\Lambda_{\mathcal{O}}\cdot z_{\Sigma}(f)^{\eta}.\end{equation}

By these facts, one has 
\begin{equation}
\lambda(\mathrm{Im}(\bold{z}_{\Sigma', 1}(\rho_x)^{\eta})/\mathrm{Im}(\bold{z}_{\Sigma, 1}(\rho_x)^{\eta}))
=\sum_{l\in \Sigma'\setminus \Sigma}\lambda(P_{\rho_x, l}(\sigma_l^{-1})^{\eta})
\end{equation}
and 
\begin{equation}
\mu(\mathrm{Im}(\bold{z}_{\Sigma', 1}(\rho_x)^{\eta})/\mathrm{Im}(\bold{z}_{\Sigma, 1}(\rho_x)^{\eta}))=0
\end{equation}
for any $\Sigma\subset \Sigma'$ since $H^1_{\mathrm{Iw}}(\mathbb{Z}[1/\Sigma], \rho_x^*(1))$ is a free $\Lambda_{\mathcal{O}}$-module and one has 
$$\bold{z}_{\Sigma', 1}(\rho_x)=\prod_{l\in \Sigma'\setminus\Sigma}P_{\rho_x,l}(\sigma_l^{-1})\bold{z}_{\Sigma, 1}(\rho_x).$$ 
Similarly, one also has 
\begin{equation}
\lambda(\mathrm{Im}(\bold{z}_{\Sigma, 1}(f)^{\eta})/\mathrm{Im}(\bold{z}_{1}(f)^{\eta}))
=\sum_{l\in \Sigma_0}\lambda(P_{f, l}(\sigma_l^{-1})^{\eta}),
\end{equation}
\begin{equation}
\mu(\mathrm{Im}(\bold{z}_{\Sigma, 1}(f)^{\eta})/\mathrm{Im}(\bold{z}_{1}(f)^{\eta}))=0,
\end{equation}
and 
\begin{equation}
\lambda_{\mathbb{F}}(\mathrm{Im}(\bold{z}_{\Sigma', 1}(\overline{\rho})^{\eta})/\mathrm{Im}(\bold{z}_{\Sigma, 1}(\overline{\rho})^{\eta}))
=\sum_{l\in \Sigma'\setminus \Sigma}\lambda_{\mathbb{F}}(P_{\overline{\rho}, l}(\sigma_l^{-1})^{\eta}). 
\end{equation}

By these arguments, the validity of Conjecture \ref{conjx} for $\rho_x$ is independent of $\Sigma$, and 
this conjecture for $\rho_{x_f}$ is equivalent to the $\eta$-parts of Kato's conjecture for $f$. 
Moreover, to show the theorems above, it suffices to show the similar statements in the theorems 
for $\bold{z}_{\Sigma, 1}(\overline{\rho})^{\eta}$, $\bold{z}_{\Sigma, 1}(f)^{\eta}$ and $\bold{z}_{\Sigma, 1}(\rho_x)^{\eta}$ 
for a sufficiently
 large $\Sigma$ instead of 
$\bold{z}_{\Sigma_{\overline{\rho}}, 1}(\overline{\rho})^{\eta}$ and $\bold{z}_1(f)^{\eta}$.

 For any $x\in \mathrm{Spec}(\mathbb{T}_{\overline{\rho}, \Sigma}
(\mathcal{O}))$, one has a canonical short exact sequence
\begin{equation}
0\rightarrow H^1_{\mathrm{Iw}}(\mathbb{Z}[1/\Sigma], \rho_x^*(1))^{\eta}/\varpi \rightarrow H^1_{\mathrm{Iw}}(\mathbb{Z}[1/\Sigma], \overline{\rho}^*(1))^{\eta}
\rightarrow H^2_{\mathrm{Iw}}(\mathbb{Z}[1/\Sigma], \rho_x^*(1))^{\eta}[\varpi]\rightarrow 0.
\end{equation}
By this exact sequence, one can identify the element 
$$z_{\Sigma}(\rho_x)^{\eta} \bmod \varpi\in H^1_{\mathrm{Iw}}(\mathbb{Z}[1/\Sigma], \rho_x^*(1))^{\eta}/\varpi$$
 with the element 
 $\overline{z}_{\Sigma}(\overline{\rho})^{\eta}\in H^1_{\mathrm{Iw}}(\mathbb{Z}[1/\Sigma], \overline{\rho}^*(1))^{\eta},$ 
 and one also obtains the following short exact sequence
 \begin{multline}\label{ggg}
 0\rightarrow (H^1_{\mathrm{Iw}}(\mathbb{Z}[1/\Sigma], \rho_x^*(1))^{\eta}/\Lambda_{\mathcal{O}}^{\eta}\cdot z_{\Sigma}(\rho_x)^{\eta})/\varpi \rightarrow 
 H^1_{\mathrm{Iw}}(\mathbb{Z}[1/\Sigma], \overline{\rho}^*(1))^{\eta}/\Lambda_{\mathbb{F}}^{\eta}\cdot \overline{z}_{\Sigma}(\overline{\rho})^{\eta}\\
\rightarrow H^2_{\mathrm{Iw}}(\mathbb{Z}[1/\Sigma], \rho_x^*(1))^{\eta}[\varpi]\rightarrow 0.
\end{multline}

 Therefore, the condition that the map $\bold{z}_{\Sigma, 1}(\overline{\rho})^{\eta}$ is non zero is equivalent to the condition that 
 $z_{\Sigma}(\rho_x)^{\eta} \bmod \varpi\not=0$. If $x=x_f$ is an element in $\mathfrak{X}(\overline{\rho})^{\mathrm{mod}}$, 
 the latter condition is also equivalent to the condition that 
 $$\mu\left(H^1_{\mathrm{Iw}}(\mathbb{Z}[1/\Sigma], \rho_f^*(1))^{\eta}/\Lambda_{\mathcal{O}}^{\eta}\cdot z_{\Sigma}(f)^{\eta}\right)=0$$ 
 since $H^1_{\mathrm{Iw}}(\mathbb{Z}[1/\Sigma], \rho_f^*(1))^{\eta}$ is a free $\Lambda_{\mathcal{O}}^{\eta}$-module of rank one. 
 Theorem \ref{thmA} follows from these equivalences. 
 
From now on, we assume that $\mathfrak{X}(\overline{\rho})_a^{\mathrm{mod}}$ is not empty and 
$\overline{z}_{\Sigma}(\overline{\rho})^{\eta}\not=0$. Take a point $x_f\in \mathfrak{X}(\overline{\rho})_a^{\mathrm{mod}}$. Then, 
one has $$\mu\left(H^1_{\mathrm{Iw}}(\mathbb{Z}[1/\Sigma], \rho_f^*(1))^{\eta}/\Lambda_{\mathcal{O}}^{\eta}\cdot z_{\Sigma}(f)^{\eta}\right)=0$$ 
by the equivalence in the previous paragraph. Then, one also 
has 
$$\mu(H^2_{\mathrm{Iw}}(\mathbb{Z}[1/\Sigma], \rho_f^*(1))^{\eta})=0$$ by 
Theorem \ref{Kato}. Therefore, 
$$(H^1_{\mathrm{Iw}}(\mathbb{Z}[1/\Sigma], \rho_f^*(1))^{\eta}/\Lambda_{\mathcal{O}}^{\eta}\cdot z_{\Sigma}(f)^{\eta})/\varpi,\ \  
H^2_{\mathrm{Iw}}(\mathbb{Z}[1/\Sigma], \rho_f^*(1))[\varpi],$$
 and 
 $$H^2_{\mathrm{Iw}}(\mathbb{Z}[1/\Sigma], \rho_f^*(1))/\varpi$$ are all
finite dimensional $\mathbb{F}$-vector spaces. Then, the short exact sequence (\ref{ggg}) for $\rho_f$ implies that 
$H^1_{\mathrm{Iw}}(\mathbb{Z}[1/\Sigma], \overline{\rho}^*(1))^{\eta}/\Lambda^{\eta}_{\mathbb{F}}\cdot \overline{z}_{\Sigma}(\overline{\rho})^{\eta}$ is also finite dimensional.
$H^2_{\mathrm{Iw}}(\mathbb{Z}[1/\Sigma], \overline{\rho}^*(1))^{\eta}$ is also finite dimensional since one has 
$$H^2_{\mathrm{Iw}}(\mathbb{Z}[1/\Sigma], \rho_f^*(1))^{\eta}/\varpi \isom H^2_{\mathrm{Iw}}(\mathbb{Z}[1/\Sigma], \overline{\rho}^*(1))^{\eta}.$$
Since $H^1_{\mathrm{Iw}}(\mathbb{Z}[1/\Sigma], \overline{\rho}^*(1))^{\eta}$ is a finite free $\Lambda_{\mathbb{F}}$-module, the finite dimensionality of 
the quotient $H^1_{\mathrm{Iw}}(\mathbb{Z}[1/\Sigma], \overline{\rho}^*(1))^{\eta}/\Lambda^{\eta}_{\mathbb{F}}\cdot \overline{z}_{\Sigma}(\overline{\rho})^{\eta}$ 
implies that $H^1_{\mathrm{Iw}}(\mathbb{Z}[1/\Sigma], \overline{\rho}^*(1))^{\eta}$ is of rank one, which shows Theorem \ref{thmB}. 

We now take any point $x\in \mathrm{Spec}(\mathbb{T}_{\overline{\rho}, \Sigma})(\mathcal{O})\subset \mathfrak{X}(\overline{\rho})$. Since 
 the quotient $$H^1_{\mathrm{Iw}}(\mathbb{Z}[1/\Sigma], \overline{\rho}^*(1))^{\eta}/\Lambda^{\eta}_{\mathbb{F}}\cdot \overline{z}_{\Sigma}(\overline{\rho})^{\eta}$$ is finite dimensional, both 
 $$(H^1_{\mathrm{Iw}}(\mathbb{Z}[1/\Sigma], \rho_x^*(1))^{\eta}/\Lambda_{\mathcal{O}}^{\eta}\cdot z_{\Sigma}(\rho_x)^{\eta})/\varpi$$ and 
 $$H^2_{\mathrm{Iw}}(\mathbb{Z}[1/\Sigma], \rho_x^*(1))^{\eta}[\varpi]$$ are also finite dimensional by the short exact sequence (\ref{ggg}). 
 The former finite dimensionality implies that $H^1_{\mathrm{Iw}}(\mathbb{Z}[1/\Sigma], \rho_x^*(1))^{\eta}/\varpi$ is a free $\Lambda^{\eta}_{\mathbb{F}}$-module of rank one, hence $H^1_{\mathrm{Iw}}(\mathbb{Z}[1/\Sigma], \rho_x^*(1))^{\eta}$ is a free $\Lambda_{\mathcal{O}}^{\eta}$-module of rank one, and 
 the quotient $H^1_{\mathrm{Iw}}(\mathbb{Z}[1/\Sigma], \rho_x^*(1))^{\eta}/\Lambda_{\mathcal{O}}^{\eta}z_{\Sigma}(\rho_x)^{\eta}$ is a torsion $\Lambda_{\mathcal{O}}^{\eta}$-module. Then, 
 the Euler-Poincar\'e characteristic formula implies that $H^2_{\mathrm{Iw}}(\mathbb{Z}[1/\Sigma], \rho_x^*(1))^{\eta}$ is a torsion 
 $\Lambda_{\mathcal{O}}^{\eta}$-module. Then, the equality $\mu\left(H^1_{\mathrm{Iw}}(\mathbb{Z}[1/\Sigma], \rho_x^*(1))^{\eta}/\Lambda_{\mathcal{O}}^{\eta}\cdot z_{\Sigma}(\rho_x)^{\eta}\right)=0$ follows from the proof of Theorem \ref{thmA} above, and one also has the equality 
 $\mu(H^2_{\mathrm{Iw}}(\mathbb{Z}[1/\Sigma], \rho_x^*(1))^{\eta})=0$ since $H^2_{\mathrm{Iw}}(\mathbb{Z}[1/\Sigma], \rho_x^*(1))^{\eta}[\varpi]$
  is finite dimensional. 
  
  We next show the condition (4). If $x$ is in $\mathfrak{X}(\overline{\rho})^{\mathrm{mod}}$, then 
 the condition (4) follows from Theorem 12.4 (3) of \cite{Ka04} and the vanishing of the $\mu$-invariants, and the assumption (3) in the beginning of \S 4 for $\varepsilon^{\pm}
  =\varepsilon^{-1}$.  When $x$ satisfies the condition (b), we define, for each $n\geqq 1$ such that $(n,\Sigma)=1$, 
  $$z_{\Sigma, n}(\rho_x)^{\eta}:=\bold{z}_{\Sigma, n}(\rho_x)(e_{\pm}(x)\otimes e_{\eta})\in H^1_{\mathrm{Iw}}
  (\mathbb{Z}[1/\Sigma_n, \zeta_n], \rho_x^*(1))^{\eta}.$$
  By the Euler system relation of $\bold{z}_{\Sigma, n}(\rho^{\mathfrak{m}})$, the set of these elements $z_{\Sigma, n}(\rho_x)^{\eta}$ forms an Euler system whose 
  first layer is $z_{\Sigma}(\rho_x)^{\eta}\not=0$ by Theorem \ref{thmA}. Therefore, the general theory of Euler system (to apply it, we need the assumption (b)) and the vanishing of 
  $\mu$-invariants which we prove above shows there exists an inclusion 
  $$\mathrm{Char}_{\Lambda^{\eta}_{\mathcal{O}}}(H^1_{\mathrm{Iw}}(\mathbb{Z}[1/\Sigma], \rho_x^*(1))^{\eta}
  /\mathrm{Im}(\bold{z}_{\Sigma,1}(\rho_x)^{\eta}))\subset
    \mathrm{Char}_{\Lambda^{\eta}_{\mathcal{O}}}(H^2_{\mathrm{Iw}}(\mathbb{Z}[1/\Sigma], \rho_x^*(1))^{\eta}),$$
  which shows Theorem \ref{thmC}. 
  
  We finally prove Theorem \ref{thmD}. For this, we first note that, for any finite generated torsion $\Lambda_{\mathcal{O}}^{\eta}$-module $M$ with $\mu(M)=0$, 
  one has 
 \begin{equation}\lambda(M)=\lambda_{\mathbb{F}}(M/\varpi)-\lambda_{\mathbb{F}}(M[\varpi]).
 \end{equation} 
 We apply this formula for $$H^1_{\mathrm{Iw}}(\mathbb{Z}[1/\Sigma], \rho_x^*(1))^{\eta}/\Lambda_{\mathcal{O}}^{\eta}\cdot z_{\Sigma}(\rho_x)^{\eta}$$
  and $$H^2_{\mathrm{Iw}}(\mathbb{Z}[1/\Sigma], \rho_x^*(1))^{\eta},$$ then one obtains
  $$\lambda\left(H^1_{\mathrm{Iw}}(\mathbb{Z}[1/\Sigma], \rho_x^*(1))^{\eta}/\Lambda_{\mathcal{O}}^{\eta}\cdot z_{\Sigma}(\rho_x)^{\eta}\right)=\lambda_{\mathbb{F}}(\left(H^1_{\mathrm{Iw}}(\mathbb{Z}[1/\Sigma], \rho_x^*(1))^{\eta}/\Lambda_{\mathcal{O}}^{\eta}\cdot z_{\Sigma}(\rho_x)^{\eta})/\varpi\right)$$
  since one has $(H^1_{\mathrm{Iw}}(\mathbb{Z}[1/\Sigma], \rho_x^*(1))^{\eta}/\Lambda_{\mathcal{O}}^{\eta}\cdot z_{\Sigma}(\rho_x)^{\eta})[\varpi]=0$, 
  and 
  $$\lambda(H^2_{\mathrm{Iw}}(\mathbb{Z}[1/\Sigma], \rho_x^*(1))^{\eta})=\lambda_{\mathbb{F}}(H^2_{\mathrm{Iw}}(\mathbb{Z}[1/\Sigma], \overline{\rho}^*(1))^{\eta})
  -\lambda_{\mathbb{F}}(H^2_{\mathrm{Iw}}(\mathbb{Z}[1/\Sigma], \rho_x^*(1))^{\eta}[\varpi])$$
  since one has $H^2_{\mathrm{Iw}}(\mathbb{Z}[1/\Sigma], \rho_x^*(1))^{\eta}/\varpi\isom H^2_{\mathrm{Iw}}(\mathbb{Z}[1/\Sigma], \overline{\rho}^*(1))^{\eta}$. 
  These equalities and the short exact sequence (\ref{ggg}) induces the equality 
  \begin{multline*}
  \lambda\left(H^1_{\mathrm{Iw}}(\mathbb{Z}[1/\Sigma], \rho_x^*(1))^{\eta}/\Lambda_{\mathcal{O}}^{\eta}\cdot z_{\Sigma}(\rho_x)^{\eta}\right)-
  \lambda(H^2_{\mathrm{Iw}}(\mathbb{Z}[1/\Sigma], \rho_x^*(1))^{\eta})\\
  =\lambda_{\mathbb{F}}\left(H^1_{\mathrm{Iw}}(\mathbb{Z}[1/\Sigma], \overline{\rho}^*(1))^{\eta}/\Lambda_{\mathbb{F}}\overline{z}_{\Sigma}(\overline{\rho})^{\eta}\right)-\lambda_{\mathbb{F}}(H^1_{\mathrm{Iw}}(\mathbb{Z}[1/\Sigma], \overline{\rho}^*(1))^{\eta}),
  \end{multline*}
  which shows Theorem \ref{thmD}.


\end{proof}

\addcontentsline{toc}{section}{Appendix}
\setcounter{section}{0}
\renewcommand{\thesection}{\Alph{section}}
 \setcounter{equation}{0} 
\renewcommand{\theequation}{\Alph{section}.\arabic{equation}}

  \section{Kato's zeta morphisms associated to Hecke eigen cusp new forms}
    
  For each integer $N\geqq 1$, we set 
 $$\Gamma_1(N):=\left\{\left. g=\begin{pmatrix}a & b \\ c & d\end{pmatrix}\in\mathrm{SL}_2(\mathbb{Z}) \right| g\equiv \begin{pmatrix} *& *\\
 0 & 1\end{pmatrix} (\bmod N)\right\}.$$ 
 For $k\in \mathbb{Z}_{\geqq 2},$ 
  we write $S^{\mathrm{new}}_k(\Gamma_1(N))$ to denote the 
  $\mathbb{C}$-vector space of cusp newforms with level $\Gamma_1(N)$ and weight $k$. Let
  $f(\tau)=\sum_{n=1}^{\infty}a_nq^n\in S^{\mathrm{new}}_k(\Gamma_1(N_f))$ ($\tau\in \mathcal{H}^+, q=\mathrm{exp}(2\pi i \tau)$) be a normalized Hecke eigen cusp newform of level $N_f$. 
  We regard the field $\mathbb{Q}(\{a_n\}_{n\geqq 0})\subset \mathbb{C}$ as a subfield of $\overline{\mathbb{Q}}$ by the fixed inclusion 
  $\iota_{\infty} : \overline{\mathbb{Q}}\hookrightarrow \mathbb{C}$. We take a number field
  $F\subset \overline{\mathbb{Q}}$ containing $\mathbb{Q}(\{a_n\}_{n\geqq 0})$, and write $E\subset \overline{\mathbb{Q}}_p$ to denote the closure of 
  the image of $F$ by the fixed inclusion $\iota_p : \overline{\mathbb{Q}}\hookrightarrow \overline{\mathbb{Q}}_p$. We set $\Sigma_f:=\mathrm{prime}(N_f)\cup\{p\}$.  
  Let $\rho_f : G_{\mathbb{Q}}\rightarrow \mathrm{GL}_2(E)$ be a continuous representation associated to $f$ defined by Shimura \cite{Sh71} for $k=2$ and Deligne \cite{De71} for $k>2$, i.e. 
  a representation satisfying 
  $$\mathrm{trace}(\rho_f(\mathrm{Frob}_l))=a_l\ \ \text{ and } \ \ \mathrm{det}(\rho_f(\mathrm{Frob}_l))=l^{k-1} \varepsilon_f(l)$$ 
  for any prime $l\not\in \Sigma_f$. One can take such a $\rho_f$ satisfying the additional property $\mathrm{Im}(\rho_f)\subset \mathrm{GL}_2(\mathcal{O})$. 
  For such $\rho_f$, we write $\overline{\rho}_f : G_{\mathbb{Q}}\rightarrow \mathrm{GL}_2(\mathbb{F})$ to denote the 
  semi-simplification of $\rho_f\otimes_{\mathcal{O}}\mathbb{F}$. We remark that the isomorphism class of $\overline{\rho}_f$ is independent of the choice 
  of such $\rho_f$, and $\overline{\rho}_f\isom \rho_f\otimes_{\mathcal{O}}\mathbb{F}$ for arbitrary such $\rho_f$ if $\overline{\rho}_f$ is absolutely irreducible.
 
 We also remark that $\rho_f|_{G_{\mathbb{Q}_p}}$ is de Rham with Hodge-Tate weights $\{-(k-1), 0\}$ (\cite{Sc90}). 
   For each integers $n\geqq 1$ such that $(n,\Sigma_f)=1$, $m\in \mathbb{Z}_{\geqq 0}$ and $r\in \mathbb{Z}$, we write  
 $$\mathrm{exp}^*_{m, r} : H^1_{\mathrm{Iw}}(\mathbb{Z}[1/\Sigma_{f,n}, \zeta_n], \rho_f)\rightarrow \mathrm{Fil}^rD_{\mathrm{dR}}(\rho_f)\otimes_{\mathbb{Q}}\mathbb{Q}(\zeta_{np^m})$$
 to denote  the following composite : 
 \begin{multline*}
 H^1_{\mathrm{Iw}}(\mathbb{Z}[1/\Sigma_{f,n}, \zeta_n], \rho_f)\xrightarrow{(a)}
 H^1_{\mathrm{Iw}}(\mathbb{Z}[1/\Sigma_{f,n}, \zeta_n], \rho_f(r))\\
 \xrightarrow{\mathrm{can}}H^1(\mathbb{Z}[1/\Sigma_{f,n}, \zeta_{np^m}], \rho_f(r))
 \xrightarrow{\mathrm{loc}_p} H^1(\mathbb{Q}_p\otimes_{\mathbb{Q}}\mathbb{Q}(\zeta_{np^m}), \rho_f(r))\\
 \xrightarrow{\mathrm{exp}^*} \mathrm{Fil}^0D_{\mathrm{dR}}(\rho_f(r)|_{G_{\mathbb{Q}_p\otimes_{\mathbb{Q}}\mathbb{Q}(\zeta_{np^m})}})
 \isom\mathrm{Fil}^rD_{\mathrm{dR}}(\rho_f)\otimes_{\mathbb{Q}}\mathbb{Q}(\zeta_{np^m}),
 \end{multline*}
 where the map $(a)$ is the twist by $((\zeta_{p^k})_{k\geqq 1})^{\otimes r}$ in Iwasawa cohomology.
 
 
 To characterize Kato's zeta morphism associated to $f$, we need not only $\rho_f$ but also a motivic structure on it. 
 In \S 6.3 and \S8.3 of  \cite{Ka04}, Kato defined $\rho_f$ as the maximal quotient $V_1(f)_E$ of $H^1(Y_1(N_f), \mathcal{V}_{k/E})$ 
 on which $T_l$ (resp. $S_l$) acts by $a_l$ (resp. $l^{k-2}\varepsilon_f(l)$) for any prime $l\not\in \Sigma_f$. 
 $V_1(f)_E$ has a $\mathrm{Gal}(\mathbb{C}/\mathbb{R})$-stable $F$-lattice $V_1(f)_F$ which is 
 the similar quotient of $H^1(Y_1(N_f), \mathcal{V}_{k/F})$. If we write $S_1(f)_F$ to denote the similar quotient of 
 $S_k(\Gamma_1(N_f))_F$, then one has a canonical $E$-linear isomorphism 
 $$S_1(f)_E:=S_1(f)_F\otimes_FE\isom \mathrm{Fil}^rD_{\mathrm{dR}}(\rho_f)$$
 for any $1\leqq r\leqq k-1$ by the comparison theorem of $p$-adic Hodge theory, by which we regard the map $\mathrm{exp}^*_{m, r}$ as the following map 
 $$\mathrm{exp}^*_{m, r} : H^1_{\mathrm{Iw}}(\mathbb{Z}[1/\Sigma_{f,n}, \zeta_n], V_1(f)_E)\rightarrow S_1(f)_E\otimes_{\mathbb{Q}}\mathbb{Q}(\zeta_{np^m})$$ 
 for every $1\leqq r\leqq k-1$. Moreover, the period map 
  $$\mathrm{per} : S_k(\Gamma_1(N_f))_F\hookrightarrow H^1(Y_1(N_f)(\mathbb{C}), \mathcal{V}_{k/\mathbb{C}})$$ induces an $F$-linear injection 
  $$\mathrm{per} : S_1(f)_F\rightarrow V_1(f)_{\mathbb{C}}:=V_1(f)_F\otimes_F\mathbb{C}.$$
   For each integers $n\geqq 1$ and $m\geqq 0$, and a finite set of primes $\Sigma'$ containing $p$  such that $(n, \Sigma')=1$,  and a character $\chi : (\mathbb{Z}/np^m)^{\times}\rightarrow \mathbb{C}^{\times}$, we set 
  $$L_{\Sigma', n}(f, \chi, s):=\sum_{k=1, (k, \Sigma'_n)=1}^{\infty}\frac{\chi(k)a_k(f)}{k^s}$$
  which is known to absolutely converge in $\mathrm{Re}(s)>\frac{k+1}{2}$, and analytically continued to the whole $\mathbb{C}$ (we mainly apply this definition for $\Sigma'=\{p\}$ or $\Sigma'=\Sigma$).

   We write $f^*(\tau):=\sum_{n=1}^{\infty}\overline{a_n}q^n\in S^{\mathrm{new}}_k(\Gamma_1(N_f))$ to denote the complex conjugation of $f$, which is known to be 
  a Hecke eigen newform.   In Theorem 12.4 of \cite{Ka04}, he (essentially) proves the following.
  \begin{thm}\label{4.2}
  For each integer $n\geqq1$ such that $(n, \Sigma_f)=1$, there is a unique $E$-linear map 
  $$\bold{z}_{n}^{Ka}(f) : V_1(f)_E\rightarrow H^1_{\mathrm{Iw}}(\mathbb{Z}[1/\Sigma_{f,n}, \zeta_n], V_1(f)_E)$$
  satisfying the following property : for each $m\geqq 0$ and $\gamma\in V_1(f)_F\subset V_1(f)_E$, the image $\omega_{\gamma, m}$ of $\bold{z}_{n}^{Ka}(f)(\gamma)$ 
  under the map 
  $$\mathrm{exp}^*_{m,1} : H^1_{\mathrm{Iw}}(\mathbb{Z}[1/\Sigma_{f,n}, \zeta_n], V_1(f)_E)\rightarrow S_1(f)_E\otimes_{\mathbb{Q}}\mathbb{Q}(\zeta_{np^m})$$
  belongs to $S_1(f)_F\otimes_{\mathbb{Q}}\mathbb{Q}(\zeta_{np^m})$, and the map 
  $$S_1(f)_F\otimes_{\mathbb{Q}}\mathbb{Q}(\zeta_{np^m})\rightarrow V_1(f)_{\mathbb{C}}^{\pm} : v\otimes x\mapsto \sum_{\sigma\in \mathrm{Gal}(\mathbb{Q}(\zeta_{np^m})/\mathbb{Q})}\chi(\sigma)\sigma(x)\mathrm{per}(v)^{\pm},$$where 
  $\chi : \mathrm{Gal}(\mathbb{Q}(\zeta_{np^m})/\mathbb{Q})\rightarrow \mathbb{C}^{\times}$ is any character and $\pm 1=\chi(-1)$, sends the element 
  $\omega_{\gamma, m}$ to 
  $$L_{\{p\}, n}(f^*,\chi, k-1)\cdot \gamma^{\pm}.$$
    \end{thm}
    \begin{proof}
    The uniqueness of the map $\bold{z}_{n}^{Ka}(f)$ can be proved in the same way as the proof of Corollary \ref{2.4} since one can prove that 
    the map 
    $$\mathrm{exp}^*:=(\mathrm{exp}^*_{m,1})_{m\geqq 1} : H^1_{\mathrm{Iw}}(\mathbb{Z}[1/\Sigma_{f,n}, \zeta_n], V_1(f)_E)\rightarrow \varprojlim_{m\geqq 1}S_1(f)_E\otimes_{\mathbb{Q}}\mathbb{Q}(
    \zeta_{np^m})$$ is an injection in the same way as in (2) of Lemma \ref{2.2}. 
    
    For the existence of the map $\bold{z}_{n}^{Ka}(f)$, Kato defined such a map 
    for $n=1$ in Theorem 12.4 of \cite{Ka04}. 
    
    For general $n\geqq 1$ such that $(n, \Sigma_f)=1$, we can similarly define the map 
    $$\bold{z}_n^{Ka}(f) : V_1(f)_E\rightarrow H^1_{\mathrm{Iw}}(\mathbb{Z}[1/\Sigma_{f,n}, \zeta_n], V_1(f)_E)\otimes_{\Lambda_n}Q(\Lambda_n)$$
    as follows. We first recall that Kato defined the element $\bold{z}_{1}^{Ka}(f)(\gamma)$ which he writes $\bold{z}^{(p)}_{\gamma}$ in \cite{Ka04} for 
    $\gamma\in  V_1(f)_F$ using the elements 
    $$\mu_1(c,d,j):=(c^2-c^{k+1-j} \sigma_c)(d^2-d^{j+1} \sigma_d) \prod_{l\in \Sigma_{f}\setminus\{p\}}(1-\overline{a_l}l^{-k}\sigma_l^{-1})\in \Lambda=\Lambda_1$$ 
    ($\sigma_c, \sigma_d\in \mathrm{Gal}(\mathbb{Q}(\zeta_{p^{\infty}})/\mathbb{Q})$) and ${}_{c,d}\bold{z}^{(p)}_{p^m}(f,k,j,\alpha, \mathrm{prime}(pN_f))\in 
    H^1(\mathbb{Z}[1/\Sigma_{f}, \zeta_{p^m}], 
    V_1(f)_E)$ defined for every $c,d\in \mathbb{Z}$ such that $(cd, 6p)=1$, $c\equiv d\equiv 1$ mod $N_f$, $c^2, d^2\not=1$, and 
    $1\leqq j\leqq k-1$, $\alpha\in \mathrm{SL}_2(\mathbb{Z})$ (see \S13.9 of \cite{Ka04}). As a direct generalization of this definition, we define the element 
    $$\bold{z}_n^{Ka}(f)(\gamma)\in H^1_{\mathrm{Iw}}(\mathbb{Z}[1/\Sigma_{f,n}, \zeta_n], V_1(f)_E)\otimes_{\Lambda_n}Q(\Lambda_n)$$
    for each $\gamma\in V_1(f)_F$ 
    using the elements
    $$\mu_n(c,d,j):=(c^2-c^{k+1-j}\sigma_c)(d-d^{j+1}\sigma_d)\prod_{l\in \Sigma_{f}\setminus\{p\}}(1-\overline{a_l}l^{-k}\sigma_l^{-1})\in \Lambda_n$$
    ($\sigma_c, \sigma_d\in \mathrm{Gal}(\mathbb{Q}(\zeta_{np^{\infty}})/\mathbb{Q})$) instead of $\mu_1(c,d,j)$, 
    and $${}_{c,d}\bold{z}^{(p)}_{np^m}(f,k,j,\alpha, \mathrm{prime}(pnN_f))\in 
    H^1(\mathbb{Z}[1/\Sigma_{f,n}, \zeta_{np^m}], 
    V_1(f)_E)$$ instead of using ${}_{c,d}\bold{z}^{(p)}_{p^m}(f,k,j,\alpha, \mathrm{prime}(pN_f))$, which are defined for every $c,d\in \mathbb{Z}$ such that $(cd, 6np)=1$, $c\equiv d\equiv 1$ mod $N_f$, $c^2, d^2\not=1$. Then, in a similar way as in \S13.12 of \cite{Ka04}, 
    we can show that the element $\bold{z}_n^{Ka}(f)(\gamma)$ is in 
    $H^1_{\mathrm{Iw}}(\mathbb{Z}[1/\Sigma_{f,n}, \zeta_n], V_1(f)_E)$, and 
    the $E$-linearization 
    $$\bold{z}_n^{Ka}(f) : V_1(f)_E\rightarrow H^1_{\mathrm{Iw}}(\mathbb{Z}[1/\Sigma_{f,n}, \zeta_n], V_1(f)_E) $$
    of the map $\gamma\in V_1(f)_F\mapsto \bold{z}_n^{Ka}(f)(\gamma)$ satisfies the desired interpolation property.

    \end{proof}
    \begin{rem}\label{integer}
   If $p$ is odd, and $\overline{\rho}_f : G_{\mathbb{Q}}\rightarrow \mathrm{GL}_2(\mathbb{F})$ is absolutely irreducible, then one can furthermore shows that 
    the map $\bold{z}_n^{Ka}(f)$ preserves the integral structure, i.e. if $V_1(f)_{\mathcal{O}}$ is a $G_{\mathbb{Q}}$-stable $\mathcal{O}$-lattice 
    of $V_1(f)_{\mathcal{O}}$, then the map $\bold{z}_n^{Ka}(f)$ also induces 
    $$\bold{z}_n^{Ka}(f) : V_1(f)_{\mathcal{O}}\rightarrow H^1_{\mathrm{Iw}}(\mathbb{Z}[1/\Sigma_{f,n}, \zeta_n], V_1(f)_{\mathcal{O}}),$$
    which is proved by the same argument as in \S13.12 of \cite{Ka04}. We also remark that all $G_{\mathbb{Q}}$-stable $\mathcal{O}$-lattices are of the form $\varpi^nV_1(f)_{\mathcal{O}}$ 
    for some $n\in \mathbb{Z}$ since $\overline{\rho}_f$ is absolutely irreducible. 
    
    \end{rem}
    
     
    For our purpose, it is convenient to define a twisted version of $\bold{z}_n^{Ka}(f)$, which is directly related with our zeta morphism $\bold{z}_{\Sigma, n}(\rho^{\mathfrak{m}})$. 
    For this, we consider the maximal quotient $V'_1(f)_A$ of $H^1(Y(N_f), \mathcal{V}^*_{k/A})(1)$, for $A=F, E, \mathbb{C}$, on which 
    $T'_l$ (resp. $S'_l$) acts by $a_l$ (resp. $l^{k-2}\varepsilon_f(l)$) for any prime $l\not\in \Sigma_f$. 
    Then, the pairing $H^1(Y(N_f), \mathcal{V}^*_{k/A})(1)\times H_c^1(Y(N_f), \mathcal{V}_{k/A})\rightarrow A$ of 
    Poincar\'e duality induces an isomorphism 
    $$V'_1(f)_A\isom V_1(f)_A^*,$$
    by which we identify the both side. 
    
    \begin{lemma}\label{4.3}
    For $A=E, F, \mathbb{C}$, the canonical isomorphism $$H^1(Y(N_f), \mathcal{V}^*_{k/A})(2-k)\isom H^1(Y(N), \mathcal{V}_{k/A})$$ 
    induces an isomorphism 
    $$V'_1(f)_A(1-k)\isom V_1(f^*)_A.$$
    
    \end{lemma}
    \begin{proof}
    By Lemma \ref{2.01}, one has equalities 
    $$S_l=l^{2(k-2)}(S'_l)^{-1}$$ 
    and 
    $$T_l=l^{2-k}S_lT'_l=l^{k-2}(S'_l)^{-1}T'_l$$
    as operators on $H^1(Y(N), \mathcal{V}_{k/A})$. Therefore, the isomorphism $$H^1(Y(N_f), \mathcal{V}^*_{k/A})(2-k)\isom H^1(Y(N), \mathcal{V}_{k/A})$$ 
    induces an isomorphism from $V'_1(f)_A(1-k)$ to the maximal quotient of $H^1(Y(N), \mathcal{V}_{k/A})$ on which 
    $T_l$ (resp. $S_l$) acts by 
    $$l^{k-2}(l^{k-2}\varepsilon_f(l))^{-1}a_l=\varepsilon_f(l)^{-1}a_l\ \  (\text{resp}.  l^{2(k-2)}(l^{k-2}\varepsilon_f(l))^{-1}=l^{k-2}\varepsilon_f(l)^{-1}).$$ 
    Since one has $\overline{a_l}=\varepsilon_f(l)^{-1}a_l$ and $\overline{\varepsilon_f(l)}=\varepsilon_f(l)^{-1}$, this quotient is equal to 
    $V_1(f^*)_A$, which proves the lemma.
    
    \end{proof}
    
    Using this lemma, we define the following map 
      \begin{defn}\label{4.4}For each $n\geqq 1$ such that $(n, \Sigma_f)=1$, 
    we define an $E$-linear map 
    $$\bold{z}_n(f) : V'_1(f)_E\rightarrow 
    H^1_{\mathrm{Iw}}(\mathbb{Z}[1/\Sigma_{f,n}, \zeta_n], V'_1(f)_E(1)) $$
    by the following composite : 
    \begin{multline*}
    V'_1(f)_E\xrightarrow{x\mapsto x\otimes ((\zeta_{p^{m}})_{m\geqq 1})^{\otimes (1-k)}}  V_1(f^*)_E\xrightarrow{\bold{z}_n^{Ka}(f^*)} 
    H^1_{\mathrm{Iw}}(\mathbb{Z}[1/\Sigma_{f,n}, \zeta_n], V_1(f^*)_E)\\
    \xrightarrow{x\mapsto x\otimes ((\zeta_{p^{m}})_{m\geqq 1})^{\otimes k}}
     H^1_{\mathrm{Iw}}(\mathbb{Z}[1/\Sigma_{f,n}, \zeta_n], V'_1(f)_E(1))
    \end{multline*}
    \end{defn}
    By definition, this map is characterized as follows : 
    \begin{corollary}\label{4.5}
    $\bold{z}_n(f)$ is a unique $E$-linear map satisfying  the following property $:$ 
     for each $m\geqq 0$ and $\gamma\in V'_1(f)_F\subset V'_1(f)_E$, the image $\omega_{\gamma, m}$ of $\bold{z}_n(f)(\gamma)$ 
  under the following composite
  \begin{multline*}
  H^1_{\mathrm{Iw}}(\mathbb{Z}[1/\Sigma_{f,n}, \zeta_n], V'_1(f)_E(1))\xrightarrow{x\mapsto x\otimes x\otimes ((\zeta_{p^{m}})_{m\geqq 1})^{\otimes (-k)}} 
    H^1_{\mathrm{Iw}}(\mathbb{Z}[1/\Sigma_{f,n}, \zeta_n], V_1(f^*)_E)\\
    \xrightarrow{\mathrm{exp}^*_{m,1}}S_1(f^*)_E\otimes_{\mathbb{Q}}\mathbb{Q}(\zeta_{np^m})
  \end{multline*}
  belongs to $S_1(f)_F\otimes_{\mathbb{Q}}\mathbb{Q}(\zeta_{np^m})$, and the map 
  $$S_1(f)_F\otimes_{\mathbb{Q}}\mathbb{Q}(\zeta_{np^m})\rightarrow V_1(f)_{\mathbb{C}}^{\pm} : v\otimes x\mapsto \sum_{\sigma\in \mathrm{Gal}(\mathbb{Q}(\zeta_{np^m})/\mathbb{Q})}\chi(\sigma)\sigma(x)\mathrm{per}(v)^{\pm},$$where 
  $\chi : \mathrm{Gal}(\mathbb{Q}(\zeta_{np^m})/\mathbb{Q})\rightarrow \mathbb{C}^{\times}$ is any character and $\pm 1=\chi(-1)$, sends the element 
  $\omega_{\gamma, m}$ to 
  $$L_{\{p\}, n}(f,\chi, k-1)\cdot \gamma^{\pm}.$$

    \end{corollary}
     \begin{rem}\label{pm1}By the construction of the maps $\bold{z}^{Ka}_n(f)$ and $\bold{z}_n(f)$, one can also show that the map $\bold{z}_n(f)$ satisfies the equality 
 $$ \bold{z}_n(f)(\tau(v))=\sigma_{-1} (\bold{z}_n(f)(v))$$
  for every $v\in  V'_1(f)_E$, and, for every prime $l\not\in \Sigma_f$, 
    $$\mathrm{Cor}\circ  \bold{z}_{nl}(f)=\begin{cases} \bold{z}_n(f)& \text{ if } l|n\\ 
  P_l(\sigma_l^{-1}) \bold{z}_n(f)& \text{ if } (l,n)=1\end{cases}$$
  for the corestriction map 
  $$\mathrm{Cor} : H^1_{\mathrm{Iw}}(\mathbb{Z}[1/\Sigma_{f,nl}, \zeta_{nl}], V'_1(f)_E(1))\rightarrow 
  H^1_{\mathrm{Iw}}(\mathbb{Z}[1/\Sigma_{f,nl}, \zeta_n], V'_1(f)_E(1)).$$ One can also prove these properties as an immediate corollary of Theorem \ref{3.2}.

    \end{rem}

   \section{Preliminaries on the $p$-adic local Langlands correspondence for $\mathrm{GL}_2(\mathbb{Q}_l)$}
   In this appendix, we collect some basic notions and facts 
   which we need to state Emerton's refined local-global compatibility.
   To state this compatibility, we need to recall the theories of the 
   $p$-adic local Langlands correspondences in families both for $G_p$  \cite{Co10}, \cite{Pas13} 
   and $G_l$ ($l\not=p)$  \cite{EH14}. 
   To recall these correspondences, we need some notions on topological modules. 
   We first recall these notions in this appendix. In fact, we need the dual version 
   of  the refined local-global compatibility for our application. For this, we also recall the duality theory of 
   these topological modules, and the notions of topological tensor products of these modules. 
   Finally, we also recall the theory of Galois (and Iwasawa) cohomology of some Galois representations whose 
   underlying spaces are these topological modules, 
   
 
 Let $A$ be a commutative ring, $G$ be a topological group. A left $A[G]$-module $V$ is called smooth, or a smooth $A$-representation of $G$ if 
 the stabilizer $\{g\in G \mid gv=v\}$ is open for any $v\in V$. If $A$ is Noetherian, a smooth $A$-representation $V$ of $G$ is called admissible, or a smooth admissible $A$-representation of $G$ if the $K$-fixed part $V^K$ is a finite generated $A$-module for any open subgroup $K$ of $G$.
 
 \subsection{Recall of some notions of topological modules}

    Let $A$ be an object in $\mathrm{Comp}(\mathcal{O})$ with its maximal ideal $\mathfrak{m}$. 
 \begin{defn}\label{5.1} We say that an $A$-module $V$ is orthonormalizable if $V$ is 
$\mathfrak{m}$-adically complete and separated, and if the quotient $V/\mathfrak{m}^iV$ is free over $A/\mathfrak{m}^i$ for 
every $i\geqq 1$. 
\end{defn}\label{5.2}
We remark that orthonormalizable $A$-modules are flat over $A$.
\begin{defn}
We say that a topological $A$-module $M$ is called a compact $A$-module if 
$M$ is compact Hausdorff, and has a fundamental system of open neighborhood of $0$ consisting of 
sub $A$-modules. 
\end{defn}
By definition, one has a canonical isomorphism $M\isom \varprojlim_{N}M/N$ of topological $A$-modules for any 
compact $A$-module $M$, where the limit is taken over all the open sub $A$-modules $N$ of $M$. For compact $A$-modules $M_1$ and $M_2$, one can define 
the tensor product $M_1\widehat{\otimes}_AM_2$ in the category of compact $A$-module by 
$M_1\widehat{\otimes}_AM_2=\varprojlim_{N_1, N_2}(M_1/N_1)\otimes_A(M_2/N_2)$, where 
$N_1$ (resp. $N_2$) runs through all the open sub $A$-modules of $M_1$ (resp. $M_2$). 
\begin{defn}\label{5.3}
We say that a compact $A$-module $M$ is pro-free if it is topologically isomorphic to a direct product of copies of $A$, each factor 
being equipped with its $\mathfrak{m}$-adic topology. 
\end{defn}
For any orthonormalizable $A$-module $V$, we set $V^*:=\mathrm{Hom}_A(V, A)$ and equip it with the pointwise convergence toplogogy. 
By Proposition B. 11 of \cite{Em}, $V^*$ is a pro-free $A$-module, and the functor $V\mapsto V^*$ induces an anti-equivalence of categories 
between the category of orthonormalizable $A$-modules and the category of pro-free $A$-modules. 
We remark that the inverse functor is given by $M\mapsto \mathrm{Hom}_A^{\mathrm{cont}}(M, A)$. 

From here until the end of this subsection, we assume that $A$ is $\mathcal{O}$-torsion free. 
For any $A$-module $M$ and any ideal $I$ of $A$, we set $M[I]:=\{v\in M\mid av=0 \text{ for any } a\in I\}$.

\begin{defn}\label{5.4}
We say that an $A$-module $X$ is cofinitely generated if it satisfies the following conditions :
 \begin{itemize}
 \item[(1)]$X$ is $\varpi$-adically completed and separated.
 \item[(2)]$X$ is $\mathcal{O}$-torsion free.
 \item[(3)]The action map $A\times X\rightarrow X$ induced by the $A$-module structure on $X$ is continuous, when 
 $A$ is given its $\mathfrak{m}$-adic topology, and $X$ is given its $\varpi$-adic topology.
 \item[(4)]$(X/\varpi X)[\mathfrak{m}]$ is finite-dimensional over $\mathbb{F}$. 
 
 \end{itemize}
\end{defn}
We recall that, if $X$ is any $\mathcal{O}$-torsion free, $\varpi$-adically complete and separated $\mathcal{O}$-module, 
then its dual $\mathrm{Hom}_{\mathcal{O}}(X, \mathcal{O})$ has a natural topology, namely the topology of the pointwise convergence, 
and the functor $X\mapsto \mathrm{Hom}_{\mathcal{O}}(X, \mathcal{O})$ induces an anti-equivalence of categories between the category of 
$\mathcal{O}$-torsion free, $\varpi$-adically complete and separated $\mathcal{O}$-modules and the category of 
$\mathcal{O}$-torsion free profinite $\mathcal{O}$-modules, with  a quasi-inverse functor given by 
$M\mapsto \mathrm{Hom}_{\mathcal{O}}^{\mathrm{cont}}(M, \mathcal{O})$. By Proposition C.5 of \cite{Em}, 
this functor $X\mapsto \mathrm{Hom}_{\mathcal{O}}(X, \mathcal{O})$ also induces an anti-equivalence 
of categories between the category of cofinitely generated $A$-modules and the category of $\mathcal{O}$-torsion free finite 
generated $A$-modules.

Let $V$ be an orthonormalizable $A$-module, and $X$ be a cofinitely generated $A$-module. 
We write $V\widehat{\otimes}_AX:=\varprojlim_nV\otimes_A(X/\varpi^nX)$ to denote the 
the $\varpi$-adic completion of $V\otimes_AX$. This is the same notation as that of the tensor product of compact $A$-modules. 
However, the author thinks that readers can easily recognize which tensor product is used in each situation. 
\begin{lemma}\label{5.5}
$V\widehat{\otimes}_AX$ is $\mathcal{O}$-torsion free.
\end{lemma}
\begin{proof}
Since $X$ is $\mathcal{O}$-torsion free, we have the following short exact sequence 
$$0\rightarrow X/\varpi^{n-1}X\xrightarrow{x\mapsto \varpi x}X/\varpi^nX\rightarrow X/\varpi X\rightarrow 0$$
of $A$-modules for any $n\geqq 1$. Since $V$ is flat over $A$, we obtain the short exact sequence 
$$0\rightarrow V\otimes_AX/\varpi^{n-1}X\xrightarrow{x\mapsto \varpi x} V\otimes_AX/\varpi^nX\rightarrow V\otimes_AX/\varpi X\rightarrow 0.$$
Taking the limit, we obtain the short exact sequence
$$0\rightarrow V\widehat{\otimes}_AX\xrightarrow{x\mapsto \varpi x} V\widehat{\otimes}_AX\rightarrow V\otimes_AX/\varpi X\rightarrow 0,$$
hence $V\widehat{\otimes}_AX$ is $\mathcal{O}$-torsion free. 
\end{proof}
By this lemma, the $A$-module $\mathrm{Hom}_{\mathcal{O}}(V\widehat{\otimes}_AX, \mathcal{O})$ becomes an 
$\mathcal{O}$-torsion free profinite $\mathcal{O}$-module. 
 \begin{lemma}\label{5.6}
 Let $V$ be an orthonormalizable $A$-module, and $X$ be a cofinitely generated $A$-module. 
 Then, one has a natural isomorphism 
 $$V^*\otimes_A \mathrm{Hom}_{\mathcal{O}}(X, \mathcal{O})\isom \mathrm{Hom}_{\mathcal{O}}(V \widehat{\otimes}_AX, \mathcal{O}): 
 f\otimes g\mapsto [v\widehat{\otimes}x\mapsto g(f(v)x)]$$
 of topological $A$-modules. 
 \end{lemma}
 \begin{proof}We first remark that the usual tensor product $V^*\otimes_A \mathrm{Hom}_{\mathcal{O}}(X, \mathcal{O})$ is 
 equal to the completed tensor product $V^*\widehat{\otimes}_A \mathrm{Hom}_{\mathcal{O}}(X, \mathcal{O})$  in the category of 
 compact $A$-modules 
 since $\mathrm{Hom}_{\mathcal{O}}(X, \mathcal{O})$ is a finite generated $A$-module. 
 
 Since $V \widehat{\otimes}_AX$ is the $\varpi$-adic completion of $V\otimes_AX$, one has 
 \begin{multline*}
 \mathrm{Hom}_{\mathcal{O}}(V \widehat{\otimes}_AX, \mathcal{O})=\varprojlim_{n}\mathrm{Hom}_{\mathcal{O}}(V\widehat{\otimes}_AX, \mathcal{O}/\varpi^n \mathcal{O})=\varprojlim_{n}\mathrm{Hom}_{\mathcal{O}}(V \otimes_A(X/\varpi^nX), \mathcal{O}/\varpi^n\mathcal{O}). 
 \end{multline*}
 Since $X$ with its $\varpi$-adic topology is equipped with a continuous action of $A$, one has $X/\varpi^nX=\varinjlim_k(X/\varpi^nX)[\mathfrak{m}^k]$ for any $n$. Hence, we obtain
$$ \mathrm{Hom}_{\mathcal{O}}(V\otimes_A(X/\varpi^nX), \mathcal{O}/\varpi^n\mathcal{O})=\varprojlim_k\mathrm{Hom}_{\mathcal{O}}((V/\mathfrak{m}^k V) \otimes_A(X/\varpi^nX)[\mathfrak{m}^k], \mathcal{O}/\varpi^n\mathcal{O}).$$

Since $V/\mathfrak{m}^kV$ is a free $A/\mathfrak{m}^k$-module and $(X/\varpi^nX)[\mathfrak{m}^k]$ is a finite length $A$-module, there exists a natural $A$-linear 
isomorphism
\begin{multline*}
\mathrm{Hom}_A(V/\mathfrak{m}^k V, A/\mathfrak{m}^k)\otimes_A\mathrm{Hom}_{\mathcal{O}}((X/\varpi^nX)[\mathfrak{m}^k], \mathcal{O}/\varpi^n\mathcal{O})\\
\isom \mathrm{Hom}_{\mathcal{O}}((V/\mathfrak{m}^k V) \otimes_A(X/\varpi^nX)[\mathfrak{m}^k], \mathcal{O}/\varpi^n\mathcal{O}).
\end{multline*}
Since the projective limit of the left hand side with respect to $k, n$ is 
equal to the completion $(V^*)\widehat{\otimes}_A \mathrm{Hom}_{\mathcal{O}}(X, \mathcal{O})$ in the category of compact $A$-modules, 
we obtain a natural isomorphism 
$$(V^*)\widehat{\otimes}_A \mathrm{Hom}_{\mathcal{O}}(X, \mathcal{O})\isom \mathrm{Hom}_{\mathcal{O}}(V \widehat{\otimes}_AX, \mathcal{O})$$
 of topological $A$-modules. By the remark in the beginning of the proof, we also obtain an isomorphism 
 $$(V^*)\otimes_A \mathrm{Hom}_{\mathcal{O}}(X, \mathcal{O})\isom \mathrm{Hom}_{\mathcal{O}}(V \widehat{\otimes}_AX, \mathcal{O}),$$
 which is easily seen to be the same map in the statement of the lemma. 

 \end{proof}

 \subsection{The local Langlands correspondence for $\mathrm{GL}_2(\mathbb{Q}_l)$ in families}
 This theory interpolates a generic version of the usual local Langlands correspondence for $G_l=\mathrm{GL}_2(\mathbb{Q}_l)$. 
 Since the genericity is very important for our application, we start to recall the notion of generic representations 
 of $G_l$. 
 

 \subsubsection{Generic representations of $G_l$}
 Let $l$ be a prime distinct from $p$. Let 
  $\psi_l : \mathbb{Q}_l\rightarrow \overline{\mathbb{Q}}^{\times}$ be the fixed additive character defined in Notation in Introduction . We freely regard this character as characters 
  $\psi_l : \mathbb{Q}_l\rightarrow \overline{\mathbb{Q}}_p^{\times}$ and 
  $\psi_l : \mathbb{Q}_l\rightarrow \mathbb{C}^{\times}$ by 
  the fixed inclusions $\iota_p : \overline{\mathbb{Q}}\hookrightarrow \overline{\mathbb{Q}}_p$ 
  and $\iota_{\infty} : \overline{\mathbb{Q}}\hookrightarrow \mathbb{C}$. Let $W(\overline{\mathbb{F}}_p)$ be the ring of Witt vectors of $\overline{\mathbb{F}}_p$. We remark that 
the image of $\psi_l : \mathbb{Q}_l\rightarrow \overline{\mathbb{Q}}_p^{\times}$ is in $W(\overline{\mathbb{F}}_p)^{\times}$, 
hence obtain the character $\psi_l:\mathbb{Q}_l\rightarrow W(\overline{\mathbb{F}}_p)^{\times}$ which we also denote by the same letter $
\psi_l$. 

We set $N_l:=\left\{\left.\begin{pmatrix}1 & b\\0 & 1\end{pmatrix} \right| b\in \mathbb{Q}_l\right\}\subset P_l:=
\left\{\left.\begin{pmatrix}a & b\\0 & 1\end{pmatrix} \right| a\in \mathbb{Q}^{\times}_l, b\in \mathbb{Q}_l  \right\}\subset G_l$. 
We identify $N_l$ with $\mathbb{Q}_l$ by the isomorphism 
$\mathbb{Q}_l\isom N_l : b\mapsto \begin{pmatrix}1 & b\\0 & 1\end{pmatrix}$, by which we also regard $\psi_l$ as a character 
$\psi_l : N_l\rightarrow \overline{\mathbb{Q}}^{\times}$. 
 
 Let $\widetilde{A}$ be a $W(\overline{\mathbb{F}}_p)$-algebra or a $\overline{\mathbb{Q}}$-algebra. For any smooth 
 $\widetilde{A}$-representation 
 $V$ of $P_l$, we write $\Phi_{l}(V)$ to denote the largest quotient of $V$ on which $N_l$
 acts by $\psi_l$. It is known that the functor $\Phi_{l}$ is an exact functor from the category of smooth 
 $\widetilde{A}$-representations of $P_l$ to the category of $\widetilde{A}$-modules.  

 In Proposition 3.1.4 of \cite{EH14}, it is shown that the functor $\Phi_{l}$ naturally descends to 
 any $\mathbb{Z}_p$-algebra, i.e. for any $\mathbb{Z}_p$-algebra $A$, there exists an exact functor $\Phi_{l}$ from the category of 
 smooth $A$-representations of $P_l$  to the category of $A$-modules with the property that there exists a canonical isomorphism 
 $$\Phi_{l}(V)\otimes_{\mathbb{Z}_p}W(\overline{\mathbb{F}}_p)\isom \Phi_{l}(V\otimes_{\mathbb{Z}_p}W(\overline{\mathbb{F}}_p))$$ 
 for any smooth $A$-representation $V$ of $P_l$. By the same proof, we can show that the functor $\Phi_{l}$ naturally descends also to 
 any $\mathbb{Q}$-algebra, i.e. for any $\mathbb{Q}$-algebra $A$, there exists an exact functor $\Phi_{l}$ from the category of 
 smooth $A$-representations of $P_l$  to the category of $A$-modules with the similar property. By definition, there exists a canonical isomorphism 
 $$\Phi_{l}(V)\otimes_AM\isom \Phi_{l}(V\otimes_AM)$$
 for any $\mathbb{Z}_p$- or $\mathbb{Q}$-algebra $A$, smooth $A$-representation $V$ of $P_l$, and $A$-module $M$. 
 
   Let $\kappa$ be a $\mathbb{Z}_p$-or $\mathbb{Q}$-algebra which is a field. 
 \begin{defn}\label{5.7}
 An absolutely irreducible smooth admissible representation $V$ of $G_l$ over 
 $\kappa$ is called generic if $\Phi_{l}(V):=\Phi_{l}(V|_{P_l})\not=0$. We remark that this condition is equivalent to 
 $\mathrm{dim}_{\kappa}\Phi_{l}(V)=1$ by Theorem 3.1.15 of \cite{EH14}. 
 \end{defn}
 For any smooth $\kappa$-representation $V$ of $G_l$, we write $\mathrm{soc}(V)$ to denote the sum of 
 all the irreducible subrepresentations of $V$.

 \begin{defn}\label{5.8}
 A smooth $\kappa$-represntation $V$ of $G_l$ is 
 called essentially AIG if :
 \begin{itemize}
 \item[(1)]$\mathrm{soc}(V)$ is absolutely irreducible and generic.
 \item[(2)]$\Phi_{l}(V/\mathrm{soc}(V))=0$. 
 \item[(3)]$V$ is the sum of its finite length subrepresentations.
 \end{itemize}
 \end{defn}
 Let $A$ be a Noether $\mathbb{Z}_p$-algebra. The following notion is defined in \cite{He16}, which we need to state the main theorem of \cite{EH14}. This notion
  is also important 
 for our application. 
 \begin{defn}\label{5.9}
 A smooth $A$-representation $V$ of $G_l$ is called a co-Whittaker $A[G_l]$-module 
 if the following hold : 
 \begin{itemize}
 \item[(1)]$V$ is an admissible $A$-representation of $G_l$.
 \item[(2)]$\Phi_{l}(V)$ is free of rank one over $A$. 
 \item[(3)]If $\mathfrak{p}$ is a prime ideal of $A$ with 
 field of fraction $\kappa(\mathfrak{p})$ of $A/\mathfrak{p}$, then the smooth 
 $\kappa(\mathfrak{p})$-dual of $V\otimes_A\kappa(\mathfrak{p})$ is essentially AIG. 
 \end{itemize}
 \end{defn}
 
 We need the following lemma, which is an easy generalization of Proposition 6.2 of \cite{He16}.
\begin{lemma}\label{5.10}
For any co-Whittaker $A[G_l]$-module $V$ and 
any $A$-module $M$,  the natural map
$$\mathrm{End}_A(M)\rightarrow\mathrm{End}_{A[G_l]}(M\otimes_AV) : f\mapsto f\otimes\mathrm{id}_V$$
is isomorphism.
\end{lemma}
 \begin{proof}
 By localizing at each prime ideal of $A$, it suffices to consider the case in which $A$ is a local ring. 
 For any smooth $A$-representation $W$ of $G_l$, we write $\mathfrak{I}(W)\subseteq W$ to denote the 
 space of Schwartz  functions in $W$ which is defined in \S 3.1 of \cite{EH14}, which is stable by the action of the subgroup 
 $P_l$.
 By the adjoint property between
 the functors $\Phi_{l}(-)$ and $\mathfrak{I}(-)$, there exists a natural isomorphism 
 $$\mathrm{End}_{A[P_l]}(\mathfrak{I}(M\otimes_AV))\isom \mathrm{End}_{A}(\Phi_{l}(M\otimes_AV))$$
 (see for example the proof of Proposition 3.1.16 of \cite{EH14}). 
 Since $V$ is co-Whittaker, we have $\Phi_{l}(V)\isom A$ and $V$ is generated as $A[\mathrm{GL}_2(\mathbb{Q}_l)]$-module by 
 $\mathfrak{I}(V)$ by Lemma 6.3.2 of \cite{EH14}. Therefore, there exists an isomorphism 
  $$\mathrm{End}_{A}(\Phi_{l}(M\otimes_AV))\isom \mathrm{End}_{A}(M\otimes_A\Phi_{l}(V))\isom\mathrm{End}_A(M),$$
  and the natural map 
  $$\mathrm{End}_{A[G_l]}(M\otimes_AV)\rightarrow 
  \mathrm{End}_{A[P_l]}(\mathfrak{I}(M\otimes_AV))$$ is 
  injective since one has $\mathfrak{I}(M\otimes_AV)\isom M\otimes_A\mathfrak{I}(V)$ by definition of $\mathfrak{I}(-)$ 
  (see \S 3.1 of \cite{EH14}), from which the lemma follows.
  
 \end{proof}
 \subsubsection{The generic local Langlands correspondence for $G_l$}
Here, we recall the generic version of the local Langlands correspondence for $G_l$-representations with coefficients in 
 some fields $\mathcal{K}$ over $\mathbb{Q}_p$ following \S 4.3 of \cite{Em} and \S 4.2 of \cite{EH14}. We first recall it in the case $\mathcal{K}=\overline{\mathbb{Q}}_p$. 
 Let $\rho : G_{\mathbb{Q}_l}\rightarrow \mathrm{GL}_2(\overline{\mathbb{Q}}_p)$ be a continuous representation. 
 Let $W(\rho)$ be the associated $\overline{\mathbb{Q}}_p$-representation of Weil-Deligne group of $\mathbb{Q}_l$, and 
 $W(\rho)^{\mathrm{ss}}$ be its Frobenius semi-simplification. Let $\pi'(\rho):=\pi'(W(\rho)^{\mathrm{ss}})$ be the absolutely irreducible 
 smooth admissible $\overline{\mathbb{Q}}_p$-representation of $G_l$ corresponding to $W(\rho)^{\mathrm{ss}}$ via the Tate's normalized local Langlands correspondence. Finally, we write $\pi(\rho)$ to denote the essentially AIG 
 $\overline{\mathbb{Q}}_p[G_l]$-module associated to $\rho$ by the generic local Langlands correspondence, i.e. 
 $\pi(\rho):=\mathrm{Ind}_{B}^{G_l}(\chi_1|\quad|_l\otimes \chi_2)$ if $\rho=\chi_1\oplus \chi_2$ with $\chi_1\chi_2^{-1}=|\quad|_l^{-1}$, and 
 $\pi(\rho)=\pi'(\rho)$ otherwise. We remark that, in the former case, $\pi(\rho)$ is not equal to $\pi'(\rho)= (\chi_1\circ \mathrm{det})$, 
 and is a non-trivial extension 
 $$0\rightarrow (\chi_1\circ \mathrm{det})\otimes \mathrm{St}\rightarrow \pi(\rho) \rightarrow \pi'(\rho)\rightarrow 0,$$
 in particular, $\pi(\rho)$ is not irreducible but generic.

 
 
 
 Since we choose Tate's normalization, we can extend this correspondence to representations over more general coefficient fiields.
 Let $A$ be a complete Noetherian local $\mathbb{Z}_p$-algebra with finite residue field. Suppose that $A$ is a domain, with 
 field of fractions $\mathcal{K}$. Let $\rho : G_{\mathbb{Q}_l}\rightarrow \mathrm{GL}_2(\mathcal{K})$ be a continuous representation, 
 here ``continuous" means that $\rho$ is conjugate to a continuous representation $G_{\mathbb{Q}_l}\rightarrow \mathrm{GL}_2(A)$. 
 For any such $\rho$, one can canonically attach an essentially AIG $\mathcal{K}[G_l]$-module $\pi(\rho)$,
 which  is compatible with the correspondence above, and also with arbitrary field extentsions $\mathcal{K}\rightarrow \mathcal{K}'$.
  See \S 4.2 of \cite{Em} and \S 4 of \cite{EH14} for more details.
 
 \subsubsection{Local Langlands correspondence in families for $G_l$} 
 Here, we assume that $A$ is a reduced complete Noetherian local 
 $\mathbb{Z}_p$-algebra with finite residue field. For each prime ideal $\mathfrak{p}$ of $A$, 
 we write $\kappa(\mathfrak{p})$ to denote the field of fractions of $A/\mathfrak{p}$. 
 \begin{thm}
\label{5.11} Let $\rho : G_{\mathbb{Q}_l}\rightarrow \mathrm{GL}_2(A)$ be a 
 continuous representation. Then there is, up to isomorphism, at most one smooth admissible $A$-representation $\pi(\rho)$ of $G_l$ such that : 
 \begin{itemize}
 \item[(1)]$\pi(\rho)$ is $A$-torsion free, 
 \item[(2)]$\pi(\rho)$ is a co-Whittaker $A[G_l]$-module,
 \item[(3)]for each minimal prime ideal $\mathfrak{a}$ of $A$, the tensor product 
 $\kappa(\mathfrak{a})\otimes_A\pi(\rho)$ is $\kappa(\mathfrak{a})[G_l]$-linearly isomorphic to 
 the smooth contragradient $\widetilde{\pi}(\rho(\mathfrak{a}))$ of $\pi(\rho(\mathfrak{a}))$, where 
 $\rho(\mathfrak{a}) : G_{\mathbb{Q}_l}\rightarrow \mathrm{GL}_2(\kappa(\mathfrak{a}))$ is the base change of 
 $\rho$ by the canonical map $A\rightarrow A/\mathfrak{a}\hookrightarrow \kappa(\mathfrak{a})$. 
 
 \end{itemize}
 
 \end{thm}
 
 \begin{rem}\label{5.12}
 This theorem is proved in Theorem 6.2.1 of \cite{EH14} 
 (even for any continuous representation $\rho : G_{F}\rightarrow \mathrm{GL}_n(A)$ 
 for any prime $p$, any finite extension $F$ of $\mathbb{Q}_l$, and any $n\geqq 1$). We follow Theorem 
 7.2.1 of \cite{He16} for the statement (1), (2), (3) which characterize $\pi(\rho)$. For the existence of $\pi(\rho)$, it is announced in 
 (the proof of) Theorem 6.1.2 of \cite{Em} and Remark 7.10 of \cite{He16} that such a $\pi(\rho)$ exists for any 
 $\rho$ as above under the assumption $p\not=2$. Moreover, for the representations $\rho$ which are considered in \cite{Em}, 
 the existence of $\pi(\rho)$ is proved in Theorem 6.1.2 of \cite{Em}, which is also enough for our application. 
 
  \end{rem}
 \subsubsection{Coadmissible representation}
 Let $G=G_l$ or $G=G_{\Sigma_0}$.
 
 \begin{defn}\label{5.13}
 We say that a smooth $A$-representation of $G$ on an $\mathcal{O}$-torsion free $A$-module $X$ is coadmissible if 
 for each open compact subgroup $K$ of $G$, the space of invariants $X^K$ is a cofinitely generated $A$-module.
 \end{defn}
 
  \begin{defn}\label{5.14}
  If $X$ is a smooth $A$-representation of $G$ on an $\mathcal{O}$-torsion free $A$-module, then we write $\widetilde{X}$ 
  to denote the smooth contragradient of $X$, i.e. 
  $$\widetilde{X}:=\{\phi\in \mathrm{Hom}_{\mathcal{O}}(X, \mathcal{O})\mid \phi \text{ is fixed by some compact open subgroup of }G\}$$
  with the natural contragradient action of $G$. 
 \end{defn}
 
  Let $V$ be a smooth $A$-representation of $G$. For any open compact subgroup $K$ of $G$ whose pro-order is prime to $p$ 
  (e.g. $K=1+l^nM_2(\mathbb{Z}_l)$ ($n\geqq 1$) for $G=G_l$), 
  one has a projector $p_K : V\rightarrow V^K$ defined by 
  $$p_K(v):=\int_{K}kv d\mu(k) \quad (v\in V), $$
  where $d\mu$ is the $\mathbb{Z}_p$-valued Haar measure on $K$ normalized so that $\int_{K}d\mu(k)=1$ (see \S2.1 of \cite{EH14}). In particular, $V^K$ is an $A$-direct summand of $V$, and the canonical map $M\otimes_A V^K\rightarrow (M\otimes_AV)^K : m\otimes v \mapsto m\otimes v$ is isomorphism for any $A$-module $M$. 
One also has a canonical $A$-linear isomorphism 
   $$\mathrm{Hom}_{\mathcal{O}}(V, \mathcal{O})^K\isom \mathrm{Hom}_{\mathcal{O}}(V^K, \mathcal{O}) : f\mapsto f|_{V^K}.$$
   for any such $K$ with the inverse $f\mapsto f\circ p_K$.
   
  Let $X$ be a coadmissible $A$-representation  of $G$. Then, for any open compact subgroup $K$ of $G$ whose pro-order is prime to $p$, 
  one has an $A$-linear isomorphism 
  $$\widetilde{X}^K=\mathrm{Hom}_{\mathcal{O}}(X, \mathcal{O})^K\isom \mathrm{Hom}_{\mathcal{O}}(X^K, \mathcal{O}).$$
  In particular, $\widetilde{X}^K$ is finite generated $A$-module for any such $K$. Since such $K$ are cofinal in the open subgroups of $G$, 
  this implies that $\widetilde{X}$ is a smooth admissible $A$-representation of $G$ on an $\mathcal{O}$-torsion free $A$-module. In fact, 
  it is shown that 
  the functor $X\mapsto \widetilde{X}$ induces an anti-equivalence of the category of coadmissible 
 smooth $A$-representations of $G$ on $\mathcal{O}$-torsion free $A$-modules, and the category of 
 admissible smooth $A$-representations of $G$ on $\mathcal{O}$-torsion free $A$-modules (Lemma C.26 of \cite{Em}).

 \begin{defn}\label{5.15}
 If $V$ is an orthonormalizable $A$-module, and $X$ is a coadmissible smooth $A$-representation of $G$, then we write 
 $$V\overset{\curlywedge}{\otimes}_AX :=\varinjlim_{K}V\widehat{\otimes}_AX^K$$
 where $K$ runs over all the open compact subgroups of $G$, and $V\widehat{\otimes}_AX^K$
  denotes the $\varpi$-adic completion of the tensor product $V\otimes_AX^K$. 
 \end{defn}
 We remark that $V\widehat{\otimes}_AX^K$ is $\mathcal{O}$-torsion free for each $K$ by lemma \ref{5.5}, hence 
 $V\overset{\curlywedge}{\otimes}_AX$ is also $\mathcal{O}$-torsion free. 
 We also remark that the natural map 
 $V\widehat{\otimes}_AX^K\rightarrow V\overset{\curlywedge}{\otimes}_AX$ is an injection for 
 any $K$ by Lemma C.44 of \cite{Em}, hence one has
 $$V\widehat{\otimes}_AX^K=(V\overset{\curlywedge}{\otimes}_AX)^K.$$
 
 \begin{lemma}\label{5.16}
 If $V$ is an orhthonormalizable $A$-module, and $X$ is a coadmissible smooth $A$-representation of $G$, 
 then there exists a natural $A[G]$-linear isomorphism
 $$V^*\otimes_A\widetilde{X}\isom \widetilde{(V\overset{\curlywedge}{\otimes}_AX)}.$$
  
 \end{lemma}
 \begin{proof}
 For any $f\otimes g\in V^*\otimes_A\widetilde{X}$, the map 
 $V\otimes_A X^K\rightarrow \mathcal{O}:v\otimes x\mapsto g(f(v)x)$ uniquely extends to an $\mathcal{O}$-linear map 
 $V\widehat{\otimes}_A X^K\rightarrow \mathcal{O}$ for any open compact subgroup $K$ of $G$. These maps for all $K$ induce an
 $\mathcal{O}$-linear map $V\overset{\curlywedge}{\otimes}_AX\rightarrow \mathcal{O}$. 
 Therefore, we obtain a map 
 $$V^*\otimes_A\widetilde{X}\rightarrow \mathrm{Hom}_{\mathcal{O}}(V\overset{\curlywedge}{\otimes}_AX, \mathcal{O}).$$ Since this map is easily seen 
 to be $A[G]$-linear, this map factors through the smooth contragradient $V^*\otimes_A\widetilde{X}\rightarrow \widetilde{(V\overset{\curlywedge}{\otimes}_AX)}$. To see that the latter map is isomorphism, it suffices to show that this map induces an 
 isomorphism $(V^*\otimes_A\widetilde{X})^K\rightarrow \left(\widetilde{(V\overset{\curlywedge}{\otimes}_AX)}\right)^K$ for any $K$ whose pro-order is prime to $p$. 
 To show this claim, we note that one has
 $(V^*\otimes_A\widetilde{X})^K=V^*\otimes_A\widetilde{X}^K$ since the pro-order of $K$ is prime to $p$. 
 One also has 
 $\left(\widetilde{(V\overset{\curlywedge}{\otimes}_AX)}\right)^K\isom \mathrm{Hom}_{\mathcal{O}}((V\overset{\curlywedge}{\otimes}_AX)^K, \mathcal{O})
 =\mathrm{Hom}_{\mathcal{O}}(V\widehat{\otimes}_AX^K, \mathcal{O})$. Since $X^K$ is a cofinitely generated  $A$-module, then the claim follows from Lemma \ref{5.6}.
 
 \end{proof}

 \subsection{The $p$-adic local Langlands correspondence for $\mathrm{GL}_2(\mathbb{Q}_p)$}
 \subsubsection{Deformation theoretic formulation of the $p$-adic local Langlands correspondence}
 Here, we recall the $p$-adic local Langlands correspondence for $G_p$ \cite{Co10} following \S3 of \cite{Em}.
 We write $\mathrm{Art}(\mathcal{O})$ to denote the full subcategory of $\mathrm{Comp}(\mathcal{O})$ consisting of Artin rings. 
 For any object $A$ in $\mathrm{Comp}(\mathcal{O})$ with its maximal ideal $\mathfrak{m}$,  we say that an $A[G_p]$-module $V$ is smooth if $V$ satisfies 
 $$V=\cup_{K, n}V^K[\mathfrak{m}^n]$$
 where the union is taken over open compact subgroups $K$ of $G_p$ and positive integers $n$. We write 
 $\mathrm{Mod}_{G_p}^{\mathrm{sm}}(A)$ to denote the category of smooth $A[G_p]$-modules. We say that an object 
 $V\in \mathrm{Mod}_{G_p}^{\mathrm{sm}}(A)$ is admissible if $V^K[\mathfrak{m}^n]$ is finite generated $A$-module 
 for any open compact subgourp $K$ of $G_p$ and $n\geqq 1$. Moreover,  $V\in \mathrm{Mod}_{G_p}^{\mathrm{sm}}(A)$ is 
 called locally admissible if, for every $v\in V$, the sub $A[G_p]$-module $A[G_p] v$ of $V$ is admissible. 
 We write $\mathrm{Mod}_{G_p}^{\mathrm{adm}}(A)$ (resp. $\mathrm{Mod}_{G_p}^{\mathrm{l.adm}}(A)$) to denote the full 
 subcategory of $\mathrm{Mod}_{G_p}^{\mathrm{sm}}(A)$ consisting of admissible (resp. locally admissible) smooth $A[G_p]$-modules.
 
 \begin{defn}\label{5.17}
 We say that an $A[G_p]$-module $\pi$ is an orthonormalizable admissible $A$-representation of $G_p$
 if the following conditions are satisfied : 
 \begin{itemize}
 \item[(1)]$\pi$ is an orthonormalizable $A$-module. 
 \item[(2)]The induced $G_p$-action on $\pi/\mathfrak{m}^n\pi$ makes this quotient an admissible smooth $(A/\mathfrak{m}^n)[G_p]$-module 
 for each $n\geqq 1$. 
 \end{itemize}
 
 \end{defn}

 Let $A$ be an object in $\mathrm{Art}(\mathcal{O})$. Colmez \cite{Co10} has defined a covariant exact functor 
 \begin{multline*}
 \mathrm{MF} : \{\text{admissible smooth $A[G_p]$-modules of finite length }\}\\
 \longrightarrow \{\text{ continuous } A\text{-representations of $G_{\mathbb{Q}_p}$ on finite generated } A\text{-modules}\}.
 \end{multline*}
 
 Precisely, Colmez \cite{Co10} defined a contravariant exact functor $\pi\mapsto D(\pi)$ to the category of \'etale $(\varphi, \Gamma)$-modules 
 with $A$-action. Then, our $\mathrm{MF}(\pi)$ is equal to the Pontryagin dual $\mathrm{Hom}_{\mathcal{O}}(V(D(\pi)), L/\mathcal{O})$ of 
 the $A$-representation $V(D(\pi))$ of $G_{\mathbb{Q}_p}$ associated to $D(\pi)$ via Fontaine's equivalence. On the other hands, Colmez 
 \cite{Co10} used $\bold{V}(\pi):=\mathrm{Hom}_{\mathcal{O}}(V(D(\pi)), L/\mathcal{O})(\varepsilon)$, hence the relation between Emerton's functor $\mathrm{MF}(-)$ and 
 Colmez' functor $\bold{V}(-)$ is 
 $$\bold{V}(\pi)=\mathrm{MF}(\pi)(\varepsilon).$$
 
 
  Suppose that $A$ is an object in $\mathrm{Comp}(\mathcal{O})$ with maximal ideal $\mathfrak{m}$. 
  If $\pi$ is an orthonormalizable admissible $A$-representation of $G_p$ over $A$ with the additional property that 
  $\pi/\mathfrak{m}\pi$ is of finite length as $A[G_p]$-module, then we define 
  $$\mathrm{MF}(\pi):=\varprojlim_n\mathrm{MF}(\pi/\mathfrak{m}^n\pi).$$
  Since $\mathrm{MF}$ is exact for every $A\in \mathrm{Art}(\mathcal{O})$, $\mathrm{MF}$ for any $A\in \mathrm{Comp}(\mathcal{O})$ also provides an exact functor 
  \begin{multline*}
 \mathrm{MF} : \{\text{orthonormalizable admissible } G\text{-representations }\pi \text{ over } A \\
\text{ such that } 
\pi/\mathfrak{m}\pi \text{ is of finite length as $A[G_p]$-module}\}\\
 \longrightarrow \{\text{ continuous } A\text{-representations of $G_{\mathbb{Q}_p}$ on finite free } A\text{-modules}\}
 \end{multline*}

Let $\mathbb{F}$ be a finite field of characteristic $p$, and let 
$\overline{\rho} : G_{\mathbb{Q}_p}\rightarrow \mathrm{GL}_2(\mathbb{F})$ be a continuous representation. 
The following theorem is due to Colmez \cite{Co10}. We quote it from Theorem 3.3.2 of \cite{Em}. 

\begin{thm}\label{5.18}
Assume that $\overline{\rho}$ satisfies the following : 
\begin{itemize}
\item[(a)] $\overline{\rho}\not\isom \chi\otimes \begin{pmatrix}1& * \\ 0 & \overline{\varepsilon}\end{pmatrix}$
for any character $\chi : G_{\mathbb{Q}_p}\rightarrow \mathbb{F}^{\times}$. 
\end{itemize}
Then, there exists a finite length object $\overline{\pi}$ in $\mathrm{Mod}_{G_p}^{\mathrm{adm}}(\mathbb{F})$, unique up to isomorphism, such that : 
\begin{itemize}
\item[(1)]$\mathrm{MF}(\overline{\pi})\isom \overline{\rho}$.
\item[(2)]$\overline{\pi}$ has central character equal to $\mathrm{det}(\overline{\rho})\overline{\varepsilon}$.
\item[(3)]$\overline{\pi}$ has no sub quotient which is finite dimensional.
\end{itemize}
\end{thm}

From now, fix $\overline{\rho}$ and $\overline{\pi}$ as above. We consider the following deformation functors.
\begin{defn}\label{5.19}
We let 
$$\mathrm{Def}(\overline{\rho}) : \mathrm{Comp}(\mathcal{O})\rightarrow (\mathrm{Set})$$ 
(resp. 
$$\mathrm{Def}(\overline{\pi}) : \mathrm{Comp}(\mathcal{O})\rightarrow (\mathrm{Set}))$$ 
denote the covariant functor from 
$\mathrm{Comp}(\mathcal{O})$ to the category of sets defined by : for any object $A\in \mathrm{Comp}(\mathcal{O})$ with maximal ideal $\mathfrak{m}$, 
the set $\mathrm{Def}(\overline{\rho})(A)$ (resp. $\mathrm{Def}(\overline{\pi})(A)$) is the isomorphism class of the pairs 
$(V, \iota)$ (resp. $(\pi, \iota)$), where $V$ is a free $A$-module of rank two equipped with a continuous $A$-linear action of 
$G_{\mathbb{Q}_p}$ (resp. $\pi$ is an orthonormalizable admissible $A$-representation of $G_p$), and $\iota$ is an $(A/\mathfrak{m})[G_{\mathbb{Q}_p}]$-linear isomorphism 
$\iota : V/\mathfrak{m}V\isom A/\mathfrak{m}\otimes_{\mathbb{F}}\overline{\rho}$ 
(resp. $(A/\mathfrak{m})[G_p]$-linear isomorphism $\iota : \pi/\mathfrak{m}\pi\isom A/\mathfrak{m}\otimes_{\mathbb{F}}\overline{\pi}$). 

\end{defn}

Since $\mathrm{MF}$ is exact, this induces a natural transformation of the deformation functors
$$\mathrm{MF} : \mathrm{Def}(\overline{\pi})\rightarrow \mathrm{Def}(\overline{\rho}).$$
The following theorem is the deformation theoretic formulation of the $p$-adic local Langlands correspondence of $G_p$ whose formulation is 
due to Kisin \cite{Ki10}. We also quote it from Theorem 3.3.13 (and Remark 3.314) of \cite{Em}.

\begin{thm}\label{5.20}
Assume that $\overline{\rho}$ satisfies the assumption $(a)$ in Theorem \ref{5.18}, and the 
following : 
\begin{itemize}
\item[(b)]$p\geqq 3$.
\item[(c)]$\mathrm{End}_{\mathbb{F}[G_{\mathbb{Q}_p}]}(\overline{\rho})=\mathbb{F}$.
\end{itemize}
Then, both the functors are representable, and the natural transformation 
$$\mathrm{MF} : \mathrm{Def}(\overline{\pi})\rightarrow \mathrm{Def}(\overline{\rho})$$
is isomorphism. 
\end{thm}
\begin{rem}\label{5.21}
Under the assumptions in the theorem above, for any $A\in \mathrm{Comp}(\mathcal{O})$ and any 
isomorphism class 
$[(\pi, \iota)]\in \mathrm{Def}(\overline{\rho})(A)$, the center of $G_p$ acts on $\pi$ via 
$A^{\times}$-valued character $\mathrm{det}(\mathrm{MF}(\pi))\varepsilon : \mathbb{Q}_p^{\times}
\rightarrow A^{\times}$ (see Remark 3.3.14 of \cite{Em}). 
\end{rem}
From now on until the end of this subsection, we assume that all the conditions (a), (b), (c) in the theorem above are satisfied. 
We let $R_{\overline{\rho}}$ denote the universal deformation ring of the functor $\mathrm{Def}(\overline{\rho})$, and 
let $\rho_{\overline{\rho}} : G_{\mathbb{Q}_p}\rightarrow \mathrm{GL}_2(R_{\overline{\rho}})$ 
denote the (unique up to isomorphism) universal deformation of $\overline{\rho}$. By the isomorphism 
$\mathrm{MF} : \mathrm{Def}(\overline{\pi})\isom \mathrm{Def}(\overline{\rho})$, $R_{\overline{\rho}}$ is also 
the universal deformation ring for $\mathrm{Def}(\overline{\pi})$, and the universal object 
$[(\pi_{\overline{\rho}}, \iota)]\in  \mathrm{Def}(\overline{\pi})(R_{\overline{\rho}})$ is the unique one satisfying isomorphism 
$$\mathrm{MF}(\pi_{\overline{\rho}})\isom \rho_{\overline{\rho}}.$$

\subsubsection{Recall of Pa\v{s}k\={u}nas' results}
For our application, another description of $\pi_{\overline{\rho}}$ due to Pa\v{s}k\={u}nas \cite{Pas13} ie very important, which we 
recall in this subsection. To recall it, we first need to consider the deformations with a fixed determinant condition.

Let $\psi : \mathbb{Q}_p^{\times}\rightarrow \mathcal{O}^{\times}$ be a continuous character such that 
$\psi \equiv \mathrm{det}(\overline{\rho})\overline{\varepsilon}$ (mod $\varpi$).

\begin{defn}\label{5.22}
We let $\mathrm{Def}(\overline{\rho})^{\psi\varepsilon^{-1}}$ (resp. $\mathrm{Def}(\overline{\pi})^{\psi}$) to denote the subfunctor of $\mathrm{Def}(\overline{\rho})$ (resp. $\mathrm{Def}(\overline{\pi})$) 
consisting of those deformations $[\rho, \iota]\in \mathrm{Def}(\overline{\rho})(A)$ (resp. $[\pi,\iota]\in \mathrm{Def}(\overline{\pi})(A)$) 
such that $\mathrm{det}(\rho)$ (resp. with central character which) is equal to the composite of 
$\psi\varepsilon^{-1}$ (resp. $\psi$) with the natural map $\mathcal{O}^{\times}\rightarrow A^{\times}$. 
\end{defn}

By Remark \ref{5.21}, the isomorphism $\mathrm{MF} : \mathrm{Def}(\overline{\pi})\isom \mathrm{Def}(\overline{\rho})$ induces the isomorphism 
$\mathrm{MF} : \mathrm{Def}(\overline{\pi})^{\psi}\isom \mathrm{Def}(\overline{\rho})^{\psi\varepsilon^{-1}}$. We let 
$R^{\psi\varepsilon^{-1}}_{\overline{\rho}}$ denote the universal deformation ring both for $\mathrm{Def}(\overline{\pi})^{\psi}$ and $\mathrm{Def}(\overline{\rho})^{\psi\varepsilon^{-1}}$, and let $[(\rho^{\psi\varepsilon^{-1}}_{\overline{\rho}},\iota)]\in 
\mathrm{Def}(\overline{\rho})^{\psi\varepsilon^{-1}}(R^{\psi\varepsilon^{-1}}_{\overline{\rho}})$ and 
$[(\pi^{\psi}_{\overline{\pi}},\iota)]\in 
\mathrm{Def}(\overline{\pi})^{\psi}(R^{\psi\varepsilon^{-1}}_{\overline{\rho}})$ denote the universal objects. Then, there is an isomorphism 
$$\mathrm{MF}(\pi^{\psi}_{\overline{\pi}})\isom \rho^{ \psi\varepsilon^{-1}}_{\overline{\rho}}$$
 of $R_{\overline{\rho}}^{\psi\varepsilon^{-1}}[G_{\mathbb{Q}_p}]$-modules.


We recall that $\mathrm{Mod}_{G_p}^{\mathrm{l.adm}}(\mathcal{O})$ is the category of locally admissible 
$\mathcal{O}[G_p]$-module (for $A=\mathcal{O}$). We write $\mathrm{Mod}_{G_p}^{\mathrm{l.adm},\psi}(\mathcal{O})$ 
to denote the full subcategory of $\mathrm{Mod}_{G_p}^{\mathrm{l.adm}}(\mathcal{O})$ consisting of objects on which the center 
of $G_p$ act by $\psi$. We remark that both categories $\mathrm{Mod}_{G_p}^{\mathrm{l.adm}}(\mathcal{O})$ and $\mathrm{Mod}_{G_p}^{\mathrm{l.adm},\psi}(\mathcal{O})$ are abelian, stable under taking inductive limits, and admit injective envelopes (see Lemma 2.3 of \cite{Pas13} for $\mathrm{Mod}_{G_p}^{\mathrm{l.adm},\psi}(\mathcal{O})$, 
and a remark after Definition 4.5 of \cite{CEG${}^{+}$18} for $\mathrm{Mod}_{G_p}^{\mathrm{l.adm}}(\mathcal{O})$). For a discrete $\mathcal{O}$-module $M$, 
we write $M^{\vee}:=\mathrm{Hom}_{\mathcal{O}}(M, L/\mathcal{O})$ to denote its Pontryagin dual. 
We write $\mathrm{Mod}_{G_p}^{\mathrm{pro}}(\mathcal{O})$ to denote the category of compact $\mathcal{O}[[\mathrm{GL}_2(\mathbb{Z}_p)]]$-modules 
with an action of $\mathcal{O}[G_p]$ such that the two actions coincide when restricted to $\mathcal{O}[\mathrm{GL}_2(\mathbb{Z}_p)]$. Then, 
the Pontryagin dual $M\mapsto M^{\vee}$ induces an anti-equivalence of categories between 
$\mathrm{Mod}^{\mathrm{sm}}_{G_p}(\mathcal{O})$ and $\mathrm{Mod}_{G_p}^{\mathrm{pro}}(\mathcal{O})$. 
We write $\mathfrak{C}(\mathcal{O})$ (resp. $\mathfrak{C}^{\psi}(\mathcal{O})$) to denote the full subcategory of 
$\mathrm{Mod}_{G_p}^{\mathrm{pro}}(\mathcal{O})$ which is the Pontryagin dual 
of $\mathrm{Mod}_{G_p}^{\mathrm{l.adm}}(\mathcal{O})$ (resp. $\mathrm{Mod}_{G_p}^{\mathrm{l.adm},\psi}(\mathcal{O})$ ). 
By definition, $\mathfrak{C}^{\psi}(\mathcal{O})$ is the full sub category of  $\mathfrak{C}(\mathcal{O})$ consisting of objects on which the center 
of $G_p$ acts by $\psi^{-1}$. The categories $\mathfrak{C}(\mathcal{O})$ and $\mathfrak{C}^{\psi}(\mathcal{O})$ are abelian, 
stable under taking projective limits
and admit projective envelopes. 

We note that the fixed representation $\overline{\pi}$ is an object in $\mathrm{Mod}_G^{\mathrm{l.adm},\psi}(\mathcal{O})$, hence 
$\overline{\pi}^{\vee}\in \mathfrak{C}^{\psi}(\mathcal{O})$. Let $\overline{\pi}_1\subseteq \overline{\pi}$ be the socle of $\overline{\pi}$. 
We 
write $\widetilde{P}^{\psi}\in \mathfrak{C}^{\psi}(\mathcal{O})$ (resp. $\widetilde{P}\in\mathfrak{C}(\mathcal{O})$) to denote the projective envelope 
of $\overline{\pi}_1^{\vee}\in  \mathfrak{C}^{\psi}(\mathcal{O})$ in the category $\mathfrak{C}^{\psi}(\mathcal{O})$ (resp. $\mathfrak{C}(\mathcal{O})$). 

As a dual version of the functors $\mathrm{MF}$ and $\bold{V}$, Pa\v{s}k\={u}nas defined in \S 5.7 of \cite{Pas13} a covariant functor $\check{\bold{V}}$ from $\mathfrak{C}^{\psi}(\mathcal{O})$ to the category of compact $\mathcal{O}$-modules with $\mathcal{O}$-linear continuous $G_{\mathbb{Q}_p}$-action. For any object $M\in \mathfrak{C}^{\psi}(\mathcal{O})$ of finite length, it is defined by 
$$\check{\bold{V}}(M):=\bold{V}(M^{\vee})^{\vee}(\varepsilon\psi),$$
where $\bold{V}(-)$ is Colmez' functor. Since one has $\bold{V}(\pi)=\mathrm{MF}(\pi)(\varepsilon)$, $\check{\bold{V}}(M)$ is equal to
$$\check{\bold{V}}(M)=\bold{V}(M^{\vee})^{\vee}(\varepsilon\psi)=(\mathrm{MF}(M^{\vee})(\varepsilon))^{\vee}(\varepsilon\psi)
=\mathrm{MF}(M^{\vee})^{\vee}(\psi).$$


\begin{lemma}\label{5.23}
For any $A\in \mathrm{Comp}(\mathcal{O})$ and 
$[(\pi, \iota)]\in \mathrm{Def}(\overline{\pi})^{\psi}(A)$, the $A$-dual $\pi^{*}:=\mathrm{Hom}_A(\pi, A)$ of $\pi$ is an object in $\mathfrak{C}^{\psi}(\mathcal{O})$, 
and there exists an isomorphism 
$$\check{\bold{V}}(\pi^*)\isom \mathrm{MF}(\pi)(\varepsilon)$$
of $A[G_{\mathbb{Q}_p}]$-modules. 
\end{lemma}
\begin{proof}Let $\mathfrak{m}$ be the maximal ideal of $A$. 
Since one has $$\pi^*=\varprojlim_n\mathrm{Hom}_{A}(\pi, A/\mathfrak{m}^n)=\varprojlim_n\mathrm{Hom}_{A/\mathfrak{m}^n}(\pi/\mathfrak{m}^n, A/\mathfrak{m}^n)=\varprojlim_n(\pi/\mathfrak{m}^n\pi)^*,$$ 
and both the functors $\check{\bold{V}}$ and $\mathrm{MF}$ commute with projective limits,  it suffices to show the lemma when $A$ is an object of $\mathrm{Art}(\mathcal{O})$. 

Let $A$ be an object in $\mathrm{Art}(\mathcal{O})$. We note that for any $A$-module $N$, the canonical isomorphism 
$A\isom (A^{\vee})^{\vee}= \mathrm{Hom}_{\mathcal{O}}(A^{\vee}, L/\mathcal{O})$ induces an isomorphism of $A$-modules
$$N^*\isom\mathrm{Hom}_A(N, \mathrm{Hom}_{\mathcal{O}}(A^{\vee}, L/\mathcal{O}))
\isom \mathrm{Hom}_{\mathcal{O}}(N\otimes_AA^{\vee}, L/\mathcal{O})=(N\otimes_AA^{\vee})^{\vee}.$$ 
Therefore, one has 
$$\pi^*\isom (\pi\otimes_AA^{\vee})^{\vee} \in \mathfrak{C}^{\psi}(\mathcal{O}).$$ 
Then, one also has
$$\check{\bold{V}}(\pi^*)\isom \check{\bold{V}}((\pi\otimes_AA^{\vee})^{\vee})=\mathrm{MF}(((\pi\otimes_AA^{\vee})^{\vee})^{\vee})^{\vee}(\psi)
=\mathrm{MF}(\pi\otimes_AA^{\vee})^{\vee}(\psi)$$
Since $\mathrm{MF}$ is exact, one has 
$$\mathrm{MF}(\pi\otimes_AA^{\vee})\isom \mathrm{MF}(\pi)\otimes_AA^{\vee},$$
therefore one also has 
$$\mathrm{MF}(\pi\otimes_AA^{\vee})^{\vee}(\psi)\isom 
(\mathrm{MF}(\pi)\otimes_AA^{\vee})^{\vee}(\psi)\isom \mathrm{MF}(\pi)^*(\psi).$$
Since one has $\mathrm{det}(\mathrm{MF}(\pi))
=\psi\varepsilon^{-1}$, and $\mathrm{MF}(\pi)$ is a free $A$-module of rank two, one also has 
$$\mathrm{MF}(\pi)^*(\psi)\isom \mathrm{MF}(\pi)((\psi\varepsilon^{-1})^{-1})(\psi)
=\mathrm{MF}(\pi)(\varepsilon).$$
\end{proof}

By the lemma above, the functor $\check{\bold{V}}$ induces a natural transformation from deforamtions of $\overline{\pi}^{\vee}$ in 
$\mathfrak{C}^{\psi}(\mathcal{O})$ to 
deformations of $\overline{\rho}(\varepsilon)$. In \cite{Pas13}, Pa\v{s}k\={u}nas studied the deformations of $\overline{\pi}^{\vee}$ in 
$\mathfrak{C}^{\psi}(\mathcal{O})$ using the projective envelope $\widetilde{P}^{\psi}$ of $\overline{\pi}_1^{\vee}$. In particular, he proved the following theorem. 
\begin{thm}\label{5.24}
Assume the following conditions : 
\begin{itemize}
\item[(d)]$p\geqq 5$. 
\item[(e)]$\mathrm{End}_{\mathbb{F}[G_{\mathbb{Q}_p}]}(\overline{\rho})=\mathbb{F}$.
\item[(f)]$\overline{\rho}\not\isom \chi\otimes \begin{pmatrix}\overline{\varepsilon}& * \\ 0 & 1\end{pmatrix}$
for any character $\chi : G_{\mathbb{Q}_p}\rightarrow \mathbb{F}^{\times}$. \end{itemize}
Then, the functor $\check{\bold{V}}$ naturally induces an 
isomorphism 
$$\mathrm{End}_{\mathfrak{C}^{\psi}(\mathcal{O})}(\widetilde{P}^{\psi})\isom R_{\overline{\rho}(\overline{\varepsilon})}^{\psi\varepsilon},$$
of compact rings, by which $\widetilde{P}^{\psi}$ is topologically flat and compact $R_{\overline{\rho}(\overline{\varepsilon})}^{\psi\varepsilon}$-module, and there is an 
isomorphism 
$$\check{\bold{V}}(\widetilde{P}^{\psi})\isom \rho^{ \psi\varepsilon}_{\overline{\rho}(\overline{\varepsilon})}$$
of $R_{\overline{\rho}(\overline{\varepsilon})}^{\psi\varepsilon}[G_{\mathbb{Q}_p}]$-modules. 

\end{thm}
\begin{proof}
We remark that one has $$\check{\bold{V}}(\overline{\pi}^{\vee})\isom \mathrm{MF}(\overline{\pi})(\varepsilon)\isom \overline{\rho}(\varepsilon)$$
by Lemma \ref{5.23}
since one has $\overline{\pi}^*=\overline{\pi}^{\vee}$. Then, the theorem follows from \cite{Pas13} Lemma 2.7, Corollary 3.12, and 
Proposition 6.3 (resp. Corollary 8.7, resp. Theorem 10.71 and 
Corollay 10.17) of \cite{Pas13}
if $\overline{\rho}$ is irreducible (resp. $\overline{\rho}$ is reducible and satisfies the condition (a) in Theorem \ref{5.18}, resp. 
$\overline{\rho}$ does not satisfy (a) in Theorem \ref{5.18}). 
\end{proof}


We remark that the map $\rho\mapsto \rho(\varepsilon)$ (twisting by $\varepsilon$) induces an isomorphism 
from the deformations of $\overline{\rho}$ to those of $\overline{\rho}(\overline{\varepsilon})$. Therefore, it also induces isomorphisms of local $\mathcal{O}$-algebras
$$R_{\overline{\rho}(\overline{\varepsilon})}\isom R_{\overline{\rho}}\quad \text{ and }\quad 
R^{\psi\varepsilon}_{\overline{\rho}(\overline{\varepsilon})}\isom R^{\psi\varepsilon^{-1}}_{\overline{\rho}}.$$ As the composite of the latter one and the isomorphism $\mathrm{End}_{\mathfrak{C}^{\psi}(\mathcal{O})}(\widetilde{P}^{\psi})\isom R_{\overline{\rho}(\overline{\varepsilon})}^{\psi\varepsilon}$ in the above theorem, 
we obtain an isomorphism 
$$\mathrm{End}_{\mathfrak{C}^{\psi}(\mathcal{O})}(\widetilde{P}^{\psi})\isom R^{\psi\varepsilon^{-1}}_{\overline{\rho}},$$
by which we also regard $\widetilde{P}^{\psi}$ as a compact $R^{\psi\varepsilon^{-1}}_{\overline{\rho}}$ module. 

\begin{lemma}\label{5.25}
Assume that the conditions $(a)$ of Theorem $\ref{5.18}$, and $(b), (c), (d)$ of Theorem $\ref{5.24}$ hold. 
There exists a $G_p$-equivariant $R^{\psi\varepsilon^{-1}}_{\overline{\rho}}$-linear topological  isomorphism
$$\widetilde{P}^{\psi}\isom (\pi_{\overline{\pi}}^{\psi})^*$$
\end{lemma}
\begin{proof}
We remark that both $\widetilde{P}^{\psi}$ and $(\pi_{\overline{\pi}}^{\psi})^*$ are objects in $\mathfrak{C}^{\psi}(\mathcal{O})$, and are deformations of 
$\overline{\pi}^{\vee}$ over $R^{\psi\varepsilon^{-1}}_{\overline{\rho}}$ such that both $\check{\bold{V}}(\widetilde{P}^{\psi})$ and 
$\check{\bold{V}}((\pi_{\overline{\pi}}^{\psi})^*)$ are isomorphic to 
$\rho_{\overline{\rho}(\overline{\varepsilon})}^{\psi\varepsilon}\isom \rho_{\overline{\rho}}^{\psi\varepsilon^{-1}}(\varepsilon)$ 
by Lemma \ref{5.23} and Theorem \ref{5.24}. Then the lemma follows from the fact that the functor 
$\check{\bold{V}}$ induces an isomorphism between the deformations of $\overline{\pi}^{\vee}$ and those of $\overline{\rho}(\overline{\varepsilon})$ 
by (the proof of) Proposition 6.3 and Corollary 8.7 of \cite{Pas13}.
\end{proof}

From now until the end of this subsection, we assume that all the assumptions in Lemma \ref{5.25} are satisfied. 
We want to generalize the isomorphism $\widetilde{P}^{\psi}\isom (\pi_{\overline{\pi}}^{\psi})^*$ for $\widetilde{P}$ and $\pi_{\overline{\pi}}^*$ over $R_{\overline{\rho}}$. 
Let $\Lambda$ be the universal deformation ring for the deformations over $\mathrm{Comp}(\mathcal{O})$ of the trivial representation
$\bold{1} : G_{\mathbb{Q}_p}\rightarrow \mathbb{F}^{\times}$. Let $\bold{1}^{\mathrm{univ}} : G_{\mathbb{Q}_p}\rightarrow 
\Lambda^{\times}$ be the universal deformation of $\bold{1}$. Under the assumption that $p\geqq 3$, the natural map 
$$R_{\overline{\rho}}\rightarrow R_{\overline{\rho}}^{\psi\varepsilon^{-1}}\widehat{\otimes}_{\mathcal{O}}\Lambda:=\varprojlim_n (R_{\overline{\rho}}^{\psi\varepsilon^{-1}}/\mathfrak{m}_0^n)\otimes_{\mathcal{O}}(\Lambda/\mathfrak{m}_{\Lambda}^n)$$
corresponding to the map $D^{\psi\varepsilon^{-1}}_{\overline{\rho}}(A)\times D_{\bold{1}}(A)\rightarrow D_{\overline{\rho}}(A) : (\rho,\chi)\mapsto \rho(\chi)$ 
($A\in \mathrm{Comp}(\mathcal{O})$) is isomorphism, where $\mathfrak{m}_0$ (resp. $\mathfrak{m}_{\Lambda}$) is the maximal ideal 
of $R_{\overline{\rho}}^{\psi\varepsilon^{-1}}$ (resp. $\Lambda$). 
Therefore, one also has a topological isomorphism
$$\rho_{\overline{\rho}}\isom \rho_{\overline{\rho}}^{\psi\varepsilon^{-1}}\widehat{\otimes}_{\mathcal{O}}\bold{1}^{\mathrm{univ}}\quad \text{ (resp. } \quad \pi_{\overline{\pi}}\isom \pi_{\overline{\pi}}^{\psi}\widehat{\otimes}_{\mathcal{O}}(\bold{1}^{\mathrm{univ}}\circ \mathrm{det}))$$
of $R_{\overline{\rho}}[G_{\mathbb{Q}_p}]$-modules (resp. $R_{\overline{\rho}}[G_p]$-modules). Moreover, it is shown in Theorem 6.18 of \cite{CEG${}^{+}$18} that there exists an isomorphism 
$$\widetilde{P}\isom \widetilde{P}^{\psi}\widehat{\otimes}_{\mathcal{O}}(\bold{1}^{\mathrm{univ}}\circ \mathrm{det}^{-1})$$
in the category $\mathfrak{C}(\mathcal{O})$, and that the natural map 
$$R_{\overline{\rho}}\isom R_{\overline{\rho}}^{\psi\varepsilon^{-1}}\widehat{\otimes}_{\mathcal{O}}\Lambda\rightarrow \mathrm{End}_{\mathfrak{C}(\mathcal{O})}(\widetilde{P}^{\psi}\widehat{\otimes}_{\mathcal{O}}(\bold{1}^{\mathrm{univ}}\circ \mathrm{det}^{-1}))$$
is isomorphism, by which we equip $\widetilde{P}$ with a structure of compact $R_{\overline{\rho}}$-module. Therefore, one obtains a natural 
isomorphism 
$$R_{\overline{\rho}}\isom \mathrm{End}_{\mathfrak{C}(\mathcal{O})}(\widetilde{P})$$
of topological $\mathcal{O}$-algebras.



\begin{prop}\label{5.26}
There exists a $G_p$-equivariant $R_{\overline{\rho}}$-linear topological  isomorphism
$$\widetilde{P}\isom \pi_{\overline{\pi}}^*$$
\end{prop}
\begin{proof}
Since one has isomorphisms
$$\widetilde{P}\isom \widetilde{P}^{\psi}\widehat{\otimes}_{\mathcal{O}}(\bold{1}^{\mathrm{univ}}\circ \mathrm{det}^{-1}),$$
$$\pi_{\overline{\pi}}\isom \pi_{\overline{\pi}}^{\psi}\widehat{\otimes}_{\mathcal{O}}(\bold{1}^{\mathrm{univ}}\circ \mathrm{det})\quad \text{ and }
\quad (\pi_{\overline{\pi}}^{\psi})^*\isom \widetilde{P}^{\psi},$$ it suffices to check that there exists a $G_p$-equivariant $R_{\overline{\rho}}$-linear isomorphism 
$$ (\pi_{\overline{\pi}}^{\psi}\widehat{\otimes}_{\mathcal{O}}(\bold{1}^{\mathrm{univ}}\circ \mathrm{det}))^*\isom  (\pi_{\overline{\pi}}^{\psi})^*\widehat{\otimes}_{\mathcal{O}}(\bold{1}^{\mathrm{univ}}\circ \mathrm{det}^{-1}).$$
But this is trivial because $\pi_{\overline{\pi}}^{\psi}$ is orthonormalizable $R_{\overline{\rho}}$-module and 
$\bold{1}^{\mathrm{univ}}\circ \mathrm{det}$ is free $\Lambda$-module of rank one. 
\end{proof}

The following easy lemma (and the corollary) is one of the key results for our application. This is crucial for removing 
the $\widetilde{P}$-parts from the (Iwasawa cohomology of) the completed homology of modular curves. 

We remark that, for any object $N\in \mathfrak{C}(\mathcal{O})$, the isomorphism $R_{\overline{\rho}}\isom \mathrm{End}_{\mathfrak{C}(\mathcal{O})}(\widetilde{P})$ naturally equip $\mathrm{Hom}_{\mathfrak{C}(\mathcal{O})}(\widetilde{P}, N)$ with a structure of compact $R_{\overline{\rho}}$-module.
\begin{lemma}\label{5.27}
For any compact $R_{\overline{\rho}}$-module $M$,  
 the $R_{\overline{\rho}}$-linear map 
$$M\rightarrow \mathrm{Hom}_{\mathfrak{C}(\mathcal{O})}(\widetilde{P}, \widetilde{P}\,\widehat{\otimes}_{R_{\overline{\rho}}}M) : m\mapsto [v\mapsto v\,\widehat{\otimes}\,m]$$
is topological $R_{\overline{\rho}}$-linear isomorphism. 

\end{lemma}
\begin{proof}
Since any such $M$ can be written by $M=\varprojlim_{i\in I} M_i$ so that all the $M_i$ are  finite $R_{\overline{\rho}}$-modules (more strongly, we can take $M_i$ of finite length), 
it suffices to show the lemma when $M$ is a finite generated $R_{\overline{\rho}}$-module. 

Assume that $M$ is finite generated. Then one has $\widetilde{P}\,\widehat{\otimes}_{R_{\overline{\rho}}}M=\widetilde{P}\,\otimes_{R_{\overline{\rho}}}M$. We take a finite presentation of $M$ 
$$R_{\overline{\rho}}^m\rightarrow R_{\overline{\rho}}^n\rightarrow M\rightarrow 0.$$
Then, this induces an exact sequence 
$$\widetilde{P}^m\rightarrow \widetilde{P}^n\rightarrow \widetilde{P}\,\otimes_{R_{\overline{\rho}}}M\rightarrow 0$$
in $\mathfrak{C}(\mathcal{O})$, 
and induces an exact sequence 
$$\mathrm{Hom}_{\mathfrak{C}(\mathcal{O})}(\widetilde{P}, \widetilde{P}^m) \rightarrow \mathrm{Hom}_{\mathfrak{C}(\mathcal{O})}(\widetilde{P}, \widetilde{P}^n) 
\rightarrow \mathrm{Hom}_{\mathfrak{C}(\mathcal{O})}(\widetilde{P},\widetilde{P}\,\otimes_{R_{\overline{\rho}}}M) \rightarrow 0$$
of $R_{\overline{\rho}}$-modules since $\widetilde{P}$ is projective in $\mathfrak{C}(\mathcal{O})$. Since one has a natural isomorphism $$R_{\overline{\rho}}\isom \mathrm{Hom}_{\mathfrak{C}(\mathcal{O})}(\widetilde{P}, \widetilde{P}): 
a\mapsto a\cdot \mathrm{id}_{\widetilde{P}},$$ it is easy to see that the map $M\rightarrow \mathrm{Hom}_{\mathfrak{C}(\mathcal{O})}(\widetilde{P}, \widetilde{P}\widehat{\otimes}_{R_{\overline{\rho}}}M):m\mapsto [v\mapsto v\,\widehat{\otimes}\,m]$ is also isomorphism by the above exact sequences.

 \end{proof}
 By this lemma, we immediately obtain the following corollary.
 \begin{corollary}\label{5.28}
 For any compact $R_{\overline{\rho}}$-modules $M_1, M_2$, the natural map 
 $$\mathrm{Hom}_{R_{\overline{\rho}}}^{\mathrm{cont}}(M_1, M_2)\rightarrow 
 \mathrm{Hom}_{R_{\overline{\rho}}[G_p]}^{\mathrm{cont}}(\widetilde{P}\widehat{\otimes}_{R_{\overline{\rho}}}M_1, \widetilde{P}\widehat{\otimes}_{R_{\overline{\rho}}}M_2) : 
 f\mapsto \mathrm{id}_{\widetilde{P}}\otimes f$$ is 
 $R_{\overline{\rho}}$-linear isomorphism. 
 
 \end{corollary}

\subsection{Complements on Galois cohomologies}
 In the last subsection, we explain a general method which we use to factor out the Galois (and the Iwasawa) cohomology of $(\rho^{\mathfrak{m}})^*$ from that of 
 $\widetilde{H}^{BM}_{1, \overline{\rho}, \Sigma}$. Let $G_*$ be a topological group which is 
 equal to $G_{\mathbb{Q}, \Sigma}$ for some finite set of primes $\Sigma$ containing $p$, or $G_{\mathbb{Q}_l}$ for some prime $l$. 
\begin{lemma}\label{5.29}
Let $A\in\mathrm{Comp}(\mathcal{O})$, and $\rho:G_*\rightarrow \mathrm{GL}_n(A)$ be a continuous representation over $A$. 
For any finite generated $A$-module $M$, 
the natural map 
$$C^{\bullet}(G_*, \rho)\otimes_AM\rightarrow C^{\bullet}(G_*, \rho\otimes_AM)$$
is isomorphism.
\end{lemma}
\begin{proof}
If $M$ is a finite free $A$-module, it is trivial. In general, take a short exact sequence 
$$0\rightarrow N\rightarrow A^n\rightarrow  M\rightarrow 0$$
for some $n\geqq 0$ and a finite generated $A$-module $N$. 
Then, we obtain a short exact sequence 
$$0\rightarrow \rho\otimes_A N\rightarrow \rho\otimes_AA^n\rightarrow  \rho\otimes_AM\rightarrow 0 $$
of topological $A[G_*]$-modules with a continuous section $s : \rho\otimes_AM\rightarrow \rho\otimes_AA^n$ (as topological spaces). Hence we also obtain the following short exact sequence
\begin{equation}\label{B1}
0\rightarrow C^{\bullet}(G_*, \rho\otimes_A N) \rightarrow C^{\bullet}(G_*,\rho\otimes_AA^n) \rightarrow  C^{\bullet}(G_*,\rho\otimes_AM)\rightarrow 0
\end{equation} of $A$-modules from the standard theory of Galois cohomology. 
On the other hand, since the $A$-module $C^{\bullet}(G_*,\rho)$ is flat over $A$ by Theorem 1.1 (1) of \cite{Po13}, we obtain another short exact sequence
\begin{equation}\label{B2}
0\rightarrow C^{\bullet}(G_*,\rho)\otimes_AN\rightarrow C^{\bullet}(G_*,\rho)\otimes_A A^n\rightarrow  C^{\bullet}(G_*,\rho)\otimes_AM\rightarrow 0, 
\end{equation}
and a natural map from the short exact sequence (\ref{B1}) to (\ref{B2}). Then, that the natural map $C^{\bullet}(G_*, \rho)\otimes_AM\rightarrow  C^{\bullet}(G_*, \rho\otimes_AM)$ is isomorphism follows from the snake lemma (precisely, we first obtain surjectivity for any $M$ (hence for $N$), then the injectivity also follows from the snake lemma)

\end{proof}
\begin{lemma}\label{5.30}
Let $P$ be a pro-free $A$-module, $M$ a finite generated $A$-module. 
Then, the natural map 
$$C^{\bullet}(G_*, \rho\otimes_AP)\otimes_AM\rightarrow C^{\bullet}(G_*, \rho\otimes_AP\otimes_AM)$$
is isomorphism 
\end{lemma}
\begin{proof}
Since $M$ and $\rho$ are  finite generated $A$-modules, and $P$ can be written as  $P\isom\prod_{j}A$ (as topological $A$-module), there exist
 isomorphisms
$$\rho\otimes_AP\isom \prod_j\rho,\quad \rho\otimes_AP\otimes_AM\isom \prod_j(\rho\otimes_AM)$$
as topological $A[G_*]$-modules. Hence, we obtain the following 
isomorphisms 
\begin{multline*}
C^{\bullet}(G_*, \rho\otimes _AP)\otimes_AM\isom C^{\bullet}(G_*, \prod_j \rho)\otimes_AM\isom \left(\prod_jC^{\bullet}(G_*, \rho)\right)\otimes_AM\\
\isom  \prod_j\left(C^{\bullet}(G_*, \rho)\otimes_AM\right)\isom \prod_jC^{\bullet}(G_*, \rho\otimes_AM)
\isom C^{\bullet}(G_*, \prod_j (\rho\otimes_AM))\\\isom C^{\bullet}(G_*, \rho\otimes_AP\otimes_AM)\\
\end{multline*}
where the fourth isomorphism follows from the previous lemma. 
\end{proof}
\begin{lemma}\label{5.31}
Assume that the complex $C^{\bullet}(G_*, \rho)$ is a perfect complex of $A$-modules, i.e.  
there exists a complex $N^{\bullet}$ of $A$-modules consisting of finite projective $R$-modules such that $N^{\pm n}=0$ for any sufficiently large $n$, and there exists an $A$-linear quasi-isomorphism 
$$\phi : N^{\bullet}\rightarrow  C^{\bullet}(G_{*}, \rho).$$ 
Let $P$ be a pro-free $A$-module. Then $\phi$ naturally induces a quasi-isomorphism 
$$P\otimes_AN^{\bullet}\otimes_AM\rightarrow C^{\bullet}(G_{*}, P\otimes_A\rho\otimes_AM)$$ 
for every finite generated $A$-module $M$.

\end{lemma}
\begin{proof}
We first remark that since both $N^{\bullet}$ and $C^{\bullet}(G_*, \rho)$ are complexes of flat $A$-modules, the quasi-isomorphism $\phi$ induces 
a quasi-isomorphism 
$$N^{\bullet}\otimes_AM\rightarrow C^{\bullet}(G_*, \rho)\otimes_AM\isom C^{\bullet}(G_*, \rho\otimes_AM)$$
for any finite generated $A$-module $M$. 
Since $N^i$ ($i\in \mathbb{Z}$) and $\rho$ are finite generated $A$-modules, an isomorphism $P\isom \prod_jA$ induces  the following isomorphisms 
$$P\otimes_AN^{\bullet}\otimes_AM\isom \prod_j N^{\bullet}\otimes_AM, \quad P\otimes_{A}\rho\otimes_AM\isom \prod_j\rho\otimes_AM.$$
Taking the product of the quasi-isomorphism $N^{\bullet}\otimes_AM\rightarrow  C^{\bullet}(G_{*}, \rho\otimes_AM)$ above, we obtain the following quasi-isomorphism 
\begin{multline*}
P\otimes_AN^{\bullet}\otimes_AM\isom \prod_j N^{\bullet}\otimes_AM\rightarrow \prod_jC^{\bullet}(G_*, \rho\otimes_AM)\\
=C^{\bullet}(G_*, \prod_j\rho\otimes_AM )\isom C^{\bullet}(G_*, P\otimes_A\rho\otimes_AM).
\end{multline*}
 \end{proof}
 Let $\Gamma$ be $G_l$ or $G_{\Sigma_0}$ for a finite set of primes not containing $p$. 
 Let $V$ be a compact $A$-module with a continuous $A$-linear $G_*$-action, $M$ be a smooth 
 admissible $A[\Gamma]$-module. We define a complex of smooth $A[\Gamma]$-modules
 $$C^*(G*, V\otimes_AM):=\varinjlim_{K}C^{\bullet}(G_*, V\otimes_A M^K),$$
where $K$ run through all the open compact subgroup of $\Gamma$. Here, we remark that each $M^K$ is finite generated over $A$, hence $V\otimes_AM^K$ is also
a compact $A$-module with a continuous $A$-linear $G_*$-action.

By this definition, two lemma above immediately imply the following corollary.

\begin{corollary}\label{5.32}
Let $P$ be a pro-free $A$-module, $M$ a smooth admissible $A[\Gamma]$-module. 
\begin{itemize}
\item[(1)]The natural map 
$$C^{\bullet}(G_*, \rho\otimes_AP)\otimes_AM\rightarrow C^{\bullet}(G_*, \rho\otimes_AP\otimes_AM)$$
is isomorphism.
\item[(2)]We furthermore assume that there exists a perfect complex $N^{\bullet}$ of $A$-modules, and there exists an $A$-linear quasi-isomorphism 
$$\phi : N^{\bullet}\rightarrow  C^{\bullet}(G_{*}, \rho),$$ then it naturally induces a quasi-isomorphism 
$$P\otimes_AN^{\bullet}\otimes_AM\rightarrow C^{\bullet}(G_{*}, P\otimes_A\rho\otimes_AM)$$
of smooth $A[\Gamma]$-modules. 

\end{itemize}

\end{corollary}

\footnote{Kentaro Nakamura

 Mathematical Science Course, Department of Science and Engineering, Saga University, 
 
 1 Honjo-machi, Saga, 840-8502, Japan
 
 nkentaro@cc.saga-u.ac.jp}

\end{document}